\documentclass[a4paper,10pt,reqno]{amsart}

\usepackage{xcolor,calc, blindtext}
\usepackage{amsmath,amssymb,mathtools,amsthm,mathrsfs}
\numberwithin{equation}{section}

\usepackage{YDLthesissymbols}

\usepackage[colorlinks]{hyperref}

\usepackage{setspace}
\usepackage[margin=1.2in]{geometry} 
\usepackage{graphicx}     
 
\usepackage{fancyhdr}  
\usepackage{appendix}
\usepackage{subfiles}
\usepackage{
bbm}
\usepackage{tikz-cd,hf-tikz}
\usepackage[utf8]{inputenc}
\usepackage[T1]{fontenc}
\usepackage[normalem]{ulem}
\usepackage{indentfirst} 
\usepackage{enumerate} 
\usepackage{enumitem} 
\usepackage{lscape}
\usepackage{etoolbox}
\usepackage{stmaryrd}
\usepackage{tabularx}
\usepackage{import}
\usepackage{multirow}

\usepackage[numbers]{natbib}
\usepackage[french,english]{babel}

\allowdisplaybreaks

\begin{document}
\title{Approche non-invariante de la correspondance de Jacquet-Langlands : analyse géométrique}
\author{Yan-Der LU}
\date{\today
}
{\selectlanguage{english}
\begin{abstract}
In this two-part series of articles, we present a new proof comparing the trace formula for a general linear group with that of one of its inner forms. Our methodology relies on the trace formula for Lie algebras, incorporating the notion of non-invariant transfer of test functions. In the appendix A, we provide a description of conjugacy classes of an inner form of a general linear group. In the appendix B, we provide explicit computations of Haar measures. This article focuses on the geometric side of the trace formula.
\end{abstract}}

\selectlanguage{french}
\maketitle

\tableofcontents

\section{Introduction}

\subsection{Préface}
La formule des traces est un des outils les plus puissants pour aborder les problèmes de fonctorialité, dont la correspondance de Jacquet-Langlands. Arthur a développé la formule des traces pour tout groupe réductif sur un corps de nombres, dont l'une des constructions cruciales est la déduction de la formule des traces invariante à partir de sa version non-invariante. En collaboration avec Clozel, il a  réussi à comparer la formule des traces invariante pour $\GL_n$ et celle pour une de ses formes intérieures (\cite{AC}). C'est en exploitant leurs résultats que Badulescu et Renard, en l'année 2010, ont pu établir la correspondance de Jacquet-Langlands  (\cite{Badu09}).

On s'interroge sur l'emploi de la formule des traces invariante ici : pourrait-elle être substituée par celle non-invariante ? La question n'est pas insignifiante, primo, cela évitera toute analyse harmonique difficile dans la déduction de la formule des traces invariante. Secundo, apparaissent de part et d'autre de la formule des traces non-invariante des objets purement géométriques ou purement spectraux, qui sont d'une certaine perspective plus naturels. La notion du transfert « non-invariante » d'une fonction test est pour la première fois abordée par Labesse (\cite{Lab95}) dans le cas du changement de base cyclique pour un groupe général linéaire, après généralisée par Chaudouard (\cite{Ch07}) dans le cas de l'espace tordu, au sens de Labesse, sous des certaines hypothèses supplémentaires. C'est ce qui nous motive à tenter une nouvelle approche de la correspondance de Jacquet-Langlands. 

\subsection{Résultats principaux}

Fournissons maintenant un bref aperçu de ce texte. Soit $E$ un corps de caractéristique 0. On prend $\overline{E}$ une clôture algébrique de $E$. Pour $H$ un groupe algébrique sur $E$, et $R$ une $E$-algèbre, on note $H_R$ le changement de base $H\times_F R$. Soient $\g^\ast=\mathfrak{gl}_{n,E}$ l'algèbre des matrices de taille $n$ sur $E$, et $\g$ une algèbre simple centrale de rang $n$ sur $E$. Il existe alors $\eta:\g_{\overline{E}}\to \g_{\overline{E}}^\ast$ un $\overline{E}$-isomorphisme d'algèbres associatives. On note $G$ (resp. $G^\ast$) le groupe des unités de $\g$ (resp. $\g^\ast$). L'algèbre de Lie de $G$ (resp. $G^\ast$) s'identifie naturellement à $\g$ (resp. $\g^\ast$). Pour un groupe algébrique $H$ on note par la même lettre en minuscule gothique, en l'occurrence $\mathfrak{h}$, son algèbre de Lie. On observe, au moyen du théorème de Skolem–Noether, que $\eta$ est automatiquement un torseur intérieur, i.e. pour tout $\sigma\in\Gal(\overline{E}/E)$ on ait que $\eta\sigma(\eta)^{-1}$ est l'automorphisme de $\g_{\overline{E}}^\ast$ défini par conjugaison par un élément de $G_{\overline{E}}^\ast$. Ici $\sigma(\eta):X\in \g_{\overline{E}}\mapsto \sigma(\eta(\sigma^{-1}(X)))\in \g_{\overline{E}}^\ast$ avec $\Gal(\overline{E}/E)$ agit sur $\g_{\overline{E}}=\g\otimes_E\overline{E}$ et $\g_{\overline{E}}^\ast=\g^\ast\otimes_E\overline{E}$ via la composante $\overline{E}$ du produit tensoriel.



Par un sous-groupe de Levi d'un groupe réductif on entend le facteur de Levi d'un sous-groupe parabolique. Fixons $M_0$ (resp. $M_{0^\ast}$) un sous-groupe de Levi minimal de $G$ (resp. $G^\ast$) défini sur $E$. On suppose que $M_{0^\ast,\overline{F}}\subseteq \eta(M_{0,\overline{F}})$. Nous pouvons définir la notion de transfert d'objets liés à $G$ et $G^\ast$ comme suit.

\begin{definition}~{} 
\begin{enumerate}
    \item Soit $L'$ (resp. $L$) un sous-groupe de Levi de $G^\ast$ (resp. $G$) défini sur $E$. Soit $Q'$ (resp. $L$) un sous-groupe parabolique de $G^\ast$ (resp. $G$) défini sur $E$ contenant $L'$ (resp. $L$). On dit que $(L',Q')$ se transfère en $(L,Q)$ si $\eta(L_{\overline{E}})=L_{\overline{E}}'$ et $\eta(Q_{\overline{E}})=Q_{\overline{E}}'$. On écrit $L'=L^\ast$ et $L\arr L^\ast$. Puis $Q'=Q^\ast$ et $Q\arr Q^\ast$. Un sous-groupe de Levi/parabolique de $G^\ast$ est dit ne pas se transférer s'il n'est pas une composante d'un couple  défini sur $E$ qui se transfère comme plus haut.
    \item Soit $L^\ast$ (resp. $L$) un sous-groupe de Levi de $G^\ast$ (resp. $G$) défini sur $E$, avec $L\arr L^\ast$. Soit $\mathfrak{o}'$ (resp. $\mathfrak{o}$) une classe de $L^\ast(E)$-conjugaison (resp. une classe de $L(E)$-conjugaison) dans $\mathfrak{l}^\ast(E)$ (resp. $\mathfrak{l}(E)$). On dit que $\mathfrak{o}'$ se transfère en $\mathfrak{o}$ si $\eta((\Ad G(\overline{E}))\mathfrak{o})=(\Ad G(\overline{E}))\mathfrak{o}'$ avec $\Ad$ l'action de conjugaison d'un groupe sur son algèbre de Lie. On écrit $\o'=\o^\ast$ et $\o\arr \o^\ast$. Une classe de $L^\ast(E)$-conjugaison dans $\mathfrak{l}^\ast(E)$ est dit ne pas se transférer si elle ne rencontre aucune image par $\eta$ d'une classe de $L(E)$-conjugaison dans $\mathfrak{l}(E)$.
\end{enumerate}    
\end{definition}

On fixe maintenant $F$ un corps de nombres. Soient $\g^\ast=\mathfrak{gl}_{n,F}$ et $\g$ une algèbre simple centrale de rang $n$ sur $F$ déployée aux places archimédiennes. Pour toute place $v$ de $F$, on fixe un plongement $\overline{F}\to\overline{F_v}$. Pour $H$ un groupe algébrique sur $F$, on abrège $H_{F_v}$ en $H_v$. Soit $\eta:\g_{\overline{F}}\to \g_{\overline{F}}^\ast$ un torseur intérieur. Pour toute place $v$ de $F$, on obtient $\eta_v:\g_{\overline{F_v}}\to \g_{\overline{F_v}}^\ast$ un torseur intérieur par l'extension des scalaires. Si $L'=L^\ast$ (resp. $L_v=L_v^\ast$) est un sous-groupe de Levi de $G^\ast$ (resp. $G_v^\ast$) qui se transfère, alors la restriction de $\eta$ (resp. $\eta_v$) définit un torseur intérieur $\eta|_{\l}:\l_{\overline{F}}\to \l_{\overline{F}}^\ast$ (resp. $\eta_{v,\l_v}:\l_{v,\overline{F_v}}\to \l_{v,\overline{F_v}}^\ast$).

Soit $S$ un ensemble fini de places de $F$. On note $F_S=\prod_{v\in S}F_v$ avec $F_v$ le complété local de $
F$ en $v$. Soient $L$ un sous-groupe de Levi de $G$, $Q\supseteq L$ un sous-groupe parabolique, $\o_S$ une classe de $L(F_S)$-conjugaison dans $\mathfrak{l}(F_S)$, et $f_S$ une fonction définie sur $\g(F_S)$. On note $J_L^Q(\o_S,f_S)$ l'intégrale orbitale pondérée semi-locale associée (numéro \ref{subsubsec:defIOP}).

Nous pouvons définir la notion de transfert non-invariant local des fonctions tests en utilisant des conditions de correspondance ou des conditions d'annulation sur les intégrales orbitales locales pondérées associées aux éléments semi-simples réguliers. Pour $V$ un espace vectoriel sur $F$, et $S$ un ensemble fini de places de $F$, on note $\S(V(F_S))$ (resp. $\S(V(\A_F))$) l'espace de Schwartz-Bruhat sur $V(F_S)$ (resp. $V(\A_F)$). Enfin les mesures sont normalisées comme au numéro \ref{subsec:normalisationmesuresfinales}.

\begin{theorem}[Chaudouard \cite{Ch07}, théorème \ref{prop:deftransfertfon}] Soit $v$ une place non-archimédienne de $F$. Pour toute fonction $f_v\in\S(\g(F_v))$ il existe une fonction $f_v^\ast\in\S(\g^\ast(F_v))$ telle que pour tout $L_v'$ sous-groupe de Levi de $G_v^\ast$, tout $Q_v'\supseteq L_v'$ sous-groupe parabolique, toute $\o_v'$ classe de $L_v'(F_v)$-conjugaison dans $\mathfrak{l}_v'(F_v)$ semi-simple et $G_v^\ast$-régulier, on ait
\begin{align*}
    J_{L_v'}^{Q_v'}(\mathfrak{o}_v',f_v^\ast)=\begin{cases*}
    J_{L_v}^{Q_v}(\mathfrak{o}_v,f_v) & \text{si $(L_v',Q_v',\mathfrak{o}_v')= (L_v^\ast,Q_v^\ast,\mathfrak{o}_v^\ast)$ ;} \\
    0 & \text{si $Q_v'$ ne se transfère pas.}
    \end{cases*}
    \end{align*} 
\end{theorem}

Pour $f_v\in \S(\g(F_v))$ et $f_v^\ast \in \S(\g^\ast(F_v))$ deux fonctions vérifiant le théorème, on note $\S
(\g(F_v))\ni f_v\arr f_v^\ast \in \S(\g^\ast(F_v))$. On étend la comparaison ci-dessus :

\begin{theorem}[théorème \ref{pro:corrloc}] Soit $v$ une place non-archimédienne de $F$. Soit $\S(\g(F_v))\ni f_v\arr f_v^\ast \in \S(\g^\ast(F_v))$. Alors pour tout $L_v'$ sous-groupe de Levi de $G_v^\ast$, tout $Q_v'\supseteq L_v'$ sous-groupe parabolique, toute $\o_v'$ classe de $L_v'(F_v)$-conjugaison dans $\mathfrak{l}_v'(F_v)$, on ait 
\begin{align*}
    J_{L_v'}^{Q_v'}(\mathfrak{o}_v',f_v^\ast)=\begin{cases*}
    J_{L_v}^{Q_v}(\mathfrak{o}_v,f_v) & \text{si $(L_v',Q_v',\mathfrak{o}_v')=(L_v^\ast,Q_v^\ast,\mathfrak{o}_v^\ast)$  ($\mathfrak{o}_v'=\mathfrak{o}_v^\ast$ via $\eta|_{v,\l_v}$) ;} \\
    0 & \text{si $Q_v'$ ne se transfère pas ou si $L_v'$ se transfère et} \\
     & \text{$\o_v'$ ne se transfère pas (via $\eta|_{v,\l_v}$).}
\end{cases*}
    \end{align*}
\end{theorem}

Nous étendons sans difficulté cette notion de transfert des fonctions tests locales pour englober les fonctions tests semi-locales (resp. globales) de type tenseur pur, notées $\S(\g(F_S))\ni f_S\arr f_S^\ast\in \S(\g(F_S))$ (resp. $\S(\g(\A_F))\ni f\arr f^\ast\in \S(\g(\A_F))$) (définition \ref{YDLgeomdef:transfertd'unefonctionglobale}).

En parallèle, nous établissons le développement fin du côté géométrique « raffiné » de la formule des traces non invariante. Pour tout groupe réductif $H$ sur $F$, nous notons $\O^{\mathfrak{h}}$ l'ensemble des classes de $H(F)$-conjugaison dans $\mathfrak{h}(F)$. \`{A} la différence de l'approche d'Arthur, où le côté géométrique est décomposé selon la partition de $\g(F)$ par les classes de « conjugaison semi-simples », notre approche va plus loin en décomposant le côté géométrique selon la partition de $\g(F)$ par les classes de $G(F)$-conjugaison usuelles. Cette approche est rendue possible grâce aux travaux de Chaudouard sur les algèbres de Lie \cite{Ch18} (ou grâce aux travaux de Finis-Lapid sur les groupes \cite{FiLa16}). Pour ce faire, nous nous appuyons sur la notion de l'induite de Lusztig-Spaltenstein généralisée (proposition \ref{prop:indprop}). Il s'agit d'une application à fibres finies $\Ind_M^G:\O^{\m}\to\O^{\g}$, où $M$ est un sous-groupe de Levi de $G$. On note dans la suite $J^\g$ la formule des traces pour l'algèbre de Lie $\g$, et $J_{\o}^\g$ la contribution due à $\o\in\O^\g$ (sous-section \ref{subsec:preludeTFdef}). On a $J^\g=\sum_{\o\in\O^\g}J_\o^\g$.

Pour $S$ un ensemble fini de places contenant les places archimédiennes, notons $\O^S$ le produit des anneaux des entiers des $F_v$ avec $v\not\in S$.

\begin{theorem}[théorème \ref{thm:devfin}] Pour tout $M$ sous-groupe de Levi d'une forme intérieure d'un groupe général linéaire, tout $S$ un ensemble fini de places assez grand, et tout $\o_M\in \O^{\m}$, il existe un unique nombre complexe $a^{M}(S,\o_M)$ vérifiant l'assertion suivante : pour tout $L$ sous-groupe de Levi d'une forme intérieure d'un groupe général linéaire, tout $\o\in \O^\l$, il existe $S_\o$ un ensemble fini de places de $F$, contenant les places archimédiennes, tel que pour tous $S\supseteq S_\o$ et $f_S\in\S(\l(F_S))$, on ait
\[J_\o^\l(f_S\otimes 1_{\l(\O^S)})=\sum_{M\in\L^L(M_0)}|W_0^M||W_0^L|^{-1}\sum_{\o_M\in \O^{\m} : \Ind_M^L(\o_M)=\o}a^M(S,\o_M)J_M^L(\o_M,f_S).\]
\end{theorem}

Nous analysons ensuite toutes les intégrales orbitales pondérées impliquées dans la formule des traces pour les algèbres de Lie. 

\begin{proposition}[proposition \ref{prop:semi-localIOPcomp}]
Il existe $\underline{S}$ un sous-ensemble fini de places, contenant les places archimédiennes, tel que pour tous $S\supseteq \underline{S}$ fini, $\S(\g(F_S))\ni f_S\underset{}{\arr} f_S^\ast \in \S(\g^\ast(F_S))$ de type tenseur pur, $L'$ un sous-groupe de Levi de $G^\ast$, $Q'\supseteq L'$ un sous-groupe parabolique, et $\o'\in \O^{\mathfrak{l}'}$, on ait
\begin{align*}
    J_{L'}^{Q'}(\mathfrak{o}',f_S^\ast)=\begin{cases*}
    J_L^Q(\mathfrak{o},f_S) & \text{si $(L',Q',\mathfrak{o}')= (L^\ast,Q^\ast,\mathfrak{o}^\ast)$ ;} \\
    0 & \text{sinon.}
    \end{cases*}
    \end{align*}
\end{proposition}

Puis nous comparons les coefficients $a^M(S,\o_M)$ dans les développements fins. 

\begin{theorem}[théorème \ref{thm:identificcationcoeffaG(S,X)transfert}]\label{Introthm:identificcationcoeffaG(S,X)transfert}Soit $\O^\g\ni\o\arr\o^\ast\in\O^{\g^\ast}$. Il existe $S_{\o}'$ un ensemble fini de places de $F$, contenant les places archimédiennes, tel que pour tout $S\supseteq S_{\o}'$ fini on ait
\begin{equation}
a^G(S,\o)=a^{G^\ast}(S,\o^\ast).    
\end{equation}
\end{theorem}

On aboutit alors à la comparaison des côtés géométriques de la formule des traces. Notons $\O^{G}$ l'ensemble des classes de $G(F)$-conjugaison dans $G(F)$. Notons $J_{\geom}^G$ le côté géométrique de la formule des traces non-invariante pour le groupe $G$, et $J_{\o}^G$ la contribution due à $\o\in\O^G$ (sous-section \ref{subsec:comparaisonTFgroupLiealg}). On a $J_{\geom}^G=\sum_{\o\in\O^G}J_\o^G$.

\begin{theorem}[théorème \ref{thm:finalres'} pour (1) ; proposition \ref{prop:Lie=Grp} et théorème \ref{thm:finalres} pour (2)] Soit $\S(\g(\A_F))\ni f\arr f^\ast \in \S(\g^\ast(\A_F))$ de type tenseur pur. 
\begin{enumerate}
    \item On a \begin{align*}
    J_{\o'}^{\g^\ast}(f^\ast)=\begin{cases*}
    J_\o^\g(f) & \text{si $\o'=\o^\ast$ ;} \\
    0 & \text{si $\o'$ ne se transfère pas.}
    \end{cases*}
    \end{align*}
    \item L'inclusion $G\hookrightarrow \g$ donne une injection naturelle $\O^{G}\hookrightarrow \O^\g$. On a $J_{\o}^G=J_{\o}^\g$ pour tout $\o\in \O^G$. Pour $\o\in \O^G$ et $\o^\ast\in \O^{\g^\ast}$, si $\eta(\mathfrak{o})$ intersecte $\mathfrak{o}^\ast$ de manière non-triviale alors $\o^\ast\in \O^{G^\ast}$. Donc $J_{\geom}^G(f)=J_{\geom}^{G^\ast}(f^\ast)$. Cette dernière égalité est valable pour les normalisations expliquées dans la remarque \ref{rem:finalremarkonmeasures}.   
\end{enumerate}    
\end{theorem}


\subsection{Plan de l'article}

Le contenu de l'article est organisé dans cet ordre : en section \ref{sec:preliminairesgeneraux} on introduit des notations fondamentales et on explique le contexte de la correspondance de Jacquet-Langlands, le cadre global et celui local traités de manière concomitante. Les sections \ref{sec:preliminaireslocaux} et  \ref{sec:analyselocale} sont dédiées à l'étude locale. En section \ref{sec:preliminaireslocaux} on introduit  des notations nécessaires puis en section \ref{sec:analyselocale} on définit la notion de transfert non-invariant d'une fonction test et on procède à la comparaison des intégrales orbitales pondérées locales. \`{A} compter de la section \ref{sec:formuledestracesI},  notre attention se tourne vers l'étude globale. Au sein des sections \ref{sec:formuledestracesI}, \ref{sec:formuledestracesII} et \ref{sec:formuledestracesIII}, se révèlent des différentes formules des traces indispensables à notre approche de la correspondance de Jacquet-Langlands. Puis dans la section \ref{sec:compatibilitélocal-global}, nous réduisons l'analyse des termes du côté géometrique, hormis les coefficients $a^M(S,\o_M)$, à une analyse locale. En dernier lieu dans la section \ref{sec:analyseglobale}, on compare lesdits coefficients $a^M(S,\o_M)$ et on termine la preuve de la comparaison des côtés géométriques de la formules des traces non-invariantes. Dans l'annexe \ref{sec:AppendixA} on décrit l'ensemble des classes de conjugaison d'une algèbre séparable sur un corps parfait via la théorie des diviseurs élémentaires, et on dépeint son lien avec l'induction de Lusztig-Spaltenstein généralisée. Enfin on explique le tranfert des classes de conjugaison d'une algèbre séparable aux celles de sa forme intérieure quasi-déployée. Dans l'annexe \ref{sec:AppendixB} on fournit des calculs explicites des measures de Haar. Nous avons choisi d'ajouter ces sections en annexes, car elles présentent un intérêt intrinsèque.

\subsection*{Remerciements}
L'auteur tient particulièrement à remercier Pierre-Henri Chaudouard, son directeur de thèse. Sans ses connaissances, conseils, temps et patience, cet article n'aurait pas été possible. L'auteur tient également à remercier l'école doctorale 386 Sciences Mathématiques de Paris Centre ainsi que l'université Paris-Cité pour avoir financé ce projet de thèse.

\section{Préliminaires généraux}\label{sec:preliminairesgeneraux}

Dans cette partie $F$ désigne un corps global ou local de caractéristique 0.

\subsection{Notations}
\subsubsection{Objets relativement à \texorpdfstring{$F$}{F}}
On fixe $\overline{F}$ une clôture algébrique de $F$, et on note par $\Gamma\eqdef\Gal(\overline{F}/F)$ le groupe de Galois absolu de $F$. Désignons par $\V_F$ (resp. $\V_{\fin}$ ; $\V_\infty$) l’ensemble des places (resp. places finies ; places archimédiennes) de $F$ lorsque ce dernier est un corps global. Si $F$ est un corps local, soit $\V_F$ le singleton où l'unique élément est la place dont $F$ est munie.

\subsubsection{Objets relativement à \texorpdfstring{$G$}{G}}
Soit $G$ un groupe algébrique connexe défini sur $F$. On note par la même lettre en minuscule gothique son algèbre de Lie, donc ici $\g=\Lie(G)$. L’action adjointe, appelée également la conjugaison, de $G$ (resp. $\g$) sur $\g$ est notée Ad (resp. ad). On note $G_\ad$ le groupe adjoint de $G$, à savoir $G_\ad=G/Z(G)$ avec $Z(G)$ le centre de $G$. On note $G_\der$ le sous-groupe dérivé de $G$. On note $A_G$ le sous-tore central $F$-déployé maximal dans $G$. On note $X^\ast(G)$ le groupe des caractères de $G$ définis sur $F$ et $a_G$ (resp. $a_G^\ast$) l’espace vectoriel réel $\Hom(X^\ast(G),\R)$ (resp. $X^\ast(G)\otimes_\Z\R$). Les espaces $a_G$ et $a_G^\ast$ sont canoniquement duaux l'un de l'autre. 

Sauf mention contraire, un sous-groupe de $G$ signifie un sous-groupe algébrique de $G$ défini sur $F$. Soit $H$ un sous-groupe de $G$ contenant $Z(G)$, on écrit $H_\AD$ pour le sous-groupe $H/Z(G)$ de $G_\ad$.

\'{E}crivons $\g=\mathfrak{z}\oplus \g_\ad$ avec $\mathfrak{z}=\Lie(Z(G))$ et $\g_\ad=[\g,\g]$. Remarquons que $\g_\ad$ est à la fois l'algèbre de Lie de $G_\der$ et celle de $G_\ad$, vu que $G_\der\twoheadrightarrow G_\ad$ est une isogénie.

Pour $S$ un sous-ensemble de $\g
$ ou de $G$, on note $\Cent(S,G)$ le centralisateur de $S$ dans $G$. Pour tout $X\in\g$, on note $H_X=\Cent(X,H)$, resp. $\g_X$, le stabilisateur de $X$ sous l’action d’un sous-groupe fermé $H$ de $G$, resp. sous l’action par adjonction de $\g$. Le sous-groupe $G_X$ sera toujours connexe lorsque $G$ est un groupe du type GL (définition \ref{def:groupetypeGL}).

On note $\g_\ss$ (resp. $\g_\nilp$) l'ensemble des éléments semi-simples (resp. nilpotents) de $\g$. Pour tout $X \in \g$, on dispose de la décomposition de Jordan $X = X_{\ss} + X_{\nilp}$ où $X_{\ss}$ et $X_{\nilp}$ sont respectivement des éléments de $\g_\ss$ et de $\g_\nilp$, qui commutent. Un élément $X$ de $\g(F)$ est dit ($F$-)elliptique s'il est semi-simple et $A_G=A_{G_X}$. Une classe de $G(F)$-conjugaison dans $\g(F)$ est dite elliptique si un (donc tout) élément dedans est elliptique. On note $\g_{\rss}$ l’ensemble des éléments semi-simples fortement réguliers de $\g$, c’est-à-dire les éléments semi-simples $X$ tels que $G_X$ soit un sous-tore maximal de $G$. On pose
\[D^{\g}(X)=\det(\text{ad}(X_\ss); \g/\g_{X_{\ss}}).\]

Pour toute $F$-algèbre $R$ on note $G_R$ le changement de base $G\times_F R$. Si $F$ est un corps global et $v$ et une place de $F$, on note $F_v$ le complété local de $F$ en $v$. On abrège aussi le changement de base $G_{F_v}$ en $G_v$.

Soit désormais $G$ un groupe réductif connexe sur $F$. On fixe $P_0$ un sous-groupe parabolique minimal de $G$, puis $M_0$ une composante de Levi de $P_0$.

On appelle sous-groupe de Levi semi-standard de $G$ un groupe $M$ contenant $M_0$ et qui est une composante de Levi d'un sous-groupe parabolique de $G$. On rappelle qu'avec nos conventions, tous ces sous-groupes sont définis sur $F$. Tout sous-groupe de Levi d'un groupe réductif connexe est aussi réductif connexe. Soient $ M \subseteq H$ des sous-groupes de $G$. On note $\L^H(M)$ l’ensemble des sous-groupes de Levi de $G$ inclus dans $H$ et contenant $M$ ; on note $\P^H(M)$ l’ensemble des sous-groupes paraboliques de $G$ inclus dans $H$, dont $M$ est un facteur de Levi ; on note $\F^H(M)$ l’ensemble des sous-groupes paraboliques de $G$ inclus dans $H$ et contenant $M$. Lorque $M = M_0$, on peut omettre $(M_0)$ dans la notation, par exemple $\L^H=\L^H(M_0)$.

Soit $M\in\L^G$ un sous-groupe de Levi semi-standard. Le groupe $A_M$ est la partie déployée du groupe de type multiplicatif $\Cent(M,G)=Z(M)$. \'{E}galement $M=\Cent(A_M,G)$. 

On définit le groupe de Weyl relatif de $(G,M)$ par
\[W^G(M)=W_M^G\eqdef\text{Norm}_{G(F)}(A_M)/M(F) =\text{Norm}_{G(F)}(M)/M(F),\]
avec $\text{Norm}_{G(F)}(A_M)$ (resp. $\text{Norm}_{G(F)}(M)$) le normalisateur de $A_M$ (resp. $M$) dans $G(F)$. Le groupe $W^G(M)$ agit par conjugaison sur $\L^G(M)$, $\P^G(M)$, et $\F^G(M)$.  On identifie le stabilisateur de $M\in \L^G$ dans $W_0^G=W_{M_0}^G$, quotienté par $W_0^M$, au groupe de Weyl de $(G,M)$.

\subsubsection{Objets relativement à un sous-groupe de Levi \texorpdfstring{$M$}{M} ou à un sous-groupe parabolique \texorpdfstring{$P$}{P}}

Soit $P\in\F^G(M)$. Notons $M_P$ l'unique facteur de Levi de $P$ contenant $M$. On note $N_P$ le radical unipotent de $P$. En particulier, $P = M_PN_P$ est une décomposition de Levi de $P$. \'{E}crivons $\overline{P}$ pour le sous-groupe parabolique opposé à $P$. 

On a $A_P= A_{M_P}$. Par restriction il y des isomorphismes canoniques de groupes $X^\ast(P)\simeq X^\ast(M_P)$, $a_P\simeq a_{M_P}$, et $a_{P}^\ast\simeq a_{M_P}^\ast$. 

Soit $H$ un sous-groupe de $G$ stable par conjugaison par $A_M$. On note $\Sigma(\mathfrak h; A_M)\subseteq a_M^\ast$ l’ensemble des racines de $A_M$ sur $\mathfrak{h}$. On a alors
\[\mathfrak{h}=\mathfrak{h}^{A_M}\oplus\bigoplus_{\alpha\in\Sigma(\mathfrak h;A_M)}\mathfrak{h}_\alpha\]
avec $\mathfrak{h}^{A_M}$ le sous-espace invariant et $\mathfrak{h}_\alpha$ le sous-espace propre relativement à $\alpha$ de $\mathfrak{h}$ sous l'action adjointe par $A_M$ :
\[\mathfrak{h}_\alpha=\{X\in\mathfrak{h}\mid \Ad(a)X=\alpha(a)X\,\,\,\,\forall a\in A_M\}.\]
Toutes les racines seront non-divisibles lorsque $G$ est un groupe de type GL (définition \ref{def:groupetypeGL}).

Soit de plus $Q\in\F^G(M)$ tel que $P\subseteq Q$. On note $\Delta_P^Q$ l’ensemble des racines simples dans $\Sigma(\m_Q\cap \n_P; A_{M_P})$, et $\rho_P^Q\eqdef \frac{1}{2}\sum_{\alpha\in\Sigma (\m_Q\cap \n_P; A_{M_P})}\left(\dim (\mathfrak{m}_Q\cap \mathfrak{n}_P)_\alpha\right)\alpha$. 

Les ensembles $\Sigma(\g; A_{M_P})$ et $\Sigma(\mathfrak{p}; A_{M_P})$ sont canoniquement inclus dans $a_{M_P}^\ast$. \`{A} chaque racine $\beta\in \Sigma(\g; A_{M})$ est associée une coracine $\beta^\vee\in a_{M}$. On définit alors dans $a_{M_P}$ le sous-ensemble $(\Delta_P^Q)^\vee$ des coracines des éléments de $\Delta_P^Q$.  Lorsque $Q = G$, on note simplement $\Delta_P = \Delta_P^G$, $\rho_P=\rho_P^G$, et $\Delta_P^\vee=(\Delta_P^G)^\vee$.  On a $\Delta_P^Q\subseteq \Delta_P$.

Soit $a_P^Q$ (resp. $(a_p^Q)^\ast$) le sous-espace vectoriel engendré par $(\Delta_P^Q)^\vee$ (resp. $\Delta_P^Q$). Notons que $a_{Q}$ (resp. $a_{Q}^\ast$) s’identifie canoniquement à un sous-espace de $a_{P}$ (resp. $a_{P}^\ast$).
On sait que 
\[a_P=a_P^Q\oplus a_Q\,\,\,\,\text{et}\,\,\,\,a_P^\ast=(a_P^Q)^\ast\oplus a_Q^\ast.\]
On note $a_{M_P}^{M_Q}=a_P^Q$ et $(a_{M_P}^{M_Q})^\ast=(a_P^Q)^\ast$.

\subsubsection{\texorpdfstring{$(G,M)$}{(G,M)}-familles}\label{subsubsec:(G,M)-familles} Le numéro présent résume \cite[sections 6, 7 et 11]{
Art81} et \cite[sections 7, 8 et 9]{Art88I}.

Pour tout espace vectoriel $V$ sur $\R$, on note $V_{\C}\eqdef V\otimes_{\R}\C$ sa complexification, et $iV\subseteq V_{\C}$ le sous-$\R$-espace évident où $i\in\C$ est l'unité imaginaire.

Soit $M$ un sous-groupe de Levi de $G$. Une $(G, M)$-famille est un ensemble d'applications lisses $\{c_P(\lambda)\mid P\in\P^G(M)\}$ de $\lambda\in ia_M^\ast$ vérifiant la propriété que si $P$ et $Q$ sont adjacents (i.e. $\Sigma(\mathfrak{p}; A_{M_0})\cap(-\Sigma(\mathfrak{q}; A_{M_0}))$ est un singleton) et si $\lambda$ appartient à l'hyperplan engendré par le mur en commun des chambres de $P$ et $Q$ dans $ia_M^\ast$ alors $c_P(\lambda) = c_{Q}(\lambda)$. 

Pour tout $P\in \P^G(M)$ on définit une fonction polynomiale homogène
\[\theta_P(\lambda)=\vol\left(a_M^G/\Z(\Delta_P^\vee)\right)^{-1}\cdot\prod_{\alpha\in \Delta_P}\lambda(\alpha^\vee),\,\,\,\,\lambda\in ia_M^\ast,\]
avec $\Z(\Delta_P^\vee)$ le réseau engrendré par $\Delta_P^\vee$ dans $a_M^G$ et $\vol\left(a_M^G/\Z(\Delta_P^\vee)\right)$ le volume de $a_M^G/\Z(\Delta_P^\vee)$. Le choix de la mesure sur $a_M^G/\Z(\Delta_P^\vee)$ sera précisé ultérieurement. Pour toute $(G,M)$-famille $(c_P)_{P\in\P^G(M)}$, la fonction
\[c_M(\lambda)=\sum_{P\in \P^G(M)}c_P(\lambda)\theta_P(\lambda)^{-1} \]
s'étend en une fonction lisse sur $ia_M^\ast$. On note $c_M$ la valeur $c_M(0)$.

Soit $(c_P)_{P\in\P^G(M)}$ une $(G,M)$-famille. D'une part pour tout $L \in \L^G(M)$, on peut définir une $(G, L)$-famille $(c_R)_{R\in \P^G(L)}$ par
\[c_R(\lambda)=c_P(\lambda),\,\,\,\,\lambda\in ia_L^\ast\]
où $P$ est un élément de $\P^G(M)$ qui vérifie $P\subseteq R$. Si l’on fixe d'autre part un élément $Q\in \F^G(M)$, on définit une $(M_Q,M)$-famille $(c_R^Q)_{R\in\P^{M_Q}(M)}$ par
\[c_R^Q(\lambda)=c_{Q(R)}(\lambda),\,\,\,\,\lambda\in ia_M^\ast\]
où $Q(R)$ est l’unique élément de $\P^G(M)$ qui est contenu dans $Q$ et dont l'intersection avec $M_Q$ vaut $R$. Pour calculer $c_M^Q$ on possède la formule, indépendante de $\lambda$ en position générale,
\begin{equation}\label{eq:GMcalcul}
c_M^Q=\frac{1}{p!}\sum_{R\in\P^{M_Q}(M)}\left(\lim_{t\to 0}\left(\frac{d}{dt}\right)^p c_R^Q(t\lambda)\right)\theta_R^{M_Q}(\lambda)^{-1},    
\end{equation}
où $t\in\R$ et $p=\dim(a_{M}^{M_Q})$.

Soit $(d_P)_{P\in\P^G(M)}$ une  $(G,M)$-famille. Pour tout élément $Q\in \F^G(M)$, il existe un
nombre complexe $d_Q'$ de sorte que pour toute $(G,M)$-famille $(c_P)_{P\in\P^G(M)}$, le produit s’écrive comme
\[(cd)_M=\sum_{Q\in\F^G(M)}c_M^Qd_Q'.\]
De plus, s'il existe pour tout $L\in \L^G(M)$ un nombre complexe $c_M^L$ tel que $c_M^Q= c_M^L$ pour tout $Q \in \P^G(L)$. Alors on a
\begin{equation}\label{YDLgeomeq:GMfamilyproductformula1}
(cd)_M=\sum_{L\in\L^G(M)}c_M^Ld_L.    
\end{equation}

Moyennant certains choix, on peux montrer qu’il existe des applications :
\begin{align*}
    d_M^G :\L^G(M)\times \L^G(M)&\rightarrow [0,+\infty[ \\
    s:\L^G(M)\times \L^G(M)&\rightarrow \F^G(M)\times \F^G(M)
\end{align*}
de sorte que, pour tout $(L_1,L_2)\in \L^G(M)\times \L^G(M)$, on ait
\begin{enumerate}
    \item si $s(L_1,L_2)=(Q_1,Q_2)$ alors $(Q_1,Q_2)\in \P^G(L_1)\times\P^G(L_2)$ ;
    \item si $s(L_1,L_2)=(Q_1,Q_2)$ alors $s(L_2,L_1)=(\overline{Q_2},\overline{Q_1})$ ;
    \item $d_M^G(L_1,L_2)\not =0$ si et seulement si l'une des flèches naturelles
    \[a_M^{L_1}\oplus a_M^{L_2}\longrightarrow a_M^G\]
    et
    \[a_{L_1}^G\oplus a_{L_2}^G\longrightarrow a_M^G\]
    est un isomorphisme, auquel cas les deux sont isomorphismes et $d_M^G(L_1,L_2)$ est le volume dans $a_M^G$ du parallélotope formé par les bases orthonormées de $a_M^{L_1}$ et de $a_M^{L_2}$ ;
    \item si $(c_P)_{P\in\P^G(M)}$ et $(d_P)_{P\in\P^G(M)}$ sont des $(G,M)$-familles, on a l'égalité
    \[(cd)_M=\sum_{(L_1,L_2)\in \L^G(M)\times \L^G(M)}d_M^G(L_1,L_2)c_M^{Q_1}d_M^{Q_2};\]
    \item Formule de descente : si $(c_P)_{P\in\P^G(M)}$ est une $(G,M)$-famille et $L\in \L^G(M)$, alors
    \[c_L=\sum_{L'\in \L^G(M)}d_M^G(L,L')c_M^{Q'}\]
    où l'on note $Q'$ la deuxième composante de $s(L,L')$.
\end{enumerate}

\subsection{Généralités sur les groupes du type GL et notion de transfert des objets}
Pour $S$ un anneau commutatif et $A$ une $S$-algèbre à gauche. On note $A^{\text{op}}$ l'algèbre opposée de $A$, qui est une $S$-algèbre à droite. Il y a une structure de $A^{\text{op}}\otimes_SA$-module à droite sur $A$ donnée par $a\cdot (b\otimes_S b')=bab'$ avec $a,b'\in A$ et $b\in A^{\text{op}}$. Une $S$-algèbre $A$ à gauche est dite séparable si $A$ est un module projectif sur $A^{\text{op}
}\otimes_SA$.
\begin{definition}\label{def:groupetypeGL}
On dit qu'un groupe réductif $G$ est du type GL (sur $F$) s'il est le groupe des unités d'une algèbre séparable sur $F$.

De façon équivalente (conséquence de \cite[Corollaire 10.7.b]{Pi82} et le théorème d'Artin-Wedderburn), un groupe réductif est du type GL (sur $F$) s'il est de la forme $\prod_{i\text{ fini}}\Res_{E_i/F}\GL_{n_i,D_i}$ avec $E_i$ une extension finie de $F$ et $D_i$ une algèbre à division de dimension finie sur son centre $E_i$.
\end{definition}

Plus précisément, si $A$ est une $F$-algèbre séparable, et $G$ est le groupe des unités de $A$, alors
\[G(B)=(A\otimes_F B)^\times\]
pour toute $F$-algèbre $B$, avec $(A\otimes_F B)^\times$ le groupe des unités de $A\otimes_F B$.


Les objets centraux dans cet article seront les groupes du type GL. Donnons des propriétés que nous utiliserons implicitement tout au long de l'article.
\begin{proposition}\label{pro:bontype}~{}
\begin{enumerate}
    \item Un produit fini de groupes du type GL l'est aussi.
    \item Tout sous-groupe de Levi d'un groupe du type GL l'est aussi.
    \item Toute forme intérieure d'un groupe du type GL l'est aussi.
    \item Toute extension des scalaires par un sur-corps d'un groupe du type GL l'est aussi (sur le sur-corps).
    \item Soient $G$ un groupe du type GL et $\g$ son algèbre de Lie. On dispose de $G \hookrightarrow \g$ une inclusion canonique $G$-équivariante ($G$ agit par conjugaison sur les deux espaces).
    \item Soient $G$ un groupe du type GL et $X\in \g_\ss(F)$, alors $G_X$ est un groupe du type GL.
    \item Soient $G$ un groupe du type GL et $X\in \g(F)$, alors $G_X$ est géométriquement connexe.
    \item Dans l'algèbre des $F$-points de l'algèbre de Lie d'un groupe $G$ du type GL, les notions de $G(F)$-conjugaison et de $G(\overline{F})$-conjugaison coïncident.  
    \item Dans $\g^\ast(\overline{F})$ l'espace des $\overline{F}$-points de l'algèbre de Lie d'un groupe du type GL quasi-déployé, toute classe de conjugaison définie sur $F$ contient un $F$-point.
\end{enumerate}
\end{proposition}
\begin{proof}
Les résultats sont classiques. Les points 1,2,4 et 5 ressortent de la définition. Le point 3 vient de la théorie des corps de classes ; le point 6 vient de \cite[section 12.7]{Pi82} ; le point 7 vient du fait que $\g_X$ est un produit semi-directe d'un groupe isomorphe en tant que schéma à une puissance du groupe additif, par un groupe du type GL, puis une puissance du groupe additif est géométriquement connexe et un groupe du type GL est aussi géométriquement connexe, car son extension à $\overline{F}$ est un produit fini de groupes généraux linéaires sur $\overline{F}$ ; le point 8 est une conséquence du théorème de Hilbert 90, conjugué avec le lemme de Shapiro ; enfin le point 9 vient de \cite[théorème 4.2]{Kott82}, on donnera aussi une autre démonstration dans l'appendice A. 
\end{proof}

Soient $G$ un groupe du type GL sur $F$, $G^\ast$ sa forme intérieure quasi-déployée, en l'occurrence un produit de restrictions des scalaires de groupes généraux linéaires, et $\eta: \g_{\overline{F}} \rightarrow \g_{\overline{F}}^\ast$ un torseur intérieur. Pour rappel cela signifie que 
\begin{enumerate}
    \item $\eta$ est un $\overline{F}$-isomorphisme d'algèbres associatives ;
    \item pour tout $\sigma\in\Gal(\overline{F}/F)$ on ait que $\eta\sigma(\eta)^{-1}$ est l'automorphisme de $\g_{\overline{F}}^\ast$ défini par conjugaison par un élément de $G_{\overline{F}}^\ast$, ici $\sigma(\eta):X\in \g_{\overline{F}}\mapsto \sigma(\eta(\sigma^{-1}(X)))\in \g_{\overline{F}}^\ast$ avec $\Gal(\overline{F}/F)$ agit de façon évidente sur $\g_{\overline{F}}$ et $\g_{\overline{F}}^\ast$.
\end{enumerate}
La deuxième condition est automatique selon le théorème de Skolem–Noether. Pour tout $\overline{F}$-isomorphisme d'algèbres associatives $\eta': \g_{\overline{F}} \rightarrow \g_{\overline{F}}^\ast$, on dit que $\eta$ et $\eta'$ sont dits équivalents. Il existe un 1-cocycle $u$ à valeurs dans $G_{\ad}^\ast(\overline{F})$, tel que pour tout $\sigma\in\Gal(\overline{F}/F)$ on ait $\eta\sigma(\eta)^{-1}=\Ad(u_\sigma)$.   La classe de $(u_\sigma)$ dans $H^1(F,G_{\ad}^\ast)$ ne dépend pas du choix de $\eta$ dans sa classe d'équivalence. Par abus de notation $\eta$ peut aussi désigner sa classe d'équivalence. 
On a un diagramme commutatif équivariant par $G_{\overline{F}}$ et $G^\ast_{\overline{F}}$ :
\[
    \begin{tikzcd}
    G_{\overline{F}} \arrow{r}{\eta\mid_{G_{\overline{F}}}}\arrow[hookrightarrow]{d}{}& G_{\overline{F}}^\ast \arrow[hookrightarrow]{d}{}\\
    \g_{\overline{F}} \arrow{r}{\eta} & \g_{\overline{F}}^\ast
    \end{tikzcd}
    \]
Le morphisme $\eta$ est aussi la différentielle de $\eta|_{G_{\overline{F}}}$.     
    
On peut omettre $\overline{F}$ dans l'indice de $ \g_{\overline{F}}\to \g_{\overline{F}}^\ast$, $G_{\overline{F}}\to G_{\overline{F}}^\ast$ etc lorsqu'on parle d'un torseur intérieur, avec pour finalité d'alléger l'écriture.

\begin{lemma}\label{lem:AMisFiso}
La restriction $\eta|_{A_G}:A_G\rightarrow A_{G^\ast}$ est un $F$-isomorphisme.
\end{lemma}
\begin{proof}
Comme $\eta$ est un torseur intérieur nous avons un $F$-isomorphisme $\eta|_{Z(G)}:Z(G)\rightarrow Z(G^\ast)$. Or $A_G$ est la partie déployée du tore $Z(G)$,
nous obtenons ainsi un $F$-isomorphisme $\eta|_{A_G}:A_G\rightarrow A_{G^\ast}$.
\end{proof}

Soit $\o$ une classe de $G(F)$-conjugaison dans $\g(F)$. On a, pour tout $\sigma\in\Gal(\overline{F}/F)$,
\[\sigma(\eta(\o))=\sigma(\eta)\left(\sigma(\o)\right)=\Ad(u_\sigma)^{-1}(\eta(\o)).\]
Il vient que la classe de $G^\ast(\overline{F})$-conjugaison de $\eta(\o)$ est définie sur $F$. 

\begin{definition}[Transfert d'une classe de conjugaison/d'un élément]~{}
\begin{enumerate}
    \item Soit $\o$ (resp. $\o'$) une classe de $G(F)$-conjugaison (resp. $G^\ast(F)$-conjugaison) dans $\g(F)$ (resp. $\g^\ast(F)$). On dit que $\mathfrak{o}'$ se transfère en $\mathfrak{o}$ si $\eta((\Ad G(\overline{F}))\mathfrak{o})=(\Ad G^\ast(\overline{F}))\mathfrak{o}'$. On écrit $\o'=\o^\ast$ et $\o\arr \o^\ast$, et on dit que $\o^\ast$ se transfère à $G$.
    \item Soient $X\in \g(F)$ et $X'\in \g^\ast(F)$. On dit que $X$ est un transfert de $X'$ s'il existe $\o$ une classe de $G(F)$-conjugaison dans $\g(F)$ tel que $X\in \o$ et $X'\in \o^\ast$ avec $\o^\ast$ se transfère en $\o$. On écrit $X'=X^\ast$ et $X\arr X^\ast$, et on dit que $X^\ast$ se transfère à $G$.
\end{enumerate}
\end{definition}

Soient $X\in \g_{\ss}(F)$ et $X^\ast\in \g_{\ss}^\ast(F)$ avec $X$ un transfert de $X^\ast$. Soit $x'\in G^\ast(\overline{F})$ tel que $(\Ad x)\eta(X)=X^\ast$. On obtient, par la restriction de $(\Ad x')\eta$, un morphisme $\eta_X:\g_X\rightarrow \g_{X^\ast}^\ast$. 

\begin{proposition}\label{prop:GXtorseurint}
Si $X\in \g_{\ss}(F)$ est un transfert de $X^\ast\in \g_{\ss}^\ast(F)$ alors $G_{X^\ast}^\ast$ est la forme intérieure quasi-déployée de $G_X$. 
\end{proposition}
\begin{proof}
Comme $\eta_X$ est un $\overline{F}$-isomorphisme d'algèbres associatives, c'est un torseur intérieur. 
Pour finir $G_{X^\ast}^\ast$ est quasi-déployé car il est de la forme $\prod_i\Res_{E_i/F}\GL_{n_i,E_i}$ avec $E_i$ des extensions finies de $F$. 
\end{proof}

\begin{proposition} Si $X\in \g(F)$ est un transfert de $X^\ast\in \g^\ast(F)$, alors  $X_\ss$ est un transfert de $X_\ss^\ast$ via $\eta$, et $X_\nilp$ est un transfert de $X_\nilp^\ast$ via $\eta_X$.
\end{proposition}
\begin{proof}
La première partie de l'énoncé est prouvée. Puis par définition on voit que $X$ est un transfert de $X^\ast$ via $\eta_X$. Or $\eta_X$ est un torseur intérieur, il induit ainsi un $F$-isomorphisme du centre de $\g_{X_\ss}$ sur le centre de $\g_{X_\ss^\ast}^\ast$. Ce $F$-isomorphisme envoie $X_\ss$ sur $X_\ss^\ast$. De ce fait $X_\nilp=X-X_\ss$ via $\eta_X$ est un transfert de $X_\nilp^\ast=X^\ast-X_\ss^\ast$.
\end{proof}

On a défini la correspondance des éléments des algèbres de Lie. Il est aussi nécessaire de définir la correspondance des sous-groupes de Levi et de sous-groupes paraboliques. 

\begin{definition}[Transfert d'un sous-groupe de Levi et d'un sous-groupe parabolique] 
Soient $M'$ un sous-groupe de Levi de $G^\ast$ et $P'$ un sous-groupe parabolique de $G^\ast$ contenant $M'$. On dit que $(M',P')$ se transfère à $G$ s'il existe $M$ un sous-groupe de Levi de $G$, $P$ un sous-groupe parabolique de $G$ contenant $M$, et $\eta_1$ dans la classe de $\eta$ tels que $\eta_1(M_{\overline{F}},P_{\overline{F}})=(M_{\overline{F}}',P_{\overline{F}}')$.

On dit que $M'$ un sous-groupe de Levi de $G^\ast$ se transfère à $G$ s'il existe $P'$ un sous-groupe parabolique de $G^\ast$ contenant $M'$ de sorte que $(M',P')$ se transfère à $G$.
\end{definition}

\begin{proposition}\label{prop:transfertdef}
Soient $M'$ un sous-groupe de Levi de $G^\ast$ et $P'$ un sous-groupe parabolique de $G^\ast$. Les assertions suivantes sont équivalentes :
\begin{enumerate} 
    \item Il existe $\eta_1$ dans la classe de $\eta$ tel que $\eta_1\sigma(\eta_1)^{-1}=\Ad(u_{\sigma})$ avec $u_\sigma\in M_{\AD}'(\overline{F})$ pour tout $\sigma\in \Gal(\overline{F}/F)$.
    \item $(M',P')$ se transfère.
    \item La classe $\eta\in H^1(F,G_{\ad}^\ast)$ est dans l'image de la flèche $H^1(F,M_{\AD}')\rightarrow H^1(F,G_{\ad}^\ast)$.
\end{enumerate}
\end{proposition}

\begin{remark}
 La flèche $H^1(F,M_{\AD}')\rightarrow H^1(F,G_{\ad}^\ast)$ du point 3 est injective, c.f. \cite[lemme 1.1]{Art99}.   
\end{remark}

\begin{proof}

1 $\Rightarrow$ 2 : posons $A_{M,\overline{F}}\eqdef\eta_1^{-1}(A_{M',\overline{F}})$, c'est un sous-tore de $G$ défini sur $F$. Puis $A_{M}$ est $F$-déployé. Soit ensuite $M=\Cent(A_M,G)$. C'est un sous-groupe de Levi défini sur $F$ car c'est le centralisateur d'un sous-tore déployé. Par restriction on a un $\overline{F}$-isomorphisme $\eta_1:M_{\overline{F}}\rightarrow M_{\overline{F}}'$. Enfin $\eta_1^{-1}(P_{\overline{F}}')$ est un sous-groupe parabolique défini sur $F$, et la restriction de $\eta_1$ à $\eta_1^{-1}(P_{\overline{F}}')$ nous gratifie de $\eta_1^{-1}(P_{\overline{F}}')\rightarrow P_{\overline{F}}'$.

2 $\Rightarrow$ 3 :  les conditions $\Ad(u_\sigma)P_{\overline{F}}'=P_{\overline{F}}'$ et $\Ad(u_\sigma)M_{\overline{F}}'=M_{\overline{F}}'$ impliquent $u_\sigma\in M_\AD'(\overline{F})$. Ainsi la classe de  $(u_\sigma)$ est dans l'image de la flèche $H^1(F,M_{\AD}')\rightarrow H^1(F,G_{\ad}^\ast)$. 

3 $\Rightarrow$ 1 : s'il existe $g'\in G^\ast(\overline{F})$ tel que $(\Ad g')\eta\circ\sigma((\Ad g')\eta)^{-1}$ est à valeur dans $M_{\AD}'(\overline{F})$ alors on pose $\eta_1=(\Ad g')\eta$.
\end{proof}

Nous nous préoccupons essentiellement du transfert des couples semi-standard. Soit $M_0$ (resp. $M_{0^\ast}$) un sous-groupes de Levi minimal de $G$ (resp. $G^\ast$). Si $M'\in\L^{G^\ast}(M_{0^\ast})$ et $P'\in\P^{G^\ast}(M')$ sont tels que $(M',P')$ se transfère, alors il existe bien  $M\in\L^G$, $P\in \F^G(M)$ et $g'\in G^\ast(F)$ tels que $(M',P')=(\Ad g')\eta(M,P)$. \`{A} l'inverse fixons pour la suite $\eta$ un élément dans sa classe qui vérifie : 
\begin{equation}
    M_{0^\ast,\overline{F}} \subseteq \eta(M_{0,\overline{F}}).
\end{equation}
Soient maintenant $M\in \L^G(M_0)$ et $P\in \P^G(M)$. On prend $P_{0}$ (resp. $P_{0^\ast}$) un sous-groupe parabolique minimal de $G$ (resp. $G^\ast$) contenant $M_0$ (resp. $M_{0^\ast}$), tels que $P_0\subseteq P$ et $P_{0^\ast,\overline{F}} \subseteq \eta(P_{0,\overline{F}})$. Le couple $\eta(M_{0,\overline{F}},P_{0,\overline{F}})$ étant standard par rapport à $(M_{0^\ast,\overline{F}},P_{0^\ast,\overline{F}})$, il est défini sur $F$. Ainsi, la restriction de $\eta$ définit un torseur intérieur
\[\eta: \m\rightarrow \m^\ast.\]
avec $M^\ast\in \L^{G^\ast}(M_{0^\ast})$ le sous-groupe de Levi sur $F$ obtenu via la descente galoisienne de $\eta(M_{\overline{F}})$. On écrit aussi $P^\ast\in \P^{G^\ast}(M^\ast)$ le sous-groupe parabolique sur $F$ obtenu via la descente galoisienne de $\eta(P_{\overline{F}})$. Par abus de notation on notera aussi $M^\ast=\eta(M)$ et $P^\ast=\eta(P)$. De la discussion on en déduit que cet élément $\eta$ induit une injection $\L^G(M_0)\rightarrow \L^{G^\ast}(M_{0^\ast})$, et pour tout $M\in\L^G(M_0)$ des bijections $\L^G(M)\rightarrow \L^{G^\ast}(M^\ast)$, $\P^G(M)\rightarrow \P^{G^\ast}(M^\ast)$, et $\F^G(M)\rightarrow \F^{G^\ast}(M^\ast)$.


On a $\eta(\mathfrak{z})=\mathfrak{z}^\ast$, et $\eta(\mathfrak{z}(F))=\mathfrak{z}^\ast(F)$ car l'action de $G_{\overline{F}}^\ast$-conjugaison sur $\mathfrak{z}_{\overline{F}}^\ast$ est triviale donc la restriction de $\eta$ sur $\mathfrak{z}_{\overline{F}}$ est definie sur $F$.

\begin{lemma}\label{lem:Weylcomp}Pour tout $M\in\L^G(M_0)$, $\eta$ induit une bijection entre les groupes de Weyl relatifs $W^G(M)$ et $W^{G^\ast}(M^\ast)$.
\end{lemma}
\begin{proof}
D'abord $\eta$ induit une bijection entre $\text{Norm}_{G}(A_M)/M$ et $\text{Norm}_{G^\ast}(A_{M^\ast})/M^\ast$. Comme $\eta\sigma(\eta)^{-1}$ est définit par un automorphisme intérieur de $M^\ast$, $\eta$ induit alors une bijection entre $\left(\text{Norm}_{G}(A_M)/M\right)^\Gamma$ et $\left(\text{Norm}_{G^\ast}(A_{M^\ast})/M^\ast\right)^\Gamma$. Il reste à voir que cela descend au niveau des $F$-points. Prouvons que l’injection canonique $W^G(M)\rightarrow \left(\text{Norm}_{G}(A_M)/M\right)^\Gamma$ est bijective : soit $x\in\left(\text{Norm}_{G}(A_M)/M\right)^\Gamma$. Soit $P\in\P^G(M)$. Le couple $(M,P)$ est standard pour un couple $(M_0,P_0')$ constituée d’un sous-groupe parabolique minimal $P_0'$ de composante de Levi $M_0$. Maintenant le couple $(\Ad(x)M,\Ad(x)P)=(M,\Ad(x)P)$ est définie sur $F$, il existe donc $g\in G(F)$ tel que  $(\Ad(gx)M,\Ad(gx)P)$ soit standard par rapport à $(M_0,P_0')$, du coup $gx\in M$, ce qui implique $x\in G(F)M$.
\end{proof}

\subsubsection{Orbites induites}\label{subsubsec:orbind}

Pour toute partie $S\subseteq \g$, on note par $S_{G-\reg}$ le lieu régulier pour l'action adjointe de $G$, autrement dit
\begin{equation}\label{eq:deflieuG-reg}
S_{G-\reg}\eqdef\{X\in S\mid \dim (\Ad G)X=\max_{Y\in S}\left( \dim (\Ad G)Y\right)\}.    
\end{equation} 
On souligne le fait qu'en général $S_{G-\reg}\not \subseteq \g_{G-\reg}$.

\begin{proposition}[\cite{YDL23a}]\label{prop:indprop}Soient $M$ un sous-groupe de Levi de $G$ et $X\in \m(F)$. On note $\o= (\Ad M)X$.
\begin{enumerate}
    \item Il existe une unique orbite $\Ind_M^G(\o)=\Ind_M^G(X)$ dans $\g$ pour l’action adjointe de $G$ telle que l’intersection
    \begin{equation}\label{YDLgeomeq:indorbdef}
    \Ind_M^G(\o)\cap \left(\o+\mathfrak{n}_{P}\right)    
    \end{equation}
    soit un ouvert Zariski dense dans $\o+\mathfrak{n}_{P}$ pour tout $P$ sous-groupe parabolique de $G$ ayant $M$ comme facteur de Levi ;
    \item l'orbite induite commute à la décomposition de Jordan, i.e.
    \begin{equation}\label{YDLgeomeq:indJordan}
    \Ind_M^G(X)=\Ad(G)(\Ind_{M_{X_{\ss}}}^{G_{X_{\ss}}}(X))=\Ad(G)(X_{\ss}+\Ind_{M_{X_{\ss}}}^{G_{X_{\ss}}}(X_{\nilp})) ;   
    \end{equation}
    \item $\text{codim}_{\mathfrak{m}}(\o)=\text{codim}_{\g}(\Ind_M^G(\o))$ ;
    \item l'intersection \eqref{YDLgeomeq:indorbdef} est également le lieu régulier de $\o+\mathfrak{n}_{P}$ ;
    \item soit $P$ un sous-groupe parabolique de $G$ ayant $M$ comme facteur de Levi, alors $\Ind_M^G(\o)\cap \left(\o+\mathfrak{n}_{P}\right)$ est une orbite dans $\p$ pour l'action adjointe de $P$ ;
    \item soit $P$ un sous-groupe parabolique de $G$ ayant $M$ comme facteur de Levi, soit $Y\in \Ind_M^G(\o)\cap (\o+\mathfrak{n}_{P})$, alors $G_{Y}\subseteq P$ ;
    \item soit $L$ un sous-group de Levi de $H$ contenant $M$, alors
    \[\Ind_{L}^G(\Ind_{M}^L(\o))=\Ind_{M}^G(\o).\]
\end{enumerate}
\end{proposition}

Si $\o$ contient un $F$-point, l'orbite induite $\Ind_M^G(\o)$ contient aussi un $F$-point comme $\n_P(F)$ est Zariski dense dans $\n_P$. On sait que $\o(F)$ est exactement une classe de $G(F)$-conjugaison (cf. point 8 de la proposition \ref{pro:bontype}).

\subsection{Espace de Schwartz-Bruhat}\label{subsec:EspaceS-B}

Soit $H$ un groupe abélien localement compact et séparé. On désigne par $\S(H)$ l'espace de Schwartz-Bruhat de $H$, muni de sa topologie usuelle (cf. \cite[page 61]{Bruhat1961}).

\begin{proposition}[{{\cite[proposition 11. a]{Bruhat1961}}}]\label{prop:topopropertiesS-Bspace}~{}
L'espace $\S(H)$ est un espace de Montel nucléaire.
\end{proposition}

Soient $H_1$ et $H_2$ deux groupes abéliens localement compacts et séparés. Alors $\S(H_1\times H_2)=\S(H_1)\widehat{\otimes}\S(H_2)$, avec $\otimes$ le produit tensoriel d'espaces vectoriels topologiques, muni de sa topologie usuelle (cf. \cite[définition 43.1 ou 43.2]{BookTreve}, $\S(H_1)$ et $\S(H_2)$ étant nucléaires, la $\epsilon$-topologie et la $\pi$-topologie sur $\S(H_1)\otimes\S(H_2)$ sont donc la même), puis $\widehat{\otimes}$ le complété pour cette topologie. 

Soit $V$ un espace vectoriel sur $F$ supposé ici un corps local de caractéristique 0. Alors, pour $F$ non-archimédien, $\S(V(F))$ est l'espace des fonctions localement constantes à support compact sur $V(F)$, muni de la topologie discrète ; pour $F$ archimédien, $\S(V(F))$ est l'espace des fonctions de classe Schwartz usuelle sur $V(F)$, muni de la topologie engendrée par les semi-normes $\|-\|_{P,Q}$ pour $P$ et $Q$ polynômes à $\dim_\R V(F)$ variables et à coefficients complexes, avec $\|f\|_{P,Q}=\sup_{x\in V(F)}|P(x)Q(\frac{\partial}{\partial x})f(x)|$.

\subsection{Conventions}
On note, en ajoutant une étoile $^\ast$, en exposant les objets relativement à $G^\ast$ qui se transfèrent à $G$, et on veille à ne pas confondre $M_{0^\ast}$, qui est le sous-groupe de Levi standard minimal de $G^\ast$, avec le  sous-groupe de Levi qui se transfère en $M_0$, noté $M_0^\ast$. Pour les objets généraux relativement à $G^\ast$ on les note en ajoutant un signe prime $'$ en exposant. Par exemple en écrivant $X^\ast\in \g^\ast(F)$ il est sous-entendu qu'il existe $X\in\g(F)$ tel que $X$ est un transfert de $X^\ast$. Alors que l'écriture $X'\in \g^\ast(F)$ n'impose pas de condition particulière sur $X'$. Le contexte devrait permettre d’enlever toute ambiguïté.

\section{Préliminaires locaux}\label{sec:preliminaireslocaux}
Dans cette partie $F$ désigne un corps local de caractéristique 0. On note $\O_F$ l'anneau des entiers de $F$ si $F$ est non-archimédien. On note $|\cdot|_F=|\cdot|$ la valeur absolue normalisée sur $F$.

Fixons $G$ un groupe du type GL sur $F$, $G^\ast$ sa forme intérieure quasi-déployée. Prenons $M_0$, resp. $M_{0^\ast}$, un sous-groupe de Lévi minimal de $G$, resp. de $G^\ast$. Ils sont bien sûr définis sur $F$. Fixons $\eta:\g\rightarrow\g^\ast$ un torseur intérieur tel que $\eta(M_0)$ contient $M_{0^\ast}$, et que la restriction de $\eta$ sur $A_{M_{0}}$ est définie sur $F$. Comme il est loisible, on choisit les sections $L\mapsto Q_{L}$ sur $G$ et $L'\mapsto Q_{L'}$ sur $G^\ast$ de sorte que
\[(Q_L)^\ast=Q_{L^\ast}.\]

Nous allons fixer $K$ un sous-groupe compact maximal de $G(F)$. Lorsque $F$ est un corps local archimédien on prend $K$ tel que les algèbres de Lie de $A_{M_0}(\R)$ et de $K$ sont orthogonales par rapport à la forme de Killing de $G(F)$. Lorsque $F$ est un corps local non-archimédien, on fixe un sommet spécial dans l’appartement associé à $A_{M_0}$ de l’immeuble de Bruhat-Tits de $G$. On en déduit un sous-groupe compact maximal $K$ de $G(F)$, qui est le fixateur dans $G(F)$ de ce sommet. Un sous-groupe $K$ fixé ainsi est dit en bonne position par rapport à $M_0$. Signalons que pour tout $M\in\L^G(M_0)$ le sous-groupe compact maximal $K\cap M(F)$ de $M(F)$ est en bonne position par rapport à $M_0$. Le procédé ci-dessus vaut pour tout groupe rédutif $G$, et on voit que le sous-groupe compact maximal $KZ(G)(F)/Z(G)(F)=K/K\cap Z(G)(F)$ de $G_\ad(F)$ est en bonne position par rapport à $M_{0,\AD}$.

\subsection{Transformée de Fourier sur les algèbres de Lie} 

Notons $\tau_\g$ la trace réduite de l'algèbre séparable sous-jacente, elle envoie alors un élément de $\g(F)$ sur un élément de $F$. Soit $\langle X,Y\rangle =\tau_\g(XY)$, c'est une forme bilinéaire sur $\g(F)$ non-dégénérée et invariante par adjonction. Nous appellerons $\langle-,-\rangle$ la forme bilinéaire canonique de $\g$ dans cet article. Remarquons que pour $X\in\g_\ss(F)$, la forme bilinéaire canonique de $\g$ restreinte à $\g_X$ est la forme bilinéaire canonique de $\g_X$. 

On se donne un caractère non-trivial $\psi$ de $F$. On définit la transformée de Fourier d'une fonction $f \in\S(\g(F))$ de classe Schwartz-Bruhat (numéro \ref{subsec:EspaceS-B}) par
\[\widehat{f}^\g(Y)=\int_{\g(F)}f(X)\psi(\langle X,Y\rangle)\,dX,\,\,\,\,\forall Y\in \g(F),\] 
où $\g(F)$ est muni de la mesure de Haar auto-duale relativement à la transformée de Fourier. On va abréger $\widehat{f}^\g$ en $\widehat{f}$ si le contexte nous le permet. 

Au bicaractère $\psi(\langle-,-\rangle)$ de $\g(F)$ est associée la constante de Weil : on note $F_\psi$ la fonction complexe sur $\g(F)$ définie pour tout $X\in\g(F)$ par $F_\psi(X)=\psi\left(\langle X,X\rangle/2\right)$. Dans ce cas, il existe une racine huitième de l'unité, notée $\gamma_\psi(\g)$, telle que pour tout $f
\in\S(\g(F))$
\[\widehat{f\ast F_\psi}=\gamma_{\psi}(\g)\widehat{f}\cdot F_{-\psi}
\]
(cf. \cite[théorème 2 et section 26]{Weil64}), on a noté $\ast$ le produit de convolution sur $\g(F)$. 

On remarque que la décomposition $\g=\mathfrak{z}\oplus\g_{\ad}$ est orthogonale par rapport à la forme bilinéaire canonique $\langle-,-\rangle$.

\subsection{Normalisations des mesures de la section \ref{sec:preliminaireslocaux}}\label{subsec:localnormalisationsdemesures}
On fixe à présent les normalisations des mesures. On se concentrera sur les objets relativement à $G$ dans la mesure où il y a des analogues évidents pour $G^\ast$.

Le groupe $W_0^G$ agit naturellement sur $a_{M_0}$. On fixe alors un produit scalaire sur $a_{M_0}$ invariant par $W_0^G$. On prend sur tout sous-espace de $a_{M_0}$ (resp. $a_{M_v,0}$) la mesure engendrée par la restriction de ce produit scalaire. On remarque qu'il existe un isomorphisme naturel $a_M^G\simeq a_{M_{\AD}}=a_{M_{\AD}}^{G_
{\ad}}$, il nous procure une mesure sur $a_{M_{\AD}}$.

On écrit $dh$ pour la mesure sur $a_P
$ et $d\lambda$ pour la mesure duale sur $ia_P^\ast$, dans le sens où
\[\int_{ia_P^\ast}\int_{a_P}f(h)e^{-\lambda(h)}\,dh\,d\lambda=f(0)\]
pour tout $f\in L^1(a_P)$.

Pour $\mathfrak{h}$ une sous-algèbre séparable de $\g$ de la forme $G_X$ avec $X\in \g_{\ss}(F)$, la forme bilinéaire $\langle-,-\rangle$ est non-dégénéré sur $\mathfrak{h}(F)$, et on prend la mesure de Haar auto-duale sur $\mathfrak{h}(F)$ relativement à sa transformée de Fourier. Puis on note $H$ le groupe des unités de $\mathfrak{h}$, et on prend sur $H(F)$ la mesure de Haar 
\begin{equation}\label{YDLgeomeq:Haarmeasureongrp}
dh=\frac{dH}{|\text{Norm}_{\mathfrak{h}}(h)|},    
\end{equation}
où $dH$ est la mesure sur $\mathfrak{h}(F)$ et $\text{Norm}_{\mathfrak{h}}$ est la norme (et non la norme réduite) sur $F$ de l'algèbre séparable $\mathfrak{h}$. On remarque cette mesure est également celle induite via l’exponentielle de la mesure sur un voisinage de $0$ de $\mathfrak{h}(F)$
, i.e. 
\[\int_{V} f(\exp H)|\det(d\exp)_H|\,dH=\int_{\exp V}f(h)\,dh\]
pour tout $V$ voisinage de $0$ de $\mathfrak{h}(F)$ tel que $\exp :V \to \exp(V)$ soit bijectif et tout $f\in L^1(\exp V)$. Tout sous-groupe de Levi, tout sous-tore maximal d'un sous-groupe de Levi de $G$, et tout sous-groupe de $G$ de la forme $M_Y$ avec $M$ un sous-groupe de Levi et $Y\in\m_\ss(F)$, est de la forme $G_X$ avec $X\in \g_\ss(F)$. Remarquons que pour tout $H$ sous-groupe de $G$ comme plus haut et $g\in G(F)$, la mesure sur $(\Ad g )H(F)$ correspond à la mesure sur $H(F)$ par conjugaison par $g$.

La mesure sur un espace quotient est la mesure quotient. On obtient alors une mesure sur $G_X(F)\backslash G(F)$ pour tout $X\in \g_\ss(F)$.

Soit $M$ un sous-groupe de Levi. Soit $\o$ une $M$-orbite dans $\m$ pour l'action adjointe, qui contient un $F$-point. On pose 
\begin{equation}\label{YDLiopeq:defaMoG-reg}
\a_{M,\o,G-\reg}\eqdef \{A\in\a_M\mid A+Y\in \Ind_M^G(A+Y),\forall Y\in \o\}.    
\end{equation}
Puis, pour rappel, pour toute partie $S\subseteq \g$, on note par $S_{G-\reg}$ le lieu régulier pour l'action adjointe de $G$.

\begin{proposition}[{{\cite{YDL23a}}}]\label{YDLiopprop:whatisaMoG-reg}~{}
\begin{enumerate}
    \item $\a_{M,\o,G-\reg}$ est un ouvert dense de $\a_{M}$ défini sur $F$.
    \item L'intersection $\a_{M,\o,G-\reg}\cap \a_{M,G-\reg}$ contient un voisinage de $0$ dans $\a_{M,G-\reg}$, avec $\a_{M,G-\reg}$ le lieu régulier de $\a_{M}$ pour la $G$-conjugaison (équation \eqref{eq:deflieuG-reg}).
    \item $\a_{M,\o,G-\reg}=\a_{M,\o_\ss,G-\reg}$, où $\o_\ss\eqdef\{Y_\ss\mid Y\in \o\}$ est une $M$-orbite semi-simple dans $\m$.
    \item Pour tout $Y\in \o$, on a 
    \[\a_{M,\o,G-\reg}\subseteq \a_{M_{Y_\ss},G_{Y_\ss}-\reg}.\]

    \item Si $\o$ est nilpotent, on a 
     \[\a_{M,\o,G-\reg}=\a_{M,G-\reg}.\]
\end{enumerate}
\end{proposition}

Pour $Y\in \m(F)$, on pose
\begin{equation}\label{YDLgeomeq:defaMYG-reg}
a_{M,Y,G-\reg}\eqdef a_{M,(\Ad M)Y,G-\reg}.    
\end{equation}

Soit maintenant $X\in \g(F)$ quelconque, on prend d'abord $M$ un sous-groupe de Levi tel que $X_\ss \in\m(F)$ elliptique et $X\in \Ind_M^G(X_\ss)$ (proposition \ref{prop:suppellitique}). Puis on prend sur $G_{X}(F)$ l'unique mesure de Haar telle que 
\begin{align}\label{eq:defmeasureonanyorb}
|D^\g(X)|_F^{1/2}&\int_{G_{X}(F)\backslash G(F)}f(g^{-1}Xg)\,dg\\
&=\lim_{\substack{A\to 0\\A\in\a_{M,X_\ss,G-\reg}(F)}}|D^\g(A+X)|_F^{1/2}\int_{M_{X_\ss}(F)\backslash G(F)}f(g^{-1}(A+X_\ss)g)\,dg    
\end{align}
pour tout $f\in C_c^\infty(G(F))$. On doit vérifier que ce procédé donne une mesure bien définie, c'est-à-dire :
\begin{proposition}[{{\cite{YDL23a}}}]On a les compatibilités suivantes :    
\begin{enumerate}
    \item la limite du membre de droite de l'égalité \eqref{eq:defmeasureonanyorb} existe, et elle ne dépend pas du choix de $M$ ;
    \item si $X$ est semi-simple, alors la mesure sur $G_X(F)$ définie par l'équation \eqref{eq:defmeasureonanyorb} coïncide avec la mesure définie par l'équation \eqref
{YDLgeomeq:Haarmeasureongrp} ;
    \item si $L$ est un sous-groupe de Levi de $G$ dont l'algèbre de Lie contient $X$ tel que $L_X=G_X$, alors la mesure sur $G_X(F)$ définie par l'équation \eqref{eq:defmeasureonanyorb} en voyant $X$ comme élément de $\g(F)$ coïncide avec la mesure sur $L_X$ définie par la même équation en voyant $X$ comme élément de $\mathfrak{l}(F)$ ;
    \item la mesure sur $G_X(F)$ définie par l'équation \eqref{eq:defmeasureonanyorb} en voyant $X$ comme élément de $\g(F)$ coïncide avec la mesure sur $(G_{X_\ss})_{X_\nilp}(F)$ définie par la même équation en voyant $X_\nilp$ comme élément de $\g_{X_\ss}(F)$ ;
    \item la mesure sur $G_{X+Z}(F)$ coïncide avec la mesure sur $G_X(F)$ pour tous $X\in \g(F)$ et $Z\in\mathfrak{z}(F)$.
\end{enumerate}
\end{proposition}

Remarquons que pour tous $X\in \g(F)$ et $g\in G(F)$, la mesure sur $G_{(\Ad g)X}(F)$ correspond à la mesure sur $G_X(F)$ par conjugaison par $g$.

Soit $P\in \F^G(M_0)$, on prend sur $\n_P(F)$ une mesure de Haar $dN$. On munit ensuite $N_P(F)$ de la mesure $dn$ qui est compatible à $dN$ via l'isomorphisme de variétés $\text{Id}+\cdot :N\in\mathfrak{n}_P(F)\mapsto \text{Id}+N\in N_P(F)$, i.e.
    \[\int_{\mathfrak{n}_P(F)} f(\text{Id}+N)\,dN=\int_{N_P(F)}f(n)\,dn\]
    pour tout $f\in L^1(N_P(F))$.

Soit $P\in\F^G(M_0)$. On définit la fonction module pour tout $p \in P(F)$ par
\[\delta_P (p)=e^{2\rho_P(H_{M_0}(p))} = |\det(\Ad(p);\mathfrak{n}_P (F))|.\]

On munit $K$ d'une mesure de Haar. Pour tout $P\in \F^G(M_0)$ on écrit $\gamma(P)=\gamma^G(P,K)>0$ la constante qui est de telle sorte que pour toute fonction $f\in L^1(G(F))$, on ait
\begin{equation}\label{eq:locmeasureiwasawa}
\begin{split}
\int_{G(F)} f(x)\,dx &=\gamma(P)\int_{M_P(F)}\int_{N_P(F)}\int_{K} f(mnk)\,dk\,dn\,dm\\
&=\gamma(P)\int_{N_P(F)}\int_{M_P(F)}\int_{K} f(nmk)e^{-2\rho_P(H_{M_0}(m))}\,dk\,dm\,dn.    
\end{split}
\end{equation}

Le groupe $Z(G)$ est du type $\GL$. On fixe alors une mesure de Haar sur $Z(G)(F)$ par l'équation \eqref{YDLgeomeq:Haarmeasureongrp}. On obtient en particulier une mesure sur $G_\ad(F)$, et aussi une mesure sur $T_\AD(F)$ pour tout tore maximal $T$ de $G$.

On définit les mesures sur les mêmes groupes relativement à $G^\ast$, en suivant les mêmes démarches et en utilisant le même caractère $\psi$ et la forme bilinéaire canonique de $\g^\ast$. Remarquons que la forme bilinéaire canonique de $\g^\ast$ est « déduit » de celle de $\g$ via le torseur intérieur $\eta$, au sens de \cite[section VIII.6]{Walds95}. On voit que (cf. \cite[sections 2.5 et 4.4]{Walds97}), si $T$ est un tore maximal de $G$ et $T^\ast$ un tore maximal de $G^\ast$, et $y\in G^\ast(\overline{F})$ est tel que $(\Ad y)\circ\eta|_T$ soit un isomorphisme défini sur $F$ de $T$ sur $T^\ast$, alors les mesures sur $T(F)$ et $T^\ast(F)$ se correspondent par cet isomorphisme. On observe que les mesures sur $T_\AD(F)$ et $T_\AD^\ast(F)$ se correspondent.

Enfin on demande la compatibilité suivante vis-à-vis du transfert : 
\begin{itemize}
    \item pour tout $M\in\L^G(M_0)$, le torseur intérieur $\eta$ induit un isomorphisme d'espaces  vectoriels réels $a_M\xrightarrow{\sim} a_{M^\ast}$, on veut que les mesures sur $a_M$ et $a_{M^\ast}$ se correspondent par cet isomorphisme. 
\end{itemize}
Ce point est cohérent avec la compatibilité envers le groupe de Weyl relatif exigée plus haut par le biais du lemme \ref{lem:Weylcomp}, il nous donne aussi l'égalité suivante pour tous $L_1,L_2\in \L^{G}(M)$ :
\[d_M^G(L_1,L_2)=d_{M^\ast}^{G^\ast}(L_1^\ast,L_2^\ast).\]



\subsection{Notations locales}

\subsubsection{Application \texorpdfstring{$H_P$}{HP}}~{}

Pour $M\in \L^G(M_0)$, on définit une application $H_M : M(F)\rightarrow a_{M}$ par $e^{\langle H_M(m),\chi\rangle}=|\chi(m)|_{F}$. Puis si $P\in\F^G(M_0)$, on définit $H_P : G(F)\rightarrow a_{P}$ par $H_P(mnk)\eqdef H_{M_P}(m)$, suivant la décomposition d'Iwasawa $G(F)=M_P(F)N_P(F)K$.

\subsubsection{Intégrales orbitales pondérées sur les algèbres de Lie}\label{subsubsec:defIOP}

Uniquement dans cette sous-sous-section $F$ désigne un corps global ou local de caractéristique 0. Soit $S$ un sous-ensemble fini non-vide de $\V_F$. On note $F_S=\prod_{v\in S}F_v$ (quand $F$ est un corps local, $S$ est un singleton et $F_S=F$). Rappelons que $\S(\g(F_S))=\widehat{\bigotimes}_{v\in S}\S(\g(F_v))$ (numéro \ref{subsec:EspaceS-B}). Une fonctionnelle sur $\g(F_S)$ est dite tempérée si elle est une distribution sur $\S(\g(F_S))$. 

Nous nous conformons à la règle suivante : soit $H$ un groupe algébrique sur $F$, si la mesure sur $H(F_v)$ est fixée pour tout $v\in S$, on prend sur $H(F_S)$ la mesure produit. 

Soit $f\in \S(\g(F_S))$ une fonction de classe Schwartz-Bruhat. On définit les intégrales orbitales pondérées sur l'algèbre de Lie en suivant \cite{YDL23a} : on fixe $K_S\eqdef \prod_{v\in S}K_v$. Soient $M\in \L^G(M_0)$, $Q\in \F^G(M)$ et $X\in \mathfrak{m}(F_S)$. Si $M_X=G_X$ alors on pose 
\begin{equation}\label{YDLgeomeq:IOPdefcaseequising}
J_M^Q(X,f)=|D^\g(X)|_{S}^{1/2}\int_{ G_X(F_S)\backslash G(F_S)}f\left(\Ad (g^{-1})X\right)v_M^Q(g)\,dg
\end{equation}
avec $|-|_{S}\eqdef\prod_{v\in S}|-|_{F_v}$ et $v_M^Q(g)$ le poids provenant de la $(G,M)$-famille $(v_P(g))_{P\in \P^G(M)}$ donnée par $v_P(\lambda,g)\eqdef e^{-\lambda(H_P(g))}$.

Dans le cas général, on définit, en généralisant l'approche d'Arthur, une  $(G,M)$-famille $(r_P(A,X))_{P\in \P^G(M)}$, avec $X\in \m(F_S)$ et $A
\in \a_{M,X_\ss,G-\reg}(F_S)$. Pour tout $Q \in \F^G(M)$, on en déduit une $(M_Q,M)$-famille $(r_R^Q(A,X))_{R\in\P^{M_Q}(M)}$ via le procédé du numéro \ref{subsubsec:(G,M)-familles}, elle ne dépend pas de $Q$ mais uniquement de $M_Q$ : $r_R^Q(A,X)=r_R^{M_Q}(A,X)$. On pose
\begin{equation}\label{YDLgeomeq:IOPdef}
J_M^Q(X,f)=\lim_{\substack{A\to 0\\A\in\a_{M,X_\ss,G-\reg}(F)}}\sum_{L\in\L^{M_Q}(M)}r_M^L(A,X)J_L^Q(X+A,f).    
\end{equation}
Notons que si $M_X=G_X$ alors
\[r_M^L(A,Y)=\begin{cases}
        1, & \text{si }L=M\\ 0,& \text{si } L\not=M,
\end{cases}\]
ce qui justifie l'écriture de $J_M^Q(X,f)$. La fonction $r_M^L(A,Y)$ dépend uniquement de la mesure sur $a_M^L$.

On appelle $J_M^Q(X,f)$ une intégrale orbitale pondérée semi-locale. Quand $S$ est un singleton, on appelle $J_M^Q(X,f)$ une intégrale orbitale pondérée locale. 

Le nombre $J_M^Q(X,f)$ ne dépend que de la classe de $M(F_S)$-conjugaison de $X$. On peut donc aussi écrire $J_M^Q(\o,f)$ à la place de $J_M^Q(X,f)$, où $\o=(\Ad M(F_S))X$.

\begin{proposition}[{{\cite{YDL23a}}} pour (1), (2), (3), (4) et (5) ; {{\cite[sections 17.2, 17.5]{Kottbook}}} pour (6) et (7)]\label{prop:IOP} Soient  $L\in\L^G(M_0),Q\in\F^G(L),X\in \mathfrak{l}(F_S)$ et $f\in\S(\g(F_S))$. Pour tout $x\in G(F_S)$, on notera $(\Ad x)f$ la fonction $(\Ad x)f(Y)=f((\Ad x^{-1})Y)$. On note $\gamma_S(Q)\eqdef\prod_{v\in S}\gamma_v(Q_v)$.
\begin{enumerate}
    \item La distribution $J_L^Q(X,-)$ est tempérée. C'est une mesure signée absolument continue par rapport à la mesure invariante sur l'orbite $\Ind_L^G(X)(F_S)$ dans $\g(F_S)$. 
    \item Pour tous $l\in L(F_S)$, $w\in \text{Norm}_{G(F)}(M_0)$ et $k\in K_S$,
    \[J_{(\Ad w)L}^{(\Ad w)Q}((\Ad wl)X,(\Ad  k)f)=J_L^Q(X,f).\]
    \item Pour tout $x\in G(F_S)$, on a
    \[J_L^Q\left(X,(\Ad x^{-1})f\right)=\sum_{P\in\F^Q(L)}J_L^{L_P}\left(X,f_{P,x}^Q\right)\]
    avec $f_{P,x}^Q$ la fonction appartenant à $\S(\mathfrak{l}_P(F_S))$ définie par
    \[f_{P,x}^Q(Z)=\gamma_S(P)\int_{K_S}\int_{\mathfrak{n}_P(F_S)}f\left((\Ad  k^{-1})(Z+U)\right)(v_P^Q)'(kx)\,dU\,dk,\,\,\,\,\forall Z\in\l_P(F_S).\]
    \item Formule de descente de l'induction : soient $M\in\L^L(M_0)$ avec $Y\in\m(F_S)$. On a 
    \[J_L^Q(\Ind_M^L(Y),f)=\sum_{M_1\in\L^{L_Q}(M)}d_M^{L_Q}(L,M_1)J_M^{Q_{M_1}}(Y,f).\]
    Ici $M_1\mapsto Q_{M_1}$ est la section dans la formule de descente pour les $(L_Q,M)$-familles.
    \item Formule de descente parabolique : on a,
    \[J_L^Q(X,f)=J_L^{L_Q}(X,f_Q)\]
    avec $f_Q$ la fonction appartenant à $\S(\mathfrak{l}_Q(F_S))$ définie par
    \[f_Q(Z)=\gamma_{S}(Q)\int_{K_S}\int_{\mathfrak{n}_Q(F_S)}f\left((\Ad  k^{-1})(Z+U)\right)\,dU\,dk,\,\,\,\,\forall Z\in\mathfrak{l}_Q(F_S).\]
    \item Homogénéité nilpotente : supposons que $X$ est nilpotent, et $L=Q=G$. On note $f^{ t}$ la fonction $f^{t}(Y)=f(tY)$ pour $t\in F_S^\times$. Alors
    \[J_G^G(X,f^{t})=\left(\prod_{v\in S}|t_v|_v^{-\frac{1}{2}\dim (\Ad G_v)X_v}\right)J_G^G(X,f)\]
    \item Indépendance linéaire des intégrales orbitales nilpotentes : notons $(\mathcal{N}_G(F_S))_S$ l'ensemble des $G(F_S)$-orbites nilpotentes dans $\g(F_S)$, alors la famille des distributions $(J_G^G(\nu,-))_{\nu\in (\mathcal{N}_G(F_S))_S}$ est linéairement indépendante sur $\S(\g(F_S))$.
\end{enumerate}
\end{proposition}

\begin{proposition}\label{prop:fonctionr}
Pour tous $M\in \L^G(M_0)$, $L\in \L^G(M)$, $A\in \a_M(F_S)$ et $X\in \m(F_S)$ on a 
\[r_M^L(A,X)=r_{M^\ast}^{L^\ast}(A^\ast,X^\ast).\]
\end{proposition}

\begin{proof}Il suffit de traiter le cas où $S$ est un singleton selon la formule de scindage (rappelée plus tard dans cet article, cf. l'équation \eqref{YDLgeomeq:formuledescindage}). On suppose donc dans la suite que $F=F_S$ est un corps local. On se réfère à \cite{YDL23a} pour la définition des fonctions $r_\alpha(\lambda,A,\o)$ et $\rho(\alpha,\o)$ suivantes. On peut et on va supposer que $L=G$, puis $X$ est nilpotent. Notons $\o$ la $G$-orbite de $X$. Soit $M_1\in\L^M(M_0)$ tel que $\o=\Ind_{M_1}^G(0)$.

Le torseur intérieur $\eta$ induit une bijection $\P^G(M)\to\P^{G^\ast}(M^\ast)$, puis une bijection $\Delta_P^{G,\vee}\to \Delta_{P^\ast}^{G^\ast,\vee}$. D'après nos normalisations des mesures nous avons $\theta_P(\lambda)=\theta_{P^\ast}(\lambda^\ast)$, avec bien sûr $\lambda^\ast$ l'image de $\lambda$ sous l'isomorphisme $\eta:a_M^G\to a_{M^\ast}^{G^\ast}$. Or $\eta$ induit également une bijection $\Sigma(\g;A_M)\to\Sigma(\g^\ast;A_{M^\ast})$, il nous suffit de prouver que 
$r_\alpha(\lambda,A,\o)=r_{\alpha^\ast}(\lambda^\ast,A^\ast,\o^\ast)$ pour tout $\alpha\in\Sigma(\g;A_M)$, i.e.
\[|\alpha(A)|^{\rho(\alpha,\o)\langle\lambda,\alpha^\vee\rangle}=|\alpha^\ast(A^\ast)|^{\rho(\alpha^\ast,\o^\ast)\langle\lambda^\ast,\alpha^{\ast,\vee}\rangle}.\]
Clairement $|\alpha(A)|=|\alpha^\ast(A^\ast)|$ et $\langle\lambda,\alpha^\vee\rangle=\langle\lambda^\ast,\alpha^{\ast,\vee}\rangle$. Or $\rho(\alpha,\o)\alpha^\vee=\sum_{\beta\in \Sigma(\g;A_{M_1}):\beta|_{A_M}=\alpha}\rho(\beta,0)\beta_{M}^\vee$ (cf. \cite[équation (4) p.193]{MW16}), ici on considère $\beta$ d'abord comme morphisme $A_{M_1}\to \GL_{1,F}$ puis comme élément de $a_{M_1}^\ast$, $\beta_M^\vee$ est la projection orthogonale de $\beta^\vee$ sur $a_M$. Grâce à la proposition \ref{prop:induitecommuteautransfert}, il nous suffit de prouver $\rho(\beta,0)=\rho(\beta^\ast,0)$. Ce nombre réel positif $\rho(\beta,0)$ est explicitement évalué dans \cite{YDL23a}, on en déduit l'égalité voulue. 
\end{proof}

\section{Transfert local étendu}\label{sec:analyselocale}
Tout au long de cette section, $F$ désigne un corps $p$-adique.

Partant du transfert régulier semi-simple (théorème \ref{prop:deftransfertfon}), nous montrerons le transfert de toutes les intégrales orbitales pondérées locales requises dans la formule des traces (théorème \ref{pro:corrloc}).

\subsection{Normalisations des mesures de la section \ref{sec:analyselocale}}
On suit les mêmes consignes données au numéro \ref{subsec:localnormalisationsdemesures} pour les normalisations des mesures. 

\subsection{Transformée de Fourier des intégrales orbitales}

On définit par dualité la transformée de Fourier $\widehat{D}$ d’une distribution $D$ sur $\S(\g(F))$. Soit $X\in \g(F)$, Harish-Chandra a montré que la distribution $\widehat{J}_G^G(X,\cdot)$ est localement intégrable (\cite[théorème 4.4]{HCbook}) : il existe une fonction 
\[\widehat{j}_G^G:\g(F)\times \g_\rss(F)\rightarrow \C,\] 
localement constante et invariante par $G(F)$-conjugaison sur la première variable telle que pour tout $f\in\S(\g(F))$
\[\widehat{J}_G^G(X,f)=\int_{\g(F)}|D^\g(Y)|^{-1/2}f(Y)\widehat{j}_G^G(X,Y)\,dY.\]


\begin{lemma}Pour tous $A,B\in\g_\ad(F)$ et $Z,W\in \mathfrak{z}(F)$, on a
\begin{equation}\label{eq:widehatjcalcul}
\widehat{j}_G^G(A+Z,B+W)=\psi(\langle Z,W\rangle)\widehat{j}_{G}^{G}(A,B). 
\end{equation}    
\end{lemma}
\begin{proof}
Soient $g\in \S(\g_\ad(F))$ et $h\in \S(\mathfrak{z}(F))$. Soit $f\in \S(\g(F))$ définie par $f(A+Z)=g(A)h(Z)$ pour tous $A\in\g_\ad(F)$ et $Z\in\mathfrak{z}(F)$. On a, sachant que $
\g=\mathfrak{z}\oplus\g_\ad$ est orthogonale, $\widehat{f}=\widehat{g}^{\g_\ad} \widehat{h}^{\mathfrak{z}}$. D'où, pour tous $A\in\g_\ad(F)$ et $Z\in\mathfrak{z}(F)$,
\[\widehat{J}_G^G(A+Z,f)=\int_{(B,W)\in \g_\ad(F)\oplus \mathfrak{z}(F)}|D^\g(B)|^{-1/2}g(B)h(W)\widehat{j}_G^G(A+Z,B+W)\,dB\,dW\]
et
\begin{align*}
\widehat{J}_G^G(A+Z,f)=J_G^G(A+Z,\widehat{f})=J_{G_\ad}^{G_\ad}(A,\widehat{g}^{\g_\ad})\widehat{h}^{\mathfrak{z}}(Z)=  \widehat{J}_{G_\ad}^{G_\ad}(A,g)\widehat{h}^{\mathfrak{z}}(Z)\\
=\int_{\g_\ad(F)}|D^\g(B)|^{-1/2}g(B)\widehat{j}_{G_\ad}^{G_\ad}(A,B)\,dB \int_{\mathfrak{z}(F)}h(W)\psi(\langle W,Z\rangle)\,dW.
\end{align*}
On en conclut que pour tous $A,B\in\g_\ad(F)$ et $Z,W\in \mathfrak{z}(F)$, il y a l'égalité $\widehat{j}_G^G(A+Z,B+W)=\psi(\langle Z,W\rangle)\widehat{j}_{G_\ad}^{G_\ad}(A,B)$. En particulier $\widehat{j}_G^G(A,B)=\widehat{j}_{G_\ad}^{G_\ad}(A,B)$.
\end{proof}

\subsection{Développement en germes de Shalika}

Pour toute sous-algèbre de Lie $\mathfrak{h}$ de $\g$, on note $\mathfrak{h}_{G-\rss}\eqdef \mathfrak{h}\cap\g_{\rss}$. 

On note $\mathcal{N}_G(F)$ l'ensemble des éléments nilpotents de $\g(F)$ et l'on fixe un ensemble $(\mathcal{N}_G(F))$ des représentants des orbites pour l'action adjointe de $G$ dans $\mathcal{N}_G(F)$. Dans l'énoncé ci-dessous, on raisonne par récurrence sur la dimension de $G$. On suppose donc que pour tout $L\in\L^G(M)$, $L \not =G$, on a défini des objets analogues à ceux que l'on va définir et qu'ils vérifient des propriétés analogues à celles que l'on va démontrer.

\begin{proposition}[Germes de Shalika]\label{prop:Shalika}~{}
\begin{enumerate}
    \item (\cite[III.7]{Walds95}) Soient $M\in \L^G,$ $Q\in\F^G(M)$ et $M_Q$ l'unique l’unique facteur de Levi de $Q$ contenant $M$. Soit $X$ un élément semi-simple de $\m(F)$. Pour tout $\nu\in (\mathcal{N}_{(M_Q)_X}(F))$ il existe une fonction $g_M^{M_Q}(\cdot,X+\nu)$ définie sur $\{Y\in \m_X(F): X+ Y\in \m_{G-\rss}(F)\}$, dont le germe en $0$ est uniquement determiné, de sorte que pour toute $f\in\S(\mathfrak{g}(F))$, il existe un voisinage $V_f$ de $0$ dans $\m_X(F)$ tel que pour tout $Y\in V_f$ vérifiant $X +Y\in\m_{G-\rss}(F)$, on ait l'égalité
    \[J_M^Q(X+Y,f)=\sum_{L\in \L^{M_Q}(M)}\sum_{\nu\in (\mathcal{N}_{L_X}(F))}g_M^L(Y,X+\nu)J_L^Q(X+\nu,f).\]
    \item (\cite[III.8]{Walds95}) Reprenons les notations du point 1. Supposons $M_X=M_{Q,X}$, alors pour tout $\nu \in (\mathcal{N}_{M_{Q,X}}(F))$, le germe $g_M^{M_Q}(\cdot,X+\nu)$ au voisinage de 0 est nul si $M\not =M_Q$. 
    \item (\cite[III.8]{Walds95}, \cite[théorème 27.5]{Kottbook}) Indépendance linéaire des germes : reprenons les notations du point 1, 
    la famille $(g_M^M(\cdot,X+\nu))_{\nu\in (\mathcal{N}_{M_{X}}(F))}$ est linéairement indépendente sur $\{Y\in \m_X(F): X+ Y\in \m_{G-\rss}(F)\}$.

    \noindent (\cite[III.8]{Walds95}, \cite[théorème 27.5]{Kottbook}) Soit $\nu \in (\mathcal{N}_{M_{X}}(F))$, 
    alors $g_M^{M}(\cdot,X+\nu)=g_{M_X}^{M_X}(\cdot,\nu)$
    \item (\cite[III.9]{Walds95}) Homogénéité : reprenons les notations du point 1 et supposons que $X=0$. On écrit $\rk G$ pour le rang de $G$. On note $d^G(u)=\frac{1}{2}(\dim G_u-\rk G)$. Pour $d\in \mathbb Z$ et $N\in\mathbb N$, notons $\mathfrak{X}^{d,N}(\g_\rss(F))$ l'ensemble des fonctions $\phi:\g_\rss(F)\rightarrow \C$ telles qu'il existe pour tout $i\in \{0,\dots,N\}$ une fonction $\phi_i:\g_\rss(F)\rightarrow \C$ de sorte que pour tous $Y\in \g_\rss(F)$ et $t\in F^\times$, on ait l'égalité 
    \[\phi(tY)=|t|^d\sum_{i=0}^N\phi_i(Y)(\log |t|)^i.\]
    Alors pour tous $M\in\L^G(M_0)$ et $\nu\in (\mathcal{N}_G(F))$, il existe un unique élément de $\mathfrak{X}^{d^{G}(\nu),\dim (a_M^G)}(\g_\rss(F))$ ayant même germe au voisinage de $0$ que $g_M^{G}(\cdot,X+\nu)$.
    
\end{enumerate}
\end{proposition}

\subsection{Intégrales orbitales pondérées}
\begin{definition}[Triplet géométrique]\label{def:trplet} Pour tous $M\in \L^G$, $Q\in \F^G(M)$ et $X\in \m_{G-\rss}(F)$, on appelle $(M,Q,X)$ un triplet géométrique de $G$. De même, pour tous $M'\in \L^{G^\ast}$, $Q'\in \F^{G^\ast}(M')$ et $X'\in \m_{G^\ast-\rss}'(F)$, on appelle $(M',Q',X')$ un triplet géométrique de $G^\ast$. 
\end{definition}

On affuble $G$ d'un chapeau lorsqu'on parle de son dual de Langlands (\cite[p.29]{Bor79}), i.e. le groupe des points complexes du groupe réductif sur $\overline{F}$ dont la donnée radicielle est duale de celle de $G$, muni de son action usuelle de $\Gamma\eqdef\Gal(\overline{F}/F)$. Pour $A$ un groupe admettant une action de $\Gamma$ par automorphismes, on note $A^\Gamma$ le sous-groupe des points fixes. On note $H_\text{ab}^1(F,H)$ le premier groupe de cohomologie abélianisée d'un groupe réductif connexe $H$ (\cite[définition 2.2]{Boro98}). La cohomologie abélianisée du groupe $G/A$, avec $G$ un groupe du type GL et $A\subseteq Z(G)$, est l'hypercohomologie galoisienne du complexe $[Z(G_\der)(\overline{F})\to (Z(G)/A)(\overline{F})]$ avec $Z(G_\der)(\overline{F})$ en degré $-1$ et $(Z(G)/A)(\overline{F})$ en degré $0$.

On dispose d'applications fonctorielles d'abélianisation 
\[\text{ab}_H^i:H^i(F,H)\rightarrow H_\text{ab}^i(F,H)\]
pour $i=0,1$ et d'une application canonique 
\[\alpha_H:H_\text{ab}^1(F,H)\rightarrow\pi_0(Z(\widehat{H})^\Gamma)^D,\]
avec $\pi_0$ le groupe des composantes connexes et $D$ le dual de Pontryagin. Les applications $\text{ab}_H^1$ et $\alpha_H$ sont bijectives au motif que $F$ est local non-archimédien (\cite[corollaire 5.4.1]{Boro98}). 

\begin{theorem}[Transfert géométrique d'une fonction {{\cite[théorème 7.2]{Ch07}}}]\label{prop:deftransfertfon}  Pour toute fonction $f\in\S(\g(F))$ il existe une fonction $f^\ast\in\S(\g^\ast(F))$ telle que pour tout triplet géométrique $(M',Q',X')$ de $G^\ast$ on ait
\begin{align*}
    J_{M'}^{Q'}(X',f^\ast)=\begin{cases*}
    J_M^Q(X,f) & \text{si $(M',Q',X')= (M^\ast,Q^\ast,X^\ast)$ ;} \\
    0 & \text{si $Q'$ ne se transfère pas à $G$.}
\end{cases*}
    \end{align*}
Dans la première ligne, la relation $X\arr X^\ast$ est définie via la restriction $\eta|_{\m}:\m\to \m^\ast$.
\end{theorem}
Pour les fonctions $f$ et $f^\ast$ vérifiant le théorème, on note $f\underset{\geom}{\arr}f^\ast$, ou plus simplement $f\arr f^\ast$.
\begin{remark}
Si $\S(\g(F))\ni f\arr f^\ast\in\S(\g^\ast(F))$ alors $\S(\m(F))\ni f_P\arr f_{P^\ast}^\ast\in\S(\m^\ast(F))$ pour tous $M\in \L^G$ et $P\in \P^G(M)$.   
\end{remark}
\begin{proof}
La version groupe du théorème est une conséquence de \cite[théorème 7.2]{Ch07} : il suffit de noter que $X'$ est potentiellement une norme de $M_{Q'}$ (\textit{Ibid.} définition 5.2) revient à dire que la classe de $\eta$ appartient à l'image de
\[H_{\text{ab}}^1(F,M_{\AD}')\rightarrow H_{\text{ab}}^1(F,G_{\ad}^\ast).\]
Or l'application d'abélianisation $H^1\rightarrow H_{\text{ab}}^1$ est une bijection ici, on termine la preuve grâce à la proposition \ref{prop:transfertdef}.

On explique rapidement comment déduire l'énoncé sur l'algèbre de Lie de celui sur le groupe. Fixons un $\mathcal{U}$ voisinage fermé de $0$ dans $\g(F)$ stable par $G(F)$ et un $\mathcal{V}$ voisinage fermé de $1$ dans $G(F)$ stable par $G(F)$, tels que la fonction exponentielle $\exp$ induit un homéomorphisme entre $\mathcal{U}$ et $\mathcal{V}$ (\cite[lemme 2.6]{HCbook}). Prenons $\mathcal{U}^\ast$ un voisinage fermé de $0$ dans $\g^\ast(F)$ stable par $G^\ast(F)$ contenant
\[\{X^\ast\in \g^\ast(F)\text{ se transfère à un élément de }\mathcal{U}\},\] 
et $\mathcal{V}^\ast$ un voisinage fermé de $1$ dans $G^\ast(F)$  stable par $G^\ast(F)$ contenant
\[\mathcal{V}^\ast=\{x^\ast\in G^\ast(F)\text{ se transfère à un élément de }\mathcal{V}\}.\]
Quitte à remplacer $f(\cdot)$ par $f(t\cdot)$ avec $t\in F^\ast$ de valeur absolue assez grande, on peut supposer que le support de $f$ est inclus dans $\mathcal{U}$. Il existe alors $\phi\in C_c^\infty(G(F))$ tel que $f(\cdot)=\phi(\exp(\cdot))$ sur $\mathcal{U}$, le support de $\phi$ est inclus dans $\mathcal{V}$. Il existe alors $\phi^\ast\in C_c^\infty(G^\ast(F))$ vérifiant les égalités en question pour le groupe. Quitte à multiplier $\phi^\ast$ par la fonction caractéristique de $\mathcal{V}^\ast$ on peut supposer que le support de $\phi^\ast$ est dans $\mathcal{V}^\ast$. Il existe alors une fonction $f^\ast\in \S(\g(F))$ voulue.
\end{proof}

Rappelons qu'un élément $X$ de $\g(F)$ est dit ($F$-)elliptique s'il est semi-simple et $A_G = A_{G_X}$. Une classe de $G(F)$-conjugaison dans $\g(F)$ est dite ($F$-)elliptique si un (donc tout) élément dedans est ($F$-)elliptique.

\begin{lemma}\label{lem:elltrans}
Si $X'\in\g_\ss^\ast(F)$ est elliptique, alors $X'$ se transfère à $G$.     
\end{lemma}
\begin{proof}
Cela vient de la proposition \ref{prop:elltransproof}. Sinon pour un argument plus général on peut également se référer à \cite[10.1. lemme]{Kott86}.
\end{proof}

\begin{lemma}\label{lem:nesetransfertpasalorsdansunLevinesetransfertpas}
Pour tout élément $X'\in \g^\ast(F)$ qui ne se transfère pas à $G$, il existe $X_1'\in (\Ad G^\ast(F))X'$ et $L'\in \L^G(M_0)$ ne se transférant pas à $G$ tels que $X_{1,\ss}'$ est un élément elliptique de $ \mathfrak{l}$.    
\end{lemma}

\begin{proof}
On prend $X_1'\in (\Ad G^\ast(F))X'$ et $L'\in \L^G(M_0)$ tels que $X_{1,\ss}'$ soit elliptique dans $\mathfrak{l}'(F)$ et $X_1'\in \Ind_{L'}^{G^\ast}(X_{1,\ss}')$ selon la proposition \ref{prop:suppellitique}. Si $L'=L^\ast$ se transfère alors $(\Ad L^\ast(F))X_1'$ se transfère à $L$ via $\eta|_{\mathfrak{l}}$ puisque c'est une classe de conjugaison elliptique, donc $(\Ad G^\ast(F))X_1'=\Ind_{L^\ast}^{G^\ast}(X_{1,\ss}')$ se transfère à $G$, une contradiction.
\end{proof}

\begin{proposition}\label{pro:localvanishing} Soit $\S(\g(F))\ni f\arr f^\ast\in \S(\g^\ast(F))$. 
\begin{enumerate}
    \item Soit $X'\in \g_{\rss}^\ast(F)$ un élément qui ne se transfère pas à $G$. Alors pour tous $M'\in \L^{G^\ast}$ et $P'\in\P^{G^\ast}(M')$ tels que $X'\in \m'(F)$, on a 
    \[J_{M'}^{P'}(X',f^\ast)=J_{G^\ast}^{G^\ast}(X',f^\ast)=0.\]
    \item Soit $M'\in \L^{G^\ast}$. Soit $Q'\in \F^{G^\ast}(M')$ un sous-groupe parabolique qui ne se transfère pas à $G$. Alors
    \[J_{M'}^{Q'}(X',f^\ast)=0,\,\,\,\,\forall X'\in\m'(F).\]
\end{enumerate}
\end{proposition}
\begin{proof}~{}
\begin{enumerate}
    \item La première égalité résulte de la définition (équation \eqref{YDLgeomeq:IOPdefcaseequising}). Il nous suffit d'établir l'annulation $J_{G^\ast}^{G^\ast}(X',f^\ast)=0.$ Comme $X'$ ne se transfère pas à $G$, il existe $X_1'\in (\Ad G^\ast(F))X'$ et $L'\in \L^G(M_0)$ ne se transférant pas à $G$ tels que $X_1'\in \mathfrak{l}'(F)$. Soit $P'\in\P^{G^\ast}(L')$. Alors $P'$ ne se transfère pas. Par le biais du point 2 du théorème \ref{prop:deftransfertfon} on obtient $J_{L'}^{P'}(X_1',f^\ast)=J_{G^\ast}^{G^\ast}(X_1',f^\ast)=0.$
    
    \item Procédons par récurrence sur $\dim M_{Q'}-\dim M'$ où $M_{Q'}$ est l’unique facteur de Lévi de $Q'$ contenant $M'$. Nous allons d'abord montrer l'annulation voulue pour tout $G$ du type GL, $G^\ast$ sa forme intérieure quasi-déployée et $Q',M'$ tels que $\dim M_{Q'}-\dim M'=0$ : soit $X'\in \m'(F)$, écrivons $X'=X_{\ss}'+X_\nilp'$ sa décomposition de Jordan. Par le développement en germes de Shalika nous avons
    \[J_{M'}^{Q'}(X_{\ss}'+Y',f')=\sum_{\nu'\in (\mathcal{N}_{M_{X_{\ss}'}'}(F))}g_{M'}^{M'}(Y',X_{\ss}'+\nu')J_{M'}^{Q'}(X_{\ss}'+\nu',f')\]
    pour tout $Y'$ dans $V_f$ un voisinage de $0$ dans $\m'_{X_{\ss}'}(F)$ vérifiant $X_{\ss}' +Y'\in(\m_{Q'})_{G^\ast-\rss}(F)$. On sait, selon le point précédent, que l'intégrale orbitale de gauche s'annule, puis par l'indépendance linéaire des germes (point 3 de la proposition \ref{prop:Shalika}) on en déduit que $J_{M'}^{Q'}(X',f^\ast)=J_{M'}^{Q'}(X_{\ss}'+X_\nilp',f')=0$.

    Soit ensuite $n$ un entier strictement positif et supposons que l'annulation ait lieu pour tous $Q',M'$ tels que $Q'$  ne se transfère pas et $\dim M_{Q'}-\dim M'<n$. Maintenant si $Q',M'$ sont tels que $Q'$ ne se transfère pas et $\dim M_{Q'}-\dim M'=n$ et $X'\in\m'(F)$, on écrit
    \[J_{M'}^{Q'}(X_{\ss}'+Y',f^\ast)=\sum_{L'\in \L^{M_{Q'}}(M')}\sum_{\nu'\in (\mathcal{N}_{L_{X_{\ss}'}'}(F))}g_{M'}^{L'}(Y',X_{\ss}'+\nu')J_{L'}^{Q'}(X_{\ss}'+\nu',f^\ast)\]
    pour $Y'$ dans $V_f$ un voisinage de $0$ dans $\m'_{X_{\ss}'}(F)$ vérifiant $X_{\ss}' +Y'\in(\m_{Q'})_{G^\ast-\rss}(F)$.
    Alors $J_{L'}^{Q'}(X_{\ss}'+\nu,f^\ast)=0$ pour tout $L'\not= M'$ par l'hypothèse de récurrence, d'où
    \[J_{M'}^{Q'}(X_{\ss}'+Y',f^\ast)=\sum_{\nu'\in (\mathcal{N}_{M_{X_{\ss}'}'}(F))}g_{M'}^{M'}(Y',X_{\ss}'+\nu')J_{M'}^{Q'}(X_{\ss}'+\nu',f^\ast).\]
    On peut donc conclure la preuve par l'argument précédent.\qedhere
\end{enumerate}
\end{proof}

\begin{lemma}\label{lem:polyproche}
On dit qu'un polynôme $P\in F[T]$ est de type de décomposition $(n_1,\dots,n_r)$ s'il se décompose en un produit de polynômes irréductibles de degrés $n_1,\dots,n_r$. 

Soit $P\in F[T]$ sans racine multiple. Tout polynôme suffisamment proche de $P$ est de même type de décomposition que $P$. Ici « suffisamment proche » porte le sens suivant : l'espace des polynômes de degré $\leq \text{deg}(P)$ étant un espace vectoriel de dimension finie sur $F$, il est muni naturellement d'une norme issue de celle de $F$.
\end{lemma}
\begin{proof}
Cela résulte de la continuité des racines dans une clôture algébrique et du lemme de Krasner. 
\end{proof}

\begin{corollary}\label{coro:transferopen}
L'ensemble des éléments de $\g_{\rss}^\ast(F)$ qui se tranfèrent à $G$ est ouvert. De même, l'ensemble des éléments de $\g_{\rss}^\ast(F)$ qui ne se tranfèrent pas à $G$ est aussi ouvert.
\end{corollary}

\begin{proof}
Une classe de conjugaison semi-simple régulière est complètement caractérisée par son polynôme caractéristique, la propriété souhaitée s'ensuit de ce fait à la lumière du lemme précédent et du théorème \ref{thm:chara}.
\end{proof}

\begin{lemma}[{{\cite[lemme 5.5]{Ch07}}}]
Soient $M'\in\L^{G^\ast}(M_{0^\ast})$ et $L_1',L_2'\in\L^{G^\ast}(M')$ tels que le coefficient d'Arthur
\[d_{M'}^{G^\ast}(L_1',L_2')\]
soit non nul. Alors le diagramme commutatif suivant est cartésien
\[
    \begin{tikzcd}
    H_{\mathrm{ab}}^1(F,M_{\AD}') \arrow[rightarrow]{r}{}\arrow[rightarrow]{d}{}& H_{\mathrm{ab}}^1(F,L_{1,\AD}')\arrow[rightarrow]{d}{}\\
    H_{\mathrm{ab}}^1(F,L_{2,\AD}') \arrow[rightarrow]{r}{} & H_{\mathrm{ab}}^1(F,G_{\ad}^\ast)
    \end{tikzcd}
\]
\end{lemma}
\begin{corollary}\label{coro:artcoeff}
Soient $M'\in\L^{G^\ast}(M_{0^\ast})$ et $L_1^\ast
,L_2^\ast\in\L^{G^\ast}(M')$ deux sous-groupes de Levi qui se transfèrent, tels que $d_{M'}^{G^\ast}(L_1^\ast,L_2^\ast)\not =0$. Alors $M'$ se transfère.
\end{corollary}
\begin{proof}
D'après la proposition \ref{prop:transfertdef}, la classe $\eta$ est dans les images des flèches $H^1(F,L_{i,\AD}^\ast)\rightarrow H^1(F,G_{\ad}^\ast)$. On évoque la définition du produit fibré de groupes pour conclure que $\eta$ est pareillement dans l'image de $H^1(F,M_{\AD}')\rightarrow H^1(F,G_{\ad}^\ast)$.
\end{proof}

Pour $X\in\g(F)$ un élément dans l'algèbre de Lie. La restriction de $\langle -,-\rangle$ à $\g_{X_\ss}(F)$ est la forme bilinéaire canonique de $\g_{X_\ss}(F)$. Posons 
\begin{equation}\label{eq:signe}
e^G(X)\eqdef
\frac{\gamma_{\psi}(\g_{X_\ss^\ast}^\ast)}{\gamma_\psi(\g_{X_\ss})}.    
\end{equation}

\begin{lemma}\label{lem:quotientofgermsShalika}Soient $\g(F)\ni X\arr X^\ast\in\g^\ast(F)$ deux éléments semi-simples. Alors pour tous $\g_X(F)\ni Y\arr Y^\ast\in\g_{X^\ast}^\ast(F)$ (défini par le torseur intérieur $\eta_X:\g_X\to \g_{X^\ast}^\ast$) proches de 0 avec $X+tY\in\g_\rss(F)$ et $X^\ast+tY^\ast\in \g_\rss^\ast(F)$ pour tout scalaire $t$ satisfaisant $0<|t|\leq 1$ on a
\begin{equation}\label{eq:quotientgermsShalika}
g_G^G(Y,X)=e^G(X)g_{G^\ast}^{G^\ast}(Y^\ast,X^\ast).    
\end{equation}
\end{lemma}

\begin{proof}On parvient au lemme en maniant des propriétés classiques. On prouve le résultat suivant en manipulant le théorème de finitude de Howe :

\begin{lemma}[{{\cite[proposition VI. 4.]{Walds95}}}]
Soit $X\in\g_{\ss}(F)$. Soit $\mathfrak{r}$ un sous-groupe ouvert compact de $\g$. II existe un voisinage $V$ de $0$ dans $\mathfrak{m}_X$ tel que pour tout $Y\in V$ vérifiant $X+Y\in\g_\rss(F)$ et tout $Z\in \mathfrak{r}\cap \g_\rss(F)$, on ait l'égalité
\[\widehat{j}_G^G(X+Y,Z)=\sum_{\nu\in (\mathcal{N}_{G_{X}}(F))}g_G^G(Y,X+\nu)\widehat{j}_G^G(X+\nu,Z).\]
\end{lemma}
On rappelle l'identité fondamentale du transfert dans notre cadre :
\begin{theorem}[{{\cite[11.3. théorème]{Walds97}}}] Si $X\arr X^\ast$ et $Y\arr Y^\ast$ sont réguliers semi-simples, alors
\[\gamma_{\psi}(\g)\widehat{j}_G^G(X,Y)=\gamma_{\psi}(\g^\ast)\widehat{j}_{G^\ast}^{G^\ast}(X^\ast,Y^\ast).\]
\end{theorem}

On peut maintenant regarder l'égalité recherchée \eqref{eq:quotientgermsShalika}. D'après le point 3 de la proposition \ref{prop:Shalika} et nos normalisations des mesures on peut remplacer $G$ par $G_X$ et supposer ainsi que $X\in \mathfrak{z}(F)$. Nous avons des développements en germes 
\begin{align*}
\widehat{j}_G^G(X+tY,Z)&=\sum_{\nu\in (\mathcal{N}_{G}(F))}g_G^G(tY,\nu)\widehat{j}_G^G(X+\nu,Z) \\   \widehat{j}_{G^\ast}^{G^\ast}(X^\ast+tY^\ast,Z^\ast)&=\sum_{\nu'\in (\mathcal{N}_{G^\ast}(F))}g_{G^\ast}^{G^\ast}(tY^\ast,\nu')\widehat{j}_{G^\ast}^{G^\ast}(X^\ast+\nu',Z^\ast)
\end{align*}
pour tout scalaire $t\in F$ avec $0<|t|\leq 1$. Il est clair que $(tY)^\ast=tY^\ast$ et donc $(X+tY)^\ast=X^\ast+tY^\ast$. En multipliant par les facteurs $\gamma_{\psi}(\g)$ et $\gamma_{\psi}(\g^\ast)$ on peut donc mettre une égalité entre les membres de gauche, puis en regardant les termes de degré 0 dans la relation d'homogénéité des germes (proposition \ref{prop:Shalika}) on en déduit
\[\gamma_{\psi}(\g)g_G^G(Y,0)\widehat{j}_G^G(X,Z) =\gamma_{\psi}(\g^\ast)g_{G^\ast}^{G^\ast}(Y^\ast,0)\widehat{j}_{G^\ast}^{G^\ast}(X^\ast,Z^\ast).\]
L'égalité étant valable pour tous $Z\arr Z^\ast$, on peut prendre $Z\in \g_{\ad,G-\rss}(F)$ pour voir que $\widehat{j}_G^G(X,Z)=\widehat{j}_G^G(0,Z)=1$, ici la première égalité résulte de l'équation \eqref{eq:widehatjcalcul} puis la deuxième s'ensuit du fait que la transformée de Fourier du Dirac en 0 est la fonction constante 1. De même $\widehat{j}_{G^\ast}^{G^\ast}(X^\ast,Z^\ast)=1$. Il vient ainsi $\gamma_{\psi}(\g)g_G^G(Y,0) =\gamma_{\psi}(\g^\ast)g_{G^\ast}^{G^\ast}(Y^\ast,0)$, soit $g_G^G(Y,0)=e^G(X)g_{G^\ast}^{G^\ast}(Y^\ast,0)$. 
\end{proof}

\begin{remark}Le rapport $e^G(X)$ coïncide avec le signe $e(G_{X_\ss})=\pm 1$ introduit par Kottwitz en \cite[p.292]{Kott83}. Pour une preuve on pourrait se référer à \cite[remarque après lemme 2.14]{KazhPo03}, où le rapport des constantes de Weil en question est calculé pour tout groupe réductif $G$ à l'aide de la théorie des formes quadratiques.
\end{remark}

\begin{lemma}[{{\cite[lemme 9.6]{HCbook}}}]\label{lem:germenonnulell}
Soit $Y^\ast\in \g_\rss^\ast(F)$ elliptique proche de 0, alors $g_{G^\ast}^{G^\ast}(Y^\ast,0)\not=0$.     
\end{lemma}

\begin{lemma}\label{lem:choixdeYstar}Soient $M^\ast\in \L^{G^\ast}$ et $X^\ast\in\m_{\ss}^\ast(F)$. Alors il existe $Y^\ast\in\m_{X^\ast,\ss}^\ast(F)$ proche de 0, $M_{X^\ast}^\ast$-régulier et elliptique dans $\m_{X^\ast}^\ast(F)$, tel que $X^\ast+tY^\ast\in \m_{G-\rss}^\ast(F)$ pour tout scalaire $t$ satisfaisant $0<|t|\leq 1$.
\end{lemma}
\begin{proof}
Sans perte de généralité on peut supposer que $G^\ast=\GL_{n,F}$ et $M_{X^\ast}^\ast=\prod_i\Res_{E_i/F}\GL_{m_i,E_i}$ avec $E_i/F$ une extension finie et $n=\sum_i m_i[E_i:F]$. On peut trouver $Y^\ast\in \m_{X^\ast,\ss}^{\ast}(F)$ proche de 0, grâce au lemme \ref{lem:polyproche}, tel que 
\begin{enumerate}
    \item le polynôme caractéristique de $Y^\ast$ sur chaque composante $\Res_{E_i/F}\GL_{m_i,E_i}$ est un polynôme irréductible de degré $m_i[E_i:F]$, et ces polynômes caractéristiques sont premiers entre eux ; et
    \item pour tout $t$ avec $0<|t|\leq 1$, le polynôme caractéristique de $X^\ast+tY^\ast$ sur chaque composante $\Res_{E_i/F}\GL_{m_i,E_i}$ est un polynôme irréductible de degré $m_i[E_i:F]$, et ces polynômes caractéristiques sont premiers entre eux.
\end{enumerate}
L'élément $Y^\ast$ est régulier et elliptique dans $\m_{X^\ast}^\ast(F)$ d'après la première condition. Puis $X^\ast+tY^\ast$ est $G^\ast$-régulier d'après la seconde condition et le fait que $n=\sum_i m_i[E_i:F]$.
\end{proof}

Pour tous $M\in\L^G$ et $X\in \m(F)$, on a que $\langle -,-\rangle|_{\g_{X_\ss}(F)}$ est la somme orthogonale de $\langle -,-\rangle|_{\mathfrak{m}_{X_\ss}(F)}$ et d'une forme quadratique déployée, d'où $\gamma_\psi(\g_{X_\ss})=\gamma_\psi(\mathfrak{m}_{X_\ss})$. De même l'égalité analogue évidente pour $G^\ast,M^\ast$ et $X_\ss^\ast$ a lieu. On en déduit $e^M(X)=e^G(X)$. 

\begin{theorem}[Correspondance locale]\label{pro:corrloc} Soit $f\arr f^\ast$. Alors pour tout $M'\in \L^{G^\ast}$, tout $Q'\in \F^{G^\ast}(M')$, et tout $X'\in \m'(F)$, on a 
\begin{align*}
    J_{M'}^{Q'}(X',f^\ast)=\begin{cases*}
    J_M^Q(X,f) & \text{si $M'=M^\ast$, $Q'=Q^\ast$, et $X'=X^\ast$  (via $\eta|_\m$) ;} \\
    0 & \text{si $Q'$ ne se transfère pas ou si $M'$ se transfère et} \\
     & \text{$X'$ ne se transfère pas (via $\eta|_\m$).}
\end{cases*}
    \end{align*}
\end{theorem}


\begin{proof}On prouve d'abord la première ligne puis la deuxième ligne.
\begin{enumerate}
    \item Soient $\m(F)\ni X\arr X^\ast\in\m^\ast(F)$ via la restriction $\eta|_\m :\m\rightarrow \m^\ast$. Soient $L\arr L^\ast$ tels que $X=\Ind_L^M(X_{\ss})$,  $X^\ast=\Ind_{L^\ast}^{M^\ast}(X_{\ss}^\ast)$, et $X_\ss$ est elliptique dans $\mathfrak{l}$ (proposition \ref{prop:suppellitique}). Puisque $L_{X_\ss^\ast}^\ast$ est la forme intérieure quasi-déployée de $L_{X_\ss}$ par restriction de $\eta$, on a $A_{L^\ast}=A_L=A_{L_{X_\ss}}=A_{L_{X_\ss^\ast}^\ast}$, i.e. $X_{\ss}^\ast$ est elliptique dans $\mathfrak{l}^\ast$. La formule de descente de l'induction (proposition \ref{prop:IOP}) nous assure qu'il suffit de prouver la proposition pour $X=X_\ss$ elliptique dans $\m^\ast(F)$. On suppose donc dans la suite que $X=X_\ss$, et que cet élément est elliptique dans $\m^\ast(F)$.

    On regarde dans un premier temps le cas où $G_{X}=M_{X}$ : on a alors $G_{X^\ast}^\ast=M_{X^\ast}^\ast$. On a
    $L_{X}=M_{X}$ pour tout Levi $L\in\L^{G}(M)$, aussi $L_{X^\ast}^\ast=M_{X^\ast}^\ast$ pour tout $L^\ast\in\L^{G^{\ast}}(M^\ast)$. On écrit, à l'aide de l'annulation de germes (point 2 de la proposition \ref{prop:Shalika}), 
    \begin{align*}
    J_{M}^{Q}(X+tY,f)&=\sum_{\nu\in (\mathcal{N}_{M_{X}}(F))}g_{M}^{M}(tY,X+\nu)J_{M}^{Q}(X+\nu,f) \\   
    J_{M^\ast}^{Q^\ast}(X^\ast+tY^\ast,f^\ast)&=\sum_{\nu'\in (\mathcal{N}_{M_{X^\ast}^\ast}(F))}g_{M^\ast}^{M^\ast}(tY^\ast,X^\ast+\nu')J_{M^\ast}^{Q^\ast}(X^\ast+\nu',f^\ast).
    \end{align*}
    pour tous $\m_X(F)\ni Y\arr Y^\ast\in \m_{X^\ast}^\ast(F)$ et $t\in F$ avec $0<|t|\leq 1$ pourvu que $X+tY\in \m_{G-\rss}(F)$ et $X^\ast+tY^\ast\in \m_{G^\ast-\rss}^\ast(F)$. Il est clair que $(tY)^\ast=tY^\ast$, on peut donc mettre une égalité aux membres de gauche. Puis en regardant les termes de degré 0 par dans les formules d'homogénéité des germes on en déduit
    \[g_M^M(Y,X)J_M^Q(X,f) =g_{M^\ast}^{M^\ast}(Y^\ast,X^\ast)J_{M^\ast}^{Q^\ast}(X^\ast,f^\ast).\]
    On prend ensuite $Y^\ast$ comme dans le lemme \ref{lem:choixdeYstar}. Alors avec l'appui du lemme \ref{lem:germenonnulell} et du point 3 de la proposition \ref{prop:Shalika} on voit que $g_{M^\ast}^{M^\ast}(Y^\ast,X^\ast)=g_{M_{X^\ast}^\ast}^{M_{X^\ast}^\ast}(Y^\ast,0)\not =0$. D'où, selon le lemme \ref{lem:quotientofgermsShalika},
    \[J_{M^\ast}^{Q^\ast}(X^\ast,f^\ast)=e^{M}(X)J_M^Q(X,f).\]
    
    Abandonnons l'hypothèse que $G_{X}=M_{X}$. L'énoncé est démontré par l'équation \ref{YDLgeomeq:IOPdef} et la proposition \ref{prop:fonctionr}. 
    
    \item Si $Q'$ ne se transfère pas on tombe sur le point 2 de la proposition \ref{pro:localvanishing}. Sinon on suppose que $(M',Q')=(M^\ast,Q^\ast)$ et que $X'$ ne se transfère pas. On prend $X_{1}'\in (\Ad M^\ast(F))X'$ et $L'\in\L^{M^\ast}(M_{0^\ast})$ qui ne se transfère pas tels que $X'=\Ind_{L'}^{M^\ast}(X_{1,\ss}')$ avec $X_{1,\ss}'$ elliptique dans $\mathfrak{l}'$ au moyen du lemme \ref{lem:nesetransfertpasalorsdansunLevinesetransfertpas}. La formule de descente de l'induction nous dit que
    \begin{align*}
    J_{M^\ast}^{Q^\ast}(X',f^\ast)&=J_{M^\ast}^{Q^\ast}(\Ind_{L'}^{M^\ast}(X_{1,\ss}'),f^\ast)\\
    &=\sum_{L_1'\in\L^{M_{Q^\ast}^\ast}(L')}d_{L'}^{M_{Q^\ast}^\ast}(L_1',M^\ast)J_{L'}^{Q_{L_1'}'}(X_{1,\ss}',f^\ast).    
    \end{align*}
    Soit $L_1'\in\L^{M_{Q^\ast}^\ast}(L')$. Si $J_{L'}^{Q_{L_1'}'}(X_{1,\ss}',f^\ast)\not =0$, alors le point 2 de la proposition \ref{pro:localvanishing} implique que $Q_{L_1'}'$ se transfère. Donc $L_1'=L_1^\ast$ se transfère. Le corollaire \ref{coro:artcoeff} nous assure de ce fait $d_{L'}^{M_{Q^\ast}^\ast}(L_1^\ast,M^\ast)=0$. En conclusion $J_{M^\ast}^{Q^\ast}(X',f^\ast)=0$. 
    \qedhere
\end{enumerate}
\end{proof}

\begin{corollary}
Soit $f\in \S(\g(F))$. Si $J_M^Q(X,f)=0$ pour tous $M\in \L^G(M)$, $Q\in\F^G(M)$ et $X\in\m_{G-\rss}(F)$, alors $J_M^Q(X,f)=0$ pour tous $M\in \L^G(M)$, $Q\in\F^G(M)$ et $X\in\m(F)$.   
\end{corollary}
\begin{proof}
On a $f\arr 0$. On conclut la preuve au moyen de la proposition précédente.     
\end{proof}

\begin{corollary}
Soit $M\arr M^\ast$. Pour tous $\m_{\ss}(F)\ni X\arr X^\ast\in\m_{\ss}^\ast(F)$,  $(\mathcal{N}_{G_{X}}(F))\ni \nu\arr \nu^\ast\in(\mathcal{N}_{G_{X^\ast}^\ast}(F))$, et $\m_X(F)\ni Y\arr Y^\ast\in \m_{X^\ast}^\ast(F)$ avec $X+Y\in \g_\rss(F)$ et  $X^\ast+Y^\ast\in \g_\rss^\ast(F)$, on a l'égalité des germes 
\[e^{G}(X)g_{M^\ast}^{G^\ast}(Y^\ast,X^\ast+\nu^\ast)=g_M^{G}(Y,X+\nu).\]
\end{corollary}
\begin{remark}
Cela est un renforcement de \cite[proposition 7.1]{AC}.
\end{remark}
\begin{proof}
Soit $n\geq 0$ un entier, on suppose que l'assertion est établie pour $\dim G -\dim M<n$. Alors pour $\dim G -\dim M=n$, on a, selon les hypothèses de la récurrence, la proposition \ref{pro:corrloc} et le fait que $e^M(X)=e^G(X)$,
\begin{align*}
    0&=J_{M}^{G}(X+Y,f)-J_{M^\ast}^{G^\ast}(X^\ast+Y^\ast,f^\ast)\\
    &=\sum_{\nu\in (\mathcal{N}_{M_{X}}(F))}g_{M}^{G}(Y,X+\nu)J_{G}^{G}(X+\nu,f)-\sum_{\nu'\in (\mathcal{N}_{M_{X^\ast}^\ast}(F))}g_{M^\ast}^{G^\ast}(Y^\ast,X^\ast+\nu')J_{G^\ast}^{G^\ast}(X^\ast+\nu',f^\ast)\\
    &=\sum_{\nu\in (\mathcal{N}_{M_{X}}(F))}(g_{M}^{G}(Y,X+\nu)-e^G(X)g_{M^\ast}^{G^\ast}(Y^\ast,X^\ast+\nu^\ast))J_{G}^{G}(X+\nu,f),
\end{align*}
pour tous $X$, $X^\ast$ etc comme dans l'énoncé, et toute fonction $f\in \S(\g(F))$ (et $f\arr f^\ast$). On conclut la preuve par l'indépendence des intégrales orbitales (point 7 de la proposition \ref{prop:IOP}).
\end{proof}

\subsection{Commutation du transfert avec la transformée partielle de Fourier}
Travaillons avec $G_\ad$ et $G_\ad^\ast$ dans cette sous-section. Notons que les notions de $G_\ad(F)$-conjugaison et de $G_\ad(\overline{F})$-conjugaison coïncident dans $\g_{\ad}(F)$. \`{A} titre de rappel si $H$ un sous-groupe de $G$ contenant $Z(G)$, on écrit $H_\AD$ pour le sous-groupe $H/Z(G)$ de $G_\ad$. On a $M_{0,\AD}$ un sous-groupe de Levi minimal de $G_\ad$. Au vue de la bijection évidente entre $\L^G(M_0)$ (resp. $\P^G(M_0)$ ; $\F^G(M_0)$) et $\L^{G_\ad}(M_{0,\AD})$ (resp. $\P^{G_\ad}(M_{0,\AD})$ ; $\F^{G_\ad}(M_{0,\AD})$), on peut définir dans la même veine que plus haut la notion du transfert d'une fonction dans $\S(\g_{\ad}(F))$ : on remplace les objects relativement à $G$ (resp. $G^\ast$) par leur quotient par $Z(G)$ (resp. $Z(G^\ast)$) dans la définition \ref{def:trplet}, puis on dit que $\S(\g_{\ad}(F))\ni h\arr h^\ast\in \S(\g_{\ad}^\ast(F))$ si la relation dans le théorème \ref{prop:deftransfertfon} sont vérifiées pour les sous-groupes de Levi, les sous-groupes paraboliques, et les éléments semi-simples réguliers dans les sous-algèbres de Levi de $G_\ad$ et $G_\ad^\ast$.

On signale que, selon les normalisations des mesures sur les objets en relation avec $G_\ad$ et $G_{\ad}^\ast$, si $\S(\g(F))\ni f\arr f^\ast\in \S(\g^\ast(F))$ alors $\S(\g_\ad(F))\ni f|_{\g_\ad}\arr f^\ast|_{\g_{\ad}^\ast}\in \S(\g_{\ad}^\ast(F))$.

\begin{proposition}[{{\cite[théorème 5.2.]{Ch05}}}]
Soit $\S(\g_\ad(F))\ni h\arr h^\ast\in \S(\g_{\ad}^\ast(F))$, alors
\[\gamma_\psi(\g_\ad)\widehat{h}^{\g_\ad}\arr \gamma_\psi(\g_\ad^\ast)\widehat{h^\ast}^{\g_{\ad}^\ast}.\]
\end{proposition}

Pour des raisons techniques nous travaillerons avec la transformée partielle de Fourier sur $\g(F)$ définie comme suit : rappelons que la décomposition $\g(F)=\mathfrak{z}(F)\oplus \g_{\ad}(F)$ est orthonogale par rapport à $\langle-,-\rangle$, et que les mesures sur $\g(F)$ et $\g_\ad(F)$ sont auto-duales. Posons $\pi_{\mathfrak{z}} : \g \to \mathfrak{z}$ la projection de $\g$ sur $\mathfrak{z}$ selon la décomposition ci-dessus. On définit alors la transformée partielle de Fourier d'une fonction $f \in\S(\g(F))$ par
\begin{equation}\label{YDLgeomeq:defpartialFourier}
\widetilde{f}(Y)=\int_{\g_{\ad}(F)}f(\pi_\mathfrak{z}(Y)+X)\psi(\langle X,Y\rangle)\,dX,\,\,\,\,\forall Y\in \g(F).    
\end{equation}
Il est clair, selon nos choix de mesures, que $\widetilde{\widetilde{f}}(Z+Y)=f(Z-Y)$ si $(Z,Y)\in\mathfrak{z}(F)\oplus \g_{\ad}(F)$. 

\begin{proposition}\label{prop:transFourier}
Soit $\S(\g(F))\ni f\arr f^\ast\in \S(\g^\ast(F))$, alors
\[\gamma_\psi(\g_\ad)\widetilde{f}\arr \gamma_\psi(\g_\ad^\ast)\widetilde{f^\ast}.\]
\end{proposition}
\begin{proof} On note $\pi_{\g_{\ad}}:\g\rightarrow \g_\ad$ la projections orthogonale selon $\g=\mathfrak{z}\oplus\g_\ad$. On note $h_{X}$ la fonction $h_{X}(Y)=h(X+Y)$ pour $h\in\S(\g(F))$ et $X\in\g(F)$, et similairement pour les fonctions de $\S(\g^\ast(F))$. On remarque que l'action adjointe de $G(F)$ sur $\g(F)$ se décompose en une somme de  l'action triviale sur $\mathfrak{z}(F)$ et l'action adjointe de $G(F)$ sur $\g_{\ad}(F)$.

Soit $Y\in \g(F)$, posons $\phi_Y$ la fonction $f_{\pi_{\mathfrak{z}}(Y)}|_{\g_\ad}$. Il vient, de la définition,
\[\widetilde{f}(Y)=\widehat{\phi_Y}^{\g_\ad}(\pi_{\g_{\ad}}(Y)),\,\,\,\,\forall Y\in\g(F).\]
La même égalité a lieu pour $G^\ast$. Il s'ensuit
\[J_{M'}^{Q'}(Z'+A',\widetilde{f^\ast})=J_{M'}^{Q'}(A',\widetilde{f_{Z'}^\ast})=J_{M'_{\AD}}^{Q'_{\AD}}(A',\widehat{\phi_{Z'}^\ast}^{\g_\ad^\ast})\]
pour tout triplet $(M',Q',Z'+A')$ avec $Z'\in\mathfrak{z}^\ast(F)$ et $A'\in\g_{\ad}^\ast(F)$, et $\phi_{Z'}^\ast$ est la fonction $f_{Z'}^\ast|_{\g_\ad^\ast}$. Le transfert cherché est actuellement clair au motif de la proposition précédente.
\end{proof}

\section{Développement fin de la formule des traces pour les algèbres de Lie}\label{sec:formuledestracesI}

Soit désormais $F$ un corps de nombres.

\subsection{Préliminaires globaux}
\subsubsection{Objets relativement à \texorpdfstring{$F$}{F}}
On note $\O_F$ l'anneau des entiers de $F$.  On pose $\A_F$ l'anneau des adèles, dans lequel s'injecte de façon diagonale le corps $F$. On note $|\cdot|_{\A_F}=\prod_{v\in \V_F} |\cdot|_{F_v}$ la norme adélique usuelle. Si $S$ est un sous-ensemble fini de $\V_F$ on pose $F_S\eqdef\prod_{v\in S} F_v$ et le considère comme un sous-anneau de $\A_F$. On écrit $F_\infty\eqdef F\otimes_\Q\R$. Pour tout $v \in\V_F$, on fixe une place $\overline{v}$ de $\overline{F}$ divisant $v$. Pour toute extension finie $E$ de $F$ incluse
dans $\overline{F}$, on note $E_v$
le complété pour la place de $E$ obtenue par restriction de $\overline{F}$ à $E$. On note
$\overline{F}_v$ la réunion de ces complétés, c’est une clôture algébrique de $F_v$. On se donne un caractère non-trivial $\psi$ de $F\backslash \A_F$. 

\subsubsection{Objets relativement à \texorpdfstring{$G$}{G}}\label{subsubsec:objetpourGcontexteglobal}

Soit $G$ un groupe du type GL sur $F$. Pour toute $F$-algèbre $R$ on note $G_R$ le changement de base $G\times_F R$. Quand $R=F_v$ (resp. $R=F_S$ pour $S$ un sous-ensemble fini de $\V_F$ ; resp. $R=F_\infty$), on abrège $G_{F_v}$ (resp. $G_{F_S}$ ; resp. $G_{F_\infty}$) en $G_v$ (resp. $G_S$ ; resp. $G_\infty$).

Le groupe $G(F_S)$ s'identifie canoniquement à un sous-groupe de $G(\A_F)$. Aussi $G(F)$ s'identifie à un sous-groupe de $G(\A_F)$ via l'injection diagonale.  

On fixe $M_0$ un sous-groupe de Levi minimal de $G$. Fixons ensuite $M_{v,0}$ un sous-groupe de Lévi minimal de $G_v$, contenu dans $M_{0,v} \eqdef M_0 \times_F F_v$ pour tout $v\in \V_F$. Nous fixons $K$ un sous-groupe compact maximal de $G(\A_F)$ en prenant 
\[K=\prod_{v\in\V_F} K_v\subseteq G(\A_F)\]
avec $K_v$ un sous-groupe compact maximal de $G(F_v)$ en position par rapport à $M_{v,0}$, et tel que pour toute immersion définie sur $F$ de $G$ dans $\GL_n$, on a $K_v=\GL_n(\O_v)\cap G(F_v)$ pour presque toutes les places $v$. On dira que $K$ est en bonne position par rapport à $M_0$. 

Prenons un modèle (schéma en groupes réductifs) de $G$ sur $\O_F\left[\frac{1}{s}\right]$ avec $s\in F^\times$. Fixons $\underline{S}_0$ un sous-ensemble fini de places, contenant les places archimédiennes, tel que toute place de $F$ divisant $s$ appartient à $\underline{S}_0$. En regardant l'algèbre de Lie du modèle de $G$ on se gratifie d'un modèle (schéma en algèbres de Lie) de $\g$ sur $\O_F\left[\frac{1}{s}\right]$.

Fixons également $\mathfrak{k}_v$ un réseau de $\g(F_v)$ pour tout $v\in\V_\fin$ comme suit : pour $v\not\in \underline{S}_0$ prenons $\mathfrak{k}_v=\g(\O_v)$, et pour $v\in \V_{\fin}\cap \underline{S}_0$ prenons $\mathfrak{k}$ un réseau quelconque de $\g(F_v)$. Par définition pour toute immersion définie sur $F$ de $G$ dans $\GL_n$, on a $\mathfrak{k}_v=\mathfrak{gl}_n(\O_v)\cap \g(F_v)$ pour presque tous les $v$.

Soit $M\in \L^G(M_0)$. On note $A_{M,\Q}$ le sous-tore central $\Q$-déployé maximal dans $\Res_{F/\Q}M$ puis $A_{M,\infty}\eqdef A_{M,\Q}(\R)^\circ$ la composante neutre de $A_{M,\Q}(\R)\subseteq A_{M,\Q}(F_\infty)$, ici $\R$ s'injecte de façon diagonale dans $F_\infty$. Pour $P\in \P^G(M)$ on pose en plus $A_{P,\infty}\eqdef A_{M,\infty}$.

\subsubsection{Application \texorpdfstring{$H_P$}{HP}}~{}

Soit $M\in \L^G(M_0)$, on définit alors un morphisme de groupes $H_M : M(\A_F)\rightarrow a_
{M}$ par $e^{\langle H_M(m),\chi\rangle}=|\chi(m)|_{\A_F}$ pour tout $\chi\in X^\ast(M)_F$. Posons $M(\A_F)^1$ le noyau de l'application $H_M$. On a $M(\A_F)=M(\A_F)^1\times A_{M,\infty}$.

Soit $P\in\F^G(M_0)$, on définit l'application $H_P : G(\A_F)\rightarrow a_{P}$ par $H_P(mnk)\eqdef H_{M_P}(m)$ pour tous $m\in M_P(\A_F), n\in N_P(\A_F)$ et $k \in K$, suivant la décomposition d'Iwasawa $G(\A_F)=M_P(\A_F)N_P(\A_F)K$. L'application $H_P$ induit un isomorphisme de $A_{P,\infty}$ sur $a_P$.

Soit $S$ un sous-ensemble fini de $\V_F$, alors l'application $H_P$ se restreint à $H_P:G(F_S)\rightarrow a_P$.

\subsection{Normalisations des mesures de la section \ref{sec:formuledestracesI}}\label{subsec:normalisationsdesmesuresdeFTI}
On fixe à présent les normalisations des mesures. 

Pour $H$ un groupe séparé localement compact muni d'une mesure de Haar et $X\subseteq H$ une partie mesurable, on note $\vol(X;H)$ le volume de $X$ dans $H$ et $\vol(H)$ le volume total de $H$.

Dans la suite, nous nous conformons à la règle suivante : soit $H$ un groupe algébrique sur $F$. Soit $S$ un sous-ensemble fini de places contenant les places archimédiennes, si la mesure sur $H(F_v)$ est fixée pour tout $v\in S$, on prend sur $H(F_S)$ la mesure produit. Si $H$ admet un modèle (schéma en groupes lisses) sur $\O_F\left[\frac{1}{s}\right]$ avec $s\in F^\times$ tel que toute place divisant $s$ est dans $S$, alors pour tout $v\not\in S$ ce modèle nous donne un modèle entier $H_v'$ de $H$ en place $v$, autrement dit $H_v'$ est un schéma en groupes sur $\O_v$ avec $H_v'\otimes_{\O_v}F_v=H_v$. Si la mesure sur $H(F_v)$ est fixée pour tout $v\in \V_F$ et est de sorte que $\vol(H_v'(\O_v);H(F_v))=1$ pour tout $v\not\in S$, on prend sur $H(\A_F)$ la mesure produit. La mesure obtenue dépend, a priori, de $S$ et de la famille $(H_v')_{v\not\in S}$.

Dans la suite le symbole $v$ désigne un élément général de $\V_F$. 

La mesure sur un espace quotient est la mesure quotient. 

Fixons $M_{v,0}$ un sous-groupe de Levi minimal de $G_v$, inclus dans $M_{0,v}$. On choisit un produit scalaire $W^G_{M_0}$-invariant (resp. $W^{G_v}_{M_{v,0}}$-invariant) sur $a_{M_0}$ (resp. $a_{M_v,0}$). On prend sur tout sous-espace de $a_{M_0}$ (resp. $a_{M_v,0}$) la mesure engendrée par ce produit scalaire. On aimerait préciser qu'a priori, la mesure sur $a_M$ n'a pas de lien avec celle sur $a_{M_v}$ pour $M\in \L^G$. Constatons que pour tout $M\in \L^G$, la restriction de $H_M$ sur $A_{M,\infty}$ donne un isomorphisme de groupes topologiques $A_{M,\infty}\xrightarrow{\sim}a_M$. On se procure ainsi une mesure sur $A_{M,\infty}$ par cet isomorphisme.

Pour $P\in \L^G(M_0)$ ou $P\in \L^{G_v}(M_{v,0})$, on écrit $dh$ pour la mesure sur $a_P$ et $d\lambda$ pour la mesure duale sur $ia_P^\ast$.

Tous les sous-groupes ouverts compacts du groupe des $F_v$-point d'un groupe du type GL sont conjugués. Soit $X\in \g_{\ss}(F_v)$. Si $v\in \V_F\backslash\underline{S}_0$, on prend sur $G_{v,X}(F_v)$ la mesure donne le volume 1 à un (donc tout) sous-groupe ouvert compact maximal de $G_{v,X}(F_v)$ (resp. $G_{v,X^\ast}^\ast$) ; si $v\in \underline{S}_0$, on prend sur $G_{v,X}(F_v)$ la mesure de Haar fixée dans le numéro \ref{subsec:localnormalisationsdemesures}, i.e. l'équation \eqref{YDLgeomeq:Haarmeasureongrp}. Finalement, on munit $G_{v,X}(F_v)$ pour tous $v\in \V_F$ et $X\in \g(F_v)$ de la mesure expliquée dans l'équation \eqref{eq:defmeasureonanyorb}.

Pour tous $P\in\P^G(M_0)$, $\sigma\in\m_{P}(F_v)$ semi-simple, et $H\eqdef N_{P_\sigma}$, on munit $\mathfrak{h}(F_v)$ d'une mesure de Haar $dN$. On munit ensuite $H(F_v)$ de la mesure $dn$ qui est compatible à $dN$ via l'isomorphisme de variétés $\text{Id}+\cdot :N\in\mathfrak{h}(F_v)\mapsto \text{Id}+N\in H(F_v)$. On exige enfin que les mesures soient normalisées de façon que :
\begin{enumerate}
    \item pour tous $H$ comme ci-dessus,
      \begin{equation}\label{eq:integralconditiononmesureofLie(N)}
      \vol(\mathfrak{k}_v\cap \mathfrak{h}(F_v);\mathfrak{h}(F_v))=1\,\,\,\,\text{pour presque tous les $v$.}   
      \end{equation} 
    Ainsi le produit des mesures sur $\mathfrak{h}(F_v)$ donne une mesure sur $\mathfrak{h}(\A_F)$ ;  
    \item munissons $\mathfrak{h}(\A_F)$ de la mesure produit, alors 
    \[\vol(\mathfrak{h}(F)\backslash \mathfrak{h}(\A_F))=1\] 
    avec $\mathfrak{h}(F)$ muni de la mesure de comptage. 
\end{enumerate}
Observons que l'équation \eqref{eq:integralconditiononmesureofLie(N)} entraîne que 
\begin{equation}\label{eq:integralconditiononmesureofN}
      \vol(K_v\cap H(F_v);H(F_v))=1\,\,\,\,\text{pour presque tous les $v$.}   
      \end{equation} 
Ainsi le produit des mesures sur $H(F_v)$ donne une mesure sur $H(\A_F)$ ; puis le point 2 entraîne que 
\[\vol(H(F)\backslash H(\A_F))=1\] 
avec $H(F)$ muni de la mesure de comptage et $H(\A_F)$ muni de la mesure produit.


Soit $P\in\F^G(M_0)$. On définit la fonction module pour tout $p \in P(\A_F)$ par
\[\delta_P (p)=e^{2\rho_P(H_{M_0}(p))} = |\det(\Ad(p);\mathfrak{n}_P (\A_F))|_{\A_F}.\]
Bien sûr $\delta_P (p)=\prod_{v\in\V_F}\delta_P (p_v)$.

On munit, pour tout $L\in \L^{G_v}(M_{v,0})$, les groupes $K_v\cap L(F_v)$ de mesures de Haar. On exige, comme il est loisible, que $\vol(K_v\cap L(F_v))=1$ pour tout $L\in \L^{G_v}(M_{v,0})$ et pour presque tous les $v$. Pour tout $L\in \L^{G_v}(M_{v,0})$ et $P\in \F^{L}(M_{v,0})$, on écrit $\gamma_{v}^{L}(P)=\gamma_{v}^L(P,K_v\cap L(F_v))>0$ la constante donnée par l'équation \eqref{eq:locmeasureiwasawa}. Selon les normalisations ci-dessus et l'équation \eqref{eq:integralconditiononmesureofN},
\begin{equation}
\gamma_{v}^{L}(P)=1\,\,\,\,\text{ pour tous $L\in \L^{G_v}(M_{v,0}), P\in \F^{L}(M_{v,0})$ et presque tous les $v$}.   
\end{equation}
Lorsque $L=G$, on abrège $\gamma_{v}^{G}(P)$ en $\gamma_{v}(P)$.

Pour tout $P\in\F^G(M_0)$ définissons la mesure sur $M_P(\A_F)^1$ en identifiant $M_P(\A_F)^1$ au quotient $A_{P,\infty}\backslash M_P(\A_F)$, elle est ainsi l'unique mesure telle que pour tout $f\in L^1(G(\A_F))$ on ait
\begin{align*}
\int_{G(\A_F)} f(x)\,dx 
&=\gamma_{\A_F}(P)\int_{A_{P,\infty}}\int_{M_P(\A_F)^1}\int_{N_P(\A_F)}\int_{K} f(amnk)\,dk\,dn\,dm\,da\\
&=\gamma_{\A_F}(P)\int_{N_P(\A_F)}\int_{A_{P,\infty}}\int_{M_P(\A_F)^1}\int_{K} f(namk)e^{-2\rho_P(H_{M_0}(a))}\,dk\,dm\,da\,dn,
\end{align*}
où $\gamma_{\A_F}(P)=\prod_{v\in \V_F}\gamma_v(P_v)$.

\subsection{Prélude}\label{subsec:preludeTFdef}
Notre approche de la correspondance de Jacquet-Langlands repose essentiellement sur la comparaison des côtés géométriques de la formule des traces (non-invariante) pour les algèbres de Lie, et on obtient de la sorte la comparaison des côtés géométriques de la formule des traces (non-invariante) pour les groupes comme un simple sous-produit.

Au cours des trois prochaines sections, nous aborderons les diverses formules des traces qui s'avèrent pertinentes pour la suite de notre propos. Elles feront intervenir la partition de $\g(F)$ par les classes de conjugaison et non la partition par les classes de conjugaison semi-simples. On note $\O^\g$ l’ensemble des classes de $G(F)$-conjugaison de $\g(F)$. Un élément général de l'ensemble sera souvent noté $\o\in\O^\g$. On écrit $\o_\ss$ pour la classe de $G(F)$-conjugaison $\{X_\ss\mid X\in \o\}$.

Soient $P_1\subseteq P_2$ deux sous-groupes paraboliques de $G$, posons $\tau_{P_1}^{P_2}$ et $\widehat{\tau}_{P_1}^{P_2}$ les fonctions caractéristiques respectives des chambres de Weyl ouvertes respectivement aigüe et obtuse dans $a_{P_1}^{P_2}$. Posons aussi $\sigma_{P_1}^{P_2}\eqdef \sum_{P: P_2\subseteq P} (-1)^{\dim a_{P_2}^{P}}\tau_{P_1}^{P}\widehat{\tau}_{P}$. On voit souvent ces fonctions comme des fonctions sur $a_0=a_{M_0}$ par projection naturelle $a_0=a_{P_1}^{P_2}\oplus (a_{M_0}^{P_1}\oplus a_{P_2})$. Dans la suite si on a des sous-groupes $H_i$ indexés par des entiers $i$, on omettra souvent $H$ dans l'écriture lorsque le contexte permet d'enlever toute ambiguïté, par exemple $\tau_{1}^{2}$ désignera $\tau_{P_1}^{P_2}$ avec $P_1\subseteq P_2$ deux sous-groupes paraboliques et $W_0^G$ désignera le groupe de Weyl relatif de $(G,M_0)$. On note enfin $F^P(-,T)$ pour $P\in\F^G$ et $T\in a_0$ assez régulier la fonction de \cite[p.941]{Art78} (cette fonction dépend en réalité du choix d'un compact de $N_0(\A_F)M_0(\A_F)^1$ et d'un vecteur assez négatif dans $a_0$, mais ces choix sont sans importance dans la suite), il s'agit de la fonction caractéristique d'un compact de $A_{P,\infty}M(F)N_P(\A_F)\backslash G(\A_F)$.

Soit $V$ un espace vectoriel sur $F$. Rappelons que $\S(V(\A_F))$ est l'espace de Schwartz-Bruhat de $V(\A_F)$ (numéro \ref{subsec:EspaceS-B}). On a $\S(V(\A_F))= \varinjlim_{S\subseteq \V_F\text{ fini}} \S(V(F_S))$, avec l'application de transition $f\in \S(V(S_{1}))\mapsto f\otimes \bigotimes_{v\in S_2\setminus S_1}1_{V(\O_v)}\in \S(V(S_2))$ pour $S_1\subseteq S_2$, ici on prend un modèle (schéma en modules libres) de $V$ sur $\O_F$, toujours noté $V$. La limite inductive ne dépend pas du choix du modèle.

Pour $\o \in \O^\g$ et $f \in\S(\g(\A_F))$, on définit
\[K_{P,\o}^\g(x,f)=\sum_{X\in\m_P(F):\Ind_{M_P}^G(X)=\o}\int_{\mathfrak{n}_P(\A_F)}f((\Ad x^{-1})(X+U))\,dU,\]
et
\[K_\o^{\g,T}(x,f)=\sum_{P:P_0\subseteq P}(-1)^{\dim a_P^G}\sum_{\delta\in P(F)\backslash G(F)}\widehat{\tau}_P^G(H_0(\delta x)-T)K_{P,\o}^\g(\delta x,f)\]
pour $T\in a_0$. Finalement,
\[J_\o^{\g,T}(f)\eqdef \int_{G(F)\backslash G(\A_F)^1} K_\o^{\g,T}(x,f)\,dx.\]
Le symbole $\g$ sera souvent négligé dans l'exposant de $J_\o^{\g,T}(f)$, $K_\o^{\g,T}(x,f)$, et $K_{P,\o}^\g(x,f)$ etc si le contexte nous le permet.

\begin{theorem}[{{\cite[théorème 3.2.1, corollaire 3.2.2, proposition 3.2.3]{Ch18}}}]\label{thm:classconjdisJ} ~{}
\begin{enumerate}
    \item Posons $d(T)\eqdef \min_{\alpha\in\Delta_{P_0}^G} \{\alpha(T)\}$. Dans la suite on laisse $T$ varier dans l'ensemble des vecteurs assez réguliers tels que $d(T)\geq \epsilon_0 \|T\|$ pour $\epsilon_0$ un nombre positif. Soit $\mathcal{K}$ une partie compacte de $\S(\g(\A_F))$, il existe $c_{\mathcal{K}}>0$ tel que pour tous $f\in \mathcal{K}$ et $T$ assez régulier (dépendant de $\mathcal{K}$),
    \[\sum_{\o\in \O^\g}\int_{G(F)\backslash G(\A_F)^1}\left|F^G(x,T)\sum_{X\in \o(F)}f\left((\Ad  x^{-1})X\right)-K_{\o}^{T}(x,f)\right|\,dx< c_{\mathcal{K}}e^{-d(T)}\]
    \item Pour tous $f\in \S(\g(\A_F))$ et $T\in a_0$ on a 
    \[\sum_{\o\in\O^\g}\int_{G(F)\backslash G(\A_F)^1}|K_\o^{T}(x,f)|\,dx<\infty.\]
    En particulier $J_\o^T(f)$ est bien défini.
    \item Pour $f$ fixé,  $T\mapsto J_\o^{T}(f)$ et $T\mapsto J^T(f)\eqdef\sum_{\o\in\O^\g}  J_\o^{T}(f)$ sont des polynômes en $T$ de degré au plus $\dim a_0^G$.
\end{enumerate}  
\end{theorem}

\begin{proof}
Le point 3 ici est plus fort que ce que Chaudouard énonce, il résulte néanmoins des mêmes calculs et la théorie générale (cf. \cite[proposition 9.3]{Art05}).    
\end{proof}

\begin{corollary}\label{coro:TFisadelictempereddist}
Pour tout $T\in a_0$, les fonctionnelles $f\mapsto J_{\o}^{T}(f)$ et $f\mapsto J^{T}(f)$ sont des distributions tempérées sur $\g(\A_F)$.   

En particulier, pour tout $S$ sous-ensemble fini de $\V_F$ contenant les places archimédiennes et tout $f^S\in \S(\g(\A_F^S))$, les fonctionnelles $f_S\mapsto J_{\o}^{T}(f^Sf_S)$ et $f_S\mapsto J^{T}(f^Sf_S)$ sont des distributions tempérées sur $\g(F_S)$.   
\end{corollary}

\begin{proof}
Il suffit de prouver le lemme pour $T$ assez régulier. En effet soit $T'\in a_0$ quelconque, comme les distributions en question sont des polynôme de $T$, il existe un entier positif $n=n(G,M_0,K)$ tel que $J^{T'}$ (resp. $J_{\o}^{T'}$) soit une combinaison linéaire de $T^iJ^{T+kE_j}$ (resp.  $T^iJ_\o^{T+kE_j}$) avec $0\leq i\leq n, -n\leq k\leq n$, et $0\leq j\leq \dim a_0$, où $(E_1,\dots,E_{\dim a_0})$ est une $\R$-base de $a_0$. On conclut la preuve par linéarité. 

Soit $T$ assez régulier. L'espace de Schwartz-Bruhat $\S(\g(\A_F))$ étant un espace de Montel (proposition \ref{prop:topopropertiesS-Bspace}), toute partie bornée est par définition relativement compacte. Ainsi le théorème \ref{thm:classconjdisJ} nous garantit que les opérateurs $J_{\o}^{T}$ et $J^{T}$ de $\S(\A_F)$ vers $\C$ sont bornés, donc continus puisque les espaces en question sont localement convexes.
\end{proof}

Expliquons le lien avec les distributions définies à partir de la partition de $\g(F)$ par les classes de conjugaison semi-simples, l'approche que préconise Arthur : deux éléments $X$ et $Y$ de $\g(F)$ sont dites équivalents si et seulement si leurs parties semi-simples sont conjuguées par $G(F)$, notons ${}_\ss\O^\g$ l'ensemble des classes d'équivalence. Pour $\o' \in {}_\ss\O^\g$ et $f \in\S(\g(\A_F))$, on définit
\[{}_\ss K_{P,\o'}^\g(x,f)=\sum_{X\in\m_P(F)\cap \o'}\int_{\mathfrak{n}_P(\A_F)}f((\Ad x^{-1})(X+U))\,dU,\]
puis ${}_\ss K_{\o'}^{\g,T}(x,f)$ et ${}_\ss J_{\o'}^{\g,T}(f)$ de façon identique. Alors, en notant $\o_\ss'\eqdef \{X_\ss\mid X\in\o'\}\in \O^{\g}$, nous avons ${}_\ss K_{P,\o'}^\g(x,f)=\sum_{\o\in\o^{\g}:\o_\ss=\o_\ss'} K_{P,\o}^\g(x,f)$. En effet cela se déduit du point 2 de la proposition \ref{prop:indprop}. Cette égalité est une somme finie grâce à la théorie des diviseurs élémentaires (cf. appendice \ref{sec:AppendixA}). On a ainsi une partition plus fine que celle d'Arthur.

Rappelons que pour $x\in G(\A_F)$ et $f\in \S(\g(\A_F))$ on note $(\Ad x)f\in \S(\g(\A_F))$ la fonction $(\Ad x)f(Y) = f((\Ad x^{-1})Y )$, et pour $P\in\P^G$  on note $f_{P,x}\in\S(\mathfrak{m}_P(\A_F))$ la fonction définie par $f_{P,x}(Z)=\gamma_{\A_F}(P)\int_{K}\int_{\mathfrak{n}_P(\A_F)}f\left((\Ad  k^{-1})(Z+U)\right)v_P'(kx)\,dU\,dk$. Pour tout $M\in \L^G$ on a l'induction $\Ind_M^G:\O^\m\rightarrow\O^\g$. Elle est à fibres finies.

\begin{proposition}\label{prop:varianceglobale}
Pour tous $\o\in \O^\g$, $x\in G(\A_F)$ et $f\in\S(\g(\A_F))$ on a 
\[J_\o^{\g,T}((\Ad x^{-1})f)=\sum_{P:P_0\subseteq P}\sum_{\o_{\m_P}\in(\Ind_{M_P}^{G})^{-1}(\o)}J_{\o_{\m_P}}^{\m_P,T}(f_{P,x}),\]
\end{proposition}
\begin{proof}
Cela résulte également des calculs de Chaudouard et la théorie générale (cf. la preuve de \cite[proposition 3.2.3]{Ch18} et \cite[pp.52-53]{Art05}).    
\end{proof}

Conformément aux sections 1 et 2 de \cite{Art81}, on peut s'affranchir de la dépendance de $J_{\o}^T$ en $P_0$ si l'on choisit un paramètre de troncature spécial $T_0\in a_0$. Modulo $a_G$, le vecteur $T_0$ est uniquement déterminé par les relations
\begin{equation}\label{eq:deftroncatureT0}
H_{0}(w_s^{-1})=T_0-s^{-1}T_0    
\end{equation}
avec $s\in W_0^G$, et $w_s\in G(F)$ un représentant fixé de $s$. Notons que pour tout $M\in\L^G(M_0)$ la projection orthogonale de $T_0$ (pour $G$) sur $a_0^M$ est le vecteur $T_0$ pour $M$. Nous noterons $J_{\o}^\g$ la distribution $J_{\o}^{\g,T_0}$ et $J^\g$ la distribution $J^{\g,T_0}$. 

\begin{remark}
Pour tout groupe $G$ du type GL, et tout sous-groupe de Levi $M_0$ minimal, il existe $K$ un sous-groupe compact de $G(\A_F)$ en bonne position par rapport à $M_0$, et un sous-groupe de $G(F)\cap K$ isomorphe à $W_0^G$. Si l'on choisit les représentants $W_0^G$ comme éléments de ce sous-groupe, alors $T_0=0$ (mod $a_G$).
\end{remark}

\subsection{Développement fin de la formule des traces pour les algèbres de Lie}
Arthur a obtenu le développement fin de la formule des traces pour les groupes comme suit :
\begin{enumerate}
    \item Obtenir le développement fin du terme $J_{\text{unip}}(f)$.
    \item Traiter le terme général $J_{\o'}(f)$ via la descente au centralisateur semi-simple.
\end{enumerate}
Nous suivrons sa méthode dans la suite. Il est utile de mentionner que le passage du groupe à l'algèbre de Lie introduit des facteurs Jacobiens, certaines de nos formules seront ainsi légèrement distinctes de celles d'Arthur. Aussi, les quelques articles qui nous serviront de références sont rédigés dans le contexte où le corps de base est $F=\Q$. Il n'est néanmoins pas difficile d'étendre les résultats à un corps de nombres quelconque via une restriction des scalaires. Nous travaillerons ainsi continuellement sur un corps de nombres général.

\subsubsection{Développement fin de \texorpdfstring{$J_{\o}(f)$}{Jo(f)} pour \texorpdfstring{$\o$}{o} nilpotent}
Soit $\mathcal{N}_G$ la sous-variété des éléments nilpotents dans $\g$, elle est définie sur $F$. 
On va noter $(\mathcal{N}_G(F))$ l'ensemble des classes de $G$-conjugaison nilpotentes contenant un $F$-point. 


Pour $V\in (\mathcal{N}_G(F))$ on peut lui associer une classe de $G(F_S)$-conjugaison
\[V(F_S)=(\Ad G(F_S))V.\]
Alors deux éléments $U,U'\in \mathcal{N}_G(F)$ engendrent la même classe de $G(F_S)$-conjugaison si et seulement s'ils sont déjà $G(F)$-conjugués, ceci est une conséquence du point 8 de la proposition \ref{pro:bontype}.

Venons-en maintenant au développement fin de $J_{\o}(f)$ pour $\o\in (\mathcal{N}_G(F))$. Soit $S$ un ensemble fini de places qui contient les places archimédiennes. Nous allons immerger $\S(\g(F_S))$ dans $\S(\g(\A_F))$ en multipliant une fonction dans $\S(\g(F_S))$ avec $1_{\mathfrak{k}^S}\eqdef\prod_{v\in\V_F,v\not\in S}1_{\mathfrak{k}_v}$. Pour toute fonction $f$ de $\S(\g(\A_F))$, il existe $S$ fini tel que $f$ est l'image d'une fonction de $\S(\g(F_S))$.

\begin{theorem}\label{thm:devlopfinnilp}
Pour tout $G$ groupe du type GL, tout $M_0$ sous-groupe de Levi minimal, tout $S$ ensemble fini de places qui contient les places archimédiennes, et tout $U\in(\mathcal{N}_G(F))$, il existe un nombre complexe
\[a^{G}(S,U)=a^{G,M_0,K}(S,U),\]
déterminé uniquement, à condition que l'on suive les normalisations des mesures énoncées dans la sous-section \ref{subsec:normalisationsdesmesuresdeFTI}, par la mesure sur $a_0$, tel que pour tous $f\in\S(\g(F_S))$ et $\o\in (\mathcal{N}_G(F))$ on ait,
\begin{equation}\label{eq:devlopfinnilp}
J_\o^\g(f)=\sum_{M\in\L^G(M_0)}|W_0^M||W_0^G|^{-1}\sum_{U\in(\mathcal{N}_M(F)), \Ind_M^G(U)=\o}a^{M,M_0,M\cap K}(S,U)J_M^G(U,f).    
\end{equation}
\end{theorem}

\begin{proof}
Les deux côtés de l'équation \eqref{eq:devlopfinnilp} sont des distributions tempérées (corollaire \ref{coro:TFisadelictempereddist} et point 1 de proposition \ref{prop:IOP}), il suffit de ce fait de prouver le théorème pour les fonctions tests dans $C_c^\infty(\g(F_S))$. Fixons $S$ et $\o$, et supposons inductivement que l'assertion est prouvée lorsque $G$ est remplacé par un sous-groupe de Levi semi-standard propre, en particulier les nombres $a^M(S,U)$ sont définis pour $M\not= G$. Soit
\begin{align*}
T_{\o}^G(f)\eqdef J_{\o}^\g(f)-\sum_{M\in\L^G(M_0),M\not=G}|W_0^M||W_0^G|^{-1}\sum_{U\in(\mathcal{N}_M(F)), \Ind_M^G(U)=\o}a^M(S,U)J_M^G(U,f)   
\end{align*}
pour toute fonction $f\in C_c^\infty(\g(F_S))$.  Il est clair, par l'hypothèse de la récurrence et la propriété de variance de $J_{\o}^{\g}(-)$ et de $J_M^G(U,-)$ envers la conjugaison, c'est-à-dire les propositions \ref{prop:varianceglobale} et \ref{prop:IOP}, que la distribution $T_{\o}^G(-)$ est $G(F_S)$-invariante. On voit aussi que la distribution $T_{\o}^G(-)$ annule toute fonction s'annulant sur $\o(F_S)$, grâce au point 1 du théorème \ref{thm:classconjdisJ} et au point 1 de la proposition \ref{prop:IOP}.

Soit $D$ la fonctionnelle sur $C_c^\infty(\o(F_S))$ définie par $D(h)=T_{\o}^G(h')$ pour toute fonction $h\in C_c^\infty(\o(F_S))$, avec $h'\in C_c^\infty(\g(F_S))$ une fonction qui prolonge $h$. Cette fonctionnelle est bien définie parce que toute fonction $h\in C_c^\infty(\o(F_S))$ se prolonge par 0 sur $\g(F_S)\setminus(\overline{\o}(F_S)\setminus \o(F_S))$ puis sur $\g(F_S)$, et la valeur $T_{\o}^G(h')$ ne dépend pas de prolongement $h'$ car $T_{\o}^G(-)$ annule toute fonction s'annulant sur $\o(F_S)$. On voit ensuite que $D$ est une distribution, encore une fois par un argument de la partition de l'unité.

La distribution $D$ est $G(F_S)$-invariante sur l'espace $G(F_S)$-homogène $\o(F_S)$, elle est donc un multiple de l'intégrale orbitale sur $\o(F_S)$ (cf. \cite[chapitre VII, section 6, théorème 3.b)]{Bou04}). Il existe alors un unique nombre complexe $a^G(S,\o)$ tel que
\[T_{\o}^G(f)=D(f|_{\o(F_S)})=a^G(S,\o)J_G^G(\o,f),\]
pour toute fonction $f\in C_c^\infty(\g(F_S))$. Ce qu'il fallait.

En parallèle il est clair que $a^{G}(S,\o)$ est uniquement déterminé par la mesure sur $a_0$. La preuve se termine. \qedhere

\end{proof}

Le nombre $a^G(S,U)$ ne dépend que de la classe de $G(F)$-conjugaison de $U$. On peut donc aussi écrire $a^G(S,\o)$ à la place de $a^G(S,U)$, où $\o=(\Ad G(F))U$.

\begin{corollary}\label{coro:valeuraG(S,0)}Nous avons $a^G(S,0)=\vol(G(F)\backslash G(\A_F)^1)$.    
\end{corollary}
\begin{proof}
On a $\vol(G(F)\backslash G(\A_F)^1)f(0)=J_{(\Ad G(F))0}^{\g}(f)=a^G(S,0)f(0)$ pour toute fonction $f\in \S(\g(F_S))$. D'où l'égalité voulue. 
\end{proof}

Le théorème met fin à l'étude de $J_{\o}(f)$ pour $\o$ nilpotent. 

\subsubsection{Développement fin de \texorpdfstring{$J_{\o}(f)$}{Jo(f)} : descente semi-simple}

Le dessein du paragraphe est de prouver la proposition qui suit. Autrement dit, de trouver une expression de $J_{\o}(f)$ en termes de distributions nilpotentes sur les sous-algèbres de Levi de $\g_\sigma$. 

Soit $\o\in\O^\g$. On peut et on va fixer $\sigma\in \g_{\ss}(F)$ un élément de $\o_\ss\eqdef \{X_\ss\mid X\in\o\}$ qui est $F$-elliptique dans l'algèbre de Lie d'un sous-groupe de Levi semi-standard $M_1$. Alors $M_{1\sigma}$ est un sous-groupe de Levi minimal de $G_\sigma$. On fixe $K_\sigma=\prod_{v\in\V_F}K_{\sigma v}$ un sous-groupe ouvert compact de $G_\sigma(\A_F)$ en bonne position par rapport à $M_{1\sigma}$. Posons $\o_{\nilp}=(\Ad G_{\sigma}(F))W\subseteq \g_\sigma(F)$ pour $W\in \g_\sigma(F)$ avec $\sigma+W\in \o$, cette classe de $G_\sigma(F)$-conjugaison ne dépend pas du choix de $W$. Posons $T_1=T_0-T_{0\sigma}$ avec $T_{0\sigma}$ le vecteur défini par l'équation \eqref{eq:deftroncatureT0} pour $(G_\sigma,M_{1\sigma},K_\sigma)$. 

Pour tous $R\in \F^{G_{\sigma}}(M_{1\sigma})$, $y\in G_{\sigma}(\A_F)\backslash G(\A_F)$, on pose $\Phi_{R,y,T_1}\in \S(\m_R(\A_F))$ la fonction
\[\Phi_{R,y,T_1}(C)=\gamma_{\A_F}^{G_\sigma}(R)\int_{K_\sigma}\int_{\mathfrak{n}_R(\A_F)}f\left((\Ad (ky)^{-1})(\sigma+C+U)\right)v_R'(ky,T_1)\,dU\,dk,\,\,\,\,\forall C\in\m_R(\A_F),\]
qui dépend de manière lisse de $y$, avec $v_R'(ky,T_1)$ le nombre complexe défini selon ce qui suit : définissons dans un premier temps la $(G,M_1)$-famille $(v_Q(ky,T_1))_{Q\in\F^G(M_1)}$ par
\[v_Q(\lambda,ky,T_1)\eqdef e^{-\lambda(H_Q(ky)-T_1)},\,\,\,\,\lambda\in ia_Q^\ast.\]
On obtient alors $v_Q'(ky,T_1)$. Prenons $\F_R^0(M_1)\eqdef\{P\in \F^G(M_1)\mid P_\sigma =R, a_P=a_R\}$. On pose finalement
\[v_R'(ky,T_1)\eqdef \sum_{Q\in\F_R^0(M_1)} v_Q'(ky,T_1).\]

Alors on a cette formule de descente semi-simple de $J_\o^\g(f)$ :

\begin{proposition}
Avec les notations précédentes,
\[J_\o^\g(f)=\int_{G_\sigma(\A_F)\backslash G(\A_F)}\left(\sum_{R\in\F^{G_\sigma}(M_{1\sigma})}|W_{M_{1\sigma}}^{M_R}||W_{M_{1\sigma}}^{G_\sigma}|^{-1}J_{\o_\nilp}^{\m_R}(\Phi_{R,y,T_1})\right)\,dy.\]
\end{proposition}

On procède à sa démonstration en se guidant d'après \cite{Art86}.  

Pour $X\in \mathfrak h(F)$ avec $H$ un groupe algébrique sur $F$, nous noterons $H(F,X)$ et $H(\A_F,X)$ le centralisateur de $X$ dans $H(F)$ et $H(\A_F)$. On voit le groupe de Weyl relatif $W_0^G$ comme un groupe des isomorphismes linéaires de $a_0$. Pour $\mathfrak{a}$ et $\mathfrak{b}$ deux sous-espaces vectoriels de $a_0$ on note $W^G(\mathfrak{a},\mathfrak{b})$ l'ensemble (possiblement vide) des isomorphismes linéaires de $\mathfrak{a}$ vers $\mathfrak{b}$ obtenus par restriction des éléments de $W_0^G$ sur $\mathfrak{a}$.

La première étape consiste à modifier le noyau tronqué, afin de mettre en évidence la dépendance de $J_\o(f)$ en la partie semi-simple de $\o$. Pour $\o \in \O^\g$ et $f \in\S(\g(\A_F))$, on définit 
\[k_{P,\o}^\g(x,f)=\sum_{X\in\m_P(F):\Ind_{M_P}^G(X)=\o}\sum_{\eta\in N_{P}(F,X_\ss)\backslash N_{P}(F)}\int_{\n_{P}(\A_F,X_\ss)}f((\Ad (\eta x)^{-1})(X+U))\,dU,\]
et
\[k_\o^{\g,T}(x,f)=\sum_{P:P_0\subseteq P}(-1)^{\dim a_P^G}\sum_{\delta\in P(F)\backslash G(F)}\widehat{\tau}_P^G(H_0(\delta x)-T)k_{P,\o}^\g(\delta x,f)\]
pour $T\in a_0$. 

Pour contourner des difficultés techniques on va supposer désormais que $f\in C_c^\infty(\g(\A_F))$, jusqu'à indication contraire. La justification de toute interversion entre intégrales dorénavant est la même qu'en \cite{Art86}.

\begin{lemma}Pour $T$ assez régulier,
\[J_\o^{\g,T}(f)= \int_{G(F)\backslash G(\A_F)^1} k_\o^{\g,T}(x,f)\,dx.\]

\end{lemma}
\begin{proof}
L'intégrale sur $G(F)\backslash G(\A_F)^1$ du nouveau noyau vaut
\[\sum_{P_1,P_2 : P_0\subseteq P_1\subseteq P_2}\int_{P_1(F)\backslash G(\A_F)^1}F^1(x,T)\sigma_1^2(H_0(x)-T)\sum_{P:P_1\subseteq P\subseteq P_2}(-1)^{\dim a_P^G}k_{P,\o}^\g(x,f)\,dx. \]
On regarde le terme en $(P_1,P_2)$ dans la somme, il est égal à
\begingroup
\allowdisplaybreaks
\begin{align*}
&\int_{P_1(F)\backslash G(\A_F)^1}F^1(x,T)\sigma_1^2(H_0(x)-T)\sum_{P:P_1\subseteq P\subseteq P_2}(-1)^{\dim a_P^G}k_{P,\o}^\g(x,f)\,dx
\\
&=\int_{M_1(F)N_1(\A_F)\backslash G(\A_F)^1}F^1(x,T)\sigma_1^2(H_0(x)-T)\sum_{P:P_1\subseteq P\subseteq P_2}(-1)^{\dim a_P^G}\sum_{X\in\m_P(F):\Ind_{M_P}^G(X)=\o}\\
&\hspace{2cm}\left(\int_{N_1(F)\backslash N_1(\A_F)}\sum_{\eta\in N_{P}(F,X_\ss)\backslash N_{P}(F)}\int_{\n_{P}(\A_F,X_\ss)}f((\Ad (\eta n_1x)^{-1})(X+U))\,dU\,dn_1\right)\,dx
\\
&=\int_{M_1(F)N_1(\A_F)\backslash G(\A_F)^1}F^1(x,T)\sigma_1^2(H_0(x)-T)\sum_{P:P_1\subseteq P\subseteq P_2}(-1)^{\dim a_P^G}\sum_{X\in\m_P(F):\Ind_{M_P}^G(X)=\o}\\
&\hspace{2cm}\left(\int_{N_1(F)\backslash N_1(\A_F)}\int_{N_{P}(F,X_\ss)\backslash N_{P}(\A_F)}\int_{\n_{P}(\A_F,X_\ss)}f((\Ad (\eta n_1x)^{-1})(X+U))\,dU\,d\eta\,dn_1\right)\,dx
\\
&=\int_{M_1(F)N_1(\A_F)\backslash G(\A_F)^1}F^1(x,T)\sigma_1^2(H_0(x)-T)\sum_{P:P_1\subseteq P\subseteq P_2}(-1)^{\dim a_P^G}\\
&\hspace{2cm}\left(\sum_{X\in\m_P(F):\Ind_{M_P}^G(X)=\o}\int_{N_1(F)\backslash N_1(\A_F)}\int_{\n_{P}(\A_F)}f((\Ad (n_1x)^{-1})(X+U))\,dU\,dn_1\right)\,dx,
\end{align*}
\endgroup
la deuxième égalité vient du fait  $\text{vol}(N_P(F)\backslash N_P(\A_F))=1$, puis la dernière est \cite[corollaire 2.5]{Ch02a}. En appliquant les mêmes démarches à 
\[\int_{P_1(F)\backslash G(\A_F)^1}F^1(x,T)\sigma_1^2(H_0(x)-T)\sum_{P:P_1\subseteq P\subseteq P_2}(-1)^{\dim a_P^G}K_{P,\o}^\g(x,f)\,dx\]
on voit aisément que cela coïncide avec la dernière expression ci-dessus. En résumé 
\[J_\o^{\g,T}(f)= \int_{G(F)\backslash G(\A_F)^1} k_\o^{\g,T}(x,f)\,dx.\qedhere\]

\end{proof}

Fixons par la suite $\o\in \mathcal{O}^\g$ une classe de conjugaison. L'ensemble $\o_\ss\eqdef \{X_\ss\mid X\in\o\}$ est bien une classe de $G(F)$-conjugaison. On peut et on va fixer $\sigma\in \o_\ss$ un élément qui est $F$-elliptique dans l'algèbre de Lie d'un sous-groupe de Levi semi-standard $M_1$. Alors $M_{1\sigma}$ est un sous-groupe de Levi minimal de $G_\sigma$. Prenons $P_1\in\F^G(M_1)$. Le groupe $P_{1\sigma}$ est un sous-groupe parabolique minimal de $G_\sigma$, ayant $M_{1\sigma}$ comme une composante de Levi. On fixe à présent $K_\sigma=\prod_{v\in\V_F}K_{\sigma v}$ un sous-groupe ouvert compact de $G_\sigma(\A_F)$ en bonne position par rapport à $M_{1\sigma}$.

On fixe dans la suite, pour tout $v\in\V_F$, $M_{v,1\sigma}$ un sous-groupe de Levi minimal de $G_\sigma\times_FF_v$ inclut dans $M_{1\sigma}\times_FF_v$. Rappelons que pour tout $v\in\V_F$ on a fixé les mesures sur les groupes des $F_v$-points des éléments de $\L^{G_{\sigma,v}}(M_{v,1\sigma})$ dans la sous-section \ref{subsec:normalisationsdesmesuresdeFTI}, et sur les groupes des $F_v$-points des radicales unipotents des éléments de $\F^{G_{\sigma,v}}(M_{v,1\sigma})$ ainsi que leur algèbre de Lie. On va supposer, comme il est loisible, qu'il existe une mesure sur $K_{\sigma v}$ telle que les mêmes exigences sur les mesures que l'on imposes aux objets relativement à $(G,M_0,K,(M_{v,0})_v,(\mathfrak{k}_v)_v)$ dans la sous-section \ref{subsec:normalisationsdesmesuresdeFTI} ont lieu sur les mesures que l'on imposes aux objets relativement à $(G_\sigma,M_{1\sigma},K_\sigma,(M_{v,1\sigma})_v,(\mathfrak{k}_{\sigma v})_v)$ avec $(\mathfrak{k}_{\sigma v})_v$ une nouvelle famille de réseaux pour $\g_\sigma$ comme dans le numéro \ref{subsubsec:objetpourGcontexteglobal}, et $\vol(K_{\sigma v})=1$ pour presque tous les $v$.

On dispose d'une application surjective
\begin{align*}
  \F^G(M_1) &\longrightarrow  \F^{G_\sigma}(M_{1\sigma}) \\
   P  &\longmapsto P_\sigma  .
\end{align*}
Plus généralement, si $M$ est un sous-groupe de Levi de $G$ tel que $a_M=a_{M_\sigma}$ alors
\begin{align*}
  \{L\in\L^G(M) \mid a_L=a_{L_\sigma}\} &\longrightarrow  \L^{G_\sigma}(M_{\sigma}) \\
   L  &\longmapsto L_\sigma  .
\end{align*}
est une bijection. Nous noterons $\F^\sigma\eqdef \F^{G_\sigma}(M_{1\sigma})$, et un sous-groupe parabolique de $G_\sigma$ est dit standard s'il contient $P_{1\sigma}$. Soit $R\in \F^\sigma$. Posons
\begin{align*}
\F_R(M_1)&\eqdef    \{P\in \F^G(M_1)\mid P_\sigma =R\} \\
\F_R^0(M_1)&\eqdef\{P\in \F^G(M_1)\mid P_\sigma =R, a_P=a_R\} \\
\overline{\F}_R(M_1)&\eqdef\{P\in \F^G(M_1)\mid P_\sigma \supseteq R\}.
\end{align*}
On a $\F_R^0(M_1)\subseteq \F_R(M_1)\subseteq \overline{\F}_R(M_1)\subseteq \F^G(M_1)$.

On fixe $Q$ un sous-groupe parabolique standard de $G$, et examine la contribution de $k_{Q,\o}^\g(x,f)$ à $J_\o^{\g,T}(f)$. Soit $X$ un élément de $\{X\in \m_Q(F):\Ind_{M_Q}^G(X)=\o\}$. La partie semi-simple $X_\ss$ commute avec un sous-tore de $G$ qui est $G(F)$-conjugué à $A_1\eqdef A_{M_1}$. Ce sous-tore est $M_Q(F)$-conjugué encore à $A_{M_{Q_1}}$ pour un $Q_1\subseteq Q$ sous-groupe parabolique standard de $G$ qui est associé à $P_1$. On peut alors écrire 
\[X=\mu^{-1}w_s(\sigma+U)w_s^{-1}\mu,\]
où 
\begin{align*}
s\in W^G(a_{P_1},a_{Q_1}),&\,\,\,\,\mu\in M_Q(F),\,\,\,\,U\in E(\o,\sigma,Q,s),\\
& E(\o,\sigma,Q,s)\eqdef \{U\in w_s^{-1}\m_Q(F)w_s\cap \mathcal{N}_{G_\sigma}(F):\Ind_{w_s^{-1}M_Qw_s}^G(\sigma+U)=\o\}.    
\end{align*}
On sait que la classe de $s$ est uniquement déterminée dans $W_0^{M_P}\backslash W^G(a_{P_1},a_{Q_1})/W_{M_{1\sigma}}^{G_\sigma}$. Une fois $s$ fixé, $\mu$ est uniquement déterminé modulo $M_Q(F)\cap w_sG(F,\sigma) w_s^{-1}$ à gauche. Une fois $s$ et $\mu$ fixés, la partie nilpotente $U$ est clairement uniquement déterminée. Soit $W^G(a_1;Q,G_\sigma)$ le sous-ensemble de $W_0^G$ défini par 
\[W^G(a_1;Q,G_\sigma)\eqdef\bigcup_{Q_1:P_0\subseteq Q_1\subseteq Q}\left\{s\in W^G(a_{P_1},a_{Q_1})\,|\, s^{-1}\alpha>0\,\,\forall \alpha\in \Delta_{Q_1}^Q,\,\, s\beta >0\,\,\forall \beta\in \Delta_{P_{1\sigma}}^{G_\sigma}\right\}.\]
Compte tenu de ce qu'on vient de dire
\begin{align*}
k_{Q,\o}^\g(x,f)=\sum_{s}\sum_{\mu}\sum_{U} \sum_{\eta}\int_{V'}f\left[\left(\Ad (\eta x)^{-1}\right)\left\{(\Ad \mu^{-1}w_s)(\sigma+U)+V'\right\}\right]\,dV', 
\end{align*}
où les sommes ou l'intégrale sont sur $s\in W^G(a_1;Q,G_\sigma)$, $\mu\in M_Q(F)\cap w_sG(F,\sigma) w_s^{-1}\backslash M_Q(F)$, $U\in E(\o,\sigma,Q,s)$, $\eta\in N_{Q}(F,(\Ad \mu^{-1}w_s)\sigma)\backslash N_Q(F)$,  enfin $V'\in \n_{Q}(\A_F,(\Ad \mu^{-1}w_s)\sigma)$. Avec des  changements de variables $V=w_s^{-1}\mu V'\mu^{-1} w_s$ puis
\begin{align*}
\left(M_Q(F)\cap w_sG(F,\sigma) w_s^{-1}\backslash M_Q(F)\right)&\times\left(\mu^{-1}\left(N_Q(F)\cap w_sG(F,\sigma) w_s^{-1}\right)\mu\backslash N_Q(F)\right) \\
&\longrightarrow Q(F)\cap w_sG_\sigma(F)w_s^{-1}\backslash Q(F) \\
(\mu,\eta) &\longmapsto q=\mu \eta
\end{align*}
on trouve
\begin{align*}
k_{Q,\o}^\g(x,f)=\sum_{s}\sum_{q}\sum_{U} \int_{V}f\left[\left(\Ad  x^{-1}q^{-1}w_s\right)\left(\sigma+U+V\right)\right]\,dV, 
\end{align*}
avec $V\in \n_{(\Ad w_s^{-1})Q}(\A_F,\sigma)$. Plongeons cette dernière expression dans 
\[\sum_{\delta\in Q(F)\backslash G(F)}\widehat{\tau}_Q^G(H_0(\delta x)-T)k_{Q,\o}^\g(\delta x,f),\]
faisons entrer la somme portée sur $\delta$ dans la somme portée sur $s$, puis effectuons un changement de variable $\xi'=q\delta$ on trouve que $J_\o^{\g,T}(f)$ vaut l'intégrale sur $x\in G(F)\backslash G(\A_F)^1$ de
\begin{align*}
\sum_{Q:P_0\subseteq Q}\sum_{s}\sum_{\xi'}\sum_{U} \int_{V}f\left[\left(\Ad  x^{-1}\xi'^{-1}w_s\right)\left(\sigma+U+V\right)\right](-1)^{\dim a_Q^G}\widehat{\tau}_Q^G(H_{Q}(\xi' x)-T)\,dV, 
\end{align*}
avec $\xi'\in Q(F)\cap w_sG_\sigma(F)w_s^{-1}\backslash G(F)$. Avec le changement de variables $\xi=w_s^{-1}\xi'$, on va changer la somme de sur $\xi'$ par la somme sur $\xi\in R(F)\backslash G(F)$, où 
\[R=w_s^{-1}Qw_s\cap G_\sigma\]
est un sous-groupe parabolique standard de $G_\sigma$ dont $R=M_RN_R=(w_s^{-1}M_Qw_s\cap G_\sigma)(w_s^{
-1}N_Qw_s\cap G_\sigma)$ est une décomposition de Levi. On en conclut donc que $J_\o^{\g,T}(f)$ vaut l'intégrale sur $x\in G(F)\backslash G(\A_F)^1$ de
\begin{align*}
\sum_{Q:P_0\subseteq Q}\sum_{s}&\sum_{\xi \in R(F)\backslash G(F)}\sum_{U\in \mathcal{N}_{M_R}(F):\Ind_{w_s^{-1}M_Qw_s}^G(\sigma+U)=\o} \\
&\int_{\n_R(\A_F)}f\left[\left(\Ad  x^{-1}\xi^{-1}\right)\left(\sigma+U+V\right)\right](-1)^{\dim a_Q^G}\widehat{\tau}_Q^G(H_{Q}(w_s\xi x)-T)\,dV,   
\end{align*}
soit l'intégrale sur $x\in G(F)\backslash G(\A_F)^1$ de
\begin{equation}\label{eq:devlopfinssdescenttempoQ}
\begin{split}
\sum_{R\in \F^\sigma:P_{1\sigma}\subseteq R}\sum_{\xi \in R(F)\backslash G(F)}&\sum_{U\in \mathcal{N}_{M_R}(F):\Ind_{M_R}^{G_\sigma}(U)=\o_{\nilp}} \int_{\n_R(\A_F)}f\left[\left(\Ad  x^{-1}\xi^{-1}\right)\left(\sigma+U+V\right)\right]\\
&\sum_{\substack{Q\in \F^G: P_0\subseteq Q\\ s\in  W^G(a_1;Q,G_\sigma)\\w_s^{-1}Qw_s\cap G_\sigma =R}}(-1)^{\dim a_Q^G}\widehat{\tau}_Q^G(H_{Q}(w_s\xi x)-T)\,dV,       
\end{split}
\end{equation}
ici on note $\o_{\nilp}=(\Ad G_{\sigma}(F))W\subseteq \g_\sigma(F)$ pour $W\in \g_\sigma(F)$ avec $\sigma+W\in \o$, cette classe ne dépend pas de $W$. Elle dépend de $\sigma$ choisi dans $\o_\ss$, on ne le fait toutefois pas figurer dans la notation. On a implicitement fait appel à
\begin{lemma}
Soit $\o\in \O^{\g}$. Soient $\sigma\in \g_\ss(F)$ et $W\in \g_{\sigma}(F)$ tels que $\sigma+W\in \o$. Soit $M$ un sous-groupe de Levi de $G$ tel que $\sigma\in \m(F)$. Alors pour tout $Y\in \mathcal{N}_{M_\sigma}(F)$, $\Ind_M^G(\sigma+Y)=\o$ si et seulement si $\Ind_{M_\sigma}^{G_\sigma}(Y)=(\Ad G_\sigma)W$.
\end{lemma}
\begin{proof}
Pour tout $Y\in \mathcal{N}_{M_\sigma}(F)$, on sait que $\Ind_M^G(\sigma+Y)=(\Ad G)(\sigma+\Ind_{M_\sigma}^{G_\sigma}(Y))$ (proposition \ref{prop:indprop}). L'implication réciproque est alors claire. Pour l'implication directe on utilise l'unicité de la décomposition de Jordan : il existe $g\in G_\sigma(F)$ et $Z\in \Ind_{M_\sigma}^{G_\sigma}(Y)$ tels que $\sigma+W = (\Ad g)(\sigma+Z)$, d'où $W=(\Ad g)Z$, enfin $\Ind_{M_\sigma}^{G_\sigma}(Y)=(\Ad G_\sigma)Z=(\Ad G_\sigma)W$.      
\end{proof}
On cherche maintenant à remplacer les expressions concernant $Q$ dans l'équation \eqref{eq:devlopfinssdescenttempoQ} par des expressions concernant $R$. Dans la dernière somme, on remarque qu'il y a une bijection 
\begin{align*}
\{(Q,s)\,|\,Q\in \F^G, P_0\subseteq Q, s\in  W^G(a_1;Q,G_\sigma), w_s^{-1}Qw_s\cap G_\sigma =R\}&\longrightarrow \F_R(M_1)\\
(Q,s)&\longmapsto P=w_s^{-1}Qw_s
\end{align*}
et que, avec $s_P$ la projection sur la deuxième coordonnée de la réciproque de la bijection, 
\[(-1)^{\dim a_Q^G}\widehat{\tau}_Q^G(H_{Q}(w_s\xi x)-T)=(-1)^{\dim a_P^G}\widehat{\tau}_P^G(H_{P}(\xi x)-s_P^{-1}(T-T_0)-T_0)\]
(\cite[pp.188-189]{Art86}), on peut faire un petit récapitulatif :
\begin{lemma}\label{Art86:lemma3.1}
Pour $T$ assez régulier, $J_\o^{\g,T}(f)$ est égal à l'intégrale sur $x\in G(F)\backslash G(\A_F)^1$ de la somme sur $R\in \{R\in \F^\sigma\,|\, P_{1\sigma}\subseteq R\}$ et la somme sur $\xi\in R(F)\backslash G(F)$ de
\begin{align*}
\sum_{U\in \mathcal{N}_{M_R}(F):\Ind_{M_R}^{G_\sigma}(U)=\o_{\nilp}} \int_{\n_R(\A_F)}f&\left[\left(\Ad  x^{-1}\xi^{-1}\right)\left(\sigma+U+V\right)\right]\,dV\\
&\times \sum_{P\in \F_R(M_1)}(-1)^{\dim a_P^G}\widehat{\tau}_P^G(H_{P}(\xi x)-s_P^{-1}(T-T_0)-T_0).   
\end{align*}
\end{lemma}

Introduisons ici des fonctions combinatoires d'Arthur. 
Pour $Y\in a_0$ un vecteur quelconque et $P\in\F^G$ on note $Y_P$ la projection orthogonale de $Y$ sur $a_P$. Il existe une fonction $\Gamma_P^G(-,Y_P)$ sur $a_0$ telle que
\[\Gamma_P^G(X,Y_P)=\sum_{Q\in\F^G:P\subseteq Q} (-1)^{\dim a_Q^G}\tau_P^Q(X)\widehat{\tau}_Q^G(X-Y_Q),\]
et
\[\widehat{\tau}_P^G(X-Y_P)=\sum_{Q\in\F^G:P\subseteq Q} (-1)^{\dim a_Q^G}\widehat{\tau}_P^Q(X)\Gamma_Q^G(X,Y_Q).\]
La fonction $\Gamma_P^G(X,Y_P)$ en $X$ ne dépend que de la projection de $X$ sur $a_P^G$, et est une fonction de support compact sur $a_P^G$.

Soit $R\in \F^\sigma$. Soit $\mathcal{Y}=(Y_P)_{P\in \F_R^0(M_1)}$ une famille de vecteurs de $a_0$ orthogonale au sens d'Arthur, autrement dit elle vérifie la condition suivante : si $P,P'$ sont adjacents dans $\F^G(M_1)$, le vecteur $Y_P-Y_{P'}$ appartient à la droite engendrée par la coracine associée à l'unique élément de $\Sigma(\mathfrak{p};A_{M_1}) \cap (-\Sigma(\mathfrak{q};A_{M_1}))$. Tout élément $Q$ de $\overline{\F}_R(M_1)$ contient un élément $P$ de $\F_R^0(M_1)$, on note $Y_Q$ la projection de $Y_P$ sur $a_Q$, ce vecteur ne dépend pas du choix de $P$ d'après la condition de compatilibité. Pour tout $B\in \F^\sigma$ avec $R\subseteq B$ on note $\mathcal{Y}_B=(Y_Q)_{Q\in \F_B(M_1)}$. On peut définir pour tout tel $B$ une fonction $\Gamma_R^G(-,\mathcal{Y}_R)$ sur $a_1$ tel que
\[\Gamma_R^G(X,\mathcal{Y}_R)=\sum_{B\in\F^\sigma : R\subseteq B}\tau_R^B(X)\left(\sum_{Q\in\F_B(M_1)} (-1)^{\dim a_Q^G}\widehat{\tau}_Q^G(X-Y_Q)\right),\]
et
\begin{equation}\label{eq:(3.3.3.1)}
\sum_{P\in\F_B(M_1)} (-1)^{\dim a_P^G}\widehat{\tau}_P^G(X-Y_P)=\sum_{B\in\F^\sigma : R\subseteq B}(-1)^{\dim a_R^B}\widehat{\tau}_R^B(X) \Gamma_B^G(X,\mathcal{Y}_B).    
\end{equation}
La fonction $\Gamma_S^G(X,\mathcal{Y}_S)$ en $X$ ne dépend que de la projection de $X$ sur $a_B$, et est une fonction de support compact sur $a_B^G$, son support dépend de façon continue de $\mathcal{Y}_R$. On a aussi 
\[\Gamma_R^G(X,\mathcal{Y}_R)=\sum_{P\in\F_R(M_1)}\Gamma_P^G(X,Y_P)\epsilon_P(X),\]
avec $\epsilon_P$ la fonction caractéristique de $a_P$. 

Revenons sur la déduction de la formule de descente semi-simple de $J_{\o}(f)$. Rappelons que l'on a fixé $K_\sigma=\prod_{v\in\V_F}K_{\sigma v}$ un sous-groupe ouvert compact de $G_\sigma(\A_F)$ en bonne position par rapport à $M_{1\sigma}$. Pour tout $x\in G_\sigma(\A_F)$, posons $K_R(x)$ un élément de $K_\sigma$ tel que $xK_R(x)^{-1}\in R(\A_F)$, il est uniquement déterminé modulo $M_R(\A_F)\cap K_\sigma$ à gauche. Tous les objets construits pour le triplet $(G,M_0,K)$ admet bien sûr des analogues pour le triplet $(G_\sigma,M_{1\sigma},K_\sigma)$, à titre d'exemple pour tout $R\in \F^\sigma$ on a la fonction $H_R: G_\sigma(\A_F)\rightarrow a_R$. Il y a aussi un vecteur $T_{0\sigma}\in a_1$ vecteur défini par l'équation \eqref{eq:deftroncatureT0} pour $(G_\sigma,M_{1\sigma},K_\sigma)$. Pour $T\in a_0$ un paramètre de troncature on note $T_\sigma$ sa projection sur $a_{M_1}$, alors $T_{0\sigma}$ et la valeur de $T_\sigma$ en $T=T_0$. On suppose que $P_1$ est choisi de sorte que si $T$ est assez régulier par rapport à $P_0$ alors $T_\sigma$ est assez régulier par rapport à $P_{1\sigma}$.

On a mis $J_\o^{\g,T}(f)$ sous forme du lemme \ref{Art86:lemma3.1}. Faisons le changement de variables
\begin{align*}
\left(R(F)\backslash G(F)\right)&\times\left(G(F)\backslash G(\A_F)^1\right) \\
&\longrightarrow \left(R(F)\backslash G_\sigma(F)\right)\times\left(G_\sigma(F)\backslash G_\sigma\cap G(\A_F)^1\right)\times\left(G_\sigma(\A_F)\backslash G(\A_F)\right) \\
(\xi,x) &\longmapsto (\delta,x,y)
\end{align*}
puis posons $\mathcal{Y}_{R}^T(\delta x,y)=(Y_P^T(\delta x,y))_{P\in \F^R(M_1)}$, $Y_P^T(\delta x,y)=H_P(K_{P_\sigma}(\delta x)y)+s_P^{-1}(T-T_0)-T_\sigma+T_0$, cette famille satisfait à la condition de compatibilité, en manipulant finalement l'équation \eqref{eq:(3.3.3.1)} on obtient
\begin{align*}
\sum_{P\in \F_R(M_1)}(-1)^{\dim a_P^G}&\widehat{\tau}_P^G(H_{P}(\xi x)-s_P^{-1}(T-T_0)-T_0)
\\&= \sum_{B\in \F^\sigma:R\subseteq B}  (-1)^{\dim a_R^B}\widehat{\tau}_R^B(H_R(\delta x)-T_\sigma) \Gamma_B^G(H_R(\delta x)-T_\sigma,\mathcal{Y}_{B}^T(\delta x,y)).
\end{align*}
La formule pour $J_\o^{\g,T}(f)$ devient alors 
\begin{equation}\label{Art86:eq6.2+6.3}
\begin{split}   
&\int_{y\in G_\sigma(\A_F)\backslash G(\A_F)}\sum_{B\in \F^\sigma:P_{1\sigma}\subseteq B}\int_{x\in G_\sigma(F)\backslash G_\sigma\cap G(\A_F)^1}\sum_{R\in \F^\sigma : P_{1\sigma}\subseteq R\subseteq B}\\
&\hspace{1cm}\sum_{\delta\in R(F)\backslash G_\sigma(F)}\sum_{U\in \mathcal{N}_{M_R}(F):\Ind_{M_R}^{G_\sigma}(U)=\o_{\nilp}} \int_{V'\in \n_R(\A_F)}f\left[\left(\Ad  y^{-1}x^{-1}\delta^{-1}\right)\left(\sigma+U+V'\right)\right]\\
&\hspace{2cm}\times  (-1)^{\dim a_R^B}\widehat{\tau}_R^B(H_R(\delta x)-T_\sigma) \Gamma_S^G(H_R(\delta x)-T_\sigma,\mathcal{Y}_{B}^T(\delta x,y))\,dV'\,dx\,dy. 
\end{split}
\end{equation}

Fisons ensuite le changement de variables
\begin{align*}
R(F)\backslash G_\sigma(F)&\longrightarrow \left(R(F)\cap M_B(F)\backslash M_B(F)\right)\times\left(B(F)\backslash G_\sigma(F)\right)\\
\delta &\longmapsto (\mu,\xi)
\end{align*}
Faisons sortir la somme sur $\xi$ en dehors de celle sur $R$, puis combinons-la avec l'intégrale sur $x$, on obtiendra une intégrale sur $B(F)\backslash G_\sigma\cap G(\A_F)^1$, décomposons-la en une intégrale quadruple sur
\[(v,a,m,k)\in (N_B(F)\backslash N_B(\A_F))\times (A_{B,\infty}\cap G(\A_F)^1)\times (M_B(F)\backslash M_B(\A_F)^1)\times K_\sigma,\]
nous aurons dans l'intégrande le terme $f\left[\left(\Ad  y^{-1}k^{-1}m^{-1}a^{-1}v^{-1}\right)\left(\sigma+U+V'\right)\right]$, la variable $a$ va disparaître dans l'argument car elle est dans $A_{B,\infty}$, puis avec un changement de variable $v^{-1}(U+V')v=U+V$, $V\in \n_R(\A_F)$ on peut également faire disparaître $v$ dans l'argument. Un facteur jacobien est introduit, mais il se simplifie avec celui introduit dans le dernier changement de variables. 
La variable $v$ disparaît tout aussi dans la dernière ligne de l'expression (\ref{Art86:eq6.2+6.3}), l'intégrale sur $v\in N_B(F)\backslash N_B(\A_F)$ disparaît en conséquence comme elle est absorbée par l'intégrale sur $V\in \n_R(\A_F)$. On fait encore un changement de variables
\begin{align*}
\n_R(\A_F)&\longrightarrow (\n_R(\A_F)\cap \m_B(\A_F))\times \n_B(\A_F)\\
V &\longmapsto (V_1,V_2)
\end{align*}
Avec ceci $J_\o^{\g,T}(f)$ vaut 
\begin{equation}
\int_y\sum_B\int_k\int_a\int_m\sum_R\sum_{\mu}\sum_{U}(-1)^{\dim a_R^B}\int_{V_1} \Phi_{B,a,k,y}^T((\Ad m^{-1}\mu^{-1})(U+V_1))\widehat{\tau}_R^B(H_R(\mu m)-T_\sigma),       
\end{equation}
avec $\Phi_{B,a,k,y}^T\in C_c^\infty(\m_B(\A_F))$ (rappelons que $f\in C_c^\infty(\g(F_S))\hookrightarrow C_c^\infty(\g(\A_F))$) la fonction dépendant de façon lisse en $y,k,a$ définie par
\begin{align*}
\Phi_{B,a,k,y}^T(C)=\gamma_{\A_F}^{G_\sigma}(B)\int_{\n_B(\A_F)}f\left((\Ad (ky)^{-1})(\sigma+C+V_2)\right)\Gamma_B^G(H_B(a)-T_\sigma,\mathcal{Y}_B^T(k,y))\,dV_2,
\end{align*}
pour tout $C\in\m_B(\A_F)$. En voyant l'intégrande de cette expression de $J_\o^{\g,T}(f)$ comme une fonction en $m$, on remarque directement que $\int_m\sum_R\sum_\mu\sum_U\int_{V_1}$ de l'intégrande vaut $J_{\o_\nilp}^{M_B,T_\sigma}(\Phi_{B,a,k,y}^T)$, ainsi 
\begin{align*}
J_\o^{\g,T}(f)=\int_{y\in G_\sigma(\A_F)\backslash G(\A_F)} \left(\sum_{B}\int_k\int_a  J_{\o_\nilp}^{M_S,T_\sigma}(\Phi_{B,a,k,y}^T)\,da\,dk\right)\,dy.   
\end{align*}
Spécialisons $T$ en $T_0$, nous voyons
\begin{align*}
J_\o^{\g}(f)=\int_{y\in G_\sigma(\A_F)\backslash G(\A_F)} \left(\sum_{B}\int_k\int_a  J_{\o_\nilp}^{M_S}(\Phi_{B,a,k,y}^{T_0})\,da\,dk\right)\,dy.   
\end{align*}
En analysant les fonctions $\Gamma_B^G$, on constate que (\cite[pp.199-200]{Art86})
\begin{align*}
J_\o^{\g}(f)=\int_{ G_\sigma(\A_F)\backslash G(\A_F)} \left(\sum_{B\in \F^\sigma : P_{1\sigma}\subseteq B} J_{\o_{\nilp}}^{M_B}(\Phi_{B,y,T_1})\right)\,dy,   
\end{align*}
où $T_1=T_0-T_{0\sigma}$, et pour tout $R\in \F^\sigma$, $\Phi_{R,y,T_1}\in C_c^\infty(\m_R(\A_F))$ est la fonction
\[\Phi_{R,y,T_1}(C)=\gamma_{\A_F}^{G_\sigma}(R)\int_{K_\sigma}\int_{\mathfrak{n}_R(\A_F)}f\left((\Ad (ky)^{-1})(\sigma+C+U)\right)v_R'(ky,T_1)\,dU\,dk,\,\,\,\,\forall C\in\m_R(\A_F),\]
qui dépend de manière lisse de $y$, avec $v_R'(ky,T_1)$ le nombre complexe défini selon ce qui suit : définissons dans un premier temps la $(G,M_1)$-famille $(v_Q(ky,T_1))_{Q\in\F^G(M_1)}$ par
\[v_Q(\lambda,ky,T_1)\eqdef e^{-\lambda(H_Q(ky)-T_1)},\,\,\,\,\lambda\in ia_Q^\ast.\]
On obtient alors $v_Q'(ky,T_1)$. On pose finalement
\[v_R'(ky,T_1)\eqdef \sum_{Q\in\F_R^0(M_1)} v_Q'(ky,T_1).\]
Or pour tous $B\in \F^\sigma$ avec $P_{1\sigma}\subseteq B$ et $w\in W_{M_{1\sigma}}^{G_\sigma}$ la définition de $T_0$ et $T_{0\sigma}$ (cf. l'équation \eqref{eq:deftroncatureT0}) nous permet d'affirmer que $\Phi_{B,y,T_1}(C_B)=\Phi_{(\Ad w)^{-1}B,y,T_1}((\Ad \tilde{w}^{-1})C_B)$ pour tout $C_B\in\m_B(\A_F)$, ici $\tilde{w}$ est un représentant de $w$ dans $K_\sigma M_{1\sigma}(\A_F)$. En l'occurrence, nous en déduisons que $J_{\o_\nilp}^{M_B}(\Phi_{B,y,T_1})=J_{(\Ad w)^{-1}\o_\nilp}^{(\Ad w)^{-1}M_B}(\Phi_{(\Ad w)^{-1}B,y,T_1})$. Cela entraîne la formule recherchée de descente semi-simple de $J_\o^\g(f)$ :
\[J_\o^\g(f)=\int_{G_\sigma(\A_F)\backslash G(\A_F)}\left(\sum_{R\in\F^{G_\sigma}(M_{1\sigma})}|W_{M_{1\sigma}}^{M_R}||W_{M_{1\sigma}}^{G_\sigma}|^{-1}J_{\o_\nilp}^{\m_R}(\Phi_{R,y,T_1})\right)\,dy.\]

\subsubsection{Développement fin de \texorpdfstring{$J_{\o}(f)$}{Jo(f)} : descente semi-simple des intégrales orbitales pondérées}

Soit toujours $S$ un sous-ensemble fini non-vide de $\V_F$.

\begin{lemma}[{{\cite{YDL23a}}}] Soit $M$ un sous-groupe de Levi de $G$. Soient $X\in \m(F)\subseteq \m(F_S)$, et $\sigma\eqdef X_\ss$. Si $a_{M}=a_{M_\sigma}$, alors pour toute fonction $f\in C_c^\infty(\g(F_S))$, $J_M^G(X,f)$ égale
\[|D^\g(X)|_S^{1/2}\int_{G_\sigma(F_S)\backslash G(F_S)}\left(\sum_{R\in\F^{G_\sigma}(M_\sigma)}J_{M_\sigma}^{M_R}(X_\nilp,\Phi_{(S),R,y,T_1})\right)\,dy,\]
avec $\Phi_{(S),R,y,T_1}\in C_c^\infty(\m_R(F_S))$ est la fonction
\[\Phi_{(S),R,y,T_1}(C)=\gamma_{S}^{G_\sigma}(R)\int_{K_{\sigma S}}\int_{\mathfrak{n}_R(F_S)}f\left((\Ad (ky)^{-1})(\sigma+C+U)\right)v_R'(ky,T_1)\,dU\,dk,\,\,\,\,\forall C\in\m_R(F_S).\]
\end{lemma}

\subsubsection{Lemmes de compacité}\label{subsubsec:lemmescompacite}
Avant d'entamer la discussion sur la forme ultime du développement fin, nous exposons deux lemmes supplémentaires. Tous les deux sont dérivés des généralisations d'un lemme de compacité de Harish-Chandra.

Dans cette sous-sous-section $F_v$ désigne un corps $p$-adique. Soit $G_v$ un groupe du type GL quasi-déployé sur $F_v$. Fixons $T$ un sous-tore maximal de $G_v$ non-ramifié, i.e. déployé sur une extension non ramifiée. Notons $F_v^{\text{nr}}$ le complété de la plus grande extension non ramifiée de $F_v$. \'{E}tant non-ramifié, $T$ admet un modèle canonique sur $\O_v$, on le note toujours par $T$. On dit qu'un élément $\sigma_v$ de $\mathfrak{t}(F_v^{\text{nr}})$ est entier si $\alpha(\sigma_v)\in \O_{F_v^{\text{nr}}}$ pour toute racine $\alpha\in \Sigma(\g_{v,F_v^{\text{nr}}};T_{F_v^{\text{nr}}})$. L’injection naturelle
\[\text{Norm}_{G_v(F_v^{\text{nr}})}(T)/T(F_v^{\text{nr}})\to \text{Norm}_{G_v(\overline{F_v})}(T)/T(\overline{F_v})\]
est bijective. Si $\sigma_v\in\mathfrak{t}(F_v^{\text{nr}})$ est entier, on dit que $\sigma_v$ est de réduction régulière si pour tout $w\in (\text{Norm}_{G(\overline{F_v})}(T)/T(\overline{F_v}))\setminus\{1\}$, il existe $\alpha\in \Sigma(\g_{v,F_v^{\text{nr}}};T_{F_v^{\text{nr}}})$ tel que $\alpha(w(\sigma_v)-\sigma_v)$ soit une unité.

\begin{lemma}[{{\cite[7.2. lemme]{Walds97}}}]\label{lem:Kottcomp1}
Soit $\sigma_v\in T(F_v)$ un élément entier de réduction régulière.
\begin{enumerate}
    \item Le groupe $G_{v,\sigma_v}$ est non-ramifié sur $F_v$, et $K_v\cap G_{v,\sigma_v}$ est en bonne position par rapport à un sous-groupe de Levi minimal de $G_{v,\sigma_v}$.
    \item Soit $\sigma'\in \mathfrak{k}_v$. Alors $\sigma_v$ et $\sigma'$ sont $G_v(F_v)$-conjugué si et seulement s'ils sont $K_v$-conjugués.
\end{enumerate}    
\end{lemma}

On a $\mathfrak{t}\subseteq \g_{v,\sigma_v}\subseteq \g_v$. Pour $H_v=G_{v,\sigma_v}$ ou $G_v$ on écrit $\mathfrak{c}_{H_v} : \mathfrak{h}_v\to \mathfrak{h}_v^{\GIT}\eqdef\text{Spec}(F_v[\mathfrak{h}_v]^{H_v})$ le morphisme de Chevalley, avec $F_v[\mathfrak{h}_v]$ la $F_v$-algèbre des fonctions régulières du $F_v$-schéma $\mathfrak{h}_v$, et $F_v[\mathfrak{h}_v]^{H_v}$ la sous-$F_v$-algèbre de $F_v[\mathfrak{h}_v]$ consistant en les éléments invariants par $H_v$-conjugaison.

Soit $\g_{v,\sigma_v}'$ la sous-variété ouverte $\{X\in \g_{v,\sigma_v}\mid\det(\ad(X);\g_v/\g_{v,\sigma_v})\not=0\}$ de $\g_{v,\sigma_v}$. On a $\sigma_v\in \g_{v,\sigma_v}'$.

\begin{lemma}[{{\cite[lemme 14.1]{Kottbook}}}]\label{lem:Kottcomp2}
Soient $G_v\underset{G_{v,\sigma_v}}{\times}\g_{v,\sigma_v}'$ le quotient de $G_v\times \g_{v,\sigma_v}'$ par l'action de $G_{v,\sigma_v}$ par $h_v\cdot (g_v,X_v)=(h_vg_v,(\Ad h_v)X_v)$, et $\g_v\times_{\mathfrak{g}_v^{\GIT}}\mathfrak{g}_{v,\sigma_v}^{\GIT}$ le produit fibré évident. Alors le morphisme 
\begin{align*}
 G_v\underset{G_{v,\sigma_v}}{\times}\g_{v,\sigma_v}'&\to \g_v\times_{\mathfrak{g}_v^{\GIT}}\mathfrak{g}_{v,\sigma_v}^{\GIT}\\
 (g_v,X_v)&\mapsto ((\Ad g_v^{-1})X_v,\mathfrak{c}_{G_{v,\sigma_v}}(X_v))
\end{align*}
est un isomorphisme.
\end{lemma}

\subsubsection{Développement fin de \texorpdfstring{$J_{\o}(f)$}{Jo(f)} : formule ultime}

Comme précédemment $\o\in \mathcal{O}^\g$ est une classe de conjugaison, et $\sigma\in \g_{\ss}(F)$ un élément semi-simple définissant la classe $\o_\ss$, $F$-elliptique dans une sous-algèbre de Levi semi-standard. Comme précédemment on va d'abord énoncer le théorème attendu.

Soient $S$ un sous-ensemble fini de $\V_F$ contenant les places archimédiennes, $M$ un sous-groupe de Levi de $G$ et $X\in\m(F)$ ayant $\sigma$ comme partie semi-simple. Définissons pour $S$ assez grand
\[a^M(S,X)\eqdef \epsilon^M(S,\sigma)a^{M_\sigma}(S,(\Ad M_\sigma(F))X_\nilp),\]
avec $\epsilon^M(S,\sigma)$ vaut 1 si $\sigma$ est $F$-elliptique dans $\m$ et sa classe de $M(\A_F^S)$-conjugaison rencontre $\mathfrak{k}^S\cap \m(\A_F^S)$, et 0 sinon. Le nombre complexe $a^M(S,X)$ ne dépend que de la classe de $M(F)$-conjugaison de $X$. On peut donc aussi écrire $a^M(S,\o')$ à la place de $a^M(S,X)$, où $\o'=(\Ad M(F))X$. 

\begin{theorem}[Développement fin]\label{thm:devfin}
Il existe $S_\o$ un sous-ensemble fini de $\V_F$, contenant les places archimédiennes, tel que pour tous $S\supseteq S_\o$ et $f\in\S(\g(F_S))$,
\[J_\o^\g(f)=\sum_{M\in\L^G(M_0)}|W_0^M||W_0^G|^{-1}\sum_{X\in \O^{\m} : \Ind_M^G(X)=\o}a^M(S,X)J_M^G(X,f).\]
\end{theorem}

\begin{remark}~{}
\begin{enumerate}
    \item Dans la définition de $a^M(S,X)$, l'analogue pour les groupes de la condition « sa classe de $M(\A_F^S)$-conjugaison rencontre $\mathfrak{k}^S\cap \m(\A_F^S)$ » - à savoir « sa classe de $M(\A_F^S)$-conjugaison rencontre $K^S\cap M(\A_F^S)$ » - n'est pas explicitement mentionnée par Arthur dans la discussion après l'équation (8.1) de son article \cite{Art86}, elle est toutefois incluse dans la preuve du théorème 8.1 \textit{Ibid}.  
    \item On se réfère au théorème \ref{thm:devlopfinnilp} pour la dépendance de $a^M(S,X)$ en les mesures.
\end{enumerate}
\end{remark}

On procède à sa démonstration.

Afin de relier $J_\o^g$, qui est a priori un objet global, à des intégrales orbitales pondérées, qui sont a priori des objets semi-locaux, nous devons bien choisir $S$ un sous-ensemble fini de places et nous restreindre à l'espace $\S(\g(F_S))$. Fixons $S_\o$ un sous-ensemble fini de $\V_F$ qui contient $\V_\infty$ et vérifie que si $v\not\in S_\o$ alors :
\begin{enumerate}[label=(\roman*)]
    \item $|D^\g(\sigma)|_v=1$ ;
    \item $\sigma\in\mathfrak{k}_v$ ;
    \item $K_v\cap G_{\sigma}(F_v)=K_{\sigma v}$ ;
    \item $(\Ad K_v)(\mathfrak{k}_v)=\mathfrak{k}_v$ ;
    \item $\vol(K_{\sigma v})=1$ ;
    \item $\gamma_{v}^{G_\sigma}(R)=1$ pour tout $R\in \F^{G_{\sigma}}(M_{1\sigma})$ ;
    \item $\vol(\mathfrak{k}_v\cap \mathfrak{n}_R(F_v);\mathfrak{n}_R(F_v))=1$ pour tous $R\in\F^{\sigma}=\F^{G_\sigma}(M_{1\sigma})$ ;
    \item si $y_v\in G(F_v)$ est tel que $(\Ad y_v^{-1})(\sigma+\mathcal{N}_{G_{\sigma}}(F_v))$ rencontre $\mathfrak{k}_v$ alors $y_v\in G_\sigma(F_v)K_v$.
\end{enumerate}
Les conditions hormis la (iii) et les deux dernières sont claires. La condition (iii) vient du point 1 du lemme \ref{lem:Kottcomp1}. L'avant-dernière condition est licite pour des raisons suivantes : on a une famille de  réseau $(\mathfrak{k}_{\sigma v})_v$ avec $\mathfrak{k}_{\sigma v}$ un réseau de l'espace vectoriel $\g_\sigma(F_v)$, qui est telle que $\mathfrak{k}_{\sigma v}=\gl_n(\O_v)\cap \g_\sigma(F_v)$ pour presque toutes les places $v$. Par la condition (ii) on a, pour presque toutes $v$, que $\mathfrak{k}_{\sigma v}\cap \mathfrak{n}_R(F_v)=\mathfrak{k}_{v}\cap \mathfrak{n}_R(F_v)$ pour tous $R$, 
la condition réclamée est ainsi loisible.  Pour finir, la dernière est une conséquence des lemmes dans la sous-sous-section précédente. En effet, prenons $T$ un tore maximal de $G$ dont l'algèbre de Lie contient $\sigma$. Bien sûr $T_v$ est non-ramifié pour presque tous les $v$. On voit que si $y_v\in G_{\sigma}(F_v)\backslash G(F_v)=(G_{\sigma}\backslash G)(F_v)$ est tel que $(\Ad y_v^{-1})(\sigma+\mathcal{N}_{G_{\sigma}}(F_v))$ rencontre $\mathfrak{k}_v$, alors pour presque tous les $v$, on a que $y_v$ appartient à la projection sur la composante $G_\sigma\backslash G$ de l'image réciproque de $\mathfrak{k}_v\times_{\text{Spec}(F_v[\mathfrak{g}_v])^{G_v}}\{\mathfrak{c}_{G_{v,\sigma_v}}(\sigma_v)\}$ par l'isomorphisme du lemme \ref{lem:Kottcomp2}. Ainsi, par le lemme \ref{lem:Kottcomp1}, $y_v$ appartient, pour presque tous les $v$, à $G_{\sigma}(F_v)\backslash G(F_v)K_v$.

Soit à présent $S$ un sous-ensemble fini de $\V_F$ contenant $S_\o$, et soit $f\in C_c^\infty(F_S)\subseteq C_c^\infty(\g(\A_F))$. La fonction
\[y\in G(\A_F)\longmapsto\sum_{R\in\F^{\sigma}}|W_{M_{1\sigma}}^{M_R}||W_{M_{1\sigma}}^{G_\sigma}|^{-1}J_{\o_\nilp}^{\m_R}(\Phi_{R,y,T_1})\]
s'annule sauf si $y=y_Sy^S$ avec $y_S\in G_\sigma(F_S)\backslash G(F_S)$ et $y^S$ appartient à
\[\prod_{v\not\in S} G_\sigma(F_v)\backslash G_\sigma(F_v)K_v\simeq\prod_{v\not\in S} (K_v\cap G_{\sigma}(F_v))\backslash K_v=\prod_{v\not\in S}K_{\sigma v}\backslash K_v\]
selon les conditions (iii), (iv) et (viii). Pour de tels $y$ il n'y a pas de doute que 
\[v_R'(ky,T_1)=v_R'(k_Sy_S,T_1)\]
pour tous $R\in\F^{\sigma}$ et $k=k_Sk^S\in K_\sigma$. Il s'ensuit que la fonction $\Phi_{R,y,T_1}\in C_c^\infty(\g(\A_F))$ est l'image de $\Phi_{(S),R,y_S,T_1}\in C_c^\infty(\g(F_S))$ sous l'inclusion $C_c^\infty(\g(F_S))\subseteq C_c^\infty(\g(\A_F))$, en effet l'égalité
\[\Phi_{R,y,T_1}(C)=\Phi_{(S),R,y_S,T_1}(C_S)\prod_{v\not\in S}1_{\mathfrak{k}_v}(C_v),\,\,\,\,\forall C\in \m_R(\A_F)\]
est équivalente à
\[\prod_{v\not\in S}\gamma_{v}^{G_\sigma}(R)\int_{K_{\sigma v}}\int_{\mathfrak{n}_R(F_v)}1_{\mathfrak{k}_v}\left((\Ad (k_vy_v)^{-1})(\sigma_v+C_v+U_v)\right)\,dU_v\,dk_v=\prod_{v\not\in S}1_{\mathfrak{k}_v}(C_v),\]
soit, en utilisant la condition sur $y_v$ et les hypothèses (ii), (iii), (iv) et (vi) ci-dessus,
\[\prod_{v\not\in S}\vol(K_{\sigma v})\int_{\mathfrak{n}_R(F_v)}1_{\mathfrak{k}_v}(C_v+U_v)\,dU_v=\prod_{v\not\in S}1_{\mathfrak{k}_v}(C_v),\]
qui est vrai d'après les hypothèses (v) et (vii). Il vient pour cette raison
\[J_\o^\g(f)=\int_{G_\sigma(F_S)\backslash G(F_S)}\left(\sum_{R\in\F^{\sigma}}|W_{M_{1\sigma}}^{M_R}||W_{M_{1\sigma}}^{G_\sigma}|^{-1}J_{\o_\nilp}^{\m_R}(\Phi_{(S),R,y_S,T_1})\right)\,dy.\]
pour toute fonction $f\in C_c^\infty(\g(F_S))$. 

On peut maintenant comparer la contribution $J_\o^{\g}(f)$ de la classe $\o$ (globale) et les intégrales orbitales pondérées (semi-locales) grâce aux formules de descente semi-simple de chacun. Posons $\L^\sigma=\L^{G_\sigma}(M_{1\sigma})$, $\F^\sigma(L)=\{P\in \F^\sigma\mid L\subseteq P\}$ pour $L\in\L^\sigma$, et $\L_\sigma^0=\{M\in \L^G(M_1)\mid a_M=a_{M_\sigma}\}$. L'application $\L_\sigma^0\rightarrow \L^\sigma$ qui envoie $M$ sur $M_\sigma$ est une bijection. Nous avons
{\allowdisplaybreaks\begin{equation}\label{Art86:lemma7.1}
\begin{split}
J_\o^\g(f)&=\int_{G_\sigma(F_S)\backslash G(F_S)}\left(\sum_{R\in\F^{\sigma}}|W_{M_{1\sigma}}^{M_R}||W_{M_{1\sigma}}^{G_\sigma}|^{-1}J_{\o_\nilp}^{\m_R}(\Phi_{(S),R,y_S,T_1})\right)\,dy\\
&=\int_{y}\sum_{L\in \F^\sigma}\sum_{R\in\F^{\sigma}(L)}|W_{M_{1\sigma}}^{L}||W_{M_{1\sigma}}^{G_\sigma}|^{-1}\sum_{U\in(\mathcal{N}_L(F))}a^L(S,U)J_L^{M_R}(U,\Phi_{(S),R,y,T_1})\,dy \\
&=\sum_{M\in \L_\sigma^0}|W_{M_{1\sigma}}^{M_\sigma}||W_{M_{1\sigma}}^{G_\sigma}|^{-1}\sum_{U\in(\mathcal{N}_M(F))}a^{M_\sigma}(S,U)\left(|D^\g(\sigma)|_S^{1/2}\int_y\sum_R J_{M_\sigma}^{M_R}(U,\Phi_{(S),R,y,T_1})\,dy\right)\\
&=\sum_{M\in \L_\sigma^0}|W_{M_{1\sigma}}^{M_\sigma}||W_{M_{1\sigma}}^{G_\sigma}|^{-1}\sum_{U\in(\mathcal{N}_M(F))}a^{M_\sigma}(S,U)J_M^G(\sigma+U,f),    
\end{split}
\end{equation}}
la troisième égalité vient de la condition (i) pour $S_\o$ et la formule du produit : $|D^{\g}(\sigma)|_S=1$.

Nous souhaitons enlever la dépendance en ce choix $\sigma$ dans sa classe de conjugaison dans la dernière expression. Soient $M$ un sous-groupe de Levi de $G$ et $X\in\m(F)$ ayant $\sigma$ comme partie semi-simple. Définissons les nombres $a^M(S,X)$ comme dans l'énoncé. Alors en partant de la dernière égalité de l'équation (\ref{Art86:lemma7.1}), on aboutit, par des mêmes raisonnements formels qu'en \cite[théorème 8.1]{Art86} (omis ici), au théorème \ref{thm:devfin}, au moins pour les fonctions à support compact. Mais l'égalité en question entre deux distributions tempérées étant prouvée pour les fonctions à support compact, elle reste valable pour les fonctions tests de classe Schwartz-Bruhat. La preuve du théorème se conclut.

\section{Formule des traces pour les algèbres de Lie à composante centrale fixée}\label{sec:formuledestracesII}

\subsection{Transformée de Fourier sur les algèbres de Lie}

Notons $\tau_\g$ la trace réduite sur $F$ de l'algèbre séparable sous-jacente de $\g$. On fixe $\langle X,Y\rangle =\tau_\g(XY)$ comme forme bilinéaire sur $\g(F)$ non-dégénérée et invariante par adjonction. Nous appellerons $\langle-,-\rangle$ la forme bilinéaire canonique de $\g$ dans cet article. La forme bilinéaire canonique s'étend sur $\g(\A_F)$, et il vaut le produit des formes bilinéaires canoniques locales, i.e. $\langle-,-\rangle=\prod_{v\in\V_F}\langle-,-\rangle_v$.

On se donne un caractère non-trivial $\psi$ de $F\backslash \A_F$. La transformée de Fourier d'une fonction $f \in\S(\g(\A_F))$ est alors 
\[\widehat{f}^{\g}(Y)=\int_{\g(\A_F)}f(X)\psi(\langle X,Y\rangle)\,dX,\,\,\,\,\forall Y\in \g(\A_F)\]
avec $dX$ la mesure auto-duale sur $\g(\A_F)$. 

Nous avons la formule du produit
\begin{equation}\label{eq:formuleproduitWeil}
\prod_{v\in\V_F} \gamma_{\psi_v}(\g_v)=1    
\end{equation}
(cf. \cite[proposition 5]{Weil64}).

On va abréger $\widehat{f}^{\g}$ en $\widehat{f}$ si le contexte nous le permet. On remarque que la décomposition $\g=\mathfrak{z}\oplus\g_{\ad}$ est orthogonale par rapport à la forme bilinéaire $
\langle-,-\rangle$. 

\subsection{Normalisations des mesures de la section \ref{sec:formuledestracesII}}

Dans la suite, pour $V=\g$ ou $\g_{\ad}$, on prend sur $V(\A_F)$ la mesure de Haar auto-duale relativement à sa transformée de Fourier. La mesure sur aucun autre groupe ne revêt d'importance dans cette section.
 
\subsection{Formule des traces pour les algèbres de Lie à composante centrale fixée} 

Rappelons au prime abord la formule des traces pour les algèbres de Lie.

\begin{theorem}[Formule des traces pour les algèbres de Lie {{\cite[théorème 4.2 et théorème 4.5]{Ch02a}}}] \label{thm:FTLiealg}
Pour tous $f\in\S(\g(\A_F))$ et $T\in a_0$,
\begin{equation}\label{eq:FTLiealg}
\sum_{\o\in\O^\g} J_\o^T(f)=\sum_{\o\in\O^\g} J_\o^T(\widehat{f}),    
\end{equation}
ces deux sommes convergent absolument.
\end{theorem}
\begin{remark}
La formule sommatoire de Poisson nous dit que $\vol(\g(F)\backslash \g(\A_F))=1$ avec $\g(F)$ se munit de la mesure de comptage et $\g(F)\backslash \g(\A_F)$ se munit de la mesure quotient. On utilise ainsi la même normalisation que l'article \cite{Ch02a} sur $\g(\A_F)$. 
\end{remark}

Introduisons une variante de la formule des traces pour les algèbres de Lie qui nous sera ultérieurement indispensable.

\'{E}crivons $\g=\mathfrak{z}\oplus \g_{\ad}$ avec $\mathfrak{z}=\Lie(Z(G))$ et $\g_{\ad}=[\g,\g]$. Posons $\pi_\mathfrak{z}: \g\rightarrow\mathfrak{z}$ la projection de $\g$ sur $\mathfrak{z}$ selon la décomposition. La conjugaison par $G$ sur $\mathfrak{z}$ est triviale. Toute classe de $G(F)$-conjugaison dans $\g(F)$ se décompose en une somme d'un élément dans $\mathfrak{z}(F)$ et une classe de $G(F)$-conjugaison dans $\g_{\ad}(F)$. 

Rappelons que la décomposition $\g(\A_F)=\mathfrak{z}(\A_F)\oplus \g_{\ad}(\A_F)$ est orthonogale par rapport à $\langle-,-\rangle$, et que les mesures sur $\g(\A_F)$ et $\g_\ad(\A_F)$ sont auto-duales. On définit alors la transformée partielle de Fourier d'une fonction $f \in\S(\g(\A_F))$ par
\[\widetilde{f}(Y)=\int_{\g_{\ad}(\A_F)}f(\pi_\mathfrak{z}(Y)+X)\psi(\langle X,Y\rangle)\,dX,\,\,\,\,\forall Y\in \g(\A_F).\]
Il est clair, selon nos choix de mesures, que $\widetilde{\widetilde{f}}(Z+Y)=f(Z-Y)$ si $(Z,Y)\in\mathfrak{z}(\A_F)\oplus \g_{\ad}(\A_F)$. 

Soit $Z\in \mathfrak{z}(F)$. Prenons $g_{Z}'\in\S(\mathfrak{z}(\A_F))$ tel que si $W\in \mathfrak{z}(F)$ alors $g_{Z}'(W)\not=0$ seulement si $W=Z$, et auquel cas $g_{Z}'(Z)=1$. L'existence de telle fonction résulte du fait que $F$ est discret dans $\A_F$ avec $F\backslash\A_F$ compact. Prenons ensuite $g_{Z}\in\S(\mathfrak{g}(\A_F))$ la fonction $g_{Z}(X)=g_{Z}'(\pi_\mathfrak{z}(X))$ pour tout $X\in\g(\A_F)$. On voit, pour toute fonction $f\in\S(\g(\A_F))$, que
\begin{align*}
\widetilde{fg_{Z}}(Y)&=\int_{\g_{\ad}(\A_F)}f(\pi_\mathfrak{z}(Y)+X)g_{Z}(\pi_\mathfrak{z}(Y)+X)\psi(\langle X,Y\rangle)\,dX   \\
&=\int_{\g_{\ad}(\A_F)}f(\pi_\mathfrak{z}(Y)+X)\psi(\langle X,Y\rangle)\,dX\cdot g_{Z}(Y) \\
&=\widetilde{f}(Y)g_{Z}(Y),\,\,\,\,\forall Y\in \g(\A_F),
\end{align*}
en d'autres termes $\widetilde{fg_{Z}}=\widetilde{f} g_{Z}$.

En plongeant la fonction $fg_{Z}$ dans la formule des traces pour les algèbres de Lie on en déduit que
\begin{lemma} Pour tous $Z\in \mathfrak{z}(F)$ et $T\in a_0$ assez régulier, la distribution
\[J_{[Z]}^T\eqdef \sum_{\o\in\O^\g,\pi_\mathfrak{z}(\o)=Z} J_\o^T\]
est en polynôme en $T$, de degré au plus $\dim a_0^G$.
\end{lemma}
\begin{proof}
Cela se déduit des théorèmes \ref{thm:classconjdisJ} et \ref{thm:FTLiealg}.
\end{proof}

\begin{lemma} Soit $\mathcal{K}$ une partie compacte de $\S(\g(\A_F))$, il existe $c_{\mathcal{K}}>0$ tel que pour tous $f\in \mathcal{K}$ on ait
\[\left| J_{[Z]}^T(f)-\int_{G(F)\backslash G(\A_F)^1}F^G(x,T)\sum_{X \in Z\oplus \g_{\ad}(F)}f\left((\Ad  x^{-1})X\right)\,dx\right|< c_{\mathcal{K}}e^{-d(T)}.\]
\end{lemma}
\begin{proof}
Cela se déduit du théorème \ref{thm:classconjdisJ}.
\end{proof}

\begin{proposition}\label{prop:TFtracefixee}
Pour tous $f\in\S(\g(\A_F)),Z\in \mathfrak{z}(F)$ et $T\in a_0$ assez régulier,
\[\sum_{\o\in\O^\g,\pi_\mathfrak{z}(\o)=Z} J_\o^T(f)=\sum_{\o\in\O^\g,\pi_\mathfrak{z}(\o)=Z} J_\o^T(\widetilde{f}).\]
Ces deux sommes convergent absolument.
\end{proposition}
\begin{proof}
Puisque $\g(\A_F)=\mathfrak{z}(\A_F)\oplus \g_{\ad}(\A_F)$ est orthonogale, la formule sommatoire de Poisson donne
\[\sum_{X \in Z\oplus \g_{\ad}(F)}f\left((\Ad  x^{-1})X\right)=\sum_{X \in Z\oplus \g_{\ad}(F)}\widetilde{f}\left((\Ad  x^{-1})X\right),\,\,\,\,\forall x\in G(\A_F).\]
Ceci et les lemmes précédents achèvent la preuve.
\end{proof}

\section{Comparaison d'un groupe du type GL avec son algèbre de Lie}\label{sec:formuledestracesIII}

\subsection{Normalisations des mesures de la section \ref{sec:formuledestracesIII}}
Pour tout $P\in\P^G(M_0)$, on choisit sur $N_P(\A_F)$ (resp. $\n_P(\A_F)$) l'unique mesure de telle sorte que $\vol(N_P(F)\backslash N_P(\A_F))=1$ (resp. $\vol(\n_P(F)\backslash \n_P(\A_F))=1$) avec $N_P(F)$ (resp. $\n_P(F)$) muni de la mesure de comptage. La mesure sur aucun autre groupe ne revêt d'importance dans cette section.

\subsection{Comparaison d'un groupe du type GL avec son algèbre de Lie}\label{subsec:comparaisonTFgroupLiealg}

Comme par ailleurs nous ferons usage de la formule des traces non-invariante pour les groupes, il est primordial d'établir le passage entre la formule des traces pour un groupe du type GL et celle pour son algèbre de Lie. Pour cela, rappelons avant tout la formule des traces pour les groupes. On note $\O^G$ l’ensemble des classes de $G(F)$-conjugaison de $G(F)$. Pour $\o \in \O^G$, $f$ une fonction test sur $G(\A_F)^1$ et $P$ un sous-groupe parabolique standard, on définit
\[K_{P,\o}^G(x,f)=\sum_{\gamma\in M_P(F):\Ind_{M_P}^G(\gamma)=\o}\int_{N_P(\A_F)}f((\Ad x^{-1})(\gamma n))\,dn,\]
et
\[K_\o^{G,T}(x,f)=\sum_{P:P_0\subseteq P}(-1)^{\dim a_P^G}\sum_{\delta\in P(F)\backslash G(F)}\widehat{\tau}_P^G(H_0(\delta x)-T)K_{P,\o}^G(\delta x,f)\]
pour $T\in a_0$ assez régulier. Finalement,
\[J_\o^{G,T}(f)=\int_{G(F)\backslash G(\A_F)^1} K_\o^{G,T}(x,f)\,dx\]
et 
\begin{equation}\label{eqdef:cotegeomTFGrpconto}
J_\o^{G}(f)=J_\o^{G,T_0}(f). 
\end{equation}
Ces distributions sont les ingrédients du coté géométrique.

La flèche $G(F)\hookrightarrow\g(F)$ induit une inclusion $\O^G\subseteq\O^\g$ parce qu'elle respecte les adjonctions par $G(F)$.

Soit $M\in \L^G$. On a défini l'induction $\Ind_M^G:\O^{\m}\to \O^{\g}$ dans le numéro \ref{subsubsec:orbind}. On voit, par le théorème \ref{thm:inddes}, qu'elle donne une application $\Ind_M^G:\O^{M}\to \O^{G}$ par restriction. Or pour tout $\gamma\in M(F)$ et $P\in \P^G(M)$, $(
\Ad M)\gamma+\mathfrak{n}_P=((\Ad M)\gamma )N_P$,  l'induite $\Ind_M^G(\gamma)$ est donc l’unique orbite dans $G$ pour l’action de conjugaison de $G$ telle que pour tout $P\in \P^G(M)$ l’intersection
\begin{equation*}
\Ind_M^G(\gamma)\cap \left((\Ad M)\gamma \right) N_{P}    
\end{equation*}
soit un ouvert Zariski dense dans $((\Ad M)\gamma )N_{P}$.

\begin{proposition}\label{prop:Lie=Grp} Soit $\o\in \O^G$, alors 
\[J_{\o}^{\g}(f)=J_{\o}^{G}(f)\]
si l'un des deux converge absolument.
\end{proposition}
\begin{proof}
Nous prouverons en fait 
\[K_{P,\o}^{\g}(x,f)=K_{P,\o}^{G}(x,f)\]
pour tous $x$ et $P$. Il est clair que $\{X\in\m_P(F):\Ind_{M_P}^G(X)=\o\}=\{\gamma\in M_P(F):\Ind_{M_P}^G(\gamma)=\o\}$. Quitte à conjuguer $f$ on est conduit à démontrer 
\begin{equation}\label{eq:unipotentdevissage}
\int_{\mathfrak{n}_P(\A_F)}f(\gamma+U)\,dU=\int_{N_P(\A_F)}f(\gamma n)\,dn,\,\,\,\,\forall \gamma\in M_P(F).    
\end{equation}
Pour ce faire, dévissons $N_P$ : on peut supposer que $G=\Res_{E/F}\GL_{m,D}$ et après une conjugaison $P$ est le sous-groupe parabolique associé à la partition $(m_1,\dots,m_k)$ de $m$, formé par des matrices triangulaires supérieures par blocs. Posons (la matrice est nulle partout  en dehors des coefficients indiqués)
\[N_i=\begin{bmatrix}
\text{Id}_{m_1}  &   &         &         &  &\\
  & \ddots   &    &        &     &\\
       &  &  \text{Id}_{m_i} &  \Res_{E/F}\Mat_{m_i\times m_{i+1}}   & \cdots &\Res_{E/F}\Mat_{m_i\times m_{k}}\\
            &         &  & \text{Id}_{m_{i+1}} &   & \\
            &         &  &  & \ddots  & \\
     &         &         &  & & \text{Id}_{m_k}
\end{bmatrix}\]
pour $i=1,\dots,k-1$. Alors $N_i$ s'identifie naturellement avec son algèbre de Lie $\n_i$. On a, selon nos choix de mesures,
\[dn=dn_{k-1}\cdots dn_1=\frac{\vol(N_P(F)\backslash N_P(\A_F))}{\vol(\n_P(F)\backslash \n_P(\A_F))}dU_{k-1}\cdots dU_{1}=dU\]
où $dn_i$ (resp. $dU_i$) est une mesure sur $\frac{N_{k-1}(\A_F)\cdots N_{i}(\A_F)}{N_{k-1}(\A_F)\cdots N_{i+1}(\A_F)}\simeq N_i(\A_F)$ (resp. $\frac{\n_{k-1}(\A_F)\oplus\cdots \oplus \n_{i}(\A_F)}{\n_{k-1}(\A_F)\oplus\cdots \oplus \n_{i+1}(\A_F)}\simeq\mathfrak{n}_i(\A_F)$), $dn$ est la mesure sur $N_P(\A_F)$ et $dU$ est une mesure sur $\mathfrak{n}_P(\A_F)$. Posons les caractères pour $i\in\{1,\dots,k\}$
\[\chi_i:(M_j)_{j\in\{1,\dots,k\}}\in M_P=\prod_{j\in\{1,\dots,k\}} \Res_{E/F}\GL_{n_j,D}\longmapsto N_{E/F}\circ\nu_{\Mat_{m_i,D}}(M_i)\]
avec $N_{E/F}$ la norme relative à l'extension $E/F$ et $\nu$ la norme réduite d'une algèbre simple centrale. 

Revenons sur la démonstration. Le morphisme de variétés  $U\in \mathfrak{n}_P\mapsto \text{Id}+U\in N_P$ étant un isomorphisme de Jacobien 1, elle nous fournit
\[\int_{N_P(\A_F)}f(\gamma n)\,dn=\int_{\mathfrak{n}_P(\A_F)}f(\gamma (1+U'))\,dU'=\text{Jac}\cdot \int_{\mathfrak{n}_P(\A_F)}f(\gamma+U)\,dU\]
avec $\text{Jac}=\prod_{i=1}^{k-1}|\chi_i(\gamma)|_{\A_F}^{(m_{i+1}+\cdots +m_{k})\sqrt{\dim_ED}}=1$, ceci prouve donc l'équation \ref{eq:unipotentdevissage} et la preuve est achevée.
\end{proof}

En sommant sur les classes $\o\in\O^G$ on obtient le côté géométrique de la formule des traces non-invariante
\begin{equation}\label{eqdef:cotegeomFTGrp}
J_{\text{géom}}^G(f)=\sum_{\o\in\O^G} J_{\o}^G(f).    
\end{equation}

\begin{corollary}
Le côté géométrique de la formule des traces non-invariante pour le groupe $G$ est une fonctionnelle continue sur $\S(\g(\A_F))$ (restreint sur $G(\A_F)^1$).     
\end{corollary}

\begin{remark}
Dans \cite{FiLa16}, la décomposition du côté géométrique de la formule des traces non-invariante selon une certaine relation d'équivalence est obtenue pour un autre espace de fonctions tests, leur résultat s'applique à tous les groupes réductifs. Dans le cas d'un groupe du type GL, leur relation d'équivalence correspond à la $G(F)$-conjugaison considérée ici.
\end{remark}

\section{Compatibilité locale-globale}\label{sec:compatibilitélocal-global}

Le développement fin du côté géométrique de la formule des traces est achevé, il nous incombe de démontrer et d'exiger encore des compatibilités locales-globales pour réduire les analyses globales aux analyses locales. Tel est le dessein de la présente section.

\subsection{Compatibilité des objets à l'égard du transfert}
\subsubsection{Sous-groupes semi-standard}
Soient $G$ un groupe du type GL sur $F$, $G^\ast$ sa forme intérieure quasi-déployée,  et $\eta \in H^1(F,G_\ad^\ast)$ la classe définissant ce torseur intérieur. Pour tout $v\in\V_F$, on choisit une injection $\overline{F}\hookrightarrow\overline{F_v}$, il y a un torseur intérieur $\eta_v\in H^1(F_v,G_{v,\ad}^\ast)$ défini par extension des scalaires de $\eta$. Le diagramme
\[
    \begin{tikzcd}
    \g \arrow{r}{\eta}\arrow{d}{}& \g^\ast \arrow{d}{}\\
    \g_v \arrow{r}{\eta_v} & \g_v^\ast
    \end{tikzcd}
    \]
commute et est compatible avec des conjugaisons.

Prenons $M_0$, resp. $M_{0^\ast}$, le produit des sous-groupes diagonaux comme le sous-groupe de Lévi minimal de $G$, resp. de $G^\ast$. Ils sont bien définis sur $F$. Quitte à conjuguer dans $G^*(F)$ nous allons supposer que $\eta(M_0)$ contient $M_{0^\ast}$, et que la restriction de $\eta$ sur $A_{M_0}$ et définie sur $F$. Pour tout $v\in\V$, fixons ensuite $M_{v,0}$ un sous-groupe de Lévi minimal de $G_v$, contenu dans $M_{0,v}\eqdef M_0\times_F F_v$, et défini sur $F_v$ tel que $\eta(M_{v,0})$ contient $M_{v,0^\ast}=M_{0^\ast,v}$ le produit des sous-groupes diagonaux qui est le sous-groupe de Lévi de $G_v^\ast$, et que la restriction de $\eta$ sur $A_{M_{v,0}}$ et définie sur $F_v$. On sait de même que $\eta(
\mathfrak{z})=\mathfrak{z}^\ast$, $\eta(\mathfrak{z}(F))=\mathfrak{z}^\ast(F)$ et $\eta(\mathfrak{z}(F_v))=\mathfrak{z}^\ast(F_v)$ car la restriction de $\eta$ à $\mathfrak{z}_{\overline{F}}$ est definie sur $F$.

\subsubsection{Formule de scindage}
Soit $S$ un sous-ensemble non-vide de $\V_F$. Le symbole $v$ désigne ici un élément général de $S$. On choisit un produit scalaire $W^G_{M_0}$-invariant (resp. $W^{G_v}_{M_{v,0}}$-invariant) sur $a_{M_0}$ (resp. $a_{M_v,0}$). On prend sur tout sous-espace de $a_{M_0}$ (resp. $a_{M_v,0}$) la mesure engrendré par ce produit scalaire. On aimerait préciser qu'a priori, la mesure sur $a_M$ n'a pas de lien avec celle sur $a_{M_v}$ pour $M\in \L^G$.

Soit $M\in \L^G$. Notons $M_S=\prod_{v\in S}M_v$ et le voit comme sous-groupe de Levi de $G_S=\prod_{v\in S}G_v$ défini sur $F_S$. On note de plus $\L^{G_S}(M_S)$ (resp. $\P^{G_S}(M_S)$ ; $\F^{G_S}(M_S)$) l'ensemble des produits $\prod_{v\in S}L_v$ (resp. $\prod_{v\in S}P_v$ ; $\prod_{v\in S}Q_v$) où $L_v$ est un sous-groupe de Levi défini sur $F_v$ contenant $M_v$ (resp. $P_v$ est un sous-groupe parabolique défini sur $F_v$ et de facteur de Levi $M_v$ ; $Q_v$ est un sous-groupe parabolique défini sur $F_v$ contenant $M_v$), par une récurrence on montre l'existence des applications, en moyennant certains choix (des détails sont dans \cite[section 9]{Art88I}) 
\begin{align*}
    d_M^G :\L^{G_S}(M_S)&\rightarrow [0,+\infty[ \\
    s:\L^{G_S}(M_S)&\rightarrow \F^{G_S}(M_S)
\end{align*}
de sorte que, pour tout $(L_v)_{v\in S}\in \L^{G_S}(M_S)$, on ait
\begin{enumerate}
    \item si $s\left((L_v)_{v\in S}\right)=\left((Q_v)_{v\in S}\right)$ alors $(Q_v)_{v\in S}\in \P^{G_S}(L_S)$ ;
    \item $d_M^G\left((L_v)_{v\in S}\right)\not =0$  si et seulement si l'une des flèches naturelles
    \[\bigoplus_{v\in S}a_{M_v}^{L_v}\longrightarrow a_M^G\]
    et
    \[\bigoplus_{v\in S}a_{L_v}^{G_v}\longrightarrow a_M^G\]
    est un isomorphisme, auquel cas les deux sont isomorphismes et $d_M^G\left((L_v)_{v\in S}\right)$ est le volume dans $a_M^G$ du parallélotope formé par les bases orthonormées des $a_{M_v}^{L_v}$ ;
    \item (Formule de scindage) supposons que nous disposons, pour tout $v\in S$, de $(c_{P_v})_{P_v\in\P^{G_v}(M_v)}$ une $(G_v,M_v)$-famille sur $F_v$. Définissons $c_P=\prod_{v\in S}c_{P_v}$. Alors $(c_P)_{P\in \P^G(M)}$ est une $(G,M)$-famille, et
    \begin{equation}\label{YDLgeomeq:formuledescindage}
    c_{M}=\sum_{(L_v)_{v\in S}\in \L^{G_S}(M_S)}d_M^G\left((L_v)_{v\in S}\right)\prod_{v\in S} c_{M_v}^{Q_v}.    
    \end{equation}
\end{enumerate}

On a des bijections naturelles entre $a_M$ et $a_{M^\ast}$, et pour tout $v\in S$ entre $a_{M_v}$ et $a_{M_v^\ast}$, de même entre $a_M^\ast$ et $a_{M^\ast}^\ast$ etc. Ces espaces vectoriels sont des espaces euclidiens et on supposera — quitte à modifier les produits scalaires — que les bijections sont isométriques. On a l'égalité suivante pour tout $(L_v)_{v\in S}\in \L^{G_S}(M_S)$ :
\[d_M^G((L_v)_{v\in S})=d_{M^\ast}^{G^\ast}((L_v^\ast)_{v\in S}).\]
Comme il est loisible, on choisit les sections $L_v\mapsto Q_{L_v}$ et $L_v'\mapsto Q_{L_v'}$ de sorte que
\[(Q_{L_v})_{v\in S}^\ast=(Q_{L_v^\ast})_{v\in S}.\]

\subsection{Recollement local-global}

\begin{proposition}\label{prop:local-globalprincipalLevi}Soient $M'$ un sous-groupe de Levi de $G^\ast$ et $P'$ un sous-groupe parabolique contenant $M'$. Alors $(M',P')$ se transfère à $G$ si et seulement si $(M_v',P_v')$ se transfère à $G_v$ pour tout $v\in \V$. 
\end{proposition}
\begin{proof}

On dispose d'un diagramme commutatif suivant (cf. \cite[section 2]{Kott86}) :
\[\begin{tikzcd}
	{H^1(F,M_{\AD}')} & { \bigoplus_{v\in\V_F}H^1(F_v,M_{\AD,v}')} & {\bigoplus_{v\in\V_F}\pi_0(Z(\widehat{M_\AD')}^{\Gamma_v})^D} & {\pi_0(Z(\widehat{M_\AD')}^{\Gamma})^D} \\
	\\
	{H^1(F,G_\ad^\ast)} & {\bigoplus_{v\in\V_F}H^1(F_v,G_{\ad,v}^\ast)} & {\bigoplus_{v\in\V_F}\pi_0(Z(\widehat{G_\ad^\ast)}^{\Gamma_v})^D} & {\pi_0(Z(\widehat{G_\ad^\ast)}^{\Gamma})^D}
	\arrow[from=1-1, to=3-1]
	\arrow[from=1-2, to=3-2]
	\arrow[from=1-3, to=3-3]
	\arrow[from=1-1, to=1-2]
	\arrow[from=3-1, to=3-2]
	\arrow[from=3-2, to=3-3]
	\arrow[from=1-2, to=1-3]
	\arrow["\text{somme}", from=1-3, to=1-4]
	\arrow["\text{somme}", from=3-3, to=3-4]
	\arrow[from=1-4, to=3-4]
\end{tikzcd}\]
Dans les lignes horizontales, l’image du premier terme dans le deuxième est égal au noyau de la composée des deux dernières flèches. La dernière flèche verticale est injective car
\[\pi_0(Z(\widehat{G_\ad^\ast)}^{\Gamma})\rightarrow\pi_0(Z(\widehat{M_\AD')}^{\Gamma})\]
est surjectif (\cite[lemme 1.1]{Art99}). L'injectivité souhaitée découle ainsi de la dualité.


Selon l'hypothèse pour tout $v\in\V_F$ il existe $\eta_{M'}^v\in H^1(F_v,M_{\AD,v}')$ qui s'envoie sur $\eta_v\in H^1(F_v,G_{\ad,v}^\ast)$. L'image de $(\eta_{M'}^v)_{v\in\V_F}$ dans $\pi_0(Z(\widehat{G_\ad)}^{\Gamma})^D$ est alors l'élément trivial. Après, l'injectivité de la dernière flèche verticale entraîne que  l'image de $(\eta_{M'}^v)_{v\in\V_F}$ dans $\pi_0(Z(\widehat{M_\AD')}^{\Gamma})^D$ est déjà triviale, on en déduit que $(\eta_{M'}^v)_{v\in\V_F}$ est l'image d'un élément $\eta_{M'}\in H^1(F,M_{\AD}')$. On conclut donc la preuve en évoquant le principe de Hasse :
\[H^1(F,G_\ad^\ast)\longrightarrow \bigoplus_{v\in\V_F}H^1(F_v,G_{\ad,v}^\ast)\]
est injectif comme $G_\ad^\ast$ est semi-simple du type adjoint.
\end{proof}

On prouve dans la même veine le principe local-global pour le transfert des classes de conjugaison. Une autre preuve est fournie dans l'appendice A (proposition \ref{prop:appendixprincipeloc-glotrans}).

\begin{proposition}\label{prop:local-globalprincipalelement} Un élément $X'\in \g^\ast(F)$ se transfère si et seulement si $X_v'\in \g^\ast(F_v)$ est se transfère pour tout $v\in \V_F$. 
\end{proposition}
\begin{proof}
\'{E}crivons $\overline{\o'}=(\Ad G^\ast(\overline{F}))X'$. D'abord $X'\in \g^\ast(F)$ se transfère si et seulement si $\sigma(\eta^{-1}(\overline{\o'}))=\eta^{-1}(\overline{\o'})$ pour tout $\sigma\in\Gamma$, des calculs simples montrent que cela équivaut à \[\eta\in \text{Im}(H^1(F,G_{X',\AD}^\ast)\rightarrow H^1(F,G_{\ad}^\ast)).\]
Soit $G_{X'}^\ast=R_u(G_{X'}^\ast)\rtimes L_{G_{X'}^\ast}$ une décomposition de Levi de $G_{X'}^\ast$ avec $R_u(G_{X'}^\ast)$ son radical unipotent et $L_{G_{X'}^\ast}$ un facteur de Levi. On peut donc raisonner comme dans la preuve précédente en remplaçant $M'$ par $L_{G_{X'}^\ast}$. Le seul point restant est de prouver la surjectivité de 
\[\pi_0(Z(\widehat{G_\ad^\ast)}^{\Gamma})\rightarrow\pi_0(Z(\widehat{L_{G_{X'}^\ast,\AD})}^{\Gamma}).\]
Pour cela, il faut examiner de près le centralisateur $G_{X'}^\ast$. On peut supposer, par le lemme de Shapiro, que $G^\ast=\GL_{n,F}$. Supposons que $X'$ correspond à  $\lambda':\Irr_F\rightarrow\{\text{partitions des entiers}\}$ par la théorie des diviseurs élémentaires (théorème \ref{thm:Yu9}), on écrit $\lambda'(p)$ comme
\[\lambda'(p)=(m_{p,1}\dots,m_{p,1},\dots,m_{p,s_p},\dots,m_{p,s_p})\]
avec $s_p\in\mathbb{N}_{>0}$, $m_{p,1}>\cdots>m_{p,s_p}\geq 1$ et chaque $m_{p,i}$ est répété $n_{p,i}\in\mathbb{N}_{>0}$ fois dans la partition. Alors des calculs classiques (sous-section \ref{subsec:reductiondeJordanetcentralisateur}) 
nous conduisent à
\[L_{G_{X'}^\ast}=\prod_{p : \lambda(p)\not =\emptyset}\prod_{i=1}^{s_p}\Res_{F_p/F}\GL_{n_{p,i},F_p},\,\,\,\,F_p\eqdef F[T]/(p).\]
Par conséquent 
\begin{align*}
Z(\widehat{L_{G_{X'}^\ast,\AD})}^{\Gamma}&=Z\left(\left(\prod_{p : \lambda(p)\not =\emptyset}\prod_{i=1}^{s_p}\prod_{[F_p:F]}\GL_{n_{p,i}}(\C)\right)^1\right)^\Gamma \\
&= \left(\prod_{p : \lambda(p)\not =\emptyset}\prod_{i=1}^{s_p}\GL_{1}(\C)\right)^1\\
&=\{(z_{p,i})\mid z_{p,i}\in \C^\times, \prod_{p,i}z_{p,i}^{(\deg p)m_{p,i}n_{p,i}}=1\},
\end{align*}
ici $(-)^1$ désigne l'intersection du groupe entre parenthèses avec $\widehat{G_{\ad}^\ast}=\text{SL}_n(\C)$, tous les deux vus comme sous-groupes de $\widehat{G^\ast}=\GL_n(\C)$. En l'occurrence, la flèche en question 
\[\pi_0(Z(\widehat{G_\ad^\ast)}^{\Gamma})\rightarrow\pi_0(Z(\widehat{L_{G_{X'}^\ast,\AD})}^{\Gamma})\]
devient 
\[\pi_0\left(\{z\mid z\in \C^\times, \prod_{p,i}z^{(\deg p)m_{p,i}n_{p,i}}=1\}\right)\rightarrow \pi_0\left(\{(z_{p,i})\mid z_{p,i}\in \C^\times, \prod_{p,i}z_{p,i}^{(\deg p)m_{p,i}n_{p,i}}=1\}\right).\]
L'application est induite par l'immersion diagonale $z\mapsto (z)_{p,i}$. Cela est alors l'application de réduction
\begin{align*}
\Z/\left(\sum_{p,i}(\deg p)m_{p,i}n_{p,i}\right)\Z &\to \Z/\text{pgcd}_{p,i}\left((\deg p)m_{p,i}n_{p,i}\right)\Z 
\end{align*}
qui est surjective. 
La démonstration se conclut.
\end{proof}

Donnons à présent une formulation plus élémentaire de la proposition \ref{prop:local-globalprincipalLevi} : on sait qu'il existe $I$ un ensemble fini, et pour tout $i\in I$ il existe $E_i/F$ une extension finie, $m_i$ un entier strictement positif, $D_i$ une algèbre à division centrale sur $E_i$ de dimension $d_i^2$, tels que 
\[G=\prod_{i\in I}\Res_{E_i/F}\GL_{m_i,D_i}.\]
Nous avons $G^\ast=\prod_{i\in I}\Res_{E_i/F}\GL_{m_id_i,E_i}$. Puis pour tout $i\in I, v\in\V$ et $w_i$ une place de $E_i$ au-dessus de $v$, il existe $D_{i,w_i}$ une algèbre à division centrale définie sur $E_{w_i}$ de dimension $d_{i,w_i}^2$, telle que 
\[G_v=\prod_{i\in I}\prod_{w_i|v}\Res_{E_{w_i}/F_v}\GL_{m_id_i/d_{i,w_i},D_{i,w_i}}.\]
Il est connu, selon la théorie des corps de classes globaux, que $d_i$ est le plus petit commun multiple des $d_{i,w_i}$. Nous avons 
\begin{align*}
\g^\ast=\prod_{i\in I}\Res_{E_i/F}\text{Mat}_{m_id_i,E_i}&,\,\,\,\,\,\,\,\, \g_v^\ast=\prod_{i\in I}\prod_{w_i|v}\Res_{E_{w_i}/F_v}\text{Mat}_{m_id_i,E_{w_i}}   \\
\g=\prod_{i\in I}\Res_{E_i/F}\text{Mat}_{m_i,D_i}&,\,\,\,\,\,\,\,\, \g_v=\prod_{i
\in I}\prod_{w_i|v}\Res_{E_{w_i}/F_v}\text{Mat}_{m_id_i/d_{i,w_i},D_{i,w_i}}.
\end{align*} 
Il est aussi utile de rappeler que $d_{i,w_i}$ (resp. $d_i$) est l'ordre de $D_{i,w_i}$ (resp. $D_i$) dans le groupe de Brauer de $E_{w_i}$ (resp. $E_i$).

On sait que tout sous-groupe de Levi standard d'une restriction des scalaires d'un groupe général linéaire correspond à une partition du rang du groupe général linéaire.
\begin{proposition}~{}
\begin{enumerate}
    \item Un sous-groupe de Levi standard $M'$ de $G^\ast$ se transfère à $G$ si et seulement si $d_i$ divise tout entier dans sa partition correspondante pour tout $i\in I$.
    \item Soit $v\in\V_F$. Un sous-groupe de Levi standard $M_v'$ de $G_v^\ast$ se transfère à $G_v$ si et seulement si $d_{i,w_i}$ divise tout entier dans sa partition correspondante pour tout $i\in I$.
\end{enumerate}
\end{proposition}

\begin{lemma}\label{lem:globalscindagetranfert}
Soit $S$ un sous-ensemble de $\V_F$ tel que $\eta_v: G_v\to G_v^\ast$ est un $F_v$-isomorphisme pour tout $v\not\in S$. Soient $M'\in \L^{G^\ast}$ et $(L_v^\ast)_{v\in S}\in (\L^{G_v^\ast}(M_v'))_{v\in S}$ tels que la flèche naturelle
\[\bigoplus_{v\in S}a_{M_v'}^{L_v^\ast}\rightarrow a_{M'}^{G^\ast}\]
soit un isomorphisme. Alors $M'$ se transfère à $G$.    
\end{lemma}
\begin{proof}
Le lemme a été prouvé par Arthur (cf. \cite[lemme 14.1]{Art81} et la preuve de \cite[proposition 8.1]{Art88II}). Nous exposons dans ces lignes une alternative à sa méthode. Quitte à se restreindre sur un facteur de $G$ et changer le corps de base, on peut et on va supposer que $G^\ast=\GL_{n,F}$ (implicitement le lemme de Shapiro et le fait que $a_{(\Res_{E/F}H)_v}=\oplus_{w\mid v} a_{\Res_{E_w/F_v}H_w}$ pour tout groupe réductif $H$ sur $E$ sont utilisés pour justifier la simplification). Remarquons que dans ce cas pour tout $v\in\V_F$, l'injection $\L^{G^\ast}\rightarrow\L^{G_v^\ast}$ est en fait bijective, et que pour tout $L'\in\L^{G^\ast}(M_{0^\ast})$, l'inclusion $a_{L'}\subseteq a_{L_v'}$ est une égalité. On est donc conduit à prouver l'assertion suivante : soit $(L^{v'})_{v\in S}\in(\L^{G^\ast}(M'))_{v\in S}$ avec $L^{v'}_{v}$ se transfère à $G_v$, tel que la flèche naturelle
\[\bigoplus_{v\in S}a_{M'}^{L^{v'}}\rightarrow a_{M'}^{G^\ast}\]
soit un isomorphisme. Alors $M'$ se transfère à $G$. Ceci est une conséquence des arguments combinatoires subséquents. Dans le dessein de raccourcir l'écriture, on note $k\mid H$ lorsque $k$ est un entier strictement positif et $H$ est un sous-groupe de Levi de $G^\ast$ avec $H\simeq\prod_i \GL_{n_i,F}$ et $k\mid n_i$ pour tout $i$.

\begin{lemma}Soient $H'\in\L^{G^\ast}$ et $H_1',H_2'\in\L^{G^\ast}(H')$, tels que $a_{H'}^{H_1'}\cap a_{H'}^{
H_2'}=\{0\}$. Si $k\mid H_1'$ et $k\mid H_2'$ alors $k\mid H'$.
\end{lemma}
\begin{proof}
Pour un vecteur $v=(x_1,\dots,x_l)\in \R^l$ on pose $\overline{v}\in\Z^l$ le vecteur qui vaut $0$ en $i$-ème coordonnée si $x_i=0$, et vaut $1$ en $i$-ème coordonnée si $x_i\not =0$. Pour un ensemble $E=\{v_1,\dots, v_r\}$ de vecteurs dans $\R^l$ on pose $\overline{E}=\{\overline{v_1},\dots,\overline{v_l}\}$.
 
Quitte à conjuguer tous les sous-groupes de Levi en question on peut et on va supposer que $H'$ est standard. Par la géométrie de données radicielles on sait que 
$a_{H'}^{H_1'}\cap a_{H'}^{H_2'}=\{0\}$ équivaut à ce que le rang de la matrice
\[
\left[ 
\begin{array}{@{}ccc@{}}
 \\&{\overline{(\Delta_{
H'}^{H_1'})^\vee}} & \\\\
\hline\\
 &{\overline{(\Delta_{H'}^{H_2'})^\vee}} & \\\\
\end{array}\right]\in\Mat_{(\dim a_{H_1'}+\dim a_{H_2'})\times\dim a_{H'} }(\Z).
\]  
soit $\dim a_{H'}$, ici on considère les éléments de $\overline{(\Delta_{H'}^{H_i'})^\vee}$ comme des équations (vecteurs lignes) définissant le sous-espace $a_{H'}^{H_i'}$ de $a_{H'}$. Cette matrice est à coefficient dans $\{0,1\}$, dont chaque ligne possède au mois 1 fois 1 et chaque colonne possède exactement 2 fois 1. La somme des colonnes de $\overline{(\Delta_{H'}^{H_1'})^\vee}$ égale à celle des colonnes de $\overline{(\Delta_{H'}^{H_1'})^\vee}$, à savoir $(1,\dots,1)$. Le rang de la matrice est donc au plus $\dim a_{H_1'}+\dim a_{H_2'}-1$. En appliquant la technique de double sommation sur la somme des coefficients de la matrice en question on voit qu'il existe une ligne qui possède exactement 1 fois 1, puis en enlevant la ligne et la colonne sur lesquelles se trouve ce coefficient 1 on retombe sur une matrice ayant la même nature : cette nouvelle matrice est de rang $\dim a_{H'}-1$, et à coefficient dans $\{0,1\}$, dont chaque ligne possède au mois 1 fois 1 et chaque colonne possède exactement 2 fois 1. Par une récurrence simple on montre alors que la matrice initiale admet une matrice mineure de taille $\dim a_{H'}\times \dim a_{H'}$, disons $X$, de détermiant $\pm 1$.

Selon l'hypothèse que $k\mid H_1'$ et $k\mid H_2'$ on constate que
\[
\left[X\right]\left[ 
\begin{array}{c}
 h_1' \\
\vdots\\
 h_{\dim a_{H'}}'
\end{array}\right]
\in\Mat_{\dim a_{H'}\times 1}(k\Z),
\]
où $h_i'$ est la taille de chaque bloc de $H'$, il vient donc 
\[\left[ 
\begin{array}{c}
 h_1' \\
\vdots\\
 h_{\dim a_{H'}}'
\end{array}\right]
\in\Mat_{\dim a_{H'}\times 1}(k\Z),\]
ce qu'il fallait.
\end{proof}

Revenons sur la démonstration du lemme de départ. Fixe $v\in\V_F$. On veut montrer que $d_v\mid M'$ avec $d_v$ l'invariant local de $\g_v$. Si $v\not \in S$, la relation de division est triviale par l'hypothèse dans l'énoncé sur $S$. Soit maintenant $v\not\in S$. Soit $p$ un nombre premier et $k$ un entier positif tel que $p^k\mid d_v$ et $p^{k+1}\nmid d_v$. Selon la théorie des corps de classes globaux la somme des invariants locaux de $\g$ vaut $0\in\Q/\Z$, il existe ainsi $v\not=w\in\V_F$ tel que $p^k\mid d_w$. Il vient, d'après le lemme précédent puisque $a_{M'}^{L^{v'}}\cap a_{M'}^{L^{w'}}=\{0\}$, que $p^k\mid M'$. On en déduit encore une fois $d_v\mid M'$. Ceci étant valable pour tout $v\in\V_F$, on en conclut que $M'$ se transfère à $G$.
\end{proof}

\section{Analyse globale}\label{sec:analyseglobale}

Nous arrivons à la comparaison des formules des traces. 

\subsection{Hypothèse}
Face à l'absence de résultat concernant l'existence du transfert non-invariant d'une fonction test à une place archimédienne, nous faisons une hypothèse technique afin d'éviter d'anticiper des résultats d'analyse harmonique archimédienne : jusqu'à la fin de l'article, $G_v$ est toujours supposé quasi-déployé pour $v\in\V_\infty$. Autrement dit $\eta_\infty:G_\infty\to G_\infty^\ast$ est un torseur trivial.

\subsection{Choix des objets relativement aux \texorpdfstring{$G$}{G} et \texorpdfstring{$G^\ast$}{G*}}

Fixons $\underline{S}_0$ un sous-ensemble fini de places, contenant les places archimédiennes, tel que pour tout $v\in (\V_F\setminus \underline{S}_0)\cup\V_{\infty}$ on ait que $\eta_v:G_v\to G_v^\ast$ est le torseur trivial. Remarquons que cette hypothèse implique que pour tout $\g_\ss(F_v)\ni X\arr X^\ast\in \g_\ss^\ast(F_v)$, $\eta_{v,X}=\eta_v|_{G_{v,X}}:G_{v,X}\to G_{v,X^\ast}^\ast$ est le torseur trivial.

On fixe, pour tout $v\in (\V_F\setminus \underline{S}_0)\cup \V_{\infty}$, un élément $x_v\in G^\ast(\overline{F_v})$, tel que $(\Ad x_v)\circ\eta_v|_{G_v}$ soit un isomorphisme défini sur $F_v$ de $G_v$ sur $G_v^\ast$. On exige que les objets relativement aux $G$ et $G^\ast$ dans le numéro \ref{subsubsec:objetpourGcontexteglobal} sont choisis de sorte que : si $v\in \V_F\setminus \underline{S}_0$ alors $(\Ad x_v)\circ\eta_v|_{\mathfrak{k}_v}$ est un isomorphisme de $\mathfrak{k}_v$ sur $\mathfrak{k}_v^\ast$, et si $v\in (\V_F\setminus \underline{S}_0)\cup \V_\infty$ alors $(\Ad x_v)\circ\eta_v|_{K_v}$ est un isomorphisme de $K_v$ sur  $K_v^\ast$.

\subsection{Normalisations des mesures de la section \ref{sec:analyseglobale}}\label{subsec:normalisationmesuresfinales}
On explique à présent les normalisations des mesures dans cette section.

Pour les objets en lien avec $G$, nous respectons les normalisations de la sous-section \ref{subsec:normalisationsdesmesuresdeFTI}, en utilisant le même ensemble $\underline{S}_0$ que là-bas. Quant aux normalisations des mesures sur les objets en lien avec $G^\ast$, on suit le même procédé en utilisant le même caractère $\psi$ et la forme bilinéaire canonique de $\g^\ast$. 

On demande en dernier lieu les compatibilités suivantes des normalisations des mesures à l'égard du transfert : 
\begin{enumerate}[label=(\roman*)]
    \item pour tous $M\in\L^G(M_0)$, et $\m_\ss(F)\ni X\arr X^\ast\in \m_\ss^\ast(F)$, le torseur intérieur $\eta$ induit un isomorphisme d'espaces  vectoriels réels $a_{M_X}\xrightarrow{\sim} a_{M_{X^\ast}^\ast}$, on veut que les mesures sur $a_{M_X}$ et $a_{M_{X^\ast}^\ast}$ se correspondent par cet isomorphisme ; 
    \item pour tous $M\in\L^{G_v}(M_{0,v})$, et $\m_\ss(F_v)\ni X\arr X^\ast\in \m_\ss^\ast(F_v)$, le torseur intérieur $\eta_v$ induit un isomorphisme d'espaces  vectoriels réels $a_{M_X}\xrightarrow{\sim} a_{M_{X^\ast}^\ast}$, on veut que les mesures sur $a_{M_X}$ et $a_{M_{X^\ast}^\ast}$ se correspondent par cet isomorphisme ;
    \item soient $v\in \V_\infty$, $\g_\ss(F_v)\ni X\arr X^\ast\in \g_\ss^\ast(F_v)$, et $y\in G^\ast(\overline{F_v})$, tels que $(\Ad y)\circ\eta|_{G_{v,X}}$ soit un isomorphisme défini sur $F_v$ de $G_{v,X}$ sur $G_{v,X^\ast}^\ast$, alors les mesures sur $G_{v,X}(F_v)$ et $G_{v,X^\ast}^\ast(F_v)$ se correspondent par cet isomorphisme.
\end{enumerate}

En outre, pour tout $M\in \L^G(M_0)$, et $X\in \m_\ss(F)$ le groupe $A_{M_X,\Q}$ a été défini comme le sous-tore central $\Q$-déployé maximal dans $\Res_{F/\Q}M_X$, mais on voit qu'il est aussi le sous-tore central $\Q$-déployé maximal dans $\Res_{F/\Q}A_{M_X}$ par un simple argument de descente galoisienne. De ce fait  $\eta$ induit un isomorphisme de groupes topologiques $A_{M_X,\infty}\xrightarrow{\sim} A_{M_{X^\ast}^\ast,\infty}$. Alors les mesures sur $A_{M_X,\infty}$ et $A_{M_{X^\ast}^\ast,\infty}$ se correspondent par cet isomorphisme car on a un diagramme commutatif de mesures
    \[\begin{tikzcd}
	{A_{M_X,\infty}} && {A_{M_{X^\ast}^\ast,\infty}} \\
	\\
	{a_{M_X}} && {a_{M_{X^\ast}^\ast}}
	\arrow["\sim", from=3-1, to=3-3]
	\arrow["\sim", from=1-1, to=3-1]
	\arrow["\sim", from=1-3, to=3-3]
	\arrow["\sim", from=1-1, to=1-3]
	\arrow["\eta"', from=1-1, to=1-3]
	\arrow["\eta"', from=3-1, to=3-3]
	\arrow["{H_{M_{X^\ast}^\ast}}"', from=1-3, to=3-3]
	\arrow["{H_{M_X}}"', from=1-1, to=3-1]
    \end{tikzcd}\]

\subsection{Transfert géométrique d'une fonction}

On définit ensuite la notion d'un transfert semi-local ou global (non-invariant) d'une fonction. Fixons dès lors $\underline{S}$ un sous-ensemble fini de places, contenant $\underline{S}_0$, de cardinal au moins $|\V_{\infty}|+1$, et de sorte que l'on ait $\widetilde{1_{\mathfrak{k}_v}}=1_{\mathfrak{k}_v}$ pour tout $v\in \V_F\setminus \underline{S}$. Pour rappel $\widetilde{-}$ est la transformée partielle de Fourier, i.e. équation \eqref{YDLgeomeq:defpartialFourier}.

Constatons que, selon nos normalisations de mesures, que $\S(\g(F_v))\ni h_v\arr h_v\circ((\Ad x_v)\circ\eta_v|_{G_v})^{-1}\in \S(\g^\ast(F_v))$ pour toute place $v\in (\V_F\setminus \overline{S})\cup \V_{\infty}$ et toute fonction $h_v$ de classe Schwartz-Bruhat.

Pour tout sous-ensemble $S$ fini de $\V_F$, un élément $f\in \S(\g(F_S))$ est dit un tenseur pur s'il admet une écriture $f=\bigotimes_{v\in S} f_v$ avec $f_v\in \S(\g(F_v))$. De même, un élément $f\in \S(\g(\A_F))$ est dit un tenseur pur s'il admet une écriture $f=\bigotimes_{v\in \V_F} f_v$ avec $f_v\in \S(\g(F_v))$ et $f_v=1_{\mathfrak{k}_v}$ pour presque tous les $v$. 

\begin{definition}[Transfert géométrique d'une fonction semi-locale ou globale]\label{YDLgeomdef:transfertd'unefonctionglobale}~{}
\begin{enumerate}
    \item Soit $S\supseteq \underline{S}$ un sous-ensemble fini de $\V_F$. Soit $f\in\S(\g(F_S))$ un tenseur pur. Nous dirons qu'un tenseur pur $f^\ast\in\S(\g^\ast(F_S))$ se transfère de façon géométrique en $f$ et nous noterons $f\underset{\geom}{\arr}f^\ast$, ou plus simplement $f\arr f^\ast$, s'il existe des décompositions  $f=\bigotimes_{v\in S} f_v$ et $f^\ast=\bigotimes_{v\in S} f_v^\ast$ telles que les conditions suivantes sont remplies.
    \begin{enumerate}
    \item Pour tout $v\in \underline{S}\setminus\V_{\infty}$, on a $f_v\arr f_v^\ast$.
    \item Pour tout $v\in\V_{\infty}$, on a $f_v\circ ((\Ad x_v)\circ\eta_v)^{-1}=f_v^\ast$.
    \item Pour tout $v\in S\setminus \underline{S}$, on  a $f_v=1_{\mathfrak{k}_v}$ et $f_v^\ast=1_{\mathfrak{k}_v^\ast}$.
    \end{enumerate}
    
    \item Soit $f\in \S(\g(\A_F))$ un tenseur pur. Nous dirons qu'un tenseur pur $f^\ast\in\S(\g^\ast(\A_F))$ se transfère de façon géométrique en $f$ et nous noterons $f\underset{\geom}{\arr}f^\ast$, ou plus simplement $f\arr f^\ast$, s'il existe des décompositions  $f=\bigotimes_{v\in \V_F} f_v$ et $f^\ast=\bigotimes_{v\in \V_F} f_v^\ast$ telles que les conditions suivantes sont remplies.
    \begin{enumerate}
    \item Pour tout $v\in \underline{S}\setminus\V_{\infty}$, on a $f_v\arr f_v^\ast$.
    \item Pour tout $v\in\V_{\infty}$, on a $f_v\circ ((\Ad x_v)\circ\eta_v)^{-1}=f_v^\ast$.
    \item Pour tout $v\in \V_F\setminus \underline{S}$, on  a $f_v=1_{\mathfrak{k}_v}$ et $f_v^\ast=1_{\mathfrak{k}_v^\ast}$.
    \end{enumerate}
    \end{enumerate}
\end{definition}

On pourrait prendre une « limite inductive » pour se débarrasser de la dépendance de la définition du transfert global en $\underline{S}$, mais nous n'adoptons pas ce point de vue ici. 

Soit $f\in\S(\g(\A_F))$. Rappelons que $\eta|_{\mathfrak{z}}:\mathfrak{z}\to\mathfrak{z}^\ast$ est un torseur trivial. On note $f_{X}$ la fonction $f_{X}(Y)=f(X+Y)$ pour $X\in\g(\A_F)$. On note $f^{t}$ la fonction $f^{t}(Y)=f(tY)$ pour $t\in \A_F^\times$. Les mêmes notations valent pour $G^\ast$.

\begin{lemma}Soient $\S(\g(\A_F))\ni f\arr f^\ast\in \S(\g^\ast(\A_F))$ deux tenseurs purs. 
\begin{enumerate}
    \item Alors $f_{Z}\arr f_{Z^\ast}^\ast$ pour tout $\mathfrak{z}(\A_F)\ni Z\arr Z^\ast \in\mathfrak{z}^\ast(\A_F)$ avec $Z_v\in \mathfrak{k}_v^\ast$ pour tout $v\in \V_F\setminus \underline{S}$.
    \item Alors $f^{t}\arr f^{\ast,t}$ pour tout $t\in \A_F^\times$ avec $t_v \mathfrak{k}_v^\ast=\mathfrak{k}_v^\ast$ pour tout $v\in \V_F\setminus \underline{S}$.
\end{enumerate}
\end{lemma}
\begin{proof}
La preuve est triviale. 
\end{proof}

\begin{proposition}[Transfert commute à la transformée partielle de Fourier]\label{prop:Fouierpcommutetrans}
Pour rappel $\widetilde{-}$ est la transformée partielle de Fourier, cf. équation \eqref{YDLgeomeq:defpartialFourier}. Soient $\S(\g(\A_F))\ni f\arr f^\ast\in \S(\g^\ast(\A_F))$ deux tenseurs purs, alors $\widetilde{f}\arr \widetilde{f^\ast}$.
\end{proposition}
\begin{proof}
Cela ressort des hypothèses sur $\underline{S
}$, de la proposition \ref{prop:transFourier}, et de la formule du produit pour la constante de Weil, i.e. l'équation \eqref{eq:formuleproduitWeil}.
\end{proof}

\subsection{Comparaisons}

On peut à présent comparer les formules des traces. Commençons par la comparaison des intégrales orbitales pondérées semi-locales.

\begin{proposition}\label{prop:semi-localIOPcomp} Soit $S\supseteq \underline{S}$ un sous-ensemble fini de $\V_F$. Soient $\S(\g(F_S))\ni f\arr f^\ast\in \S(\g^\ast(F_S))$ deux tenseurs purs.  
\begin{enumerate}
    \item On a 
    \[J_{M^\ast}^{G^\ast}(X^\ast,f^\ast)=J_M^G(X,f)\]
    si $M\arr M^\ast$ et $\m(F)\ni X\arr X^\ast\in \m^\ast(F)$ via la restriction $\eta|_\m :\m\rightarrow \m^\ast$.
    \item Soient $M'\in \L^{G^\ast}$ et $X'\in\m'(F)$. On a 
    \[J_{M'}^{G^\ast}(X',f^\ast)=0\]
    sauf si $M'=M^\ast$ et $X'$ se transfèrent à $M$.
\end{enumerate}
\end{proposition}

\begin{proof}~{}
\begin{enumerate}
    \item L'énoncé est semi-local. On va d'abord prouver qu'il y a $J_{M^\ast}^{G^\ast}(X^\ast,f^\ast)=J_M^G(X,f)$ si l'on respecte les normalisations du numéro \ref{subsec:localnormalisationsdemesures} partout en $v\in S$. La formule de scindage pour une $(G,M)$-famille, autrement dit l'équation \eqref{YDLgeomeq:formuledescindage}, et l'égalité $d_{M^\ast}^{G^\ast}((L_v^\ast)_{v\in S})=d_{M}^{G}((L_v)_{v\in S})$ selon nos normalisations de mesures nous assurent que, afin de prouver l'égalité des intégrales orbitales pondérées semi-locales, il suffit de prouver les égalités correspondantes des intégrales orbitales pondérées locales. En une place archimédienne les égalités voulues sont triviales d'après l'hypothèse que $\eta_\infty: G_\infty\to G_\infty^\ast$ est un torseur trivial et nos choix de transfert de fonctions (définition \ref{YDLgeomdef:transfertd'unefonctionglobale}) et nos normalisations des mesures. Puis en une place non-archimédienne on évoque la proposition \ref{pro:corrloc} et la formule du produit pour la constante de Weil, c'est-à-dire l'équation \eqref{eq:formuleproduitWeil}. On remarque que $e^{G_v}(X_v)=1$ pour $v\not\in S$ comme $G_v$ est quasi-déployé. 
    
    Enfin on reprend les normalisations des mesures de cette section. Il est facile de vérifier que le changement des mesures n'entraîne pas l'apparition de facteurs supplémentaires, car seules les mesures en $v\in \underline{S}_0\subseteq S$ changent, tout en maintenant la même variance pour les objets en question par rapport à $G$ que ceux par rapport  à $G^\ast$.  
    
    \item Supposons que $M'$ ne se transfère pas à $G$. La formule de scindage, i.e. l'équation \eqref{YDLgeomeq:formuledescindage}, nous dit
    \[J_{M'}^{G^\ast}(X',f^\ast)=\sum_{(L_v')_{v\in S}\in \L^{G_S^\ast}(M_S')}d_{M'}^G\left((L_v')_{v\in S}\right)\prod_{v\in S}J_{M_v'}^{Q_v'}(X_v',f_v^\ast).\]
    Le théorème \ref{pro:corrloc} nous assure l'annulation $\prod_{v\in S}J_{M_v'}^{Q_v'}(X_v',f_v^\ast)=0$ lorsque $Q_v'$ ne se transfère pas à $G_v$ pour au moins un $v\in S$ , puis le lemme \ref{lem:globalscindagetranfert} entraîne l'annulation $d_{M'}^G\left((L_v')_{v\in S}\right)=0$ lorsque $Q_v'$ se transfère à $G_v$ pour tout $v\in S$. On a ainsi $J_{M'}^{G^\ast}(X',f^\ast)=0$ quand $M'$ ne se transfère pas à $G$. Puis l'annulation $J_{M^\ast}^{G^\ast}(X',f^\ast)$ lorsque $X'$ ne se transfère pas à $G$ vient de la formule de scindage, la proposition \ref{pro:corrloc}, et le principe local-global pour le transfert, autrement dit la proposition \ref{prop:local-globalprincipalelement}. \qedhere
\end{enumerate}
\end{proof}

On rappelle que, pour $X\in \g(F)$ et $S$ un sous-ensemble fini de $\V_F$ assez grand (dépendant de $X_\ss$),
\[a^G(S,X)\eqdef \epsilon^G(X_\ss)a^{G_{X_\ss}}(S,X_\nilp),\]
avec $\epsilon^G(X_\ss)$ vaut 1 si $X_\ss$ est $F$-elliptique dans $\g$ et sa classe de $G(\A^S)$-conjugaison rencontre $\mathfrak{k}^S$, et 0 sinon. Le nombre complexe $a^G(S,X)$ ne dépend que de la classe de $G(F)$-conjugaison de $X$.

\begin{theorem}\label{thm:identificcationcoeffaG(S,X)transfert}Soient $G\arr G^\ast$ et $\g(F)\ni X\arr X^\ast\in \g^\ast(F)$. Alors pour $S$ assez grand on a
\begin{equation}\label{eq:identificcationcoeffaG(S,X)transfert}
a^G(S,X)=a^{G^\ast}(S,X^\ast).    
\end{equation}
\end{theorem}
\begin{proof}
La démonstration se fait par récurrence sur $\dim G$. Supposons que $G$ est commutatif (cela comprend le cas où $\dim G=1$) alors $\eta:G\to G^\ast$ est un torseur trivial. Soit $\g(F)\ni X\arr X^\ast\in \g^\ast(F)$. Bien sûr $X$ (resp. $X^\ast$) est elliptique dans $\g$ (resp. $\g^\ast$). Or, pour toute place $v$, la mesure sur $G(F_v)$ coïncide avec la mesure sur $G^\ast(F_v)$. Parce que la mesure donnée par l'équation \eqref{YDLgeomeq:Haarmeasureongrp} est inchangée par un torseur trivial. Ainsi, la mesure sur $G(\A_F)$ coïncide avec la mesure sur $G^\ast(\A_F)$.  Mais on demande également que la mesure sur $A_{G,\infty}$ correspond à la mesure sur $A_{G^\ast,\infty}$, et $G(F)$ (resp. $G^\ast(F)$) est muni de la mesure de comptage. L'égalité \eqref{eq:identificcationcoeffaG(S,X)transfert} découle maintenant de la définition du coefficient $a^G(S,X)$ et le corollaire \ref{coro:valeuraG(S,0)}.


Soit $n>1$ un entier, supposons que l'assertion est établie pour tous $G$ et $X\arr X^\ast$ tel que $\dim G< n$. 

Supposons maintenant que $\dim G=n$, et $G$ n'est pas commutatif. Soit $\g(F)\ni X\arr X^\ast\in \g^\ast(F)$. On voit, par la proposition \ref{prop:GXtorseurint} et le lemme \ref{lem:AMisFiso}, que $X_\ss$ est elliptique dans $\g$ si et seulement si $X_\ss^\ast$ est elliptique dans $\g^\ast$. Par définition 
\[a^G(S,X)= \epsilon^G(S,X_\ss)a^{G_{X_\ss}}(S,X_\nilp),\]
la preuve se conclut donc à la suite de la comparaison flagrante des coefficients $\epsilon$ et l'hypothèse de récurrence si $X_\ss\not\in \mathfrak{z}(F)$. On tient à signaler que les mesures sur les groupes en lien avec $G_X$ et $G_{X^\ast}^\ast$ vérifient les compatibilités nécessaires pour la comparaison. On traite maintenant le cas où $X_\ss\in \mathfrak{z}(F)$. 

Supposons $X_\ss=0$. Soient $\S(\g(\A_F))\ni f\arr f^\ast\in \S(\g^\ast(\A_F))$ deux tenseurs purs. Bien sûr pour tout $S\supseteq \underline{S}$ sous-ensemble fini de $\V_F$, la fonction $f$ (resp. $f^\ast$) est dans l'image de l'injection $\S(\g(F_S))\hookrightarrow \S(\g(\A_F))$ (resp. $\S(\g^\ast(F_S))\hookrightarrow \S(\g^\ast(\A_F))$). Soit maintenant $\o'\in \O^{\g^\ast}$ qui n'est pas une classe de conjugaison nilpotente. D'une part, si $\o'$ ne se transfère pas, alors le développement fin, i.e. le théorème \ref{thm:devfin}, et le point 2 de la proposition \ref{prop:semi-localIOPcomp} nous garantissent l'annulation $J_{\o'}(f^\ast)=0$. D'autre part, si $\o'=\o^\ast$ se transfère, alors le développement fin, le point 1 de la proposition \ref{prop:semi-localIOPcomp} et l'hypothèse de récurrence sur l'égalité \eqref{eq:identificcationcoeffaG(S,X)transfert} nous garantissent que $J_{\o^\ast}(f^\ast)=J_{\o}(f)$.

On rappelle que $\pi_{\mathfrak{z}}:\g\to \mathfrak{z}$ est la projection (orthogonale) suivant $\g=\mathfrak{z}\oplus \g_{\ad}$. Fixons dans la suite $v_0\in \underline{S}\setminus\V_{\infty}$ une place. Soit $D_f$ la fonction sur $F_{v_0}^\times$ définie par
\begin{equation}\label{eq:deffonctionD(t)}
D_f(t)=\sum_{\o\in\O^{\g},\pi_{\mathfrak{z}}(\o)=0}J_{\o}(f^{t})-\sum_{\o'\in\O^{\g^\ast},\pi_{\mathfrak{z}^\ast}(\o)=0}J_{\o'}(f^{\ast,t}).    
\end{equation}
On déduit, de la discussion du paragraphe précédent, que
\[D_f(t)=\sum_{\o\in (\mathcal{N}_G(F))}J_{\o}(f^{t})-\sum_{\o'\in (\mathcal{N}_{G^{\ast}}(F))}J_{\o'}(f^{\ast,t}).\]
Grâce au développement fin on obtient, pour $S\supseteq \underline{S}$ assez grand,
\begin{align*}
D_f(t)=&\sum_{M\in\L^G(M_0)}|W_0^M||W_0^G|^{-1}\sum_{Y\in ( \mathcal{N}_M(F))}a^M(S,Y)J_M^G(Y,f^{t})\\
&-\sum_{M'\in\L^{G^\ast}(M_{0^\ast})}|W_{0^\ast}^{M'}||W_{0^\ast}^{G^\ast}|^{-1}\sum_{Y'\in (\mathcal{N}_{M'}(F))}a^{M'}(S,Y')J_{M'}^{G^\ast}(Y',f^{\ast,t}),
\end{align*}
puis l'hypothèse de récurrence et les propositions \ref{lem:Weylcomp} et  \ref{prop:semi-localIOPcomp} nous emmènent à
\[
D_f(t)=\sum_{Y\in (\mathcal{N}_G(F))}\left(a^G(S,Y)-a^{G^\ast}(S,Y^\ast)\right)J_G^G(Y,f^{t}).\]
Ainsi la propriété d'homogénéité des intégrales orbitales nilpotentes (proposition \ref{prop:IOP}) nous dit que $D_f(t)$ est de la forme
\[D_f(t)=\sum_{Y\in (\mathcal{N}_G(F))}D_{f,Y}|t|_{v_0}^{-(1/2)\dim \g/\g_Y},\,\,\,\,\forall t\in F_{v_0}^\times\]
pour $D_{f,Y}=\left(a^G(S,Y)-a^{G^\ast}(S,Y^\ast)\right)J_G^G(Y,f)$.

En parallèle, \`{a} l'aide de la proposition \ref{prop:TFtracefixee} on voit que
\[D_f(t)=\sum_{\o\in\O^{\g},\pi_{\mathfrak{z}}(\o)=0}J_{\o}(\widetilde{f^{t}})-\sum_{\o'\in\O^{\g^\ast},\pi_{\mathfrak{z}^\ast}(\o)=0}J_{\o'}(\widetilde{f^{\ast,t}}).\]
Comme le transfert commute à la transformée partielle de Fourier (proposition \ref{prop:Fouierpcommutetrans}), la même démarche nous conduit à
\[
D_f(t)=\sum_{Y\in (\mathcal{N}_G(F))}\left(a^G(S,Y)-a^{G^\ast}(S,Y^\ast)\right)J_G^G(Y,\widetilde{f^{t}}).\]
D'où
\begin{align*}
D_f(t)&=\sum_{Y\in (\mathcal{N}_G(F))}\left(a^G(S,Y)-a^{G^\ast}(S,Y^\ast)\right)J_{G_\ad}^{G_\ad}(Y,\widehat{f^{t}|_{\g_\ad}}^{\g_\ad})\\
&=\sum_{Y\in (\mathcal{N}_G(F))}\left(a^G(S,Y)-a^{G^\ast}(S,Y^\ast)\right)J_{G_\ad}^{G_\ad}\left(Y,|t|_{v_0}^{-\dim \g_{\ad}}(\widehat{f|_{\g_\ad}}^{\g_\ad})^{ t^{-1}}\right).
\end{align*}
Cette fois-ci la propriété d'homogénéité nous dit que $D(t)$ est de la forme
\[D(t)=\sum_{Y\in (\mathcal{N}_G(F))}D_{f,Y}'|t|_{v_0}^{-\dim \g_\ad+(1/2)\dim \g/\g_Y}\]
pour  $D_{f,Y}'=\left(a^G(S,Y)-a^{G^\ast}(S,Y^\ast)\right)J_G^G(Y,\widetilde{f})$.

Pourtant, pour tous $Y_1,Y_2\in\mathcal{N}_G(F)$ on a
\[-\frac{1}{2}\dim \g/\g_{Y_1}>-\dim \g_\ad+\frac{1}{2}\dim \g/\g_{Y_2}.\]
En effet cela équivaut à
\[\frac{1}{2}\dim\g_{Y_1}+\frac{1}{2}\dim\g_{Y_2}>\dim \mathfrak{z}.\]
Le membre de gauche est au moins le rang réductif de $G$, qui est strictement supérieur au membre de droite si $G$ n'est pas commutatif. En raison de l'incompatibilité d'homogénéité il vient $D_{f,Y}=0$ pour tous $Y,f$, autrement dit
\[\left(a^G(S,Y)-a^{G^\ast}(S,Y^\ast)\right)J_G^G(Y,f)=0\]
pour tous $Y,f$. L'égalité souhaitée vient évidemment de la non-nullité des intégrales orbitales invariantes des éléments nilpotents (proposition \ref{prop:IOP}). Le cas général $X_\ss\in\mathfrak{z}(F)$ se déduit de la même manière en remplaçant $f$ par $f_{X_\ss}$.
\end{proof}

\begin{remark}
Une autre façon de raisonner, est de regarder les distributions tempérées invariantes
\[D(f_{v_0})=\sum_{Y\in (\mathcal{N}_G(F))}\left[\left(a^G(S,Y)-a^{G^\ast}(S,Y^\ast)\right)\prod_{v\not=v_0}J_{G_{\ad,v}}^{G_{\ad,v}}(Y_v,f_v)\right]J_{G_{\ad,v_0}}^{G_{\ad,v_0}}(Y_{v_0},f_{v_0})\]
et
\[D'(f_{v_0})=\sum_{Y\in (\mathcal{N}_G(F))}\left[\left(a^G(S,Y)-a^{G^\ast}(S,Y^\ast)\right)\prod_{v\not=v_0}J_{G_{\ad,v}}^{G_{\ad,v}}(Y_v,\widehat{f_v}^{\g_{\ad,v}})\right]J_{G_{\ad,v_0}}^{G_{\ad,v_0}}(Y_{v_0},f_{v_0})\]
sur $\g_{\ad}(F_{v_0})$, toutes les deux supportées dans le cône nilpotent. On a $D=\widehat{D'}^{\g_{\ad,v_0}}$. Or une distribution à support nilpotent ainsi que sa transformée de Fourier est nulle. En effet le théorème de régularité de Harish-Chandra proclame que $D$ est représentée par une fonction, analytique sur $\g_{\ad,\rss}(F_{v_0})$ (si $v_0$ est archimédienne, la distribution $D'$ étant supportée dans le cône nilpotent, sa transformée de Fourier est $\mathcal{U}(\mathfrak{g}_{\ad}(F_{v_0})_\C)$-finie et on applique \cite[I.5]{Vara77}, ici $\mathcal{U}(\mathfrak{g}_{\ad}(F_{v_0})_\C)$ est bien sûr l'algèbre enveloppante de $\mathfrak{g}_{\ad}(F_{v_0})_\C$ ; si $v_0$ est non-archimédienne il existe $\Omega$ une partie compacte de $\g_{\ad}(F_{v_0})$ telle que $D'$ est supportée dans la clôture de $(\Ad G_{\ad}(F_{v_0}))\Omega$ et on applique \cite[théorème 4.4]{HCbook}). Cette fonction s'annule sur le lieu semi-simple régulier puisqu'elle est supportée dans le cône nilpotent. On conclut encore une fois par la non-nullité des intégrales orbitales invariantes nilpotentes.
\end{remark}

\begin{theorem}\label{thm:finalres'}On a
\begin{align*}
    J_{\o'}^{\g^\ast}(f^\ast)=\begin{dcases*}
    J_\o^\g(f) &, \text{si $\o'=\o^\ast$ ;} \\
    0 &, \text{si $\o'$ ne se transfère pas.}
    \end{dcases*}
    \end{align*}
\end{theorem}
\begin{proof}
En manipulant le théorème précédent, le développement fin, les propositions \ref{lem:Weylcomp} et  \ref{prop:semi-localIOPcomp}, on démontre ce théorème. 
\end{proof}

On rappelle que $J_\geom^G$ (resp. $J_\geom^{G^\ast}$) est le côté géométrique de la formule des traces de $G$ (resp. $G^\ast$), donné par les équations \eqref{eqdef:cotegeomTFGrpconto} et \eqref{eqdef:cotegeomFTGrp}. 

\begin{theorem}\label{thm:finalres}On a 
\[J_{\text{géom}}^G(f)=J_{\text{géom}}^{G^\ast}(f^\ast).\]
\end{theorem}
\begin{proof}
C'est une conséquence du théorème précédent et la proposition \ref{prop:Lie=Grp}.
\end{proof}

\begin{remark}\label{rem:finalremarkonmeasures}
Certaines normalisations des mesures fixées au début de la section s'avèrent désormais superflues pour les résultats capitaux de cet article, à savoir les théorèmes \ref{thm:finalres'} et \ref{thm:finalres}. Conscients du caractère potentiellement gênant de ces redondances, nous les éliminons en vue des besoins futurs. 

Les théorèmes \ref{thm:finalres'} et \ref{thm:finalres} sont valables sous les normalisations suivantes :

\begin{enumerate}[label=(\roman*)]
    \item pour tout $v\in (\V_F\setminus \underline{S})\cup \V_{\infty}$, les mesures sur $G_v(F_v)$ et $G_v^\ast(F_v)$ se correspondent par l'isomorphisme $(\Ad x_v)\circ\eta_v|_{G_v(F_v)}$ ;
    \item pour tout $P\in \F^G$ ou $\F^{G^\ast}$, $\vol(N_P(F)\backslash N_P(\A_F))=1$ ;
    \item pour tous $v\in \underline{S}$,  $M\in\L^G(M_{0,v})$, le torseur intérieur $\eta_v$ induit un isomorphisme d'espaces  vectoriels réels $a_{M}\xrightarrow{\sim} a_{M^\ast}$, on veut que les mesures sur $a_{M}$ et $a_{M^\ast}$ se correspondent par cet isomorphisme ;
    \item le torseur intérieur $\eta$ induit un isomorphisme d'espaces  vectoriels réels $a_{G}\xrightarrow{\sim} a_{G^\ast}$, on veut que les mesures sur $a_{G}$ et $a_{G^\ast}$ se correspondent par cet isomorphisme, ainsi la mesure sur $A_{G,\infty}$ correspondent à celle sur $A_{G^\ast,\infty}$ par $\eta$ ;
    \item pour tous $v\in \underline{S}$, si $T$ est un tore maximal de $G_v$ et $T^\ast$ un tore maximal de $G_v^\ast$, et $y\in G_v^\ast(\overline{F_v})$ est tel que $(\Ad y)\circ\eta|_T$ soit un isomorphisme défini sur $F_v$ de $T$ sur $T^\ast$, alors les mesures sur $T(F_v)$ et $T^\ast(F_v)$ se correspondent par cet isomorphisme.
\end{enumerate}
\end{remark}

\appendix

\section{Classes de conjugaison d'une algèbre séparable sur un corps parfait}\label{sec:AppendixA}
Dans cet annexe on décrit l'ensemble des classes de conjugaison d'une algèbre séparable définie sur un corps parfait via la théorie des diviseurs élémentaires, et on dépeint son lien avec l'induite de Lusztig-Spaltenstein généralisée. Enfin on explique le tranfert des classes de conjugaison d'une algèbre séparable aux celles de sa forme intérieure quasi-déployée.

Il est manifeste que la problématique devrait s'ancrer dans la théorie élémentaire de l'algèbre linéaire. Malgré cette orientation, certains outils nécessaires à notre propos semblent manquer d'une documentation adéquate dans la littérature, c'est ainsi que nous choisissons d'adjoindre cette partie du contenu.

\subsection{Diviseurs élémentaires}
Notre référence principale de ce numéro est \cite{Yu13}. Dans cette sous-section $F$ désigne un corps parfait. Dans un premier temps on ne considère que des groupes abstraits (resp. anneaux abstraits, modules abstraits...) mais non groupes algébriques (resp. anneaux algébriques, modules algébriques...).

\begin{definition}Soit $A$ une algèbre simple centrale sur $F$. On écrit $A=\Mat_m(D)$ avec $n$ un entier positif et $D$ une algèbre à division centrale sur $F$. L'entier $n$ est uniquement déterminée par $A$ ; l'algèbre $D$ est uniquement déterminée à isomorphisme d'algèbres associatives près.
\begin{enumerate}
    \item On définit respectivement le degré de $A$, la capacité de $A$, et l'indice de $A$ par
    \[\deg_F(A)\eqdef \sqrt{\dim_FA},\,\,\,\,\,\,\,\,c_F(A)\eqdef m,\,\,\,\,\,\,\,\, \ind_F(A)\eqdef \sqrt{\dim_FD}. \]
    \item Soit $X\in A=\Mat_n(D)$. Le polynôme caractéristique de $X$ est défini comme le polynôme caractéristique de son image par l'injection
    \[A\longrightarrow A\otimes_F\overline{F}\simeq \Mat_{md}(\overline{F})\]
    avec $\overline{F}$ une clôture algébrique de $F$ et $d$ le degré de $D$. Le polynôme obtenu ne dépend pas du choix de l'isomorphisme $A\otimes_F\overline{F}\simeq \Mat_{md}(\overline{F})$ et est défini sur $F$.
\end{enumerate}
\end{definition}

Désormais $A=\text{End}_{D}(V)$ est une algèbre simple centrale sur $F$, avec $D$ une $F$-algèbre à division et $V$ un espace vectoriel à droite sur $D$ de dimension $m$. Posons $d\eqdef \deg_F(D)$.

Soit $X\in A$. Notons $F[X]\subseteq A$ la sous-algèbre engendrée par $X$, c'est un anneau artinien. Décomposons $m_X(T)$ le polynôme minimal de $X$ en un produit de polynômes irréductibles $m_X(T)=\prod_{p\in \Irr_F}p^{e_p}$ avec $e_p$ des entiers positifs et $\Irr_F$ l’ensemble des polynômes irréductibles non-constants de coefficient dominant 1 de $F[T]$. \'{E}crivons $F_p=F[T]/(p)$ et $\widetilde{F_p}=F[T]/(p^{e_p})$. On a la décomposition de $F[X]$ en un produit d'anneaux locaux artiniens
\begin{equation}\label{eq:2.1}
F[X]=F[T]/(m_X(T))=\prod_{p\in \Irr_F} \widetilde{F_p}.
\end{equation}
Pour $p$ tels que $e_p\geq 1$ l'anneau $\widetilde{F_p}$ est en plus un anneau local noetherien complet de corps résiduel $F_p$, il existe donc un monomorphisme d'anneaux $s_p : F_p\rightarrow \widetilde{F_p}$ tel que $\pi_p \circ s_p = \Id_{F_p}$, où $\pi_p :\widetilde{F_p}\rightarrow F_p$ est la projection naturelle. Prenons $\varepsilon_p\in \widetilde{F_p}$ un générateur de l'idéal maximal de $\widetilde{F_p}$, on a $\widetilde{F_p}=s_p(F_p)[\varepsilon_p]/(\varepsilon_p^{e_p})$.

On sait que l'espace vectoriel $V$ est un $(F[X],D)$-bimodule, ou, de manière équivalente un $D\otimes_F F[X]$-module à droite. La décomposition \eqref{eq:2.1} fournit une décomposition de $V$ en une somme directe de sous-$D$-modules $V_p$ sur lequel la $F$-algèbre $\widetilde{F_p}$ agit de façon fidèle
\[V=\bigoplus_{p\in \Irr_F}V_p,\,\,\,\,\,\,\,\,m_p\eqdef\dim_D V_p\geq 0,\,\,\,\,\,\,\,\,\sum_{p\in \Irr_F} m_p=m.\] 
On peut et on va choisir l'application $s_p$ telle que $F\subseteq s_p(F_p)$ par un théorème de Cohen (cf. \cite[p.354]{Yu13}). Supposons que $D\otimes_F s_p(F_p)=\Mat_{c_p}(D_p)$ avec $c_p$ la capacité de $D\otimes_F s_p(F_p)$ et $D_p$ une $s_p(F_p)$-algèbre à division. Chaque $V_p$ est un $D\otimes_F\widetilde{F_p}$-module à droite. On a 
\[D\otimes_F\widetilde{F_p}= (D\otimes_F s_p(F_p))\otimes_{s_p(F_p)}\widetilde{F_p}=\Mat_{c_p}(\widetilde{D_p})\]
avec
\[\widetilde{D_p}\eqdef D_p\otimes_{s_p(F_p)}\widetilde{F_p}=D_p[\varepsilon_p]/(\varepsilon_p^{e_p}).\]
Par l'équivalence de Morita on en déduit $V_p=W_p^{\oplus c_p}$ pour $W_p$ un $\widetilde{D_p}$-module. Puisque $\widetilde{D_p}$ est un anneau principal (un anneau est dit principal si tout idéal à gauche ou à droite est principal, donc ici anneau principal correspond à « principal ideal ring » en anglais, et non à « principal ideal domain ») 
on obtient un isomorphisme de $\widetilde{D_p}$-modules 
\begin{equation}\label{eq:3.2}
    W_p=\bigoplus_{j=1}^{t_p}W_{p,j},\,\,\,\,W_{p,j}\eqdef D_p[\varepsilon_p]/(\varepsilon_p^{v_j})
\end{equation} 
pour $t_p\in \mathbb{N}_{>0}$ et $0\leq v_1\leq\cdots\leq v_{t_p}\leq e_p$ des entiers positifs. Comme $\widetilde{F_p}$ agit fidèlement sur $W_p$, on a $v_{t_p} = e_p$. La décomposition \eqref{eq:3.2} donne la décomposition de $V$ en composantes invariantes.

\begin{theorem}[{{\cite[théorème 4.]{Yu13}}}]\label{thm:chara}
Soit $F$ un corps parfait. Soit $f(T)\in F[T]$ un polynôme de coefficient dominant 1 et de degré $md$, avec $f(T)=\prod_{p\in \Irr_F} p^{a_p}$. \'{E}crivons $F_p=F[T]/(p)$. Alors $f(T)$ est le polynôme caractéristique d'un élément de $A$ si et seulement si pour tout $p\in \Irr_F$ on a
\begin{enumerate}
    \item $a_p\deg p=m_pd$ pour $m_p$ un entier positif ;
    \item $[F_p:F]\mid m_p c(D\otimes_F F_p)$.
\end{enumerate}
\end{theorem}

\begin{remark}
Lorsque $F$ est un corps local, la seconde condition peut être négligée : soit $f(T)=\prod_{p\in\Irr_F} p^{a_p}$ un polynôme comme dans l'énoncé vérifiant $a_p\deg p=m_pd$ pour tout $p$. Par la théorie des corps de classes locaux on a $\text{inv}(D\otimes_F F_p)=[F_p:F]\text{inv}(D)$, où $\text{inv}$ est l'invariant local. Ainsi
\[c(D\otimes_F F_p)=\text{pgcd}(d,[F_p:F]).\]
La seconde condition est de ce fait équivalente à
\[[F_p:F]\mid\frac{a_p\deg p}{d}\text{pgcd}(d,[F_p:F]),\]
soit
\[d\mid a_p\text{pgcd}(d,[F_p:F]),\]
ce qui est une conséquence de la première condition.
\end{remark}

Pour $p\in \Irr_F$ on fixe $D_p$ une $F_p$-algèbre à division telle que $D\otimes_F F_p\simeq \Mat_{c(D\otimes_F F_p)}(D_p)$. Une partition $\lambda=(\lambda_1,\dots,\lambda_k)$ est une suite finie décroissante d'entiers positifs, et on pose $|\lambda|\eqdef \lambda_1+\cdots+\lambda_k$. Par convention $|\lambda| = 0$ si $\lambda = \emptyset$.

\begin{theorem}[{{\cite[théorème 9.]{Yu13}}}]\label{thm:Yu9}
L'application qui à $X\in A$ associe son polynôme caractéristique et les partitions des $\dim_{D_p}W_p$ par l'équation \eqref{eq:3.2} induit une bijection entre l'ensemble des classes de $A^\times$-conjugaison dans $A$ et l'ensemble des fonctions $\lambda:\Irr_F\rightarrow\{\text{partitions des entiers}\}$ telle que
\[\sum_{p\in \Irr_F
}\deg(p)|\lambda(p)|\deg_{F_p}(D_p)=\deg_{F}(A).\]
\end{theorem}
\begin{proof}
Dans l'article cité la bijection est établie entre l'ensemble des classes de $A^\times$-conjugaison dans $A^\times$ et l'ensemble des fonctions $\lambda:\Irr_F\setminus\{T\}\rightarrow\{\text{partitions des entiers}\}$ vérifiant l'égalité dans l'énoncé. Mais pour tout $X\in A$ il existe $Z\in F$ tel que $X+Z\in A^\times$, d'où le résultat cherché.
\end{proof}

Soit $\lambda=(\lambda_1,\dots, \lambda_k)$ une partition d'un entier positif. Pour tout entier strictement positif $e$, nous notons par $e\times \lambda$ la partition $(\lambda_1,\dots,\lambda_1,\dots,\lambda_k,\dots,\lambda_k)$ de $e|\lambda|$, où chaque $\lambda_i$ est répété $e$ fois. On note $e\mid_\times\lambda$ s'il existe $\lambda'$ une partition d'un entier positif telle que $\lambda=e\times\lambda'$.

\begin{corollary}\label{coro:classconjratpt}Soient $A=\Mat_m(D)$ et $A_{\overline{F}}=\Mat_{m}(D\otimes_F\overline{F})$. Alors une classe de $A_{\overline{F}}^\times$-conjugaison dans $A_{\overline{F}}$ définie sur $F$ (i.e. stable par $\Gal(\overline{F}/F)$) contient un élément de $A$ si et seulement si la fonction $\lambda_{\overline{F}}$ correspondante vérifie $\deg_{F_p}(D_p)\mid_{\times}\lambda_{\overline{F}}(p')$ pour tout $p'\in \Irr_{\overline{F}}$, ici $p\in \Irr_{F}$ est le multiple irréductible de $p'$.

Spécialement si $D=F$ alors toute classe de $A_{\overline{F}}^\times$-conjugaison dans $A_{\overline{F}}$ définie sur $F$ contient un élément de $A$.
\end{corollary}

\begin{proof}Démontrons le sens direct.
Soit $\o$ une classe de $A_{F}^{\times}$-conjugaison dans $A_{F}$ définie sur $F$ correspondante à $\lambda_{F}:\Irr_{F}\rightarrow\{\text{partitions des entiers}\}$. Soit $X$ un élément de $\o$. Les notations $V$, $V_p$, $D_p$ etc portent les sens identiques que précédemment. 

Décomposons le polynôme minimal de $X$ en
\[m_{X}(T)=\prod_{p\in \Irr_{F}}p^{e_p}=\prod_{p\in \Irr_{F}}\prod_{p'\in\Irr_{\overline{F}},p'\mid p}{p'}^{e_p}.\]
On note alors
\[\overline{F}_{p'}\eqdef \overline{F}[T]/(p'),\,\,\,\,\text{et}\,\,\,\,\widetilde{\overline{F}_{p'}}\eqdef\overline{F}[T]/(p'^{e_p}).\]
Nous avons que 
\[D\otimes_F \widetilde{F_p}\otimes_F \overline{F}=\Mat_{d}(\widetilde{F_p}\otimes_F\overline{F})=\prod_{p'\mid p}\Mat_{d}(\widetilde{\overline{F}_{p'}})\]
agit à droite sur 
\begingroup
\allowdisplaybreaks
\begin{align*}
V_p\otimes_F \overline{F}&=W_p^{\oplus c_p}\otimes_F \overline{F}    \\
&=\left[\bigoplus_{j=1}^{t_p}(D_p\otimes_{F_p}F_p\otimes_F\overline{F})[\varepsilon_p]/(\varepsilon_p^{v_j})\right]^{\oplus c_p}\\
&=\left[\bigoplus_{j=1}^{t_p}\prod_{p'\mid p}(D_p\otimes_{F_p}\overline{F}_{p'})[\varepsilon_p]/(\varepsilon_p^{v_j})\right]^{\oplus c_p}\\
&=\prod_{p'\mid p}\left[\bigoplus_{j=1}^{t_p}\Mat_{\deg_{F_p}D_p}(\overline{F}_{p'})[\varepsilon_p]/(\varepsilon_p^{v_j})\right]^{\oplus c_p}\\
&=\prod_{p'\mid p}\left[\left[\bigoplus_{j=1}^{t_p}\overline{F}_{p'}[\varepsilon_p]/(\varepsilon_p^{v_j})\right]^{\oplus \deg_{F_p}D_p}\right]^{\oplus (\deg_{F_p}D_p)\cdot c_p}
\end{align*}
\endgroup
En suivant l'action galoisienne à chaque étape on voit que la classe de $A_{\overline{F}}^{\times}$-conjugaison dans $A_{\overline{F}}$ de $X$ correspond à la partition 
\[\lambda_{\overline{F}} : p'\in\Irr_{\overline{F}}\longmapsto (\deg_{F_p}D_p)\times \lambda_F(p),\,\,\,\,p\in \Irr_F \text{ est tel que }p'\mid p\]
Ce qu'il fallait. Réciproquement la fonction \[\lambda_F:p\in\Irr_{F}\longmapsto \lambda_{\overline{F}}(p'),\,\,\,\,p'\in \Irr_{\overline{F}} \text{ est tel que }p'\mid p\]
(bien définie car la classe de $A_{\overline{F}}^\times$-conjugaison étant définie sur $F$ on sait que pour tout $\sigma\in\Gal(\overline{F}/F)$, $\lambda_{\overline{F}}(p')=\lambda_{\overline{F}}(\sigma(p'))$) détermine une classe de $A_F^\times$-conjugaison, son extension sur $\overline{F}$ correspond à la même fonction $\lambda_{\overline{F}}$ eu égard à ce qui précéde, cette classe de $A_{\overline{F}}^\times$-conjugaison en question contient donc un élément de $A_F$.
\end{proof}

\begin{remark}
\`{A} propos de la seconde partie de l'assertion, on pourrait également appliquer \cite[théorème 4.2]{Kott82}, qui stipule que dans $H$ un groupe réductif quasi-déployé de sous-groupe dérivé simplement connexe, toute classe de conjugaison de $H({\overline{F}})$ définie sur $F$ contient un $F$-point. Il suffit de noter que pour tout $X\in A$ il existe $Z\in F=Z(A)$ tel que $X+Z\in A^\times$. 
\end{remark}

\subsection{Réduction de Jordan et centralisateur}\label{subsec:reductiondeJordanetcentralisateur}
Partant du théorème \ref{thm:Yu9} on souhaite obtenir la réduction de Jordan. Du théorème, il n'est pas difficile de voir que si $p\in \Irr_F$ et $a\in\mathbb{N}_{\geq 1}$, alors $p^a$ est un polynôme caractéristique d'un élément de $\GL_m(D)$ pour un certain $m$ si et seulement s'il existe $k$ un entier strictement positif tel que
\[m=\underline{m}_pk \,\,\,\,\text{ et }\,\,\,\,a=\underline{a}_pk\]
avec
\[\underline{m}_p\eqdef\frac{\deg p \deg_{F_p}D_p}{\deg_F D}\,\,\,\,\text{ et }\,\,\,\,\underline{a}_p\eqdef\deg_{F_p}D_p.\]

Pour tout $p\in \Irr_F$ on fixe un $\textbf{X}_p$ élément de $\Mat_{\underline{m}_p}(D)$ de polynôme caractéristique $p^{\underline{a}_p}$, il est bien sûr semi-simple. Un bloc de Jordan de paramètre $p\in\Irr_F$ et d'échelon $k\in\mathbb{N}$ est la matrice par blocs
\[\text{J}(p,k)= \begin{pmatrix} 
    \textbf{X}_p & \textbf{1} &  &  &  &  \\
     & \textbf{X}_p & \textbf{1} &  &  &  \\
     &  & \ddots & \ddots &  &  \\
     &  &  & \ddots & \ddots &  \\
     &  &  &  & \textbf{X}_p & \textbf{1} \\
     &  &  &  &  & \textbf{X}_p \\
\end{pmatrix}\in \Mat_{\underline{m}_pk}(D)\]
avec $\textbf{1}$ la matrice identité de $\Mat_{\underline{m}_P}(D)$. Alors la classe de conjugaison correspondant à la fonction $\lambda:\Irr_F\rightarrow\{\text{partitions des entiers}\}$ dans le théorème précédent est bien la classe de conjugaison de la matrice diagonale par blocs
\[\text{diag}(\text{J}(p,\lambda(p)))_{p\in\Irr_F}\]
avec $\text{J}(p,\lambda(p))$ notant la concaténation diagonale des matrice $\text{J}(p,\lambda(p)_i)_{i=1,\dots,k_p}$ si on écrit $\lambda(p)=(\lambda(p)_1,\dots,\lambda(p)_{k_p})$. 

On calcule au passage le centralisateur d'un élément $X$ de $A=\Mat_m(D)$. Reprenons les notations de la dernière section. Un élément de $A$ commute à $X$ doit fixer les « espaces propres généralisées » $V_p$ de $X$ pour $p\in\Irr_F$, et il est complètement déterminé par la façon dont il agit sur un générateur fixé de chaque $\widetilde{D_p}$-module cyclique $W_{p,j}$ dans l'équation \eqref{eq:3.2}, ce générateur peut être envoyé à n'import quel élément de $W_{p,1}\oplus\cdots\oplus W_{p,j}$, puis il peut aussi être envoyé dans $W_{p,i}$ avec $i>j$, dans ce cas l'image de ce générateur doit appartenir à l'unique copie isomorphique de $W_{p,j}$ dans $W_{p,i}$, sans aucune restriction de plus.

Si $G$ le groupe algébrique $\GL_{m,D}$, $\g=\Mat_{m,D}$ son algèbre de Lie, et $X\in \g(F)$ dont la classe de conjugaison correspond à $\lambda:\Irr_F\rightarrow\{\text{partitions des entiers}\}$, et on écrit $\lambda(p)\not=\emptyset$ comme
\[\lambda(p)=(v_{p,1}\dots,v_{p,1},\dots,v_{p,k_p},\dots,v_{p,k_p})\]
avec $k_p\in \N_{>0}$, $1\leq v_{p,1}<\cdots<v_{p,k_p}$ et chaque $v_{p,i}$ est répété $n_{p,i}\in \N_{>0}$ fois dans la partition. Alors le facteur de Levi de la décomposition du centralisateur $\g_X$ est isomorphe à
\[\prod_{p : \lambda(p)\not =\emptyset}\prod_{i=1}^{k_p}\Res_{\underline{D}_p^c/F}\Mat_{n_{p,i},\underline{D}_p^c},\]
avec $\underline{D}_p^c$ est le centralisateur de $\textbf{X}_p$ dans $\Mat_{\underline{m}_p}(D)$, c'est une algèbre à division de centre $F_p$ et d'indice $\underline{a}_p$. En effet $\underline{D}_p^c$ est une algèbre simple, s'il n'est pas une algèbre à division alors il contiendra d'éléments nilpotents non-nuls. En translatant $\textbf{X}_p$ par un tel nilpotent non-nul on obtiendra un autre élément de  $\Mat_{\underline{m}_p}(D)$ de même polynôme caractéristique $p^{\underline{a}_p}$, cela contredira le théorème \ref{thm:Yu9}. L'algèbre $\underline{D}_p^c$ contient le corps $F_p$ puisqu'elle contient la sous-$F$-algèbre engendrée par $\textbf{X}_p$ dans $\Mat_{\underline{m}_p}(D)$, le corps $F_p$ est de surcroît le centre de $\underline{D}_p^c$ : l'algèbre $\underline{D}_p^c\otimes_FF_p$ est le centralisateur de $\textbf{X}_p$ dans $\Mat_{\underline{m}_p}(D)\otimes_FF_p=\Mat_{\underline{m}_pc_p}(D_p)=\Mat_{\deg p}(D_p)$, elle est ainsi une forme intérieure du centralisateur d'un élément de polynôme caractéristique $p^{\underline{a}_p}$ dans $\Mat_{\underline{a}_p\deg p}(F_p)$, ce dernier vaut $(\Mat_{\underline{a}_p}(F_p))^{\deg p}$. Par conséquent le centre de $\underline{D}_p^c\otimes_FF_p$ vaut $(\Mat_{1}(F_p))^{\deg p}$, une comparaison de dimension nous donne que le centre de  $\underline{D}_p^c$ vaut $F_p$. Enfin le théorème du double centralisateur nous dit que $\dim_{F_p}\underline{D}_p^c=\underline{a}_p^2$.

\subsection{Orbites induites}

La notion d'une orbite induite nilpotente est initialement proposée par Richardson et ensuite généralisée par Lusztig et Spaltenstein, la théorie (sur un corps algébriquement clos de caractéristique 0) est résumée dans le monographe \cite[chapitre 7]{CM93}. Pour la théorie des orbites induites générales on peut se référer à \cite{YDL23a}. Nous entamons la partie par des rappels préliminaires. Soient $F$ un corps parfait, $D$ une algèbre à division définie sur $F$, $G$ le groupe algébrique $\GL_{n,D}$, et $\g=\Mat_{m,D}$ son algèbre de Lie, enfin $d\eqdef\deg_F(D)$. On note $Z(G)$ le centre de $G$ et $Z(\g)$ son algèbre de Lie.

Supposons pour le moment que $F$ est algébriquement clos, implicitement $D=F$. Soient $M\subseteq L$ deux sous-groupes de Levi de $G$, et $P\in\P^G(M)$ un sous-groupe parabolique de $G$. On sait que $P\cap L$ est un sous-groupe parabolique de $L$ ayant $M$ comme facteur de Levi. On note $\mathfrak{n}_{L\cap P}^L$ l'algèbre de Lie du rqdicql unipotent de $P\cap L$. Soit $\o$ une orbite dans $\mathfrak{m}$ pour l’action adjointe de $M$. Il existe une unique $L$-orbite $\Ind_M^L(\o)$ telle que 
\[\Ind_M^L(\o)\cap (\o+\mathfrak{n}_{L\cap P}^L)\]
soit un ouvert de Zariski dense dans $\o+\mathfrak{n}_{L\cap P}^L$. Cette intersection est également le lieu régulier de $\o+\mathfrak{n}_{L\cap P}^L$ pour la $G$-conjugaison, i.e.
\[\Ind_M^L(\o)\cap (\o+\mathfrak{n}_{L\cap P}^L)=\{X\in \o+\mathfrak{n}_{L\cap P}^L :\dim (\Ad G)X \geq \dim (\Ad G)Y\,\,\forall Y\in \o+\mathfrak{n}_{L\cap P}^L\}.\]
L'orbite $\Ind_M^L(\o)$ ne dépend pas de $P$. 

L’induction est transitive : soient $M\subseteq L\subseteq I$ trois sous-groupes de Levi de $G$ et $\o$ une orbite dans $\mathfrak{m}$ pour l’action adjointe de $M$ alors
\[\Ind_L^I(\Ind_M^L(\o))=\Ind_M^I(\o).\]

Lorsque $\o$ contient un $F$-point, l'orbite induite $\Ind_M^L(\o)$ l'est aussi car $\mathfrak{n}_{L\cap P}^L(F)$ est de Zariski dense dans $\mathfrak{n}_{L\cap P}^L$, et si $X\in \Ind_M^L(\o)(F)$ alors
\[\Ind_M^L(\o)(F)=\text{ la classe de }L(F)\text{-conjugaison de }X.\]
On ne s'intéresse dorénavant qu'aux orbites possédant de $F$-points et aux $F$-points dans ces orbites. 

Nous définissons une relation d'ordre sur l'ensemble des classes de conjugaison en posant $\o_1\leq \o_2$ lorsque la clôture de Zariski de $\o_1$ est incluse dans celle de $\o_2$. Nous définissons en outre une relation d'ordre sur l'ensemble des partitions des entiers en posant $\lambda_1\leq \lambda_2$ si $\lambda_{i}=(\lambda_{i,1},\dots,\lambda_{i,k})$ avec $\sum_{j=1}^s\lambda_{1,j}\leq \sum_{j=1}^s\lambda_{2,j}$ pour tous $s$. Fixons également d'autres notations supplémentaires : soit $\lambda=(\lambda_1,\dots,\lambda_k)$ une partition d'un entier. Alors on note par $\lambda^t$ la partition de $|\lambda|$ associée à la transposée du tableau de Young de $\lambda$. Si $e$ est un entier positif, on note par $e\cdot \lambda$ la partition $(e\lambda_1,\dots,e\lambda_k)$ de $e|\lambda|$. Enfin on écrit $e\mid_\cdot\lambda$ s'il existe $\lambda'$ une partition d'un entier positif telle que $\lambda=e\cdot\lambda'$.

\begin{lemma}
Soient $\lambda$ une partition d'un entier positif et $e$ un entier strictement positif, on a
\[(e\cdot \lambda)^t=e\times \lambda,\text{ et }(e\times\lambda)^t=e\cdot\lambda.\]
\end{lemma}

On dispose des lemmes suivants, généralisations évidentes du lemme 6.2.2, du théorème 6.2.5, et du théorème 7.2.3 de \cite{CM93}.

\begin{lemma}\label{lem622}Soient $\o_1$ et $\o_2$ deux classes de conjugaison dans $\g(F)$ correspondant aux fonctions $\lambda_1$ et $\lambda_2$ par la théorie des diviseurs élémentaires. Soit $X_i\in \o_i$ pour $i=1,2$. Alors $\lambda_1(p)\leq \lambda_2(p)$ pour tout $p\in \Irr_F$ si et seulement si $\text{rk}((X_1-a)^k)\leq \text{rk}((X_2-a)^k)$ pour tous $a\in F$ et $k\in\mathbb{\N}$, ici $\text{rk}$ note le rang d'une matrice.
\end{lemma}

\begin{lemma}[Gerstenhaber, Hesselink]\label{thm625} Soient $\o_1$ et $\o_2$ deux classes de conjugaison dans $\g(F)$ correspondant aux fonctions $\lambda_1$ et $\lambda_2$ par la théorie des diviseurs élémentaires. Alors $\o_1\leq \o_2$ si et seulement si $\lambda_1(p)\leq \lambda_2(p)$ pour tout $p\in \Irr_F$.
\end{lemma}

\begin{lemma}[Kraft, Ozeki, Wakimoto]\label{thm723} Soit $L=\prod_{i=1}^k \GL_{m_i,D}$ avec $(m_1,\dots,m_k)$ une partition de $m$, $X=\prod_{i=1}^k X_i\in Z(\mathfrak{l})(F)$ avec $X_i\in Z(\Mat_{m_i}(D))$, et $\o$ la classe de $L(F)$-conjugaison de $X$. Supposons que la classe de $\GL_{m_i}(D)$-conjugaison de $X_i$ correspond à $\lambda_i:\Irr_F\rightarrow\{\text{partitions des entiers}\}$ et qu'il existe un unique $p\in\Irr_F$ tel que $\lambda_i(q)\not=\emptyset$ si et seulement si $q=p$ pour tout $i$, alors $\Ind_L^G(\o)$ est la classe de $G(F)$-conjugaison correspondant à $\lambda:\Irr_F\rightarrow\{\text{partitions des entiers}\}$ avec
\[\lambda(q)=\begin{cases*}
(m_1,\dots,m_k)^t & si $q=p$\\
\emptyset& sinon\end{cases*}.\]
\end{lemma}

On vise à complètement expliciter le processus de l'induction. Pour cela, un dernier lemme est indisponsable.

\begin{lemma}\label{lem:diffeigen}
Soit $L=\prod_{i=1}^k \GL_{m_i,D}$ avec $(m_1,\dots,m_k)$ une partition de $m$, $X=\prod_{i=1}^k X_i\in \mathfrak{l}(F)$ avec $X_i\in\Mat_{m_i}(D)$, et $\o$ la classe de $L(F)$-conjugaison de $X$. Supposons que la classe de $\GL_{m_i}(D)$-conjugaison de $X_i$ correspond à $\lambda_i:\Irr_F\rightarrow\{\text{partitions des entiers}\}$ et que pour tout $i$ existe un unique $p_i\in\Irr_F$ tel que $\lambda_i(q)\not=\emptyset$ si et seulement si $q=p_i$, et que $p_i\not =p_j$ si $i\not = j$, alors $\Ind_L^G(\o)=\Ad(G(F))\o$.
\end{lemma}
\begin{proof}
Soient $P=LN_P\in\P^G(L)$, $N\in \mathfrak{n}_P(F)$ et $a\in F$. En écrivant $X+N-a$ en une matrice triangulaire supérieure on voit aisément que pour tout $k$, l'entier $\text{rk}((X+N-a)^k)$ est une constante lorsque $N$ varie. La relation $\Ind_L^G(\o)=\Ad(G(F))\o$ s'ensuit alors en vertu des lemmes \ref{lem622} et \ref{thm625}.
\end{proof}

\begin{theorem}\label{thm:inddes}
Soit $L=\prod_{i=1}^k \GL_{m_i,D}$ avec $(m_1,\dots,m_k)$ une partition de $m$, $X=\prod_{i=1}^k X_i\in \mathfrak{l}(F)$ avec $X_i\in\Mat_{n_i}(D)$, et $\o$ la classe de $L(F)$-conjugaison de $X$. Supposons que la classe de $\GL_{m_i}(D)$-conjugaison de $X_i$ correspond à $\lambda_i:\Irr_F\rightarrow\{\text{partitions des entiers}\}$, alors $\Ind_L^G(\o)$ est la classe de $G(F)$-conjugaison correspondant à $\lambda:\Irr_F\rightarrow\{\text{partitions des entiers}\}$ avec
\[\lambda(p)=(\lambda_1(p)^t,\dots,\lambda_k(p)^t)^t.\]
\end{theorem}
\begin{proof}
On peut écrire, avec le concours des deux lemmes précédents, 
\[\o=\Ind_M^L(X_\ss),\]
où $M=\prod_{i=1}^k M_i$ et $M_i$ est un sous-groupe de Levi de $\GL_{n_i,D}$ associé à la partition $(\lambda_i(p)^t)_{p\in \Irr_F}$ de $n_i$. Puis la transitivité et toujours les deux lemmes précédents nous garantissent que  $\Ind_L^G(\o)=\Ind_L^G(\Ind_M^L(X_\ss))$ est la classe de $G(F)$-conjugaison correspondant à $\lambda:\Irr_F\rightarrow\{\text{partitions des entiers}\}$ avec $\lambda(p)=(\lambda_1(p)^t,\dots,\lambda_k(p)^t)^t.$
\end{proof}

Enlevons à présent l'hypothèse que $F$ est algébriquement clos, toutes les assertions des discussions précédentes, sauf le lemme \ref{lem622}, restent valables par un argument de descente galoisienne, ici on remarque que la coïncidence entre la conjugaison rationelle et celle géométrique joue un rôle essentiel. 


Pour un groupe algébrique $H$ on note $A_H$ le sous-tore central $F$-déployé maximal dans $H$. Un élément $X$ de $\g(F)$ est dit $F$-elliptique s'il est semi-simple et $A_G=A_{G_X}$. Une classe de $G(F)$-conjugaison dans $\g(F)$ est dit $F$-elliptique si un (donc tout) élément dedans est $F$-elliptique. Une classe de $G(F)$-conjugaison dans $\g(F)$ est $F$-elliptique si et seulement si la fonction $\lambda:\Irr_F\rightarrow\{\text{partitions des entiers}\}$ correspondante vérifie la condition suivante :
\[\text{il existe un unique $p\in \Irr_F$ tel que }\lambda(p)\not=\emptyset.\]

On prouve une dernière propriété de rationnalité.

\begin{proposition}\label{prop:suppellitique}
Pour tout $X\in\g(F)$, il existe $L$ un sous-groupe de Levi de $G$ dont l'algèbre de Lie contient $X_\ss$ en tant qu'élément $F$-elliptique tel que
\[\Ad(G)X=\Ind_L^G(X_\ss).\]
Si $(L_1,\o_1)$ et $(L_2,\o_2)$ sont deux paires avec $L_i$ un sous-groupe de Levi et $\o_i$ une $L_i(F)$-orbite $F$-elliptique dans $\mathfrak{l}_i$, telles que $\Ind_{L_1}^G(\o_1)=\Ind_{L_2}^G(\o_2)$, alors il existe $g\in G(F)$ tel que $L_1=(\Ad g)L_2$ et $\o_1=(\Ad g)\o_2$.
\end{proposition}

\begin{proof}
\'{E}bauchons par la première partie de l'énoncé. On suppose que la classe de $G(F)$-conjugaison de $X$ correspond à $\lambda:\Irr_F\rightarrow\{\text{partitions des entiers}\}$ avec $\sum_{p\in \Irr_F}\deg(p)|\lambda(p)|\deg_{F_p}(D_p)=nd.$
On sait que la classe de $G(\overline{F})$-conjugaison de $X$ correspond à $\lambda_{\overline{F}}:\Irr_{\overline{F}}\rightarrow\{\text{partitions des entiers}\}$ avec $\lambda_{\overline{F}}(p')=\deg_{F_p}(D_p)\times \lambda_F(p)$ pour tout $p'\in\Irr_{\overline{F}}$, ici $p\in\Irr_F$ est le multiple irréductible de $p'$.

Avec l'aide du théorème \ref{thm:inddes}, on voit que
\[(\Ad(G)X)_{\overline{F}}=\Ind_{L'}^{G_{\overline{F}}}(X_\ss)_{\overline{F}},\]
avec $L'$ le sous-groupe de Levi standard de $G_{\overline{F}}$ asssocié à la partition 
\[(\lambda_{\overline{F}}(p')^t)_{p'\in \Irr_{\overline{F}}}.\]
Or $(\lambda_{\overline{F}}(p')^t)_{p'\in \Irr_{\overline{F}}}=(\deg_{F_p}(D_p)\cdot(\lambda(p)^t)_{p'\mid p})_{p'\in \Irr_{\overline{F}}}$, on sait en outre que $d\mid \deg(p)\deg_{F_p}(D_p)$ pour tout $p\in \Irr_F$ (cf. \cite[proposition. 13.4(v)]{Pi82}), $L'$ se descend ainsi en un sous-groupe de Levi de $G$, à savoir celui standard associé à la partition
\[\left(\left(\frac{\deg(p)\deg_{F_p}(D_p)}{d}\right)\cdot \lambda(p)^t\right)_{p\in\Irr_F}.\]
Son algèbre de Lie contient $X_\ss$ en tant qu'élément $F$-elliptique, ce qu'il fallait.

Avançons sur la deuxième partie de l'énoncé. Soit $L=\prod_{i=1}^k \GL_{m_i,D}$ avec $(m_1,\dots,m_k)$ une partition de $m$, $\o=\prod_{i=1}^k \o_i\subseteq \mathfrak{l}(F)$, avec $\o_i$ une $\GL_{m_i}(D)$-orbite $F$-elliptique dans $\Mat_{m_i,D}$. Supposons que l'orbite $\o_i$ correspond à $\lambda_i:\Irr_F\rightarrow\{\text{partitions des entiers}\}$. L'ellipticité de $\o_i$ est équivalent à ce que $\lambda_i(p)$ est la partition $|\lambda_i(p)|\times (1)$ pour tout $p\in \Irr_F$ et il existe un unique $p$ tel que $|\lambda_i(p)|\not=0$. De ce fait $\Ind_L^G(\o)$ est la classe de $G(F)$-conjugaison correspondant à $\lambda:\Irr_F\rightarrow\{\text{partitions des entiers}\}$ avec
\[\lambda(p)=(|\lambda_1(p)|,\dots,|\lambda_k(p)|)^t\]
et on voit aussi que $\Ind_L^G(\o)$ détermine le couple $(L,\o)$ à une conjugaison près.\qedhere
\end{proof}
Le sous-groupe de Levi $L$ dans la proposition dépend bien sûr du corps de base $F$ car l'ellipticité est une notion arithmétique tandis que l'induction est une notion géométrique. Cela se voit de la démonstration.

\subsection{Lien avec le transfert}\label{subsec:appendiceAlienavecletransfert}
Posons $G^\ast=\GL_{md,F}$ la forme intérieure (quasi-)déployée de $G$, $\g^\ast=\Mat_{md,F}$ son algèbre de Lie, et $\eta: \g\rightarrow \g^\ast$ un torseur intérieur. On reprend dans la suite les notations de l'introduction de l'article, avec $F$ un corps parfait, non nécessairement un corps global ou local. \`{A} $X$ une classe de $G(F)$-conjugaison dans $\g(F)$ (ou par abus de notation $X$ peut être un élément dans la classe) on peut associer, via $\eta$, une classe de $G^\ast(F)$-conjugaison dans $\g^\ast$ notée $X^\ast$ (idem $X^\ast$ peut être un élément dans la classe). D'après les calculs antérieurs, si $X$ correspond à $\lambda:\Irr_F\rightarrow\{\text{partitions des entiers}\}$ alors $X^\ast$ correspond à $\lambda^\ast:\Irr_F\rightarrow\{\text{partitions des entiers}\}$ avec $\lambda^\ast(p)=\deg_{F_p}(D_p)\times \lambda(p)$ pour tout $p\in \Irr_F$.

\begin{proposition}\label{prop:indtransssibothtrans}
Soit $F$ un corps parfait. Soient $L'$ un sous-groupe de Levi de $G^\ast$, et $\o'$ une classe de $L'(F)$-conjugaison dans $\mathfrak{l}'(F)$. Alors $\Ind_{L'}^{G^\ast}(\o')$ se transfère à $G$ si et seulement si
\begin{enumerate}
    \item $L'=L^\ast$ se transfère à $G$ ;
    \item la classe de $L^\ast(F)$-conjugaison $\o'$ dans $\mathfrak{l}^\ast(F)$ se transfère à $L$.
\end{enumerate} 
\end{proposition}


\begin{proof}
Soient $L'=\prod_{i=1}^k \GL_{l_i',F}$ avec $(l_1',\dots,l_k')$ une partition de $nd$, $X'=\prod_{i=1}^k X_i'\in \mathfrak{l}'(F)$ avec $X_i'\in\Mat_{l_i'}(F)$, et $\o'$ la classe de $L'(F)$-conjugaison de $X'$. Supposons que la classe de $\GL_{l_i'}(F)$-conjugaison de $X_i'$ correspond à $\lambda_i':\Irr_F\rightarrow\{\text{partitions des entiers}\}$. On sait que l'induction  $\Ind_{L'}^{G^\ast}(\o')$ est la classe de $G^\ast(F)$-conjugaison correspondant à $\lambda':\Irr_F\rightarrow\{\text{partitions des entiers}\}$ avec
\[\lambda'(p)=(\lambda_1'(p)^t,\dots,\lambda_k'(p)^t)^t.\]
Selon l'hypothèse $\deg_{F_p}(D_p)\mid_{\times} \lambda'(p)$ pour tout $p\in \Irr_F$. Il vient alors $\deg_{F_p}(D_p)\mid_{\cdot}\lambda_i'(p)^t$ pour tout $i$. D'où $d\mid \deg (p)\deg_{F_p}(D_p)\mid \deg (p)|\lambda_i'(p)|$. Or $l_i'=\sum_{p\in\Irr_F}\deg (p)|\lambda_i'(p)|$, on obtient ainsi $d\mid l_i'$ pour tout $i$. De ce fait $L'$ est conjugué à $L^\ast$ un sous-groupe de Levi qui se transfère. Pour ne pas introduire encore des notations on suppose directement que $L'=L^\ast$. En tenant compte de ce qui précède $\deg_{F_p}(D_p)\mid_{\cdot}\lambda_i'(p)^t$, donc $\deg_{F_p}(D_p)\mid_{\times}\lambda_i'(p)$, autrement dit $\o'$ se transfère (à $L$).
\end{proof}

\begin{proposition}[Induite commute au transfert]\label{prop:induitecommuteautransfert} Soit $F$ un corps parfait. Soient $\o$ une classe de $G(F)$-conjugaison dans $\g(F)$, et $\o^\ast$ la classe de $G^\ast(F)$-conjugaison dans $\g^\ast(F)$ associée par $\eta$. Si $L$ est un sous-groupe de Levi de $G$ et $\o_1$ une classe de $L(F)$-conjugaison dans $\mathfrak{l}(F)$, tels que 
\[\o=\Ind_L^G(\o_1),\]
alors pour tous $L\arr L^\ast$ et $\o_1\arr\o_1^\ast$, 
\[\o^\ast=\Ind_{L^\ast}^{G^\ast}(\o_1^\ast).\]
Réciproquement, si $L\arr L^\ast$ et $\o_1\arr \o_1^\ast$ alors
\[\Ind_L^G(\o_1)\arr \Ind_{L^\ast}^{G^\ast}(\o_1^\ast).\]
\end{proposition}
\begin{proof}
Trivial puisque l'on peut vérifier l'énoncé sur $\overline{F}$.
\end{proof}

\begin{proposition}
Soit $F$ un corps parfait. Un élément $X'\in \g^\ast(F )$ se transfère si et seulement si $X_\ss'=X_\ss^\ast\in\g^\ast(F )$ se transfère et $X_\nilp'\in\g_{X_\ss^\ast}^\ast(F )$ se transfère (via la restriction de $\eta$ dont l'image est $\g_{X_\ss^\ast}^\ast$).    
\end{proposition}
\begin{proof}
Trivial puisque l'on peut vérifier l'énoncé sur $\overline{F}$.\qedhere
\end{proof}

\begin{proposition}\label{prop:elltransproof}
Soit $F$ un corps local de caractérisque 0. Toute classe de $G^\ast(F)$-conjugaison semi-simple elliptique dans $\g^\ast(F)
$ se transfère à $G$.    
\end{proposition}
\begin{proof}
Soit $\o'$ une classe de $G^\ast(F)$-conjugaison semi-simple elliptique dans $\g^\ast(F)
$. Supposons que $\o'$ correspond à $\lambda_i':\Irr_F\rightarrow\{\text{partitions des entiers}\}$. Selon l'hypothèse il existe un unique polynôme $p\in \Irr_F$ tel que $\deg(p)\mid md$ et  \[\lambda'(q)=\begin{cases*}
\frac{md}{\deg(p)}\times (1) & , si $q=p$\\  
\emptyset & , sinon.
\end{cases*}\]
On est alors conduit à prouver $\deg_{F_p}(D_p)\mid\frac{md}{\deg(p)}$. Comme $F$ est un corps local, la théorie des corps de classes locaux nous dit $\deg_{F_p}(D_p)=\frac{d}{\text{pgcd}(d,\deg(p))}$, on se ramène alors à prouver $\frac{\deg(p)}{\text{pgcd}(d,\deg(p))}\mid m$, qui s'ensuit bien de l'hypothèse $\deg(p)\mid md$, et la preuve s'achève.   
\end{proof}

\begin{corollary}
Soit $F$ un corps local de caractérisque 0. Soient $L'$ un sous-groupe de Levi de $G^\ast$, $\o'$ une classe de $L'(F)$-conjugaison, semi-simple et elliptique, dans $\mathfrak{l}'(F)$. Supposons que la classe de conjugaison $(\Ad G^\ast(F))\o'$ est $G^\ast$-régulier, alors elle se transfère à $G$ si et seulement si $L'$ se transfère à $G$.
\end{corollary}
\begin{proof}
L'hypothèse que $(\Ad G^\ast(F))\o'$ est $G^\ast$-régulier implique que $(\Ad G^\ast(F))\o'=\Ind_{L'}^{G^\ast}(\o')$, on sait donc que $(\Ad G^\ast(F))\o'$ se transfère à $G$ si et seulement si $L'=L^
\ast$ se transfère à $G$ et $\o'$ se transfère à $L$. On conclut par la proposition précédente.
\end{proof}

\begin{proposition}[Principe local-global pour le transfert]\label{prop:appendixprincipeloc-glotrans} Soit $F$ un corps de nombres. Un élément $X'\in \g^\ast(F)$ se transfère si et seulement si $X_v'\in \g^\ast(F_v)$ se transfère pour tout $v\in \V_F$. 
\end{proposition}

\begin{proof} Le sens direct étant trivial, on procède au sens inverse. Supposons que $X'$ correspond à $\lambda':\Irr_F\rightarrow\{\text{partitions des entiers}\}$. Soit $p\in\Irr_F$. On veut montrer que $\deg_{F_p}(D_p)\mid_{\times}\lambda'(p)$.
Soit $v$ une place de $F$. \'{E}crivons $p=p_{v,1}\cdots p_{v,s_v}$ la décomposition de $p$ en produit de polynômes irréductibles dans $F_v[T]$. Posons $F_p=F[T]/(p)$, $F_{v,i}=F_v[T]/(p_{v,i})$, et $D_{v,i}$ une $F_{v,i}$-algèbre à division telle que $D\otimes_FF_{v,i}\simeq \Mat_{c(D\otimes_FF_{v,i})}(D_{v,i})$. Puisque 
\[\prod_{i=1}^s F_{v,i}=F_v[T]/(p)=F_p\otimes_FF_v=\prod_{w\mid v}F_{p,w},\]
par l'unicité de décomposition d'un anneau artinien en un produit d'anneaux artiniens locaux on sait que $\{F_{v,i}\mid i=1,\dots,s_v\}=\{F_{p,w}\mid w|v\}$. L'hypothèse que $X_v'$ est se transfère entraîne que 
\[\deg_{F_{v,i}}(D_{v,i})\mid_\times \lambda'(p),\,\,\,\,\forall i=1,\dots,v_s.\]
Par la théorie des corps de classes on a
\[\deg_{F_p}(D_p)=\text{ppcm}_{v\in \V_F}(\text{ppcm}_{w\mid v} \deg_{F_{p,w}}(D_{p,w})).\]
Il s'ensuit de ce fait
\[\deg_{F_p}(D_p)\mid_{\times}\lambda'(p).\]
La preuve est achevée.
\end{proof}

\begin{example}[Un élément qui ne se transfère pas globalement mais qui se transfère partout non-archimédien-localement]
Soit $F=\Q$. Il existe une algèbre simple centrale $A=\Mat_m(D)$ telle que 
\[\text{inv}_v(A)=\begin{cases*}
 \frac{1}{2} &, \text{ si $v=3$ ou $\R$}\\
 0 &, \text{ sinon.}
\end{cases*}\]
Prenons $G=\GL_{m,D}$. Soit $\lambda':\Irr_{\Q}\rightarrow\{\text{partitions des entiers}\}$ donné par 
\[\lambda'(p)=\begin{cases*}
 (md) &, \text{ si $p=T^2-2$}\\
 \emptyset &, \text{ sinon,}
\end{cases*}\]
et $X'$ la classe de conjugaison de $G^\ast=\GL_{md,\Q}$ correspondante. On prétend que $X_v'\in \g^\ast(\Q_v)$ se transfère pour tout $v\in\V_{\fin}$ mais ne se transfère pas pour $v=\R$. L'assertion est triviale lorsque $v\in\V_{\fin}\setminus\{3\}$. Ensuite $3$ est une place totalement ramifiée dans $\Q[\sqrt{2}]$, on note $\bar{3}$ l'unique place de $\Q[\sqrt{2}]$ divisant $3$, par la théorie des corps de classes on a $\text{inv}_{\bar{3}}(A\otimes_{\Q}\Q[\sqrt{2}])=2\text{inv}_{3}(A)=0\in \Q/\Z$, donc $X_3'\in \g^\ast(\Q_3)$ se transfère. Enfin il existe deux places, $\R_1$ et $\R_2$, de $\Q[\sqrt{2}]$ divisant la place réelle de $\Q$, des calculs nous gratifient $\text{inv}_{\R_1}(A\otimes_{\Q}\Q[\sqrt{2}])=\text{inv}_{\R_2}(A\otimes_{\Q}\Q[\sqrt{2}])=\frac{1}{2}$, par suite $\deg_{\R_1}(D_{T^2-2})=\deg_{\R_2}(D_{T^2-2})=2
\nmid_{\times} (md)$, autrement dit $X_{\R}'\in \g^\ast(\R)$ ne se transfère pas.
\end{example}



\section{Calculs explicites des mesures de Haar}\label{sec:AppendixB}

Les choix des mesures de Haar peuvent s'avérer complexes, notamment en raison des difficultés potentielles à vérifier les conditions de compatibilité associées à ces choix. Nous fournissons ici des calculs explicites des mesures pour des applications ultérieures.

Dans la suite $F$ sera un corps local de caractéristique 0. On fixe $D$ une algèbre à division centrale sur $F$, de degré $d$, i.e. la dimension de $D$ en tant que $F$-espace vectoriel vaut $d^2$. Quand $F=\R$ et $D\not= F$, on supposera que $D=\mathbb{H}$, l'algèbre des quaternions de Hamilton. En tant que $\R$-algèbre, $\mathbb{H}=\R \oplus\R i\oplus\R j\oplus\R k$, avec $i^2=j^2=k^2=-1$, $ij=-ji=k$, $jk=-kj=i$, et $ki=-ik=j$. Quand $F$ est non-archimédien, on écrit $\O_D$ l'anneau des entiers de $D$, et on fixe $\varpi_D\in \O_D$ un uniformisant. On note $q$ le cardinal du corps résiduel de $F$ et $p$ sa caractéristique. On sait que $\O_D/\varpi_D\O_D$ est un corps fini à $q^d$ éléments (\cite[(14.3) théorème]{Rei03}). Quand $F$ est archimédien, on pose, par abus de notation, $\O_D$ l'anneau suivant :
\[\O_D=\begin{cases*}
  \Z & \text{si $D=\R$ ;}\\
  \Z\oplus \Z i & \text{si $D=\C$ ;}\\
  \Z\oplus\Z i\oplus\Z j\oplus\Z k &\text{si $D=\mathbb{H}$ (donc $F=\R$).}
\end{cases*}\]

Soit $m\in \mathbb{N}_{>0}$. Soit $G=G_m$ l'unique schéma en groupes affines défini sur $\O_F$ tel que $G(A)=\GL_
{m}(A\otimes_{\O_F} \O_D)$ pour toute $\O_F$-algèbre $A$. 
Un élément général de $G$ est une matrice inversible de taille $m$. Pour un groupe algébrique on note par la même lettre en minuscule gothique son algèbre de Lie. On écrit donc $\g=\g_m$ l'algèbre de Lie de $G$. Ainsi $\g(A)=\gl_{m}(A\otimes_{\O_F} \O_D)$ pour toute $\O_F$-algèbre $A$. Un élément général $X$ de $\g$ est une matrice de taille $m$, on l'écrit comme $X=(X_{\alpha,\beta})_{1\leq \alpha,\beta\leq m}$. 

Si $F$ est archimédien, on note $X\mapsto\overline{X}$ l'involution principale sur $\g_1(F)$. Plus précisément, si $F=\R$ et $D=F$ alors $\overline{a}=a$ pour tout $a\in \R$ ; si $F=\C$ et $D=F$ alors $\overline{a+bi}=a-bi$ pour tous $a,b\in \R$ ; si $F=\R$
 et $D=\mathbb{H}$ alors $\overline{a+bi+cj+dk}=a-bi-cj-dk$ pour tous $a,b,c,d\in \R$. On note $X\mapsto \overline{X}$ l'involution principale et $X\mapsto {}^tX$ la transposée  sur $\g(F)$. Donc si $X=(X_{\alpha,\beta})_{1\leq \alpha,\beta\leq m}\in \g(F)$ alors $\overline{X}=(\overline{X_{\alpha,\beta}})_{1\leq \alpha,\beta\leq m}$ et ${}^tX=(X_{\beta,\alpha})_{1\leq \alpha,\beta\leq m}$.

On fixe $K=K_m$ le sous-groupe de $G(F)$ suivant. Primo, si $F$ est non-archimédien, alors $K\eqdef\GL_m(\O_D)$. C'est un sous-groupe ouvert compact maximal de $G(F)$. Secundo, si $F$ est archimédien, alors on pose $K\eqdef\{g\in \GL_m(D)\mid g\cdot{}^t\overline{g}=\text{Id}\}$. C'est un sous-groupe compact maximal de $G(F)$.

On note $\GL_{1,F}$ le groupe multiplicatif sur $F$ et $\gl_{1,F}$ le groupe additif sur $F$. Notons $\nu_m:\g\to \gl_{1,F}$ la norme réduite et $\tau_{m}:\g\to \gl_{1,F}$ la trace réduite sur $\g=\g_m$. Soit $\langle X,Y\rangle =\tau_{m}(XY)$, c'est une forme bilinéaire sur $\g(F)$ non-dégénérée et invariante par $G(F)$-adjonction. On appelle $\langle -,-\rangle$ la forme bilinéaire canonique sur $\g(F)$. Pour $E/F$ une extension de corps, on note $\tau_{E/F}:E\to F$ la fonction trace. On définit un caractère local additif $\psi$ de $F$ comme suit :
\begin{enumerate}
    \item si $F$ est archimédien, posons 
    \[\psi(x)=e^{-2i\pi \tau_{F/\R}(x)} ;\]
    \item si $F$ est non-archimédien, posons 
    \[\psi(x)=\psi_{\Q_p}(\tau_{F/\Q_p}(x))\] 
    avec $\psi_{\Q_p}$ le caractère local additif de $\Q_p$ défini par 
    \[\psi_{\Q_p}(x)=e^{2i\pi (x_{-r}p^{-r}+\cdots+x_{-1}p^{-1})}\] 
    où $x=x_{-r}p^{-r}+\cdots+x_{-1}p^{-1}+x'$,
    avec $x_{i}\in \{1,2,\dots,p-1\}$ pour $i\in\{-r,\dots,-1\}$ et $x'\in\Z_p$.
\end{enumerate}

Quand $F$ est non-archimédien, on note $\Delta_{m}$ le discriminant de $\g(\O_F)$ par rapport à $\Z_p$ (\cite[p.126]{Rei03}), c'est un idéal de $\Z_p$. En manipulant \cite[(10.2) théorème]{Rei03}, on voit facilement que $\Delta_m=\Delta_1^{m^2}$. 
On écrit $\mathcal{N}(\Delta_{m})\in \mathbb{N}_{>0}$ pour la norme de $\Delta_{m}$, autrement dit si $\Delta_{m}=(p)^a\Z_p$ alors  $\mathcal{N}(\Delta_{m})=p^a$. On a $\mathcal{N}(\Delta_{m})=\mathcal{N}(\Delta_{1})^{m^2}$.

On fixe désormais 
\begin{equation}\label{appBeq:LeviMdef}
M\eqdef G_{m_1}\times \cdots\times G_{m_l}    
\end{equation}
avec $m_1+\cdots+m_l=m$. C'est un sous-groupe de Levi de $G=G_m$. On a $M(F)\cap K= K_{m_1}\times\cdots\times K_{m_l}$.

Pour $H$ un groupe séparé localement compact muni d'une mesure de Haar à droite $d\mu$ et $S\subseteq H$ une partie mesurable, on note $\vol(S;H)$ le volume de $S$ dans $H$ et $\vol(H)$ le volume total de $H$. Soit $h\in H$, on note $d\mu(h^{-1}\cdot)$ la mesure de Haar à droite qui à une partie mesurable $S$ associe $d\mu(h^{-1}S)$. On note $\text{mod}_H(h)>0$ la dérivée de Radon-Nikodym de $d\mu(h^{-1}\cdot)$ par rapport à $d\mu$. Cela définit un morphisme de groupes $\text{mod}_H:H\to \R_{>0}^\times$. On dit que $H$ est unimodulaire si $\text{mod}_H(H)=\{1\}$. Les groupes suivants sont unimodulaires :
\begin{enumerate}
    \item tout groupe des $F$-points d'un groupe réductif ;
    \item tout groupe des $F$-points d'un groupe algébrique n'ayant pas de caractère non-trivial ;
    \item tout groupe des $F$-points du centralisateur d'un élément de l'algèbre de Lie dans un groupe réductif (\cite[théorème 1]{Rao72}) ;
    \item tout groupe abélien ;
    \item tout groupe compact ;
    \item tout groupe discret.
\end{enumerate}
Soit $H'$ un sous-groupe fermé unimodulaire de $H$ muni d'une mesure de Haar à droite $d\mu'$. On peut former la mesure quotient $d\mu'\backslash d
\mu$ sur l'espace quotient $H'\backslash H$ (\cite[chapitre 7, section 6, corollaire 1]{Bou04}), c'est une mesure de Radon relativement invariante à droite pour le caractère $\text{mod}_H$. 

\subsection{Mesures sur \texorpdfstring{$a_M$}{aM} et \texorpdfstring{$a_M^\ast$}{aM*}}\label{appBsubsec:mesureona0}

Pour $H$ un groupe algébrique connexe défini sur $F$, on note $A_H$ le sous-tore central $F$-déployé maximal de $H$, on note $X^\ast(H)$ le groupe des caractères de $H$ définis sur $F$ et $a_H$ (resp. $a_H^\ast$) l’espace vectoriel réel $\Hom(X^\ast(H),\R)$  (resp. $X^\ast(H)\otimes_\Z\R$). Les espaces $a_H$ et $a_H^\ast$ sont canoniquement duaux l'un de l'autre. 

Soit $M_0=G_1^m$ le sous-groupe diagonal de $G=G_m
$. Alors le groupe additif $X^\ast(M_0)$ est isomorphe à $\Z^m$ via
\[(t_1,\dots,t_m)\in \Z^m\mapsto\left(\nu_{(t_1,\dots,t_m)}:(x_{i})_{1\leq i\leq m}\in M_0\mapsto\prod_{i=1}^m\nu_1(x_i)^{t_i}\in \GL_{1,F}\right)\in X^\ast(M_0).\]
On identifie ainsi $a_{M_0}^\ast$ à $\R^m$ via l'isomorphisme naturel $\Z^m\otimes_\Z\R =\R^m$. On identifie ensuite $a_{M_0}$ à $\R^m$ de sorte que l'accouplement canonique entre $a_0$ et $a_0^\ast$ est donné par le produit scalaire canonique $\langle-,-\rangle_m$ sur $\R^m$ : $\langle (y_i)_{1\leq i\leq m},(z_i)_{1\leq i\leq m}\rangle_m=\sum_{1\leq i\leq m} y_iz_i$ pour tous $(y_i)_{1\leq i\leq m}, (z_i)_{1\leq i\leq m}\in \R^m$.

Prenons $M$ le sous-groupe de Levi \eqref{appBeq:LeviMdef}. On écrit $x=(x_{i})_{1\leq i\leq l}$ pour un élément général de $M$, avec $x_i\in G_{m_i}$ pour tout $i$. Alors le groupe additif $X^\ast(M)$ est isomorphe à $\Z^l$ via
\[(t_1,\dots,t_l)\in \Z^l\mapsto\left(\nu_{(t_1,\dots,t_l)}:(x_{i})_{1\leq i\leq l}\in M\mapsto\prod_{i=1}^l\nu_{m_i}(x_i)^{t_i}\in \GL_{1,F}\right)\in X^\ast(M).\]
On identifie ainsi $a_{M}^\ast$ à $\R^l$ via l'isomorphisme naturel $\Z^l\otimes_\Z\R =\R^l$. On identifie ensuite $a_M$ à $\R^l$, de sorte que l'accouplement canonique entre $a_M$ et $a_M^\ast$ est donné par le produit scalaire $\langle -,-\rangle_l$ canonique sur $\R^l$ : $\langle (y_i)_{1\leq i\leq l},(z_i)_{1\leq i\leq l}\rangle_l=\sum_{1\leq i\leq l} y_iz_i$ pour tous $(y_i)_{1\leq i\leq l}, (z_i)_{1\leq i\leq l}\in \R^l$. La restriction des caractères donne une injection canonique $\iota_M^\ast:a_M^\ast\hookrightarrow a_0^\ast$, cette application est
\[\iota_M^\ast:(t_1,\dots,t_l)\in \R^{l}\mapsto (t_1,\dots,t_1,\dots,t_l,\dots,t_l)\in \R^m\]
où chaque $t_i$ est répété $m_i$ fois dans $\R^m$. Dualement on a une surjection canonique $\iota_M: a_0\twoheadrightarrow a_M$, elle s'écrit
\[\iota_M:(t_1,\dots,t_m)\in \R^{m}\mapsto (t_1+\cdots+ t_{m_1},\dots,t_{m_1+\cdots+m_{l-1}+1}+\cdots+t_m)\in \R^l\]
où la $i$-ème coordonnée de l'image dans $\R^l$ vaut $t_{m_1+\cdots+m_{i-1}+1}+\cdots+t_{m_1+\cdots+m_{i}}$. Soient $t\in a_M$ et $t'\in a_{M}^\ast$, on a $\langle t,t'\rangle_k=\langle t'',\iota_M^\ast(t')\rangle_m$ pour tout $t''\in \iota_M^{-1}(t)$. 

La restriction des caractères donne un isomorphisme canonique $X^\ast(M)\otimes_\Z \R\xrightarrow{\sim} X^\ast(A_M)\otimes_\Z \R$, aussi une surjection canonique $X^\ast(A_{M_0})\otimes_\Z \R\twoheadrightarrow X^\ast(A_{M})\otimes_\Z \R$. On a $A_M\simeq G_{1}\times \cdots\times G_{1}$, où chaque $G_1$ s'injecte de façon diagonale dans $G_{m_i}$, $1\leq i\leq l$. On obtient de ce fait une surjection canonique $\phi_M^\ast:a_0^\ast \twoheadrightarrow  a_{M}^\ast$. Plus précisément, cette application s'écrit
\[\phi_M^\ast:(t_1,\dots,t_m)\in \R^{m}\mapsto \left(\frac{t_1+\cdots+ t_{m_1}}{m_1},\dots,\frac{t_{m_1+\cdots+m_{l-1}+1}+\cdots+t_m}{m_l}\right)\in \R^l\]
où la $i$-ème coordonnée de l'image dans $\R^l$ vaut $\frac{t_{m_1+\cdots+m_{i-1}+1}+\cdots+t_{m_1+\cdots+m_{i}}}{m_i}$. Dualement on a une injection canonique $\phi_M: a_M\hookrightarrow a_0$, elle s'écrit
\[\phi_M:(t_1,\dots,t_l)\in \R^{l}\mapsto \left(\frac{t_1}{m_1},\dots,\frac{t_1}{m_1},\dots,\frac{t_l}{m_l},\dots,\frac{t_l}{m_l}\right)\in \R^m\]
où chaque $\frac{t_i}{m_i}$ est répété $m_i$ fois dans $\R^m$. 

La restriction du produit scalaire canonique $\langle-,-\rangle_m$ sur le sous-espace $\phi_M(a_M)$ reste un produit scalaire. On munit $a_M=\R^l$ de la mesure de Haar associée à ce produit scalaire. On a 
\[\vol([0,1]^l;a_M)=\prod_{i=1}^l m_i^{-1/2}.\]
On munit $a_M^\ast=\R^l$ de l'unique mesure de Haar telle que
\[\vol([0,1]^l;a_M^\ast)=(2\pi)^{-l}\prod_{i=1}^l m_i^{1/2}.\]
Ces mesures sont duales l'une de l'autre dans le sens où
\[\int_{a_M^\ast}\int_{a_M}f(h)e^{-\sqrt{-1}\langle h,\lambda\rangle}\,dh\,d\lambda=f(0)\]
pour tout $f\in C_c^\infty(a_M)$. 

On définit le groupe de Weyl relatif de $(G,M)$ par
\[W_M^G=\text{Norm}_{G(F)}(M)/M(F),\]
avec $\text{Norm}_{G(F)}(M)$ le normalisateur de $M$ dans $G(F)$. Alors $W_M^G$ agit naturellement sur $a_M$ et $a_M^\ast$.  Les mesures sur $a_M$ et $a_M^\ast$ sont invariantes par l'action du groupe $W_M^G$. 

Soit $H$ un sous-groupe de $G$ stable par conjugaison par $A_M$. On note $\Sigma(\mathfrak h; A_M)\subseteq a_M^\ast$ l’ensemble des racines de $A_M$ sur $\mathfrak{h}$. On a alors
\[\mathfrak{h}=\mathfrak{h}^{A_M}\oplus\bigoplus
_{\alpha\in\Sigma(H;A_M)}\mathfrak{h}_\gamma\]
avec $\mathfrak{h}^{A_M}$ le sous-espace invariant et $\mathfrak{h}_\gamma$ le sous-espace propre relativement à $\gamma$ de $\mathfrak{h}$ sous l'action adjointe par $A_M$ :
\[\mathfrak{h}_\gamma=\{X\in\mathfrak{h}\mid \Ad(a)X=\gamma(a)X\,\,\,\,\forall a\in A_M\}.\]

Supposons que $P$ est le groupe des matrices  inversibles triangulaires supérieures par blocs de tailles $m_1,\dots,m_{l}$. Alors $\Sigma(\mathfrak p; A_M)\subseteq a_M^\ast=\R^l$ est l'ensemble
\begin{align*}
\Sigma(\mathfrak p; A_M)=\{v=&(v_1,\dots,v_l)\in \Z^l\mid \text{il existe $i<j$}\\
&\text{tels que $v_i=1$, $v_j=-1$ et $v_k=0$ pour tout autre $1\leq k\leq l$}\}.    
\end{align*}

\`{A} chaque racine $\gamma\in \Sigma(\g;A_{M_0})$ est naturellement associée une coracine $\gamma^\vee\in a_{M_0}$ d'après la théorie classique. Plus précisément $\gamma^\vee=\frac{2\gamma}{\langle \gamma,\gamma\rangle_m}$ dans $\R^m$. Ainsi, \`{a} chaque racine $\gamma\in \Sigma(\g;A_{M})$ on peut lui associer une coracine $\gamma^\vee\in a_{M}$ via $\gamma^\vee=\iota_M(\iota_M^\ast(\gamma)^\vee)$. Plus précisément 
\[-^\vee :(t_1,\dots,t_l)\in\R^l\mapsto \left(\frac{2m_1t_1}{\sum_{j=1}^km_jt_j^2},\dots, \frac{2m_lt_l}{\sum_{j=1}^lm_jt_j^2}\right)\in \R^l,\]
où la $i$-ème coordonnée de l'image dans $\R^l$ vaut $\frac{2m_it_i}{\sum_{j=1}^lm_jt_j^2}$.

\subsection{Mesures sur \texorpdfstring{$\m(F)$}{m(F)} et \texorpdfstring{$\n_P(F)$}{nP(F)}}\label{appBsubsec:mesmnPLie}

On désigne par $\S(\g(F))$ l'espace de Schwartz-Bruhat de $\g(F)$. On définit la transformée de Fourier d'une fonction $f \in\S(\g(F))$ par
\begin{equation}\label{appBeq:Fourierdef}
\widehat{f}(Y)=\int_{\g(F)}f(X)\psi(\langle X,Y\rangle)\,dX,\,\,\,\,\forall Y\in \g(F),    
\end{equation} 
où $\g(F)$ est munit d'une mesure de Haar. Il existe une unique mesure de Haar $dX$ sur $\g(F)$, dite auto-duale relativement à la transformée de Fourier, telle que $\widehat{\widehat{f}}(Y)=f(-Y)$ pour tous $f\in\S(\g(F))$ et $Y\in \g(F)$.

Soient $F$ archimédien et $V$ un $\O_F$-module libre de rang fini. On appelle la mesure de Lebesgue sur $V(F)$ l'unique mesure de Haar sur $V(F)$ telle que $\vol(V(\O_F)\backslash V(F))=1$, avec $V(\O_F)$ muni de la mesure de comptage et sur un espace quotient on prend la mesure quotient.

\begin{proposition}\label{appBprop:autodualmeasureg}
La mesure de Haar $dX$ auto-duale relativement à la transformée de Fourier sur $\g(F)=\g_m(F)$ est décrite comme suit :
\begin{enumerate}
    \item si $F=\R$, alors $dX$ est la mesure de Lebesgue ; 
    \item si $F=\C$, alors $dX$ est $2^{m^2}$ fois la mesure de Lebesgue ;
    \item si $F$ est non-archimédien, alors $dX$ est l'unique mesure de Haar telle que $\vol(\g(\O_F);\g(F))=\mathcal{N}(\Delta_1)^{-\frac{1}{2}m^2}$.
\end{enumerate}
\end{proposition}
\begin{proof}On munit $\g(F)$ de la mesure $dX$ décrite dans l'énoncé.
\begin{enumerate}
    \item Si $F=\R$ et $D=\R$ on prend $f(X)=e^{-\pi\tau_m(X{}^t\overline{X})}$. Si $F=\R$ et $D=\mathbb{H}$ on prend $f(X)=e^{-\frac{1}{2}\pi\tau_m(X{}^t\overline{X})}$.     On a $\widehat{f}=f$ donc $\widehat{\widehat{f}}=f$.
    \item Si $F=\C$, on prend $f(X)=e^{-2\pi\tau_m(X{}^t\overline{X})}$. On a $\widehat{f}=f$ donc $\widehat{\widehat{f}}=f$.
    \item Si $F$ est non-archimédien, on prend $f=1_{\g(\O_F)}$, la fonction indicatrice de $\g(\O_F)$. On note  $\mathfrak{D}^{-1}$ la codifférente de $\g(\O_F)$ par rapport à $\Z_p$ (\cite[page 60]{Ser80}), c'est un idéal fractionnaire de $\Z_p$, et sa norme vaut $\mathcal{N}(\Delta_m)^{-1}=\mathcal{N}(\Delta_1)^{-m^2}$. 
    Alors $\widehat{f}=\mathcal{N}(\Delta_1)^{-\frac{1}{2}m^2} 1_{\mathfrak{D}^{-1}}$ avec $1_{\mathfrak{D}^{-1}}$ la fonction indicatrice de $\mathfrak{D}^{-1}$, puis $\widehat{\widehat{f}}=f$.\qedhere
\end{enumerate}    
\end{proof}

On munit dans la suite $\g(F)$ de la mesure auto-duale relativement à la transformée de Fourier.

Prenons $M$ le sous-groupe de Levi \eqref{appBeq:LeviMdef}. On remarque que la restriction sur $\mathfrak{m}(F)$ de la forme bilinéaire canonique sur $\g(F)$ est exactement la somme des formes bilinéaires canoniques sur $\g_{m_i}(F)$ où $i$ parcourt $\{1,\dots,l\}$. On définit la transformée de Fourier sur $\S(\m(F))$ en remplaçant partout dans l'équation \eqref{appBeq:Fourierdef} l'espace $\g(F)$ par $\m(F)$. On munit $\m(F)$ de la mesure auto-duale relativement à la transformée de Fourier. Elle est aussi la mesure produit de celles sur $\g_{m_i}(F)$ où $i$ parcourt $\{1,\dots,l\}$.

Soit $P$ un sous-groupe parabolique de $G$ ayant $M$ comme facteur de Levi. On note $N_P$ le radical unipotent de $P$. On note $\overline{P}$ l'opposé de $P$, c'est l'unique sous-groupe parabolique de $G$ ayant $M$ comme facteur de Levi tel que $P\cap \overline{P}=M$. On note $N_{\overline{P}}$ le radical unipotent de $\overline{P}$.

On va fixer une mesure de Haar $dN$ sur $\n_P(F)$. Supposons d'abord que $F$ est archimédien. Si $F=\R$, on munit $\n_P(F)$ de la mesure de Lebesgue ; si $F=\C$, on munit $\n_P(F)$ de $2^{\frac{1}{2}(m^2-\sum_{i=1}^lm_i^2)}$ fois la mesure de Lebesgue. Supposons ensuite que $F$ est non-archimédien, on munit $\n_P(F)$ de l'unique mesure de Haar telle que $\vol(\n_P(\O_F);\n_P(F))=\mathcal{N}(\Delta_1)^{-\frac{1}{4}(m^2-\sum_{i=1}^lm_i^2)}$. 

On fixe aussi une mesure de Haar $d\overline{N}$ sur $\n_{\overline{P}}(F)$ suivant le même procédé. Alors
\begin{equation}\label{appBeq:G=NMNLie}
\int_{\g(F)}f(X)\,dX=\int_{\n_P(F)}\int_{\m(F)}\int_{\n_{\overline{P}}(F)}f(N+Y+\overline{N})\,d\overline{N}\,dY\,dN    
\end{equation}
pour tout $f\in L^1(\g(F))$.

\subsection{Mesure sur \texorpdfstring{$K$}{K}}

On prend l'unique mesure de Haar sur $K$ tel que $\vol(K)=1$.

\subsection{Mesures sur \texorpdfstring{$M(F)$}{M(F)} et \texorpdfstring{$N_P(F)$}{NP(F)}}\label{appBsubsec:mesuresonMNP}
Pour $H$ un groupe séparé localement compacte muni d'une mesure de Haar à droite $d\mu$, et $\phi$ un endomorphisme du groupe $H$, on note $\text{mod}_H(\phi)$ la quantité suivante : si $\phi$ n'est pas un automorphisme alors $\text{mod}_H(\phi)=0$ ; si $\phi$ est un automorphisme alors $\text{mod}_H(\phi)>0$ est la dérivée de Radon-Nikodym de $\phi^{-1}(d\mu)$ par rapport à $d\mu$, i.e.
\[\int_H f(\phi^{-1}(h))d\mu(h)=\text{mod}_H(\phi)\int_H f(h)d\mu(h)\]
pour tout $f\in L^1(H)$.

Soit $X\in \g(F)$. On note $X\cdot$ l'endomorphisme $Y\in \g(F)\mapsto XY\in \g(F)$. On note $|\cdot|$ la valeur absolue normalisée sur $F$, i.e. $|Z|=\text{mod}_{\gl_{1}(F)}(Z\cdot)$ pour tout $Z\in F$.

Au moyen des calculs directs on voit que $\text{mod}_{\g(F)}(x^{-1}\cdot)=|\nu_m(x)|^{-md}$ pour tout $x\in G(F)$. On prend de ce fait sur $G(F)=G_m(F)$ la mesure de Haar 
\begin{equation}\label{appBeq:mesureonG}
dx=\frac{dX}{|\nu_m(x)|^{md}},    
\end{equation}
où $dX$ est la mesure de Haar auto-duale relativement à la transformée de Fourier sur $\g(F)$. Cette mesure est également celle induite via l’exponentielle de la mesure sur un voisinage de $0$ de $\mathfrak{g}(F)$, i.e. 
\[\int_{V} f(\exp X)|\det(d\exp)_X|\,dX=\int_{\exp V}f(x)\,dx\]
pour tout $V$ voisinage de $0$ de $\mathfrak{g}(F)$ tel que $\exp :V \to \exp(V)$ soit bijectif et tout $f\in L^1(\exp V)$.

\begin{proposition}\label{appBprop:volKinG}
Lorsque $F$ est non-archimédien, on a 
\[\vol(K;G(F))=\mathcal{N}(\Delta_1)^{-\frac{1}{2}m^2}\prod_{i=1}^{m}\left(1-q^{-id}\right).\]
\end{proposition}
\begin{proof}
Soit $k\in\mathbb{N}_{>0}$ tel que $\exp:(\varpi_D^k\g(\O_F),+)\to (\text{Id}+\varpi_D^k\g(\O_F),\times)$ soit un isomorphisme de groupes topologiques (d'inverse $\log:\text{Id}+X\mapsto X-\frac{X^2}{2}+-\cdots$). 
On a
\begin{align*}
\vol(K;G(F))&=\vol(\text{Id}+\varpi_D^k\g(\O_F);G(F))[K:\text{Id}+\varpi_D^k\g(\O_F)]  \\
&=\vol(\varpi_D^k\g(\O_F);\g(F))| \GL_m(\O_D/\omega_D^k\O_D))|\\
&=q^{-dm^2k}\vol(\g(\O_F);\g(F))\prod_{i=0}^{m-1}\left(q^{dmk}-q^{(i-m+mk)d}\right)\\
&=\mathcal{N}(\Delta_1)^{-\frac{1}{2}m^2}\prod_{i=0}^{m-1}\left(1-q^{(i-m)d}\right).\\
&=\mathcal{N}(\Delta_1)^{-\frac{1}{2}m^2}\prod_{i=1}^{m}\left(1-q^{-id}\right).\qedhere
\end{align*}
\end{proof}

Prenons $M$ le sous-groupe de Levi \eqref{appBeq:LeviMdef}. On munit $M(F)$ de la mesure produit de celles sur $G_{m_i}(F)$ où $i$ parcourt $\{1,\dots,l\}$, avec chaque $G_{m_i}(F)$ muni de la mesure auto-duale relativement à la transformée de Fourier sur $\g_{m_i}(F)$. On observe que cette mesure sur $M(F)$ est aussi celle induite via l’exponentielle de la mesure sur un voisinage de $0$ de $\m(F)$.

Soit $P$ un sous-groupe parabolique de $G$ ayant $M$ comme facteur de Levi. On note $N_P$ le radical unipotent de $P$. On note $\overline{P}$ l'opposé de $P$, c'est l'unique sous-groupe parabolique de $G$ ayant $M$ comme facteur de Levi tel que $P\cap \overline{P}=M$. On note $N_{\overline{P}}$ le radical unipotent de $\overline{P}$.

On définit l'application $H_M : M(F)\rightarrow a_{M}$ par $e^{\langle H_M(y),\chi\rangle}=|\chi(y)|$ pour tout $\chi\in a_M^\ast$ et $y\in M(F)$. Puis $H_P : G(F)\rightarrow a_{M}$ par $H_P(ynk)\eqdef H_{M}(y)$, avec $y\in M(F)$, $n\in N_P(F)$ et $k\in K$.

Le groupe $A_M$ agit par conjugaison sur $P$. On note $\rho_P\eqdef \frac{1}{2}\sum_{\gamma\in\Sigma (\n_P; A_{M})}(\dim \n_{P,\gamma})\gamma$ avec $\n_{P,\gamma}$ le sous-espace propre relativement à $\gamma$. 

On munit $\n_P(F)$ et $\n_{\overline{P}}(F)$ des mesures expliquées dans la sous-section \ref{appBsubsec:mesmnPLie}. Il existe plusieurs façons naturelles d'imposer la compatibilité des mesures de Haar sur $\mathfrak{n}_P(F)$ et $N_P(F)$, que l'on note respectivement par $dN$ et $dn$ :
\begin{enumerate}
    \item $dN$ et $dn$ sont compatibles via l'isomorphisme de variétés $\exp :\mathfrak{n}_P(F)\to N_P(F)$, i.e.
    \[\int_{\mathfrak{n}_P(F)} f(\exp N)|\det(d\exp)_N|\,dN=\int_{N_P(F)}f(n)\,dn\]
    pour tout $f\in L^1(N_P(F))$ ; 
    \item $dN$ et $dn$ sont compatibles via l'isomorphisme de variétés $\text{Id}+\cdot :N\in\mathfrak{n}_P(F)\mapsto \text{Id}+N\in N_P(F)$, i.e.
    \[\int_{\mathfrak{n}_P(F)} f(\text{Id}+N)|\det(d\,\text{Id}+\cdot)_N|\,dN=\int_{N_P(F)}f(n)\,dn\]
    pour tout $f\in L^1(N(F))$ ; 
    \item soit $N_P = N_0 \supseteq N_1 \supseteq\cdots\supseteq N_r = \{1\}$ une filtration par des sous-groupes normaux telle que $N_k/N_{k+1}$ est abélien et $[N_0,N_k] \subseteq N_{k+1}$ pour tout $k$ (une telle filtration existe toujours : supposons que $N_{k+1}$ est défini, on pose $N_k$ le sous-groupe unipotent de l'algèbre de Lie l'espace associé au poids maximal de $A_M$ dans $\n_P /\n_{k+1}$, notons que cette filtration est stable par l'action par $A_M$), le groupe $N_k/N_{k+1}$ s'identifie canoniquement avec $\n_k/\n_{k+1}$. On exige que les mesures sur $N_P(F)$ et $\n_P(F)$ sont normalisées de sorte qu'il existe des mesures de Haar sur les groupes $N_{k}(F)$ (resp. $\n_k(F)$) telles que la mesure de $N_k(F)/N_{k+1}(F)$ coïncide avec celle de $\n_k(F)/\n_{k+1}(F)$. 
\end{enumerate}
Par des calculs directs on vérifie aisément que la compatibilité (3) ne dépend ni de la filtration ni des mesures intermédiaires choisies, et que les trois conditions (1), (2), et (3) sont équivalentes. On munit $N_P(F)$ (resp. $N_{\overline{P}}(F)$) de la mesure compatible à celle sur $\n_P(F)$ (resp. $\n_{\overline{P}}(F)$) au sens ci-dessus expliqué. On observe que $\text{Id}+\n_P(\O_F)=N_P(\O_F)$, donc si $F$ est archimédien on a $\vol(N_P(\O_F)\backslash N_P(F))=\vol(\n_P(\O_F)\backslash\n_P(F))$ avec $N_P(\O_F)$ et $\n_P(\O_F)$ tous munis des mesures de comptage et sur un espace quotient on prend la mesure quotient, puis si $F$ est non-archimédien on a  $\vol(N_P(\O_F);N_P(F))=\vol(\n_P(\O_F);\n_P(F))$. De l'équation \eqref{appBeq:G=NMNLie} on déduit que 
\begin{equation}\label{appBeq:G=NMNGrp}
\int_{G(F)}f(x)\,dx=\int_{N_P(F)}\int_{M(F)}\int_{N_{\overline{P}}(F)}f(ny\overline{n})e^{-2\langle H_P(y),\rho_P\rangle}\,d\overline{n}\,dy\,dn    
\end{equation}
pour tout $f\in L^1(G(F))$.

On pose 
\begin{equation}\label{appBeq:gammaPdef}
\gamma(P)=\gamma_M^G(P)\eqdef\int_{N_{\overline{P}}(F)} e^{2\langle H_P(\overline{n}),\rho_P\rangle}d\overline{n}>0.    
\end{equation}
Soit $L$ un sous-groupe de Levi de $G$ comprenant $M$, soit $Q$ un sous-groupe parabolique de $G$ ayant $L$ comme facteur de Levi, soit $R$ un sous-groupe  parabolique de $L$ ayant $M$ comme facteur de Levi. On sait que $RN_Q$ est un sous-groupe parabolique de $G$ ayant $M$ comme facteur de Levi. On définit les facteurs $\gamma_M^L(R)$ et $\gamma_L^G(Q)$ de la même façon. Grâce à nos normalisations des mesures le facteur $\gamma$ est transitif :
\begin{equation}\label{appBeq:gammatransitif}
\gamma_M^G(RN_Q)=\gamma_M^L(R)\gamma_L^G(Q).    
\end{equation}

De l'équation \eqref{appBeq:G=NMNGrp} on déduit que $\gamma(P)$ est aussi la dérivée de Radon-Nikodym de la mesure sur $G(F)$ par rapport à la mesure produit sur $M(F)N_P(F)K=G(F)$, i.e.
\begin{equation}\label{appBeq:Iwasawa}
\int_{G(F)}f(x)\,dx=\gamma(P)\int_{M(F)}\int_{N_P(F)}\int_{K}f(ynk)dk\,dn\,dy    
\end{equation}
pour tout $f\in L^1(G(F))$.

\begin{proposition}On note $\Gamma$ la fonction Gamma d'Euler.
\begin{enumerate}

    \item Si $F=\R$ et $D=\R$, alors 
    \[\gamma(P)=\pi^{\frac{1}{4}(m^2-\sum_{i=1}^lm_i^2)}\frac{\prod_{j=1}^l\prod_{i=1}^{m_j}\Gamma\left(\frac{i}{2}\right)}{\prod_{i=1}^m\Gamma\left(\frac{i}{2}\right)}.\]
    \item Si $F=\R$ et $D=\mathbb{H}$, alors 
    \[\gamma(P)=\pi^{(m^2-\sum_{i=1}^lm_i^2)}\frac{\prod_{j=1}^l\prod_{i=1}^{m_j}\Gamma\left(2i\right)}{\prod_{i=1}^m\Gamma\left(2i\right)}.\]
    \item Si $F=\C$, alors 
    \[\gamma(P)=(2\pi)^{\frac{1}{2}(m^2-\sum_{i=1}^lm_i^2)}\frac{\prod_{j=1}^l\prod_{i=1}^{m_j}\Gamma\left(i\right)}{\prod_{i=1}^m\Gamma\left(i\right)}.\]
    \item Si $F$ est non-archimédien, alors
    \[\gamma(P)=\mathcal{N}(\Delta_1)^{\frac{1}{4}(m^2-\sum_{i=1}^lm_i^2)}\frac{\prod_{j=1}^l\prod_{i=1}^{m_j}\left(1-q^{-id}\right)}{\prod_{i=1}^{m}\left(1-q^{-id}\right)}.\]
\end{enumerate}
\end{proposition}
\begin{remark}
Soit $n\in \mathbb{N}$, alors 
\[\Gamma(n+1)=n!\,\,\,\,\text{et}\,\,\,\,\Gamma\left(n+\frac{1}{2}\right)=\frac{(2n)!}{2^{2n}n!}\sqrt{\pi}.\]
\end{remark}
\begin{proof}On observe que $G(F)$ est un ouvert dense dans $\g(F)$. Grâce à la transitivité \eqref{appBeq:gammatransitif}, on peut et on va supposer que $M=M_0=G_1^m$ est le sous-groupe diagonal de $G$, et que $P=P_0$ est le groupe des matrices inversibles triangulaires supérieures (les résultat des calculs seront les mêmes pour tous les $P_0$). Dans la suite l'écriture $d\cdot^{Leb}$ signifie la mesure de Lebesgue sur $\R$.
\begin{enumerate}
    \item Si $F=\R$ et $D=\R$, alors on prend $f(x)=e^{-\pi\tau_\g(x{}^t\overline{x})}|\nu_m(x)|^{m}$ dans l'équation \eqref{appBeq:Iwasawa}. On a
    \[\int_{G(\R)}f(x)\,dx=\prod_{1\leq \alpha,\beta\leq m}\left(\int_{\R}e^{-\pi x_{\alpha,\beta}^2} \,dx_{\alpha,\beta}^{Leb}\right)=1.\]
    Puis, en notant 
    \[y=\begin{pmatrix}
    y_1 && \\
    &\ddots&\\
    &&y_m
    \end{pmatrix}\]
    un élément général de $M_0(\R)$, on voit que
    \begin{align*}
     \int_{M_0(\R)}\int_{N_{P_0}(\R)}\int_K f(ynk)\,dk\,dn\,dy&=\int_{M_0(\R)}\int_{N_{P_0}(\R)}f(yn)\,dn\,dy\\
     &=\prod_{i=1}^m\int_{\R}e^{-\pi y_i^2}|y_i|^{i-1}dy_{i}^{Leb}=\prod_{i=1}^m\frac{\Gamma(\frac{i}{2})}{\pi^{\frac{i}{2}}}\\
     &=\frac{\prod_{i=1}^m\Gamma\left(\frac{i}{2}\right)}{\pi^{\frac{m(m+1)}{4}}}. 
    \end{align*}
    Autrement dit
    \[\gamma(P_0)=\frac{\pi^{\frac{m(m+1)}{4}}}{\prod_{i=1}^m\Gamma\left(\frac{i}{2}\right)}.\]

    \item Si $F=\R$ et $D=\mathbb{H}$, alors on prend $f(x)=e^{-\frac{1}{2}\pi\tau_\g(x{}^t\overline{x})}|\nu_m(x)|^{2m}$ dans l'équation \eqref{appBeq:Iwasawa}. On a
    \[\int_{G(\R)}f(x)\,dx=\prod_{1\leq \alpha,\beta\leq m}\left(\int_{\R}\int_{\R}\int_{\R}\int_{\R}e^{-\pi (a_{\alpha,\beta}^2+b_{\alpha,\beta}^2+c_{\alpha,\beta}^2+d_{\alpha,\beta}^2)} \,da_{\alpha,\beta}^{Leb}\,db_{\alpha,\beta}^{Leb}\,dc_{\alpha,\beta}^{Leb}\,dd_{\alpha,\beta}^{Leb}\right)=1.\]
     Puis, en notant 
    \[y=\begin{pmatrix}
    a_1+b_1i+c_1j+d_1k && \\
    &\ddots&\\
    &&a_m+b_mi+c_mj+d_mk
    \end{pmatrix}\]
    un élément général de $M_0(\R)$, on voit que
    \begin{align*}
     \int_{M_0(\R)}&\int_{N_{P_0}(\R)}\int_K f(ynk)\,dk\,dn\,dy\\
     &=\int_{M_0(\R)}\int_{N_{P_0}(\R)}f(yn)\,dn\,dy\\
     &=\prod_{i=1}^m\int_{\R}\int_{\R}\int_{\R}\int_{\R}e^{-\pi (a_i^2+b_i^2+c_i^2+d_i^2)}(a_i^2+b_i^2+c_i^2+d_i^2)^{2(i-1)}\,da_{i}^{Leb}\,db_{i}^{Leb}\,dc_{i}^{Leb}\,dd_{i}^{Leb} \\
     &=\prod_{i=1}^m\frac{\Gamma(2i)}{\pi^{2(i-1)}}=\frac{\prod_{i=1}^m\Gamma(2i)}{\pi^{m(m-1)}}. 
    \end{align*}
    Autrement dit
    \[\gamma(P_0)=\frac{\pi^{m(m-1)}}{\prod_{i=1}^m\Gamma\left(2i\right)}.\]

    \item  Si $F=\C$, alors on prend $f(x)=e^{-\pi\tau_\g(x{}^t\overline{x})}|\nu_m(x)|^{m}$ dans l'équation \eqref{appBeq:Iwasawa}. On a
    \[\int_{G(\R)}f(x)\,dx=2^{m^2}\prod_{1\leq \alpha,\beta\leq m}\left(\int_{\R}\int_{\R}e^{-\pi (a_{\alpha,\beta}^2+b_{\alpha,\beta}^2)}\,\,da_{\alpha,\beta}^{Leb}\,db_{\alpha,\beta}^{Leb}\right)=2^{m^2}.\]
     Puis, en notant 
    \[y=\begin{pmatrix}
    a_1+b_1i && \\
    &\ddots&\\
    &&a_m+b_mi
    \end{pmatrix}\]
    un élément général de $M_0(\R)$, on voit que
    \begin{align*}
     \int_{M_0(\R)}&\int_{N_{P_0}(\R)}\int_K f(ynk)\,dk\,dn\,dy\\
     &=\int_{M_0(\R)}\int_{N_{P_0}(\R)}
f(yn)\,dn\,dy\\
     &=2^{\frac{m(m-1)}{2}}2^m\prod_{i=1}^m\int_{\R}\int_{\R}e^{-\pi (a_i^2+b_i^2)}(a_i^2+b_i^2)^{i-1}\,da_{i}^{Leb}\,db_{i}^{Leb}\\
     &=2^{\frac{m(m+1)}{2}}\prod_{i=1}^m\frac{\Gamma(i)}{\pi^{i-1}}=2^{\frac{m(m+1)}{2}}\frac{\prod_{i=1}^m\Gamma(i)}{\pi^{\frac{m(m-1)}{2}}}. 
    \end{align*}
    Autrement dit
    \[\gamma(P_0)=\frac{(2\pi)^{\frac{m(m-1)}{2}}}{\prod_{i=1}^m\Gamma\left(i\right)}.\]
    
    \item Si $F$ est non-archimédien, alors on prend $f=1_{K}$, la fonction indicatrice de $K$, dans l'équation \eqref{appBeq:Iwasawa}. Il est clair que si $y\in M_0(F)$, $n\in N_{P_0}(F)$ et $k\in K$, alors $ynk\in K$ si et seulement si $y\in M_0(F)\cap K=K_{1}\times\cdots\times K_{1}$ et $n\in N_{P_0}(F)\cap K=N_{P_0}(\O_F)$. En conséquence $\vol(K;G(F))=\gamma(P_0)\vol(M_0(F)\cap K;M_0(F))\vol(N_{P_0}(\O_F);N_{P_0}(F))$. Autrement dit 
    \begin{align*}
    \gamma(P_0)&=\vol(K;G(F))^{-1}\vol(M_0(F)\cap K;M_0(F))\vol(N_{P_0}(\O_F);N_{P_0}(F))  \\
    &=\mathcal{N}(\Delta_1)^{\frac{m(m-1)}{
4}}\frac{\left(1-q^{-d}\right)^m}{\prod_{i=1}^{m}\left(1-q^{-id}\right)}.
    \end{align*}
    Pour la deuxième égalité on utilise la sous-section \ref{appBsubsec:mesmnPLie} et la proposition \ref{appBprop:volKinG}.\qedhere
\end{enumerate}    
\end{proof}

\subsection{Mesure sur \texorpdfstring{$G_X(F)\backslash G(F)$}{GX(F)G(F)}}

Soit $X\in \g(F)$. L'objectif de cette partie est de définir une mesure $dg_X$ sur l'espace homogène $G_X(F)\backslash G(F)$, ou de façon équivalente, sur l'orbite de $X$ pour la $G(F)$-conjugaison.

Si $X$ est semi-simple, alors $G_X$ est le groupe des unités d'une algèbre séparable sur $F$. On peut ainsi suivre les mêmes démarches qu'au numéro \ref{appBsubsec:mesuresonMNP} : on pose d'abord sur $\g_X(F)$ la mesure de Haar auto-duale relativement à la transformée de Fourier défini par le caractère $\psi$ et la forme bilinéaire associée à la trace réduite de l'algèbre séparable sous-jacente. Puis on prend sur $G_X(F)$ la mesure de Haar induite via l'exponentielle. On munit enfin $G_X(F)\backslash G(F)$ de la mesure quotient.

Soit maintenant $X\in \g(F)$ quelconque. On écrit $X_\ss$ pour la partie semi-simple de $X$. On définit $D^{\g}(X)=\det(\text{ad}(X_\ss); \g/\g_{X_{\ss}})$. On prend d'abord $L$ un sous-groupe de Levi tel que $X_\ss \in\mathfrak{l}(F)$ soit elliptique et $X\in \Ind_L^G(X_\ss)$ (proposition \ref{prop:suppellitique}). On munit ensuite $L_{X_\ss}(F)$ d'une mesure de Haar suivant les démarches ci-dessus. Puis on prend sur $G_{X}(F)$ l'unique mesure de Haar telle que 
\begin{equation}\label{eqappB:defmeasureonanyorb} 
\begin{split}
|D^\g(X)|_F^{1/2}&\int_{G_{X}(F)\backslash G(F)}f(g^{-1}Xg)\,dg_X\\
&=\lim_{\substack{A\to 0\\A\in\mathfrak{a}_L(F)=\text{Lie}(A_L)(F):L_{A+X_\ss}=G_{A+X_\ss}}}|D^\g(A+X)|_F^{1/2}\int_{G_{A+X_\ss}(F)\backslash G(F)}f(g^{-1}(A+X_\ss)g)\,dg_{A+X_{\ss}}   \\  
\end{split}    
\end{equation}
pour tout $f\in C_c^\infty(G(F))$. On a noté $dg_X$ dans le membre de gauche la mesure quotient sur $G_{X}(F)\backslash G(F)$, et $dg_{A+X_\ss}$ dans le membre de droite la mesure quotient sur $G_{A+X_\ss}(F)\backslash G(F)=L_{X_\ss}(F)\backslash G(F)$. 

\begin{proposition}[{{\cite{YDL23a}}}]On a les compatibilités suivantes :    
\begin{enumerate}
    \item la limite du membre de droite de l'égalité \eqref{eqappB:defmeasureonanyorb} existe, et elle ne dépend pas du choix de $L$ ;
    \item si $X$ est semi-simple, alors la mesure sur $G_X(F)$ définie par l'équation \eqref{eqappB:defmeasureonanyorb} coïncide avec la mesure définie plus haute ;
    \item si $H$ est un sous-groupe de Levi de $G$ dont l'algèbre de Lie contient $X$ tel que $H_X=G_X$, alors la mesure sur $G_X(F)$ définie par l'équation \eqref{eqappB:defmeasureonanyorb} en voyant $X$ comme élément de $\g(F)$ coïncide avec la mesure sur $H_X$ définie par la même équation en voyant $X$ comme élément de $\mathfrak{h}(F)$ ;
    \item la mesure sur $G_X(F)$ définie par l'équation \eqref{eqappB:defmeasureonanyorb} en voyant $X$ comme élément de $\g(F)$ coïncide avec la mesure sur $(G_{X_\ss})_{X_\nilp}(F)$ définie par la même équation en voyant $X_\nilp$ comme élément de $\g_{X_\ss}(F)$ ;
    \item la mesure sur $G_{X+Z}(F)$ coïncide avec la mesure sur $G_X(F)$ pour tous $X\in \g(F)$ et $Z$ dans le groupe des $F$-points du centre de $G$.
\end{enumerate}
\end{proposition}

Reprenons les notations de l'équation \eqref{eqappB:defmeasureonanyorb}. Nous disposons également de la formule suivante, sans faire référence à la limite, avec $Q$ un sous-groupe parabolique de $G$ ayant $L$ comme facteur de Levi :
\begin{align*}
|D^\g(X)|_F^{1/2}&\int_{G_{X}(F)\backslash G(F)}f(g^{-1}Xg)\,dg_X\\
&=\gamma(Q)\lim_{\substack{A\to 0\\A\in\mathfrak{a}_L(F):L_{A+X_\ss}=G_{A+X_\ss}}}|D^\mathfrak{l}(X_\ss)|_F^{1/2}\int_{L_{X_\ss}(F)\backslash L(F)}\int_{\n_Q(F)}\int_K \\
&\hspace{5cm} f\left(k^{-1}\left(l^{-1}(A+X_\ss)l+U\right)k\right)\,dk\,dU\,dl_{A+X_\ss}   \\
&=\gamma(Q)|D^\mathfrak{l}(X_\ss)|_F^{1/2}\int_{L_{X_\ss}(F)\backslash L(F)}\int_{\n_Q(F)}\int_K f\left(k^{-1}\left(l^{-1}X_\ss l+U\right)k\right)\,dk\,dU\,dl_{X_\ss}. 
\end{align*}
En effet, la première égalité vient de l'équation \eqref{eqappB:defmeasureonanyorb}, la décomposition d'Iwasawa, et le changement de variables $n^{-1}l^{-1}(A+X_\ss)ln=l^{-1}(A+X_\ss)l+U$ avec $n\in N_{Q}(F)$ et $U\in \mathfrak{n}_Q(F)$. Ce changement de variables induit un isomorphisme analytique $N_{Q}(F)\xrightarrow{\sim}\mathfrak{n}_Q(F)$ de Jacobien $|\det(\ad (A+X_\ss);\mathfrak{n}_Q(F))|_F^{-1}=|D^\g(A+X)|_F^{-1/2}|D^\l(A+X_\ss)|_F^{1/2}=|D^\g(A+X)|_F^{-1/2}|D^\l(X_\ss)|_F^{1/2}$. Puis la deuxième égalité vient du théorème de convergence dominée.

Grâce à la transitivité \eqref{appBeq:gammatransitif} du facteur $\gamma$ et de l'orbite induite (proposition \ref{prop:indprop}) on dispose d'une formule d'intégration plus générale : soient $L$ un sous-groupe de Levi, $Q$ un sous-groupe parabolique de $G$ ayant $L$ comme facteur de Levi, $Y\in \mathfrak{l}(F)$, et $X\in \Ind_L^G(Y)(F)\cap ((\Ad L(F))Y+\n_Q(F))$. Alors
\begin{align*}
|D^{\mathfrak{g}}(X)|_F^{1/2}&\int_{G_X(F)\backslash G(F)}f\left(g^{-1}Xg\right)\,dg_X \\
&=\gamma(Q)|D^{\mathfrak{l}}(Y)|_F^{1/2}\int_{L_{Y}(F)\backslash L(F)}\int_{\n_Q(F)}\int_{K}f\left(k^{-1}\left(l^{-1}Yl+U\right)k\right)\,dk\,dU\,dl_Y      
\end{align*}
pour tout $f\in C_c^\infty(G(F))$. 

Il existe une autre mesure naturelle sur une orbite pour la $G(F)$-conjugaison, dérivée de la forme bilinéaire canonique sur $\g(F)$ et la structure symplectique de l'orbite. Cependant, cette définition n'est pas adaptée à la formule des traces. Nous estimons qu'il pourrait être utile d'en expliquer les détails, ce que nous allons faire à présent. On définit $\g^\vee$ l'espace dual de $\g$. Fixons $X^\vee\in \g^\vee$. On note $\Ad^\vee$ la représentation co-adjointe de $G$ sur $\g^\vee$. On note $\O_{X^\vee}\eqdef \Ad^\vee(G)X^\vee$. La forme symplectique de Kirillov-Kostant-Souriau $\omega_{\O_{X^\vee}}$ sur $\O_{X^\vee}$ est l'unique structure symplectique telle que l'inclusion $\O_{X^\vee}\hookrightarrow \g^\vee$ soit un morphisme de Poisson, avec $\g^\vee$ muni de sa structure de Lie-Poisson (\cite[proposition 7.7]{LGPV12}). Nous rappelons la définition de $\omega_{\O_{X^\vee}}$. On note $G_{X^\vee}$ le centralisateur de $X^\vee$ dans $G$ pour l'action co-adjointe, puis $\g_{X^\vee}$ son algèbre de Lie. Pour une variété $V$ et un point $v\in V$, on note $T_vV$ l'espace tangent en $v$ de $V$. On a une application $g\in G\mapsto \Ad^\vee(g)X^\vee \in \g^\vee$. Le différentiel en $1\in G$ de cette application nous donne un isomorphisme $\g_{X^\vee}\backslash \g\xrightarrow{\sim}T_{X^\vee}\O_{X^\vee}$. On note 
\[\Phi_{X^\vee}:\g_{X^\vee}\backslash \g\to \g^\vee\]
la composée de $\g_{X^\vee}\backslash \g\xrightarrow{\sim}T_{X^\vee}\O_{X^\vee}$ avec $T_{X^\vee}\O_{X^\vee}\hookrightarrow T_{X^\vee}\g^\vee=\g^\vee$.

On note $\text{ad}^\vee$ la représentation co-adjointe de $\g$ sur $\g^\vee$. Posons $\omega_{X^\vee}$ la forme bilinéaire alternée sur $\g$ définie par
\[\omega_{X^\vee}(Y,Z)\eqdef -\langle \text{ad}^\vee(Y)X^\vee,Z\rangle. \]
Son noyau est $\g_{X^\vee}$. La forme $\omega_{X^\vee}$ peut donc être vue comme une forme bilinéaire alternée non-dégénérée sur $\g_{X^\vee}\backslash \g$, i.e. l'espace tangent en $X^\vee$ de l'orbite co-adjointe $\O_{X^\vee}$. On obtient ainsi une forme $\omega$ sur $\O_{X^\vee}$ : elle vaut $\omega_{X^{\vee}{}'}$ en $X^{\vee}{}'\in \O_{X^\vee}$, et c'est une 2-forme $G$-invariante. En particulier, la variété $\O_{X^\vee}$ est symplectique. Elle est donc de dimension paire. On obtient de ce fait $\omega_{\O_{X^\vee}}\eqdef\bigwedge^{(\dim \O_{X^\vee})/2}\omega$ une forme volume algébrique $G$-invariante sur $\O_{X^\vee}$.

Identifions à présent $\g$ et son dual $\g^\vee$ via la forme bilinéaire canonique. Soit $X\in \g(F)$. On obtient une forme volume algébrique $G$-invariante sur $G_X\backslash G$. Or pour toute variété $V$ muni d'une forme volume algébrique et toute mesure de Haar sur $F$, on a une forme volume analytique canoniquement associée sur $V(F)$ (\cite[section 10.1]{Bou07}). Prenons maintenant la mesure de Haar auto-duale relativement à la transformée de Fourier sur $F=\gl_{1,F}(F)$. Nous obtenons une forme volume analytique sur $G_X(F)\backslash G(F)$. Notons $dg_X'$ la mesure obtenue.

Cette mesure $dg_X'$ n'est pas appropriée pour définir une intégrale orbitale pondérée. En effet, la limite du membre de droite de l'équation \eqref{eqappB:defmeasureonanyorb} n'existera pas, lorsque $G_{A+{X_\ss}}(F)\backslash G(F)$ est muni de la mesure $dg_{A+X_\ss}'$. Considérons l'exemple suivant pour illustrer ce propos : $G=\GL_{2,F}$ et $X=\begin{pmatrix}
0 & 1 \\
0 & 0
\end{pmatrix}\in \g(F)$. Dans ce cas on peut prendre $L=\GL_{1,F}\times \GL_{1,F}$ le sous-groupe diagonal. On note $E_{\alpha,\beta}$ la matrice qui vaut 1 à l'intersection de la $\alpha$-ième ligne et la $\beta$-ième colonne, et qui vaut 0 ailleurs. On note $(E_{\alpha,\beta}^\vee)_{1\leq \alpha,\beta\leq 2}$ la base duale pour la forme bilinéaire canonique sur $\g(F)$. Plus précisément $E_{\alpha,\beta}^\vee=E_{\beta,\alpha}$. On note
\[A=\begin{pmatrix}
A_1 & 0 \\
0 & A_2
\end{pmatrix}\]
un élément général de $\{A\in \mathfrak{a}_L(F)
\mid L_A=G_A\}$, avec $A_1,A_2\in F$ et $A_1\not= A_2$. On a $A^\vee=A$ et $\dim \O_{A^\vee}=2$. Alors par des calculs directs on voit que $\Phi_{A^\vee}:\g_{A^\vee}\backslash \g\to \g^\vee$ est l'application 
\[\Phi_{A^\vee}:\begin{pmatrix}
Y_{1,1} & Y_{1,2} \\
Y_{2,1} & Y_{2,2}
\end{pmatrix}\mapsto (A_1-A_2)Y_{1,2}E_{2,1}^\vee+(-A_1+A_2)Y_{2,1}E_{1,2}^\vee,\,\,\,\,Y_{\alpha,\beta}\in F.\]
Puis 
\[\omega_{A^\vee}=(A_1-A_2)^{-1}dE_{1,2}^\vee\wedge dE_{2,1}^\vee.\]
D'où 
\[\omega_{A}=-(A_1-A_2)^{-1}dE_{1,2}\wedge dE_{2,1}.\]
En parallèle, notre mesure $dg_A$ sur $G_{A}(F)\backslash G(F)$ décrite par l'équation \eqref{eqappB:defmeasureonanyorb} correspond à la forme volume $dE_{1,2}\wedge dE_{2,1}$, car notre mesure sur $G(F)$ (sous-section \ref{appBsubsec:mesmnPLie}) correspond à la forme volume $dE_{1,1}\wedge dE_{1,2}\wedge dE_{2,1}\wedge dE_{2,2}$ et celle sur $G_{A}(F)$ correspond à la forme volume $dE_{1,1}\wedge dE_{2,2}$. Ainsi $dg_A'=|A_1-A_2|_F^{-1}dg_A$ et
\[\lim_{\substack{A\to 0\\A\in\mathfrak{a}_L(F):L_{A}=G_{A}}}|D^\g(A)|_F^{1/2}\int_{G_{A}(F)\backslash G(F)}f(g^{-1}Ag)\,dg_{A}'\]
n'existe pas.

\subsection{Définition de Gross}
Gross a défini une mesure de Haar « canonique », notée $|\omega_G|$, sur $G(F)$ dans \cite{Gro97}. Notons $dx$ notre mesure sur $G(F)$ décrite par l'équation \eqref{appBeq:mesureonG}. On a
\begin{enumerate}
    \item Si $F$ est non-archimédien, alors $dx=|\omega_G|$.
    \item Si $F=\R$, alors $dx=|\omega_G|$.
    \item Si $F=\C$, alors $dx=2^{m^2}|\omega_G|$.
\end{enumerate}

Pour $F$ non-archimédien et $D=F$ cela résulte de la définition (\cite[section 4]{Gro97} et point 3 de la proposition \ref{appBprop:autodualmeasureg}), puis pour $D$ quelconque cela résulte du fait que notre construction de $dx$ est respectée par un torseur intérieur, de même pour $|\omega_G|$. Justifions ensuite les égalités pour $F$ archimédien. Soit $V_\C$ un $\C$-espace vectoriel. Soit $V$ une $\R$-forme de $V_\C$, cela signifie que $V$ est un sous-$\R$-espace vectoriel vérifiant $V_\C=V+i_\C V$. Ici $i_{\C}$ est l'unité imaginaire dans $\C=\R\oplus i_{\C}\R$. On suppose que $V$ est muni d'une mesure de Haar $dv$. On note  $dv\otimes_\R\C$ l'unique mesure de Haar sur $V_\C$ telle que pour tout $\Z$-réseau $\mathcal{R}$ de $R$ on ait 
\[\vol(\mathcal{R}\backslash V)^2=\vol((\mathcal{R}\oplus i_\C\mathcal{R})\backslash V_\C).\]
Ici on prend sur $\mathcal{R}\backslash V$ (resp. $(\mathcal{R}\oplus i_\C\mathcal{R})\backslash V_\C$) la mesure quotient de $dv$ (resp. $dv\otimes_\R\C$) par la mesure de comptage sur $\mathcal{R}$ (resp. $\mathcal{R}\oplus i_\C\mathcal{R}$). Dans la suite on va identifier tout groupe algébrique défini sur $\C$ à son groupe des $\C$-points. On se réfère à \cite[sections 7 et 11]{Gro97} pour la construction de $|\omega_G|$.

Supposons maintenant que $F=\R$ et $D=\R$. On constate que la forme compacte de $G$ est $U(m)$, le groupe unitaire sur $\R$ associé à la forme hermitienne canonique de $\C^m$. Pour toute $\R$-algèbre $A$ on a $U(m)(A)=\{g\in \GL_m(A\otimes_\R\C)\mid g\cdot{}^t\overline{g}=\text{Id}\}$. On a noté $X\mapsto \overline{X}$ l'involution principale et $X\mapsto {}^tX$ la transposée sur $\gl_m(A\otimes_\R\C)$. Donc si $X=(X_{\alpha,\beta})_{1\leq \alpha,\beta\leq m}\in \gl_m(A\otimes_\R\C)$ alors ${}^tX=(X_{\beta,\alpha})_{1\leq \alpha,\beta\leq m}$ ; si $X=X_1+i_\C X_2\in\gl_m(A\otimes_\R\C)=\gl_{m}(A)\oplus i_\C \gl_{m}(A)$ avec $X_1,X_2\in \gl_m(A)$ alors $\overline{X}=X_1-i_\C X_2$. On va noter, pus l'abus de notation, le nombre $i_\C$  par $i_{U(m)}$. 

On a un isomorphisme de groupes $G\otimes_\R\C\xrightarrow{\sim} U(m)\otimes_\R\C$ dont le différentiel vaut
\begin{align*}
\psi: \g\otimes_\R \C&\xlongrightarrow{\sim} \mathfrak{u}(m)\otimes_\R\C \\
X &\longmapsto \left(\frac{X-{}^t\overline{X}}{2}\right)+i_{\C}\left(\frac{X+{}^t\overline{X}}{2i_{U(m)}}\right).
\end{align*}
Pour rappel $i_{\C}$ est l'unité imaginaire dans $\C=\R\oplus i_{\C}\R$ du produit tensoriel $\otimes_\R\C$, et $i_{U(m)}$ est le nombre « $i_{\C}$ » dans la définition de $U(m)$. On a identifié $\g\otimes_\R \C$ (resp. $\mathfrak{u}(m)\otimes_\R\C$) à son groupe des $\C$-points, et l'involution principale $\overline{X}$ ainsi que la transposée ${}^t\overline{X}$ sont effectuées dans le groupe $\mathfrak{u}(m)(\C)$. 

Soit $T=U(1)^m\subseteq U(m)$ le sous groupe diagonal de $U(m)$, c'est un sous-groupe de Cartan. On note $M_0=G_1^m\subseteq G$ le sous-groupe diagonal de $G$, c'est un sous-groupe de Cartan. On note $K(\m_0)$ (resp. $K(\mathfrak{t})$) le noyau de $\exp: \m_0\to M_0$ (resp. $\exp:\mathfrak{t}\to T$). Alors $\psi(K(\m_0))=i_{\C} K(\mathfrak{t})$.

L'ensemble des racines $\Sigma(\g\otimes_{\R}\C;M_0\otimes_\R\C)$ s'identifie naturellement à $\{(\alpha,\beta)\mid 1\leq \alpha,\beta\leq m, \alpha\not=\beta\}$. Soit $\gamma\in \Sigma(\g\otimes_{\R}\C;M_0\otimes_\R\C)$. On suppose qu'il correspond à $(\alpha,\beta)$. On note $X_\gamma\in \g\otimes_\R\C$ la matrice suivante (\cite[p.195 (IX)]{Bou05}) : si $\alpha<\beta$ alors $X_\gamma$ vaut 1 à l'intersection de la $\alpha$-ième ligne et la $\beta$-ième colonne, et vaut 0 ailleurs ; si $\alpha>\beta$ alors $X_\gamma$ vaut $-1$ à l'intersection de la $\alpha$-ième ligne et la $\beta$-ième colonne, et vaut 0 ailleurs. L'isomorphisme $\phi$ induit une bijection $\psi:\Sigma(\g\otimes_{\R}\C;M_0\otimes_\R\C)\xrightarrow{\sim}\Sigma(\mathfrak{u}(m)\otimes_{\R}\C;T\otimes_\R\C)$. Pour $\delta\in \Sigma(\mathfrak{u}(m)\otimes_{\R}\C;T\otimes_\R\C)$ on pose $Y_\delta:=\psi(X_{\psi^{-1}(\delta)})$. Par des calcals directs on voit que $(Y_{\delta})_{\delta\in \Sigma(\mathfrak{u}(m)\otimes_{\R}\C;T\otimes_\R\C)}$ est un système de Chevalley de $\mathfrak{u}(m)\otimes_\R\C$, tel que $Y_{\delta}$
 et $Y_{-\delta}$ soient conjugués (\cite[p.296]{Bou05}).

On note $\mathcal{R}_{1}$ le sous-groupe additif de $\g\otimes_\R\C$ engendré par :
\[\frac{i_{\C}}{2\pi} K(\m_0),\,\,\,\,X_\gamma+X_{-\gamma},\,\,\,\, i_{\C}(X_\gamma-X_{-\gamma}),\,\,\,\,\gamma\text{ parcourt } \Sigma(\g\otimes_{\R}\C;M_0\otimes_\R\C).\]
Alors par des calculs directs on voit que  $\mathcal{R}_{1}$ est un $\Z$-réseau de $\mathcal{R}_{1}\otimes_\Z\R$, et $\mathcal{R}_{1}\otimes_\Z\R$ est une $\R$-forme de $\g\otimes_\R\C$. Notons $dX$ notre mesure de Haar sur $\g(\R)$ décrite dans la proposition \ref{appBprop:autodualmeasureg}. Elle induit $dx$ via l'exponentielle sur un voisinage de $0$ de $\mathfrak{g}(\R)$. Par des calculs directs on voit que 
\[\vol(\mathcal{R}_{1}\oplus i_\C \mathcal{R}_{1}\backslash \g\otimes_\R\C)=2^{m^2-m}\]
pour la mesure quotient de $dX\otimes_\R\C$ par la mesure de comptage.

De même on note $\mathcal{R}_{2}$ le sous-groupe additif de $\mathfrak{u}(m)\otimes_\R\C$ engendré par :
\[\frac{1}{2\pi}K(\mathfrak{t}),\,\,\,\,Y_\delta+Y_{-\delta},\,\,\,\, i_{\C}(Y_\delta-Y_{-\delta}),\,\,\,\,\delta\text{ parcourt } \Sigma(\mathfrak{u}(m)\otimes_{\R}\C;T\otimes_\R\C).\]
Alors $\mathcal{R}_{2}\otimes_\Z\R=\mathfrak{u}(m)$ (\cite[p.369, 9) a)]{Bou05}), et $\mathcal{R}_{2}$ est un $\Z$-réseau de $\mathfrak{u}(m)$. Notons $|\omega_{\mathfrak{u}(m)}|$ la mesure de Haar sur $\mathfrak{u}(m)(\R)$ qui induit $|\omega_{U(m)}|$ via l'exponentielle sur un voisinage de $0$ de $\mathfrak{u}(m)(\R)$. On a, selon la définition de $|\omega_\mathfrak{u}(m)|$ de Gross (\cite[section 7]{Gro97}) et \cite[p.369, 9) c)]{Bou05}, 
\[\vol(\mathcal{R}_{2}\oplus i_\C \mathcal{R}_{2}\backslash \mathfrak{u}(m)\otimes_\R\C)=2^{m^2-m}\]
pour la mesure quotient de $|\omega_{\mathfrak{u}(m)}|\otimes_\R\C$ par la mesure de comptage. Or, selon la construction $|\omega_\g|=\psi^\ast|\omega_{\mathfrak{u}(m)}|$ (\cite[section 11]{Gro97}), on a donc $dX=|\omega_\g|$. D'où $dx=|\omega_G|$.

Pour $F=\R$ et $D=\mathbb{H}$ on a $dx=|\omega_G|$ car notre construction de $dx$ est respectée par un torseur intérieur, de même pour $|\omega_G|$. Pour $F=\C$, on constate que la forme compacte de $\Res_{\C/\R}G$ est $U(m)\times U(m)$, et le reste des calculs est similaire à ceux dans le cas de $F=\R$ et $D=\R$.

\nocite{*}
\bibliographystyle{alpha}
\bibliography{main.bib}

\end{document}